%% file: uctest1_revised_.tex
\documentclass[oneside,final]{ucr}
\usepackage{amssymb, amsthm}
\usepackage{bm}
\usepackage{amsmath}
\usepackage{mathrsfs}
\usepackage{graphics}
\usepackage{subfigure}
\usepackage{flafter}
\usepackage{sw20uctd}

\usepackage{ifpdf}
\ifpdf
\usepackage[pdftex]{graphicx}
\else
\usepackage[dvips]{graphicx}
\fi
\usepackage{tikz}
 	 \usetikzlibrary{arrows,backgrounds,snakes} 
\usepackage[all]{xy}

\usepackage[pdftex,plainpages=false,hypertexnames=false,pdfpagelabels]{hyperref}

\newcommand{\C}{\mathbb{C}}
\newcommand{\R}{\mathbb{R}}

\newcommand{\proj}[2]{\text{Proj}_{#1}(#2)}
\newcommand{\Norm}[1]{\left\Vert #1 \right\Vert}
\newcommand{\ip}[1]{\left\langle #1 \right\rangle}
\newcommand{\id}{\boldsymbol{1}}
 \newcommand{\norm}[1]{\left\vert #1 \right\vert}
  \newcommand{\Hom}[1]{\operatorname{Hom}(#1)}
  \newcommand{\tr}[1]{\operatorname{tr}(#1)}
  \newcommand{\lp}[1]{\mathbf{#1}}
   \newcommand{\pth}[1]{\mathbf{#1}}
  \newcommand{\JW}[1]{f^{(#1)}}
\newcommand{\upsidedown}[1]{\begin{scope}[y=-1cm] #1 \end{scope}}
\newcommand{\Mat}[1]{\operatorname{\mathbf{Mat}}\left(#1\right)}
\newcommand{\cC}{\mathcal{C}}

\newtheorem{theorem}{Theorem}[section]
\newtheorem{Defn}[theorem]{Definition}
\newtheorem{prop}[theorem]{Proposition}
\newtheorem{cor}[theorem]{Corollary}
\newtheorem{lem}[theorem]{Lemma}

\newtheorem{claim}[theorem]{Claim}

\newtheorem{example}[theorem]{Example}

\newtheorem{notation}[theorem]{Notation}

\newtheorem{remark}[theorem]{Remark}

\newtheorem{fact} [theorem] {Fact}
\def\dsp{\def\baselinestretch{2.0}\large\normalsize}
\dsp
\begin{document}
\input{TikzStyles}


\title{A Construction of the ``2221'' Planar Algebra}
\author{Richard Han}
\degreemonth{December}
\degreeyear{2010}
\degree{Doctor of Philosophy}
\chair{Professor Feng Xu}
\othermembers{Professor Marta Asaeda\\
Professor David Rush}
\numberofmembers{3}
\field{Mathematics}
\campus{Riverside}

\maketitle
\copyrightpage{}
\approvalpage{}

\degreesemester{Fall}

\begin{frontmatter}

\begin{acknowledgements}
I would like to thank my committee.  I thank Dr. Xu for his guidance of my overall progress.  I would like to thank Dr. Asaeda for her support as well.  I am grateful to Emily Peters, Noah Snyder, and Stephen Bigelow for their generous help.  I am also grateful to Jason Wong and Jonathan Sarhad for their technical help on preparing the manuscript.
\end{acknowledgements}

\begin{dedication}
\null\vfil
{\large
\begin{center}
To my parents Peter and Kathy Han.
\end{center}}
\vfil\null
\end{dedication}

\begin{abstract}
In this paper, we construct the ``2221'' subfactor planar algebra  by finding it as a subalgebra of the graph planar algebra of its principal graph.  In particular, we give a presentation of the ``2221'' subfactor planar algebra consisting of generators and relations.  As a corollary, we have a planar algebra proof of the existence of a subfactor with principal graph ``2221''.  To show the subfactor property, we use the jellyfish algorithm for evaluating closed diagrams.  Lastly, we show uniqueness up to conjugation of ``2221''.
\end{abstract}

\tableofcontents
\listoffigures

\end{frontmatter}

\chapter{Introduction}
In this paper, we investigate the ``2221'' subfactor planar algebra which has Perron-Frobenius eigenvalue $\delta=\sqrt{\frac{5+\sqrt{21}}{2}}$.  In a preprint dated 2001 \cite{Xu} on conformal inclusions, Xu had constructed a subfactor which has ``2221'' as its principal graph. Around the same time, M. Izumi \cite{Iz} had also constructed it using different methods.  For further details, see the appendix (written by Ostrik) of \cite{Cal}.\\
\indent  In this paper, we give a planar algebra proof of its existence.  We construct the planar algebra as a subalgebra of the graph planar algebra of its principal graph 
$$H=
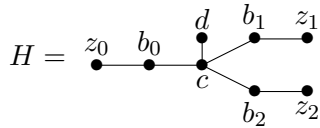
\begin{tikzpicture}[baseline,scale=.7]
\node at (0,0) {$\bullet$};
\node at (0,0) [above] {$z_0$};
\node at (1,0) {$\bullet$};
\node at (1,0) [above] {$b_0$};
\node at (2,0) {$\bullet$};
\node at (2,0) [below] {$c$};
\node at (2,.5) {$\bullet$};
\node at (2,.5) [above] {$d$};
\node at (3,.5) {$\bullet$};
\node at (3,.5) [above] {$b_1$};
\draw (0,0)--(2,0);
\draw (2,0)--(2,.5);
\node at (4,.5) {$\bullet$};
\node at (4,.5) [above] {$z_1$};
\node at (3,-.5) {$\bullet$};
\node at (3,-.5) [below] {$b_2$};
\node at (4,-.5) {$\bullet$};
\node at (4,-.5) [below] {$z_2$};
\draw (2,0)--(3,.5)--(4,.5);
\draw (2,0)--(3,-.5)--(4,-.5);
\end{tikzpicture}$$  
We follow the method which Peters \cite{EP} carries out for the Haagerup planar algebra.\\  
\indent The main theorem of this paper is Theorem 4.0.7, which gives a presentation of the ``2221'' planar algebra consisting of the two generators and some quadratic relations.  As a corollary, this proves the existence of a subfactor which has principal graph ``2221''.  One interesting difference between the ``2221'' planar algebra and some of the planar algebras previously constructed, such as the Haagerup planar algebra, is that there are two generators instead of one.  An interesting feature about this paper is that we use the Jellyfish algorithm, first introduced in \cite{ExtH}, to show the subfactor property of the ``2221'' planar algebra.  It is further used to prove other facts about our planar algebra at higher-level n-box spaces.  \\
\indent In chapter 2, the requisite background on planar algebras is given, with many definitions adopted from \cite{EP} for the sake of convenience to the reader.  In chapter 3, we look at the decomposition of the ``2221'' planar algebra into irreducible modules, which will tell us what sorts of quadratic relations to expect.  In chapter 4, we look for and identify the two generators $T$ and $Q$ of our planar algebra.  In chapter 5, we prove several quadratic relations involving the two generators.  In chapter 6, we obtain some important one-strand braiding substitutes, and in chapter 7 we prove the subfactor property.  In chapter 8, we consider the tensor category associated with the ``2221'' planar algebra and show that the planar algebra has principal (and dual principal) graph $H$.  Lastly, in chapter 9, we prove uniqueness up to conjugation of ``2221''.

\chapter{Background on Planar Algebras}  All of the definitions in this chapter are adopted from \cite{EP} and are repeated here for the convenience of the reader.  However, the original sources of all the definitions are as follows:  the original source of the definitions in sections 2.0, 2.1, and 2.2 is \cite{VJ1}, and the original sources of the definitions in section 2.3 are \cite{VJ5} and \cite{GR}.\\
\indent A planar algebra is a collection of vector spaces which has an action by the shaded planar operad.  For more details, see \cite{JP}). We look at the objects which act on the vectors of our vector spaces.

\begin{Defn}[\cite{VJ1},\cite{EP}]
Elements of the shaded planar operad are {\em shaded planar tangles}, which consist of 
\begin{itemize} 
\item
one outer disk $D_o$
\item
$k$ disjoint inner disks $D_i$ 
\item 
A set $S$ of non-intersecting strings between the disks, such that $| S \cap D_o |, |S \cap D_i | \in 2 \mathbb{Z}$
\item
A checkerboard shading on the regions (i.e., connected components of $(D_o \setminus  \bigcup_{i} D_i ) \setminus S$).
\item
A star in a region near the boundary of each disk $D_o, D_i$
\end{itemize}
See Figure \ref{tangleEG} for an example.  We consider two shaded planar tangles equal if they are isotopic.  

\begin{figure}[!ht]
$$\begin{tikzpicture}[scale=.7]
	\clip (0,0) circle (3cm);
	
	\begin{scope}[shift=(10:1cm)]	
		\draw[shaded] (0,0)--(0:6cm)--(90:6cm)--(0,0);	
		\draw[shaded] (0,0) .. controls ++(180:2cm) and ++(-90:2cm) .. (0,0);
	\end{scope}
	
	\draw[shaded] (-150:1cm) -- (120:4cm) -- (180:4cm) -- (-150:1cm);
	\draw[shaded] (-150:1cm) -- (-120:4cm) -- (-60:4cm) -- (-150:1cm);
	
	\begin{scope}[shift=(10:1cm)]	
		\node at (0,0) [Tbox, inner sep=2mm] {};
		\node at (90:1.5cm) [Tbox, inner sep=2mm] {};
		\node at (-45:.7cm) {$\star$};
		\node at (120:1.6cm) {$\star$};
	\end{scope}
	\node at (-150:1cm) [Tbox, inner sep=2mm] {};
	\node at (-120:1.6cm) {$\star$};
	\node at (-30:2.7cm) {$\star$};
	
	\draw[very thick] (0,0) circle (3cm);
\end{tikzpicture}$$
\caption{A shaded planar tangle}\label{tangleEG}
\end{figure}
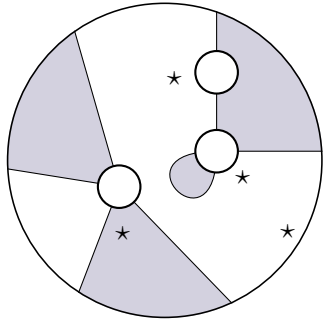

By the {\em type} of a disk in a tangle, we mean the pair $(k,\pm)$ determined by half the number of strings on the boundary of the disk ($k$) and whether the star is in an unshaded ($+$) or shaded ($-$) region.  For example, the inner disks in the above figure are of types $(1,+)$, $(2,+)$ and $(2,-)$. 

The operadic structure is given by composition.  Let $A$, $B$ be shaded planar tangles; $A \circ_i B$ exists if the $i^{th}$ inner disk of $A$ has the type as the outer disk of $B$.  In this case, $A \circ_i B$ is the shaded planar tangle built by inserting $B$ in the $i^{th}$ inner disk of $A$ (with the starred regions matching up) and connecting the strings.  For example, 
$$
\begin{tikzpicture}[PAdefn]
	\clip (0,0) circle (3cm);
	
	\begin{scope}[shift=(10:1cm)]	
		\draw[shaded] (0,0)--(0:6cm)--(90:6cm)--(0,0);	
		\draw[shaded] (0,0) .. controls ++(180:2cm) and ++(-90:2cm) .. (0,0);
	\end{scope}
	
	\draw[shaded] (-150:1cm) -- (120:4cm) -- (180:4cm) -- (-150:1cm);
	\draw[shaded] (-150:1cm) -- (-120:4cm) -- (-60:4cm) -- (-150:1cm);
	
	\begin{scope}[shift=(10:1cm)]	
		\node at (0,0) [Tbox, inner sep=1.4mm] {\tiny{\textcolor{gray}{2}}};
		\node at (90:1.5cm) [Tbox, inner sep=1.4mm] {\tiny{\textcolor{gray}{1}}};
		\node at (-45:.7cm) {$\star$};
		\node at (120:1.6cm) {$\star$};
	\end{scope}
	\node at (-150:1cm) [Tbox, inner sep=1.4mm] {\tiny{\textcolor{gray}{3}}};
	\node at (-120:1.6cm) {$\star$};
	\node at (-30:2.7cm) {$\star$};
	
	\draw[very thick] (0,0) circle (3cm);

\end{tikzpicture}
\,
\circ_{2}
\,
\begin{tikzpicture}[PAdefn]
	\clip (0,0) circle (2cm);

	\draw[shaded] (0:4cm)--(0,0)--(90:4cm);
	\draw[shaded] (180:4cm)--(180:2cm) .. controls ++(0:1cm) and ++(90:1cm) .. (270:2cm) -- (270:4cm);
	
		\node at (0,0) [Tbox, inner sep=2mm] {};

	\node at (-45:1.7cm) {$\star$};
	\node at (-45:.7cm) {$\star$};	

	\draw[very thick] (0,0) circle (2cm);

\end{tikzpicture}
\,
=
\,
\begin{tikzpicture}[PAdefn]
	\clip (0,0) circle (3cm);
	
	\begin{scope}[shift=(10:1cm)]	
		\draw[shaded] (0:6cm)--(90:6cm)--(90:2cm) .. controls ++(-90:1cm) and ++(180:.5cm) .. (0:1cm)--(0:6cm);	
		\draw[shaded] (-135:.6cm) circle (.5cm);
	\end{scope}
	
	\draw[shaded] (-150:1cm) -- (120:4cm) -- (180:4cm) -- (-150:1cm);
	\draw[shaded] (-150:1cm) -- (-120:4cm) -- (-60:4cm) -- (-150:1cm);
	
	\begin{scope}[shift=(10:1cm)]	
		\node at (0:.8cm) [Tbox, inner sep=2mm] {};
		\node at (90:1.5cm) [Tbox, inner sep=2mm] {};
		\node at (-45:.9cm) {$\star$};
		\node at (120:1.6cm) {$\star$};
	\end{scope}
	\node at (-150:1cm) [Tbox, inner sep=2mm] {};
	\node at (-120:1.6cm) {$\star$};
	\node at (-30:2.7cm) {$\star$};
	
	\draw[very thick] (0,0) circle (3cm);

\end{tikzpicture}
$$
\end{Defn}

\begin{Defn}[\cite{VJ1},\cite{EP}]
  A {\em planar algebra} is a collection of vector spaces $\{ V_{i,\pm} \} _{i=0,1,2,\ldots}$ which is acted on by the shaded planar operad; that is, shaded planar tangles act on tensor products of the $V_{i,\pm}$, in a way that is compatible with composition of tangles.  A shaded planar tangle whose $i^{th}$ inner disk has type $(k_i,\pm_i)$ and outer disk has type $(k_o, \pm_o)$ gives a map $\bigotimes_{i} V_{k_i,\pm_i} \rightarrow V_{k_o, \pm_o}$.  
For example, the first tangle above represents a map $V_{1,+} \otimes V_{2,+} \otimes V_{2,-} \rightarrow V_{3,+}$.
  
Compositional compatibility means composition of tangles and composition of multilinear maps should produce the same result, e.g.  $$(A \circ_1 B) (v_1, \ldots, v_i, w_2, \ldots, w_j) = A( B(v_1, \ldots v_i), w_2, \ldots, w_j).$$  
  
  \end{Defn}
  
  \begin{Defn}[\cite{VJ1},\cite{EP}]
Elements of $V_n$ are called {\em $n$-boxes} and we will sometimes refer to $V_{n}$ as the {\em $n$-box space} of $V$, or as the $n^{th}$ {\em level} of $V$.
\end{Defn}

\begin{remark}[\cite{VJ1},\cite{EP}]
For $n \geq 1$, $V_{n,+} \simeq V_{n,-}$ (as vector spaces only, not in any algebraic sense), via the shading-changing rotation map
$$\rho^{1/2} = 
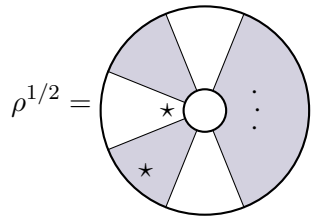
\begin{tikzpicture}[annular]
	\clip (0,0) circle (2cm);

	\filldraw[shaded] (158:4cm)--(0,0)--(112:4cm);
	\filldraw[shaded] (-158:4cm)--(0,0)--(-112:4cm);
	
	\draw[shaded] (68:4cm)--(0,0)--(-68:4cm)--(0:10cm);
	
	\draw[ultra thick] (0,0) circle (2cm);
	
	\node at (0,0)  [empty box] (T) {};
	\node at (T.180) [left] {$\star$};
	\node at (-135:2cm) [above right] {$\star$};
	\node at (0:1cm) {$\cdot$};
	\node at (20:1cm) {$\cdot$};
	\node at (-20:1cm) {$\cdot$};
\end{tikzpicture}.$$
The isomorphism between $V_{n,+}$ and $V_{n,-}$ is the reason that previous definitions of planar algebras required stars to be in unshaded regions, and only had spaces $V_+$, $V_-$ and $V_i$, $i=1,2,3, \ldots$.  We will drop $+$ and $-$ signs from subscripts when it is either clear, or unimportant, whether we are working in $V_{n,+}$ or $V_{n,-}$.
\end{remark}
  
\begin{Defn}[The Temperley-Lieb algebra][\cite{VJ1},\cite{EP}]
The first example of a planar algebra is the Temperley-Lieb algebra with parameter $\delta$ (this algebra was introduced in \cite{TL} and formulated diagrammatically by Kauffman in \cite{KF}).  The vector spaces $TL_{i,\pm}$ have a basis (called $B(TL_{i,\pm})$) consisting of non-crossing pairings on $2i$ numbers; these can be drawn as planar tangles with no input disks, $2i$ points on the output disk, and no closed circles (all strings have endpoints on boundary disks).   The number of such pictures is the $i^{th}$ Catalan number $\frac{1}{i+1}\binom{2i}{i}$.

\begin{example}[\cite{VJ1},\cite{EP}]
$TL_{3,+}=\{
\begin{tikzpicture}[TLEG]
	\filldraw[shaded]  (30:1cm) arc (30:90:1cm) arc (30:-150:5mm) arc (150:210:1cm) -- cycle; 
	\filldraw[shaded]  (0,-1) arc (-90:-30:1cm) arc (30:210:5mm);
	\draw[thick] (0,0) circle (1cm);
	\node at (120:1.3cm) {$\star$};
\end{tikzpicture},
\begin{tikzpicture}[TLEG]
	\filldraw[shaded]  (0,1) arc (90:30:1cm) arc (90:270:5mm) arc (-30:-90:1cm) -- cycle; 
	\filldraw[shaded]  (-1,0) arc (180:210:1cm) arc (-90:90:5mm) arc (150:180:1cm);
	\draw[thick] (0,0) circle (1cm);
	\node at (120:1.3cm) {$\star$};
\end{tikzpicture},
\begin{tikzpicture}[TLEG, rotate=180]
	\filldraw[shaded]  (150:1cm) arc (150:90:1cm) arc (-210:-30:5mm) arc (30:-30:1cm)--cycle;
	\filldraw[shaded] (-90:1cm) arc (-90:-150:1cm) arc (150:-30:5mm);
	\draw[thick] (0,0) circle (1cm);
	\node at (-60:1.3cm) {$\star$};
\end{tikzpicture},
\begin{tikzpicture}[TLEG]
	\filldraw[shaded]  (0,-1) arc (-90:-30:1cm) arc (30:210:5mm);
	\filldraw[shaded]  (-1,0) arc (180:210:1cm) arc (-90:90:5mm) arc (150:180:1cm);
	\filldraw[shaded] (90:1cm) arc (90:30:1cm) arc (-30:-210:5mm);
	\draw[thick] (0,0) circle (1cm);
	\node at (120:1.3cm) {$\star$};
\end{tikzpicture},
\begin{tikzpicture}[TLEG]
	\filldraw[shaded]  (90:1cm) arc (30:-150:5mm) arc (150:210:1cm) arc (150:-30:5mm) arc (-90:-30:1cm) arc (-90:-270:5mm) arc (30:90:1cm);
	\draw[thick] (0,0) circle (1cm);
	\node at (120:1.3cm) {$\star$};
\end{tikzpicture}
\}$
\end{example}

The action of shaded planar tangles on this basis is straightforward: put the pictures inside the tangle, smooth all strings and throw out closed circles by multiplying the picture by $\delta$ (i.e., if $\tau=\tau' \sqcup \tikz \draw[thick] (0,0) circle (1mm);$ or $\tau=\tau' \sqcup \tikz \filldraw[thick, shaded] (0,0) circle (1mm);$, then $\tau = \delta \tau'$.)  For example,
$$
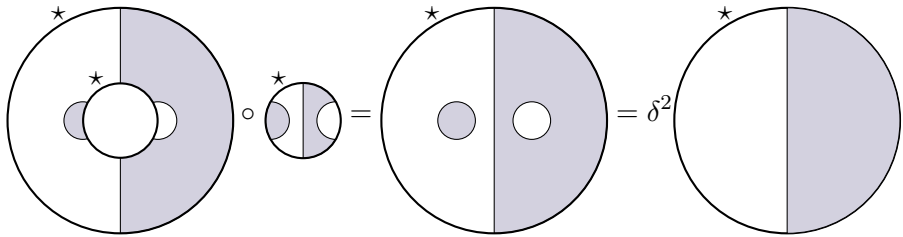
\begin{tikzpicture}[TLEG]
	\filldraw[shaded] (90:3cm) -- (-90:3cm) arc (-90:90:3cm);
	\filldraw[unshaded] (0:1cm) circle (.5cm);
	\filldraw[shaded] (180:1cm) circle (.5cm);
	\node[Tbox,inner sep=3.5mm] at (0,0) {};
	\draw[thick] (0,0) circle (3cm);
	\node at (120:1.3cm) {$\star$};
	\node at (120:3.3cm) {$\star$};
\end{tikzpicture}
\circ
\begin{tikzpicture}[TLEG]
	\filldraw[shaded]  (0,1) arc (90:30:1cm) arc (90:270:5mm) arc (-30:-90:1cm) -- cycle; 
	\filldraw[shaded]  (-1,0) arc (180:210:1cm) arc (-90:90:5mm) arc (150:180:1cm);
	\draw[thick] (0,0) circle (1cm);
	\node at (120:1.3cm) {$\star$};
\end{tikzpicture}
=
\begin{tikzpicture}[TLEG]
	\filldraw[shaded] (90:3cm) -- (-90:3cm) arc (-90:90:3cm);
	\filldraw[unshaded] (0:1cm) circle (.5cm);
	\filldraw[shaded] (180:1cm) circle (.5cm);
	\draw[thick] (0,0) circle (3cm);
	\node at (120:3.3cm) {$\star$};
\end{tikzpicture}
=\delta^2
\begin{tikzpicture}[TLEG]
	\draw[thick]  (-3,0)arc (180:-180: 3cm);
	\filldraw[shaded]  (0,3)--(0,-3) arc (-90:90:3cm);
	\node at (120:3.3cm) {$\star$};
\end{tikzpicture}
$$
\end{Defn}

We have the following common shaded planar tangles, which we will use frequently in this paper.
\begin{Defn}[\cite{VJ1},\cite{EP}] The tangles
$$  	 \begin{tikzpicture}[PAdefn]
	\clip [draw] (2,2) arc (0:180:2cm) -- (-2,-2) arc (-180:0:2cm) -- (2,2);
	
	\draw[shaded] (0,2) .. controls ++(-150:1.5cm) and ++(150:1.5cm) .. (0,-2) .. controls ++(110:1.5cm) and ++(-110:1.5cm) .. (0,2);

	\draw (0,2) .. controls ++(-30:1.5cm) and ++(30:1.5cm) .. (0,-2);
	
	\draw[shaded] (0,2) -- +(110:3cm) -- +(130:3cm) -- (0,2);
	\draw (0,2) -- ++(50:3cm);
		
	\draw (0,-2) -- ++(-50:3cm);
	\draw[shaded] (0,-2) -- +(-110:3cm) -- +(-130:3cm) -- (0,-2);
	
	\node at (.4,0) {$\ldots$};
	\node at (0,0) [below] {$\underbrace{\qquad \qquad}_{k}$};
	\node at (.2,3.2) {$\dots$};
	\node at (.2,-3.2) {$\dots$};

	\node at (0,2) [Tbox,inner sep=1.4mm] (A) {\tiny{\textcolor{gray}{1}}};
	\node at (0,-2) [Tbox,inner sep=1.4mm] (B) {\tiny{\textcolor{gray}{2}}};
	\node at (A.180) [left] {$\star$};
	\node at (B.180) [left] {$\star$};
	\node at (-1.5,2.5)  {$\star$};	
	
	\draw[ultra thick] (2,2) arc (0:180:2cm) -- (-2,-2) arc (-180:0:2cm) -- (2,2);
\end{tikzpicture} \; , \;
 	\begin{tikzpicture}[scale=.25,baseline]
	\clip (8,6) arc (0:180:6cm) -- (-4,-6) arc (-180:0:6cm) -- (8,6);

	\draw (0,0) .. controls ++(-67:6cm) and ++(-90:5cm) .. (3,0) .. controls ++(90:5cm) and ++(67:6cm) .. (0,0);
	\filldraw[shaded] (0,0) .. controls ++(130:4cm) and ++(90:15cm) .. (5,0) .. controls ++(-90:15cm) and ++(-130:4cm) .. (0,0) .. controls ++(-157:7cm) and ++(-90:20cm) .. (6,0) .. controls ++(90:20cm) and ++(157:7cm) .. (0,0);

	\node at (0,0) [Tbox,inner sep=2mm] (T1) {};
	\node at (T1.180) [left] {$\star$};
	\node at (4,0) {$\cdots$};
	\node at (4.5,0) [below] {$\underbrace{\qquad \quad}_{k}$};

	\draw[ultra thick] (8,6) arc (0:180:6cm) -- (-4,-6) arc (-180:0:6cm) -- (8,6);
\end{tikzpicture} \; , \text{and} \;
 	\begin{tikzpicture}[baseline,scale=.5]
	\clip (0,0) circle (3cm);
	
	\filldraw[shaded] (0,0)--(120:4cm)--(100:5cm)--(0,0);
	\filldraw[shaded] (0,0)--(-120:4cm)--(-100:5cm)--(0,0);
	
	\draw (0,0)--(60:4cm);
	\draw (0,0)--(-60:4cm);
	
	\draw (30:3cm) .. controls ++(-120:1.5cm) and ++(120:1.5cm) .. (-30:3cm);
	
	\node at (0,0) [Tbox, inner sep=2mm] {};
	\node at (80:2cm) {$\cdots$};
	\node at (-80:1.7cm) {$\cdots$};

	\node at (90:2.2cm) [below] {$\underbrace{\qquad \quad}_{k}$};
	\node at (-90:1.5cm) [below] {$\underbrace{\qquad \qquad}_{k}$};
	\node at (-180:1cm) [left] {$\star$};
	
	\draw[ultra thick] (0,0) circle (3cm);		
\end{tikzpicture} 
$$
are multiplication, trace and inclusion for $V_{i,+}$; the reverse-shaded tangles are multiplication, trace and inclusion for $V_{i,-}$.
\end{Defn}
\begin{fact}[\cite{VJ1},\cite{EP}]
\begin{itemize}
	\item
	The multiplication tangle 
 	gives an associative map $$m: V_{k,\pm} \otimes V_{k,\pm} \rightarrow V_{k,\pm}.$$  As usual, we will frequently denote multiplication by putting two elements next to each other, maybe with a dot between them.
 	\item 
	The trace tangle
 	gives a cyclically commutative map $$\operatorname{tr}: V_{k,\pm} \rightarrow V_{0,\pm}.$$  
 	\item 
	The inclusion tangle gives a map $$\iota: V_{k,\pm}  \rightarrow V_{k+1,\pm}$$ which is compatible with multiplication and trace.  
 \end{itemize}
 \end{fact} 
\begin{example}[\cite{VJ1},\cite{EP}] The identity for multiplication on Temperley-Lieb is $n$ vertical strands, i.e.
$$\id = \begin{tikzpicture}[scale=.6, baseline]	
	\draw (140:2cm) -- (-140:2cm) arc (-140:-110:2cm) -- (110:2cm) arc (110:140:2cm);
	\draw (40:2cm)--(-40:2cm) arc (-40:-70:2cm)--(70:2cm) arc (70:40:2cm);
	
	\node at (0,0) {$\cdots$};
	\node at (180:2cm) [left] {$\star$};
	
	\draw[thick] (0,0) circle (2cm);
\end{tikzpicture}$$
Temperley-Lieb is multiplicatively generated by the {\em Jones projections} 
$$e_i = \frac{1}{\delta} \cdot 
\begin{tikzpicture}[scale=.6, baseline]
	\draw[thick] (0,0) circle (2cm);
	
	\draw (145:2cm) -- (-145:2cm) arc (-145:-110:2cm) -- (110:2cm) arc (110:145:2cm);
	\draw (35:2cm)--(-35:2cm) arc (-35:-70:2cm)--(70:2cm) arc (70:35:2cm);
	
	\draw (80:2cm) .. controls ++(-90:1cm) and ++(-90:1cm) .. (100:2cm);
	\draw (-80:2cm) .. controls ++(90:1cm) and ++(90:1cm) .. (-100:2cm);

	\node at (-1.1,0) {\tiny{$\cdots$}};
	\node at (1.2,0) {\tiny{$\cdots$}};	
	\node at (180:2cm) [left] {$\star$};
\end{tikzpicture}$$
\end{example}
\section{Subfactor Planar Algebras}
\begin{remark}[\cite{VJ1},\cite{EP}]Subfactor planar algebras are planar algebras with additional structure.    They are called `subfactor' because the planar algebra of a subfactor always has these properties.  Furthermore, a planar algebra with these properties is always the standard invariant of some subfactor.  Further details on this connection are found in chapter 4 of \cite{VJ1}.

The properties that define subfactor planar algebras make them easier to work with than general planar algebras.  In particular, the requirements that each space be finite dimensional and have an inner product means the tools of linear algebra are available to us when we work with subfactor planar algebras.
\end{remark}
\begin{Defn}[\cite{VJ1},\cite{EP}]
A {\em subfactor} planar algebra is a planar algebra (over $\mathbb{C}$) which has
\begin{enumerate}
\item Involution:
 a $*$ on each $V_{i,\pm}$ which is compatible with reflection of tangles (so that $\tau^* (v_1^*,v_2^*,\ldots v_n^*)=\tau (v_1,v_2,\ldots v_n)^*$), 
 \item Dimension restrictions:
  $\dim (V_{0,+})=\dim(V_{0,-})=1$ and $\dim{V_{k,\pm}}<\infty$ for all $k$,
   \item Sphericality:
The left trace $\operatorname{tr}_l:V_{1,\pm} \rightarrow V_{0,\mp}\simeq \mathbb{C}$ and the right trace $\operatorname{tr}_r:V_{1,\pm} \rightarrow V_{0,\pm }\simeq \mathbb{C}$ are equal (equivalently, the action of planar tangles is invariant under spherical isotopy),
 \item  Inner Product:
The bilinear form on $V_{n,\pm}$ given by  $\left< a,b \right>:=\tr{b^* a}$ is positive definite.
\end{enumerate}
\end{Defn}

One noteworthy property of subfactor planar algebras is that closed circles count as constant multiples of the empty diagram and that shaded closed circles have the same value as unshaded closed circles.

\begin{notation}[\cite{VJ1},\cite{EP}]
In a subfactor planar algebra, $\delta$ will denote the value of closed circles.  Sometimes we change variables so that 
$$ \delta = [2]_q= (q+q^{-1})$$
in order to use {\em quantum numbers:}
$$[n]_q :=\frac{q^n-q^{-n}}{q-q^{-1}} = q^{n-1} + q^{n-3} + \cdots + q^{-n+1} + q^{-n+3}.$$
We will often write $[n]$ instead of $[n]_q$ in situations where the value of $q$ is known.
\end{notation}

\begin{example}[\cite{VJ1},\cite{EP}]
The Temperley-Lieb planar algebra always meets conditions 1, 2, and 3:  The involution is defined by reflection;  $TL_{0,+}=\C \{ \tikz \draw[thick] (0,0) circle (1mm); \}$  and $TL_{0,-}=\C \{ \tikz \draw[thick, shaded] (0,0) circle (1mm); \}$ are both $1$-dimensional; and shaded and unshaded circles both count for $\delta$, hence the planar algebra is spherical.

If $\delta \geq 2$, the bilinear form defined by the trace tangle is positive definite, and so Temperley-Lieb is a subfactor planar algebra.  If $\delta=2 \cos{\frac{\pi}{n}}$ for some $n \geq 3$, the bilinear form is positive semidefinite, and we can form a subfactor planar algebra by quotienting Temperley-Lieb by all $x \in TL$ such that $\tr{x^* x}=0$.
\end{example}

\begin{fact}[\cite{VJ1},\cite{EP}]\label{TLinjects}
Let $V$ be a subfactor planar algebra with parameter $\delta$.  Then the map
$$TL(\delta) \hookrightarrow \Hom{\mathbb{C},V} \cong V $$
given by interpreting a Temperley-Lieb diagram as a shaded planar tangle with no inputs 
\begin{itemize}
\item is injective if $\delta \geq 2$;
\item has kernel $\operatorname{Rad}(\left<, \right>)$ if $\delta < 2$.
\end{itemize}
\end{fact}

In Chapter 3, we will look at the Bratelli diagram of the ``2221'' subfactor.  This diagram shows how minimal idempotents get included into the next level.  We will need the following definitions:

\begin{Defn}[\cite{VJ1},\cite{EP}]
\begin{itemize}
	\item	An {\em idempotent} is an element $p$ of some $V_{k,\pm}$ such that $p=p^*=p p$.  We often draw idempotents in rectangles instead of disks, with an implicit star on the left side.
	\item An idempotent $p \in V_{k,\pm}$ is {\em minimal} if $p \cdot V_{k,\pm} \cdot p$ is one-dimensional.  
	\item Two idempotents $p \in V_{i,\pm}$ and $q \in V_{j,\pm}$ are {\em isomorphic} if there is an element $f$ in $V_{i \rightarrow j,\pm}$ (this is the same space as $V_{(i+j)/2,\pm}$, but drawn with $i$ strings going up and $j$ strings going down) such that $f^*f=p$ and $f f^*=q$:
	$$\begin{tikzpicture}[scale=.7,baseline]
	\node at (0,1) [circle,draw,fill=white,inner sep=1mm] (Fstar) {$f^*$};
	\node at (0,-1) [circle, draw,fill=white] (F) {$f$};
	
	\node at (Fstar.180) [left] {$\star$};
	\node at (F.180) [left] {$\star$};	
	
	\draw (Fstar.-90)--(F.90);
	\draw (Fstar.90)-- ++(90:5mm);
	\draw (Fstar.60)-- ++(60:5mm);
	\draw (Fstar.120)-- ++(120:5mm);
	\draw (F.-120)--++(-120:5mm);
	\draw (F.-90)-- ++(-90:5mm);
	\draw (F.-60)--++(-60:5mm);
\end{tikzpicture}
=
\begin{tikzpicture}[scale=.7,baseline]
	\node at (0,0) [rectangle, draw] (p) {$\;\; p \;\;$};
	\draw (p.40)--++(90:5mm);
	\draw (p.90)--++(90:5mm);
	\draw (p.140)--++(90:5mm);
	\draw (p.-40)--++(-90:5mm);
	\draw (p.-90)--++(-90:5mm);
	\draw (p.-140)--++(-90:5mm);
\end{tikzpicture}
\quad , \quad
\begin{tikzpicture}[scale=.7,baseline]
	\node at (0,-1) [circle,draw,fill=white,inner sep=1mm] (Fstar) {$f^*$};
	\node at (0,1) [circle, draw,fill=white] (F) {$f$};
	
	\node at (Fstar.180) [left] {$\star$};
	\node at (F.180) [left] {$\star$};
	
	\draw (Fstar.90)--(F.-90);
	\draw (Fstar.45) .. controls (.6,0) .. (F.-45);
	\draw (Fstar.135) .. controls (-.6,0) .. (F.-135);
	\draw (F.90)-- ++(90:5mm);
	\draw (Fstar.-90)-- ++(-90:5mm);
\end{tikzpicture}
=
\begin{tikzpicture}[scale=.7,baseline]
	\node at (0,0) [rectangle, draw] (p) {$\;\; q \;\;$};
	\draw (p.90)--++(90:5mm);
	\draw (p.-90)--++(-90:5mm);
\end{tikzpicture}$$
\end{itemize}
\end{Defn}

\begin{example}[\cite{VJ1},\cite{EP}][The Jones-Wenzl idempotents  of Temperley-Lieb] A {\em Jones-Wenzl idempotent}, first defined in \cite{HW} and denoted $f^{(n)}$, is the unique idempotent in $TL_{n,\pm}$  which is orthogonal to all $TL_{n,\pm}$  basis elements except the identity.  In symbols this says $f^{(n)}\neq 0$, $f^{(n)}f^{(n)}=f^{(n)}$ and $e_i f^{(n)}=f^{(n)}e_i=0$ for all $i \in \{ 1, \ldots , n-1\}$ (since $e_i$ multiplicatively generate $TL$).  These have $\tr{f^{(n)}}=[n+1]$, are minimal, and satisfy {\em Wenzl's relation}:
$$[n+1] \cdot
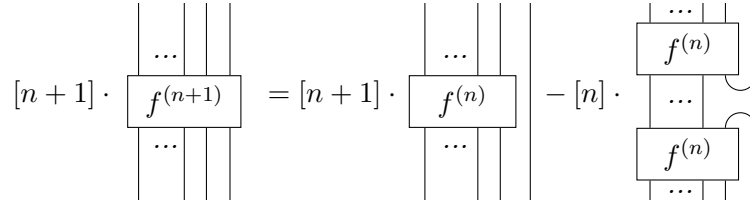
\begin{tikzpicture}[baseline]
    \clip (-.9,-1.3) rectangle (1,1.3);
    \node at (0,0) (box) [rectangle,draw] {$\hspace{.1cm} f^{(n + 1)} \hspace{.1cm} $};

    \draw (box.-50)-- ++(-90:1cm);
    \draw (box.-90) -- ++(-90:1cm);
    \draw (box.-150)--++(-90:1cm);
    \node at (-.25,-.6) {...};
    \draw (box.-30) -- ++(-90:1cm);

    \draw (box.50)-- ++(90:1cm);
    \draw (box.90) -- ++(90:1cm);
    \draw (box.150)--++(90:1cm);
    \node at (-.25,.6) {...};
    \draw (box.30) -- ++(90:1cm);
\end{tikzpicture}
=
[n+1] \cdot
\begin{tikzpicture}[baseline]
    \clip (-.8,-1.3) rectangle (1,1.3);
    \node at (0,0) (box) [rectangle,draw] {\; $f^{(n)}$ \;};

    \draw (box.-145)--++(-90:1cm);
    \node at (-.1,-.6) {...};
    \draw (box.-35) -- ++(-90:1cm);
    \draw (box.-60) -- ++(-90:1cm);

    \draw (box.35) -- ++(90:1cm);
    \draw (box.60) -- ++(90:1cm);
    \node at (-.1,.6) {...};
    \draw (box.145) -- ++(90:1cm);

    \draw (.9,1.5)--(.9,-1.5);
\end{tikzpicture}
-
[n] \cdot
\begin{tikzpicture}[baseline]
    \clip (-.8,-1.3) rectangle (1,1.3);
    \node at (0,.7) (top) [rectangle,draw] {$\;\; f^{(n)} \; \;$};
    \node at (0,-.7) (bottom) [rectangle,draw]{$\; \;f^{(n)} \; \;$};

    \draw (top.-60)--(bottom.60);
    \draw (top.-145)--(bottom.145);
    \node at (-.1,0) {...};
    \draw (top.-35) arc (-180:0:.2cm) -- ++(90:1cm);
    \draw (bottom.35) arc (180:0:.2cm) -- ++(-90:1cm);

    \draw (top.35) -- ++(90:1cm);
    \draw (top.60) -- ++(90:1cm);
    \node at (-.1,1.2) {...};
    \draw (top.145) -- ++(90:1cm);

    \draw (bottom.-35) -- ++(-90:1cm);
    \draw (bottom.-60) -- ++(-90:1cm);
    \node at (-.1,-1.2) {...};
    \draw (bottom.-145) -- ++(-90:1cm);
\end{tikzpicture}$$
\end{example}
The Jones-Wenzl idempotents also satisfy the following properties, which we cite without proof:

\begin{lem}[\cite{ExtH}]
\label{lem:basicrecursion}
\begin{align*}
\begin{tikzpicture}[baseline=0,scale=0.3]
\foreach \x in {-3,...,3} { \draw (\x,-3.5) -- (\x,3.5); }
\filldraw[thick,fill=white] (-4,-1.5) rectangle (4,1.5);
\node at (0,0) {$\JW{k}$};
\end{tikzpicture}
= 
\begin{tikzpicture}[baseline=0,scale=0.3]
\foreach \x in {-3,...,2} { \draw (\x,-3.5) -- (\x,3.5); }
\draw (4,-3.5) -- (4,3.5);
\filldraw[thick,fill=white] (-4,-1.5) rectangle (3,1.5);
\node at (-0.5,0) {$\JW{k-1}$};
\end{tikzpicture}
+
\frac{1}{[k]}
\sum_{a = 1}^{k-1} (-1)^{a+k+1} [a] \;
\begin{tikzpicture}[baseline=0,scale=0.3]
\draw[decorate,decoration={brace,raise=2pt}] (-4,3.5) -- (-2,3.5);
\node at (-3,4.6) {$a$};
\foreach \x in {-4,-3} { \draw (\x+1,0) -- (\x,3.5); }
\draw (-1,3.5) arc (360:180:0.5);
\foreach \x in {0,1} { \draw (\x-1,0) -- (\x,3.5); }
\draw (1,0)--(1.5,1.75) .. controls (2,3.5) and (3,3.5) .. (3,1.5)--(3,-3.5);
\foreach \x in {-3,...,1} { \draw (\x,-3.5) -- (\x,0); }
\filldraw[thick,fill=white] (-4,-1.5) rectangle (2,1.5);
\node at (-1,0) {$\JW{k-1}$};
\end{tikzpicture}.
\end{align*}
\end{lem}

\begin{lem}[\cite{ExtH}]
\label{lem:yetanother}
\begin{align*}
\begin{tikzpicture}[baseline=0,scale=0.3]
\foreach \x in {-3,...,3} { \draw (\x,-3.5) -- (\x,3.5); }
\filldraw[thick,fill=white] (-4,-1.5) rectangle (4,1.5);
\node at (0,0) {$\JW{k}$};
\end{tikzpicture}
= 
\begin{tikzpicture}[baseline=0,scale=0.3]
\foreach \x in {-3,...,2} { \draw (\x,-3.5) -- (\x,3.5); }
\draw (4,-3.5) -- (4,3.5);
\filldraw[thick,fill=white] (-4,-1.5) rectangle (3,1.5);
\node at (-0.5,0) {$\JW{k-1}$};
\end{tikzpicture}
+
\frac{1}{[k][k-1]}
\sum_{a,b = 1}^{k-1} (-1)^{a+b+1} [a][b] \;
\begin{tikzpicture}[baseline=0,scale=0.3]
\draw[decorate,decoration={brace,raise=2pt}] (-4,3.5) -- (-2,3.5);
\node at (-3,4.6) {$a$};
\draw[decorate,decoration={brace,mirror,raise=2pt}] (-4,-3.5) -- (-1,-3.5);
\node at (-2.5,-4.85) {$b$};
\foreach \x in {-4,-3} { \draw (\x,-3.5) -- (\x+1,0) -- (\x,3.5); }
\draw (-2,-3.5) -- (-1,0);
\draw (-1,-3.5) arc (180:0:0.5);
\draw (-1,3.5) arc (360:180:0.5);
\draw (0,3.5) -- (-1,0);
\foreach \x in {1,2} { \draw (\x,-3.5) -- (\x-1,0) -- (\x,3.5); }
\filldraw[thick,fill=white] (-4,-1.5) rectangle (2,1.5);
\node at (-1,0) {$\JW{k-2}$};
\end{tikzpicture}.
\end{align*}
\end{lem}

\begin{Defn}[\cite{VJ1},\cite{EP}]
The {\em principal graph} consists of vertices and edges.
The vertices of the principal graph are the isomorphism classes of idempotents in $V_{k,+}$ for any $k$.  
There are $n$ edges between $p$ and $q$ if $\iota(p) =
\begin{tikzpicture}[baseline]
	\foreach \x in {-3,-1,1,3,5} \draw (\x mm,-4mm)--(\x mm, 4mm);
	\node[rectangle,draw,fill=white] {$\; p \;$};
\end{tikzpicture}
$ contains $n$ copies of $q$, meaning that when $\iota(p)$ is decomposed into minimal idempotents, $n$ of these are isomorphic to $q$.  The dual principal graph is constructed in the same way, for the idempotents in $V_{k,-}$.
\end{Defn}

\begin{example}[\cite{VJ1},\cite{EP}][The principal graph of Temperley-Lieb is $A_n$ or $A_\infty$]
The vertices of the principal graph of Temperley-Lieb are the Jones-Wenzl idempotents.
Wenzl's relation says that if $[n+1]\neq 0$, the projection $\iota (f^{(n)})$ decomposes into minimal projections isomorphic to $f^{(n+1)}$ and $f^{(n-1)}$.  So there is an edge in the principal graph between $f^{(k)}$ and $f^{(k+1)}$ as long as $[k+1] \neq 0$.  

If $\delta>2$, we never have $[k] = 0$, so the principal graph is the Dynkin diagram $A_{\infty}$.
 If $\delta = 2 \cos{\frac{\pi}{n}}$ then $[n]=0$, so by the positive definiteness of the inner product $f^{(n-1)}=0$.  Thus the principal graph has $n-1$ vertices and is the Dynkin diagram $A_{n-1}$.
 \end{example}

In regards to the Jones-Wenzl idempotents, we will often need to know the coefficients of the $TL$ pictures in the Jones-Wenzl idempotents.  We will make frequent use of an explicit formula found in \cite{SM} to do this.

\section{The Planar Algebra of a Bipartite Graph}

\begin{remark}[\cite{VJ1},\cite{EP}]The planar algebra of a bipartite graph $G$, defined in \cite{VJ2}, is another example of a planar algebra.  It shares with Temperley-Lieb the property that closed circles count for a constant --- in this case, 
$\delta= \left\Vert G \right\Vert $ (where $\Norm{G}$ is the operator norm of the adjacency matrix of $G$).  It also has an involution, is spherically invariant and has a positive definite inner product.  Usually, however, it is not a subfactor planar algebra, because for most graphs, the zero-box spaces are too big.
\end{remark}
\begin{notation}[\cite{VJ1},\cite{EP}]
All graphs in this paper are simply laced.  Therefore paths or loops can, and will, be entirely described by the vertices they pass through. 
When we concatenate two paths written this way, we have to drop a vertex:
$$a b \ldots c d \sqcup d e \ldots f g = a b \ldots c d e \ldots  f g.$$
When a loop is described in this notation, we might forget to notate the last vertex.

Paths and loops are written in boldface, and indices indicate the position of a vertex in a path or loop: $\pth{p} = p_1 p_2 \ldots p_k$.
\end{notation}

\begin{Defn}[\cite{VJ1},\cite{EP}]
$PABG(G)_{0,+}$ has the even vertices $U_+$ of $G$ as a basis; $PABG(G)_{0,-}$ has the odd vertices $U_-$ as a basis.  The space $PABG(G)_{i,+}$ has a basis consisting of all based loops of length $2i$ on the graph $G$, based at any even vertex; similarly the space $PABG(G)_{i,-}$ has a basis consisting of all based loops of length $2i$ on the graph $G$, based at any odd vertex.
\end{Defn}

To see how the graph planar algebra is a planar algebra, we need to show how a shaded planar tangle acts on the basis elements.  We do this by using the idea of a state.  

\begin{Defn}[\cite{VJ1},\cite{EP}] Suppose $\tau$ is a tangle with $k$ input disks.  A {\em state} on $\tau$ is an assignment of even vertices to unshaded regions, odd vertices to shaded regions, and edges to strings in such a way that if an edge is assigned to a string, its endpoints are assigned to the two regions which touch that string.   Given a state $\sigma$ on $\tau$, we can read clockwise around any disk and get a loop on $G$ (the starred region tells us where to base the loop).  Let $\partial_i (\sigma)$ be the loop read from the $i^{\text{th}}$ inner disk, and let $\partial_o (\sigma)$ be the loop read from the outer disk.
\end{Defn}

\begin{Defn}[\cite{VJ1},\cite{EP}]
Define the action of $\tau$ on loops $\lp{p_i}$ by
$$\tau(\lp{p_1},\ldots,\lp{p_k}):= \sum_{\mbox{states } \sigma \mbox{ s.t. } \partial_i(\sigma)=\lp{p_i}} c(\tau, \sigma) \cdot \partial_o (\sigma) ,$$
where $c(\tau, \sigma)$ is a number defined below.  Extend this action from the basis of loops to all of $PABG(G)_{i,\pm}$ by linearity.
\end{Defn}
Our bipartite graph G will have Perron-Frobenius data.

\begin{Defn}[\cite{VJ1},\cite{EP}]
Let $\Lambda$ be the adjacency matrix of $G$.   Its  {\em Perron-Frobenius eigenvalue} is the maximal modulus eigenvalue (which is necessarily real); call this $\delta$.   Let $\lambda$ be the eigenvector with eigenvalue $\delta$, called the {\em Perron-Frobenius eigenvector}.  We normalize $\lambda$ to have norm $\sqrt{2}$, and let $\lambda(a)$ be the entry of $\lambda$ corresponding to vertex $a$.  
\end{Defn}

The normalization above is chosen so that $\sum_{a \in U_+} \lambda(a)^2=1$ and $\sum_{a \in U_-} \lambda(a)^2=1$.

\begin{Defn}[\cite{VJ1},\cite{EP}]
To define $c(\sigma, \tau)$, we first put $\tau$ in a `standard form':
isotope $\tau$ so that all strings are smooth, and all its boxes are rectangles, with the starred region on the left, half the strings coming out of the top and the other half coming out the bottom.
Let $E(\tau)$ be the set of maxima and minima on strings of the standard form of $\tau$.  If $t \in E(\tau)$ is a max or min on a string of $\tau$, let 
$\sigma(t_{\mbox{convex}})$ 
be the vertex assigned by $\sigma$ to the region touching $t$ where the string is convex, and 
${\sigma}(t_{\mbox{concave}})$
 be the vertex assigned by $\sigma$ to the region touching $t$ where the string is concave.  Then
$$c(\tau, \sigma):=\prod_{t \in E(\tau)} \sqrt{\frac{\lambda({\sigma}(t_{\mbox{convex}}))}{\lambda({\sigma}(t_{\mbox{concave}}))}}$$
\end{Defn}

Let us consider our specific bipartite graph for ``2221'':

\begin{example}
Take our bipartite graph to be 
$$H=
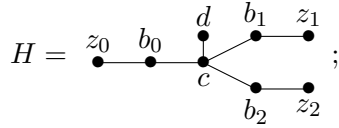
\begin{tikzpicture}[baseline,scale=.7]
\node at (0,0) {$\bullet$};
\node at (0,0) [above] {$z_0$};
\node at (1,0) {$\bullet$};
\node at (1,0) [above] {$b_0$};
\node at (2,0) {$\bullet$};
\node at (2,0) [below] {$c$};
\node at (2,.5) {$\bullet$};
\node at (2,.5) [above] {$d$};
\node at (3,.5) {$\bullet$};
\node at (3,.5) [above] {$b_1$};
\draw (0,0)--(2,0);
\draw (2,0)--(2,.5);
\node at (4,.5) {$\bullet$};
\node at (4,.5) [above] {$z_1$};
\node at (3,-.5) {$\bullet$};
\node at (3,-.5) [below] {$b_2$};
\node at (4,-.5) {$\bullet$};
\node at (4,-.5) [below] {$z_2$};
\draw (2,0)--(3,.5)--(4,.5);
\draw (2,0)--(3,-.5)--(4,-.5);
\end{tikzpicture};$$ 
the Perron-Frobenius data for $H$ is $\delta=\sqrt{\frac{5+\sqrt{21}}{2}}$ and 
\begin{align*}
\lambda(z_i) & = \sqrt{\frac{7-\sqrt{21}}{42}}, & 
\lambda(b_i) & =\sqrt{\frac{7+\sqrt{21}}{42}}, \\
\lambda(d) & =\sqrt{\frac{21-3\sqrt{21}}{42}}, & 
 \lambda(c) & =\sqrt{\frac{21+3\sqrt{21}}{42}}. 
\end{align*}
Here is an example of how shaded planar tangles act on paths:
\begin{align*}\begin{tikzpicture}[scale=.7,baseline]
	\filldraw[shaded] (0,2)--(0,1) arc (-180:0:5mm) -- (1,2);
	\filldraw[shaded] (0,-1)--(0,-.7) arc (180:0:5mm) -- (1,-1);
	\node at (0,1.2) [rectangle, draw, thick, fill=white, inner sep=2mm] {};
	\draw[thick] (-1,-1) rectangle (2,2);
\end{tikzpicture}( c b_0)  = 
\begin{tikzpicture}[scale=.7,baseline]
	\filldraw[shaded] (0,2)--(0,1) arc (-180:0:5mm) -- (1,2);
	\filldraw[shaded] (0,-1)--(0,-.7) arc (180:0:5mm) -- (1,-1);
	\node at (0,1.2) [rectangle, draw,thick, fill=white, inner sep=2mm] {};
	\draw[thick] (-1,-1) rectangle (2,2);
	\node at (.6,1.2) {{\textcolor{white}{$b_0$}}};
	\node at (-.6,1.2) {{\textcolor{gray}{$c$}}};
	\node at (.5,-.6) {{\textcolor{white}{?}}};
\end{tikzpicture}
&=
\frac{\lambda(b_0)}{\lambda(c)} c b_0 c b_0 + \frac{\sqrt{\lambda(b_0) \lambda(b_1)}}{\lambda(c)} c b_0 c b_1  \\
& + \frac{\sqrt{\lambda(b_0) \lambda(b_2)}}{\lambda(c)} c b_0 c b_2 +\frac{\sqrt{\lambda(b_0) \lambda(d)}}{\lambda(c)} c b_0 c d
\end{align*}
\end{example}

\begin{remark}[\cite{VJ1},\cite{EP}] Note that the planar algebra of a bipartite graph as defined is not often a subfactor planar algebra; $\dim(PABG(G)_{0,+})$ equals the number of even vertices of $G$, $\dim(PABG(G)_{0,-})$ equals the number of odd vertices of $G$, and these are both $1$ only in rather dull cases.  However, this planar algebra does have an involution $*$ defined on loops by traversing them backwards.  What's more, the planar algebra of a bipartite graph has a genuine trace (that is, a cyclically commutative map $PABG(G)_{i,\pm} \rightarrow \C$):
\end{remark}
\begin{Defn}[\cite{VJ1},\cite{EP}] The trace $Z: PABG(G)_{i,\pm} \rightarrow \mathbb{C}$ is defined as the composition 
\begin{equation*}
	\xymatrix{
	PABG(G)_{i,\pm} \ar@{->}[r]^{\operatorname{tr}} & PABG(G)_{0,\pm} \ar@{->}[r]^{Z_0} & \C\qquad\qquad\quad }
\end{equation*}
where $\operatorname{tr}$ is the trace tangle, and $Z_0$ is the linear extension of the map $a \mapsto \lambda(a)^2$.
\end{Defn}

\begin{remark}[\cite{VJ1},\cite{EP}]This trace gives us a positive definite inner product, and by construction (and normalization of $\lambda$) agrees with the trace on $TL(\delta) \subset PABG(G)$.
\end{remark}
 
\section{Annular Temperley-Lieb Modules}
We will need the following definitions.

\begin{Defn}[\cite{VJ5},\cite{EP}]
An {\em annular Temperley-Lieb tangle} is any shaded planar tangle having exactly one input disk; for example
$$\,
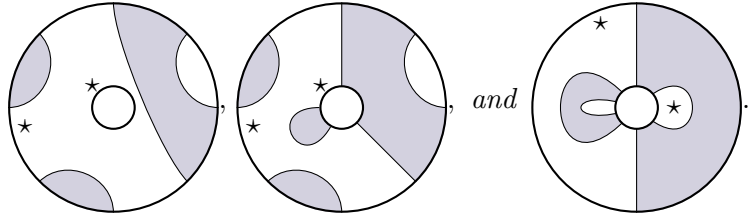
\begin{tikzpicture}[annular]
	\clip (0,0) circle (2cm);

	\filldraw[shaded] (0:2cm) .. controls ++(180:7mm) and ++(-135:7mm) .. (45:2cm) -- (45:3cm)--(90:3cm) -- (90:2cm) .. controls ++(-90:7mm) and ++(135:7mm) .. (-45:2cm) -- (-45:3cm);		
	\filldraw[shaded] (-90:2cm) .. controls ++(90:7mm) and ++(45:7mm) .. (-135:2cm) -- (-135:3cm);
	\filldraw[shaded] (-180:2cm) .. controls ++(0:7mm) and ++(-45:7mm) .. (135:2cm) -- (135:3cm);
		
	\draw[ultra thick] (0,0) circle (2cm);
	
	\node at (0,0)  [empty box] (T) {};
	\node at (T.160) [above] {$\star$};
	\node at (-170:2cm) [right] {$\star$};
\end{tikzpicture} ,
\,
\begin{tikzpicture}[annular]
	\clip (0,0) circle (2cm);

	\filldraw[shaded] (0:2cm) .. controls ++(180:7mm) and ++(-135:7mm) .. (45:2cm) -- (45:3cm)--(90:3cm) -- (90:2cm) --(0,0) -- (-45:2cm) -- (-45:3cm);		
	\filldraw[shaded] (-90:2cm) .. controls ++(90:7mm) and ++(45:7mm) .. (-135:2cm) -- (-135:3cm);
	\filldraw[shaded] (-180:2cm) .. controls ++(0:7mm) and ++(-45:7mm) .. (135:2cm) -- (135:3cm);
	\filldraw[shaded] (T.180) .. controls ++(180:11mm) and ++(-120:11mm) .. (T.-120);
		
	\draw[ultra thick] (0,0) circle (2cm);
	
	\node at (0,0)  [empty box] (T) {};
	\node at (T.160) [above] {$\star$};
	\node at (-170:2cm) [right] {$\star$};
\end{tikzpicture} , \text{ and }
\,
\begin{tikzpicture}[annular]
	\clip (0,0) circle (2cm);

	\filldraw[shaded] (0,0) -- (90:3cm) -- (0:4cm) -- (-90:3cm) -- (0,0) .. controls ++(-45:2cm) and ++(45:2cm) .. (0,0);
	\filldraw[shaded] (0,0) .. controls ++(-130:30mm) and ++(130:30mm) .. (0,0) .. controls ++(160:15mm) and ++(-160:15mm) .. (0,0);	

	\draw[ultra thick] (0,0) circle (2cm);
	
	\node at (0,0)  [empty box] (T) {};
	\node at (T.0) [right] {$\star$};
	\node at (110:2cm) [below] {$\star$};
\end{tikzpicture}. $$
We write $ATL_{k,\pm \rightarrow m,\pm}$ for the set of all annular Temperley-Lieb tangles with type $(k,\pm)$ inner disk and type $(m,\pm)$ outer disk.  If $ATL$ is being applied to an element $R$ of type $(k,\pm)$, we may simply write $ATL_{m,\pm}(R)$ for $ATL_{k,\pm \rightarrow m,\pm}(R)$.
\end{Defn}
\begin{Defn}[\cite{VJ5},\cite{EP}] An {\em annular Temperley-Lieb module} is a family of vector spaces $V_{i,\pm}$ which has an action by annular Temperley-Lieb.  This action should be compatible with composition.
\end{Defn}
All planar algebras are annular Temperley-Lieb modules and we can break up a planar algebra into modules.

\begin{Defn}[\cite{VJ5},\cite{EP}]
An annular Temperley-Lieb module is {\em irreducible} if it has no proper submodules.  This is equivalent to being indecomposable.
\end{Defn}

If $\delta>2$, then irreducible annular Temperley-Lieb modules are classified by an eigenvector and its eigenvalue.

\begin{Defn}[\cite{VJ5},\cite{GR},\cite{EP}]
If $V$ is an irreducible annular Temperley-Lieb (for $\delta>2$) module, it has a {\em low weight space:}  some $k$ such that $V_{i,\pm} = 0$ if $i<k$ and $V_{k,+} \neq 0$ or $V_{k,-} \neq 0$.  
Then $\dim (V_{k,\pm})\leq 1$, and any non-zero element $T \in V_{k,\pm}$ generates all of $V$ as an annular Temperley-Lieb module. In particular, we may choose $T=T^*$.

If $V$ has low weight $(0,+)$,  $T$ is an eigenvector of the double-circle operator
$$\sigma
=
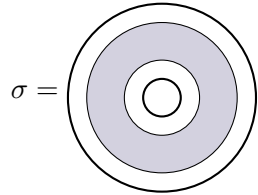
\begin{tikzpicture}[scale=.5,baseline]
	\draw[shaded] (0,0) circle (20mm);
	\draw[unshaded] (0,0) circle (10mm);
	\draw[thick] (0,0) circle (5mm);
	\draw[thick] (0,0) circle (25mm);
\end{tikzpicture},$$
and the eigenvalue $\mu^2$ of $T$ can be any number in $[0,\delta^2]$.  We write $V^{(0,+),\mu}$ for this module.

If $V$ has low weight $(0,-)$,  $T$ is an eigenvector of the double-circle operator, $\sigma$ from above with its shading reversed.
The eigenvalue $\mu^2$ of $T$ can be any number in $[0,\delta^2]$.  We write $V^{(0,-),\mu}$ for this module.

If $V$ has low weight $k>0$, $T$ is an eigenvector of the rotation operator 
$$\rho
=
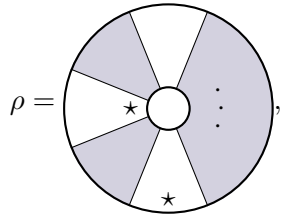
\begin{tikzpicture}[annular]
	\clip (0,0) circle (2cm);

	\filldraw[shaded] (158:4cm)--(0,0)--(112:4cm);
	\filldraw[shaded] (-158:4cm)--(0,0)--(-112:4cm);
	
	\draw[shaded] (68:4cm)--(0,0)--(-68:4cm)--(0:10cm);
	
	\draw[ultra thick] (0,0) circle (2cm);
	
	\node at (0,0)  [empty box] (T) {};
	\node at (T.180) [left] {$\star$};
	\node at (-90:2cm) [above] {$\star$};
	\node at (0:1cm) {$\cdot$};
	\node at (20:1cm) {$\cdot$};
	\node at (-20:1cm) {$\cdot$};
\end{tikzpicture},$$ 
and the eigenvalue $\zeta$ of $T$ can be any $k^{th}$ root of unity.  We write $V^{k,\zeta}$ for this module.
\end{Defn}

\begin{remark}[\cite{VJ5},\cite{EP}]
For $V^{k,\zeta}$ the condition that the generator $T$ is a low weight element (i.e., if $m<k$ 
then $ATL_{m,\pm} (T)=0$) is equivalent to the easier-to-check condition that
all of the capping-off operators in $ATL_{k,\pm \rightarrow k-1,\pm}$ give $0$:  if we define
$$\epsilon_1=
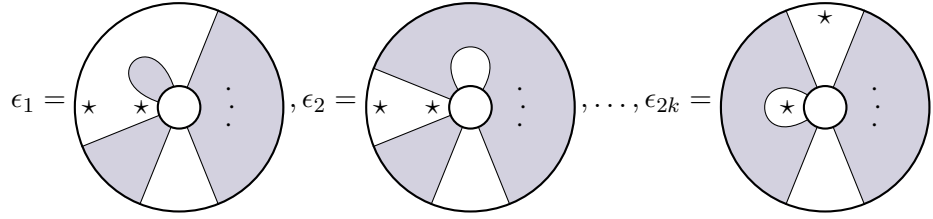
\begin{tikzpicture}[annular]
	\clip (0,0) circle (2cm);

	\filldraw[shaded] (0,0) .. controls ++(170:2cm) and ++(100:2cm) .. (0,0);
	\filldraw[shaded] (-158:4cm)--(0,0)--(-112:4cm);

	
	\draw[shaded] (68:4cm)--(0,0)--(-68:4cm)--(0:10cm);
	
	\draw[ultra thick] (0,0) circle (2cm);
	
	\node at (0,0)  [empty box] (T) {};
	\node at (T.180) [left] {$\star$};
	\node at (180:2cm) [right] {$\star$};
	\node at (0:1cm) {$\cdot$};
	\node at (20:1cm) {$\cdot$};
	\node at (-20:1cm) {$\cdot$};
	
\end{tikzpicture}\, , 
\epsilon_2=\begin{tikzpicture}[annular]
	\clip (0,0) circle (2cm);

	\filldraw[shaded] (158:4cm) -- (0,0) .. controls ++(130:2cm) and ++(50:2cm) .. (0,0)--(-68:4cm) arc (-68:158:4cm);
	\filldraw[shaded] (-158:4cm)--(0,0)--(-112:4cm);

	\draw[ultra thick] (0,0) circle (2cm);
	
	\node at (0,0)  [empty box] (T) {};
	\node at (T.180) [left] {$\star$};
	\node at (180:2cm) [right] {$\star$};
	\node at (0:1cm) {$\cdot$};
	\node at (20:1cm) {$\cdot$};
	\node at (-20:1cm) {$\cdot$};
	
\end{tikzpicture}\, , \ldots,
\epsilon_{2k}=\begin{tikzpicture}[annular]
	\clip (0,0) circle (2cm);

	\filldraw[shaded] (112:4cm) -- (0,0) .. controls ++(140:2cm) and ++(-140:2cm) .. (0,0)--(-112:4cm) arc (-112:-202:4cm);
	
	\draw[shaded] (68:4cm)--(0,0)--(-68:4cm)--(0:10cm);
	
	\draw[ultra thick] (0,0) circle (2cm);
	
	\node at (0,0)  [empty box] (T) {};
	\node at (T.180) [left] {$\star$};
	\node at (90:2cm) [below] {$\star$};
	\node at (0:1cm) {$\cdot$};
	\node at (20:1cm) {$\cdot$};
	\node at (-20:1cm) {$\cdot$};
	
\end{tikzpicture}$$
then  $\epsilon_i(T)=0$ for all $i$.  
\end{remark}

 To decompose a planar algebra into irreducible annular Temperley-Lieb modules, we count dimensions and consider the action of $\rho$, the rotation operator.  First, we need some information on the dimensions of the irreducible modules.

\begin{theorem}[\cite{VJ5},\cite{EP}] If $k>0$, for $m > k$
$$\dim (V_m^{k,\zeta}) = \binom{2m}{m-k};$$
If $k=0$, the dimensions depend on $\mu$:
\begin{align*}
\dim (V_{0,\pm}^{(0,\pm),0}) = 1, \; &
\dim (V_{0,\mp}^{(0,\pm),0}) = 0, &&
\dim (V_m^{(0,\pm),0}) = \frac{1}{2} \binom{2m}{m}, \\
\dim (V_{0,\pm}^{0,\delta})&=1  && \dim (V_m^{0,\delta})= \frac{1}{m+1} \binom{2m}{m}  , \\
\dim (V_{0,\pm}^{0,\mu})&=1  && \dim (V_m^{0,\mu})  = \binom{2m}{m} &
\text{ if } \mu \in (0,\delta)
\end{align*}
\end{theorem}

We now want to use this information to break up the hypothetical ``2221'' planar algebra into irreducible modules.

\chapter{Decomposition of ``2221'' planar algebra into irreducibles}
Consider a hypothetical planar algebra $\mathcal{P}$ having principal graph $H$.  It will have $\delta=\sqrt{\frac{5+\sqrt{21}}{2}}$, and from the principal graph we can reconstruct the Bratelli diagram and sequence of dimensions: (See Figure 3.1)
\begin{figure}
$$\begin{tikzpicture}
	\node at (0,0) (zero) {$1$};
	\node at (1,-1) (one) {$1$};
	\draw (zero)--(one);
	\node at (0,-2) (two1) {$1$};
	\node at (2,-2) (two2) {$1$};
	\draw (one)--(two1);
	\draw (one)--(two2);	
	\node at (1,-3) (three1) {$2$};
	\node at (1.5,-3) (three2) {$1$};
	\node at (2.5,-3) (three3) {$1$};
	\node at (3,-3) (three4) {$1$};
	\draw (two1)--(three1);
	\draw (two2)--(three1);
	\draw (two2)--(three2);
	\draw (two2)--(three3);
	\draw (two2)--(three4);
	\node at (0,-4) (four1) {$2$};
	\node at (2,-4) (four2) {$5$};
	\node at (3,-4) (four3) {$1$};
	\node at (4,-4) (four4) {$1$};
	\draw (three1)--(four1);
	\draw (three1)--(four2);
	\draw (three2)--(four2);
	\draw (three3)--(four2);
	\draw (three3)--(four3);
	\draw (three4)--(four2);
	\draw (three4)--(four4);
	\node at (1,-5) (five1) {$7$};
	\node at (1.5,-5) (five2) {$5$};
	\node at (2.5,-5) (five3) {$6$};
	\node at (3,-5) (five4) {$6$};
	\draw (four1)--(five1);
	\draw (four2)--(five1);
	\draw (four2)--(five4);
	\draw (four2)--(five2);
	\draw (four2)--(five3);
	\draw (four3)--(five3);
	\draw (four4)--(five4);
	\node at (0,-6) (six1) {$7$};
	\node at (2,-6) (six2) {$24$};
	\node at (3,-6) (six3) {$6$};
	\node at (4,-6) (six4) {$6$};
	\draw (five1)--(six1);
	\draw (five1)--(six2);
	\draw (five2)--(six2);
	\draw (five3)--(six2);
	\draw (five3)--(six3);
	\draw (five4)--(six2);
	\draw (five4)--(six4);
	\node at (2,-6.5) {$\vdots$};
\end{tikzpicture}$$
\caption{The Bratelli diagram of the ``2221'' planar algebra}
\end{figure}
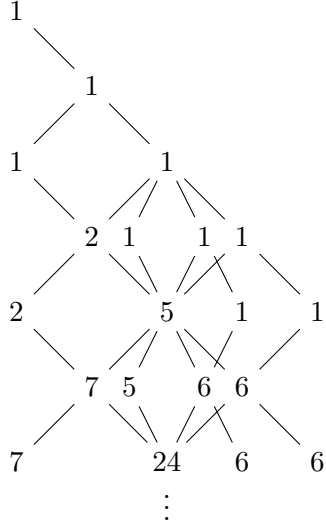

The sum of the squares of the numbers at the kth row gives us the dimension of $\mathcal{P}_k$.  We thus see that 
$\dim (\mathcal{P}_{0,\pm})=1 , \, \dim (\mathcal{P}_1)=1 , \, \dim (\mathcal{P}_2)= 2, \, \dim (\mathcal{P}_3)= 7, \\
 \dim (\mathcal{P}_4)= 31, \, \dim (\mathcal{P}_5)=146 , \, \ldots$

Fact 2.1.5 implies that $\mathcal{P}$ contains a copy of Temperley-Lieb, and Temperley-Lieb is an irreducible annular Temperley-Lieb module: $TL \simeq V^{0,\delta}$.
The sequence of dimensions of $TL$ is given by the $ith$ Catalan number $\frac{1}{i+1}{2i\choose i}$.
The sequence of dimensions of $TL$ is

$\dim (TL_{0,\pm})=1 , \, \dim (TL_1)=1 , \, \dim (TL_2)= 2, \, \dim (TL_3)= 5, \\
 \dim (TL_4)= 14, \, \dim (TL_5)=42 , \, \dim (TL_6)=132 , \ldots$

\noindent Compare the sequences:

$\mathcal{P}:1,1,2,7,31,146,\ldots, \\
\indent TL:1,1,2,5,14,42,\ldots$

\noindent These sequences differ for the first time at level 3.
$\mathcal{P}$ contains no other low weight 0 modules and no low weight 1 or 2 modules.  But $\mathcal{P}$ does contain one low weight 3 module:
$$\dim(V_3^{3,\zeta})=1 , \, \dim(V_4^{3,\zeta})=8 , \, \dim(V_5^{3,\zeta})=45 , \ldots$$
 By counting dimensions we see that $\mathcal{P}_3=TL_3 \oplus 2V_3^{3,\zeta}$.
  Then $\dim(\mathcal{P}_4)=31$ and $\dim(TL_4)=14$ and dim$(V_4^{3,\zeta})=8$.  We want $m$ and $n$ in $\mathbb{N}$ such that $\mathcal{P}_4=TL_4\oplus mV_4^{3,\zeta}\oplus nV_4^{4,\beta}$ where $8m+n=17$.  The only choices for $(m,n)$ are $(0,17),(1,9),\mbox{and}\  (2,1)$.  Thus, $\mathcal{P}$ does contain a low weight 4 module:
  $$\mathcal{P} = (TL\simeq V^{0,\delta}) \oplus mV^{3,\zeta} \oplus nV^{4, \beta} \oplus \cdots$$
  If we continue this process, we see that dim$(\mathcal{P}_5)=146 $ and 
\begin{eqnarray*}
\dim(TL_5\ \oplus mV_5^{3,\zeta} \oplus nV_5^{4,\beta})\\
 = 42+45m+10n\\
 = 212,\mbox{or},177,\mbox{or} 142
\end{eqnarray*}

\noindent This implies that $(m,n)=(2,1)$.\\
So $\mathcal{P}_4=TL_4\oplus 2V_4^{3,\zeta}\oplus V_4^{4,\beta}$ and\\ $\mathcal{P}_5=TL_5\oplus 2V_5^{3,\zeta}\oplus V_5^{4,\beta}\oplus 4V_5^{5,\gamma}$\\
Thus, $$\mathcal{P} = (TL\simeq V^{0,\delta}) \oplus 2V^{3,\zeta} \oplus V^{4, \beta} \oplus 4V^{5, \gamma} \oplus \cdots$$
Let $T_1,T_2$ be the generators of each copy of $V_3^{3,\zeta}$.  Any diagram with some $T_i$'s inside it, and only 6 strands on the outer boundary, is contained in $ATL_3(T_1)\oplus ATL_3(T_2)\oplus TL_3$.\\For instance, if R is a diagram with two $T_i$'s connected by 3 strands, then R $\in$ $ATL_3(T_1)\oplus ATL_3(T_2)\oplus TL_3$.  Let M be the generator of $V_4^{4,\beta}$.  Then any diagram with some $T_i$'s and $M$'s inside it, and only 8 strands on the outer boundary is contained in $ATL_4(M)\oplus ATL_4(T_1)\oplus ATL_4(T_2)\oplus TL_4$.  Let $N_i (i=1,2,3,4)$ be the generators of $V_5^{5,\gamma}$.  Any diagram with some $T_i$'s, $M$'s, and $N_i$'s inside it and only 10 strands on the outer boundary is contained in $\bigoplus_i{ATL_5(T_i)}\oplus ATL_5(M)\bigoplus_i{ATL_5(N_i)}\oplus TL_5$.

\chapter{Potential Generators T and Q for the ``2221'' Planar Algebra}
Since $\mathcal{P} = (TL\simeq V^{0,\delta}) \oplus 2V^{3,\zeta} \oplus V^{4, \beta} \oplus 4V^{5, \gamma} \oplus \cdots$, we will look for two generators T and Q, which suffice to generate the planar algebra.  We want to provide a presentation for the ``2221'' planar algebra.  
\begin{theorem}  There are two elements T,Q $\in PABG(H)_{3,+}$ satisfying the following relations, and the planar algebra generated by $T,Q$ is a subfactor planar algebra with principal graph $H$.
\begin{enumerate}
\item [(1)]$T=T^*, Q=Q^*$
\item [(2)]$\epsilon_i(T)=0 \mbox{ for }i=1,\ldots,6$ and $\epsilon_i(Q)=0 \mbox{ for }i=1,\ldots,6$
\item [(3)]$\rho(T)= T$ and $\rho(Q)=\omega Q$, where $\omega=e^{\frac{2\pi i}{3} }$
\item [(4)]The following 3-box relations hold:

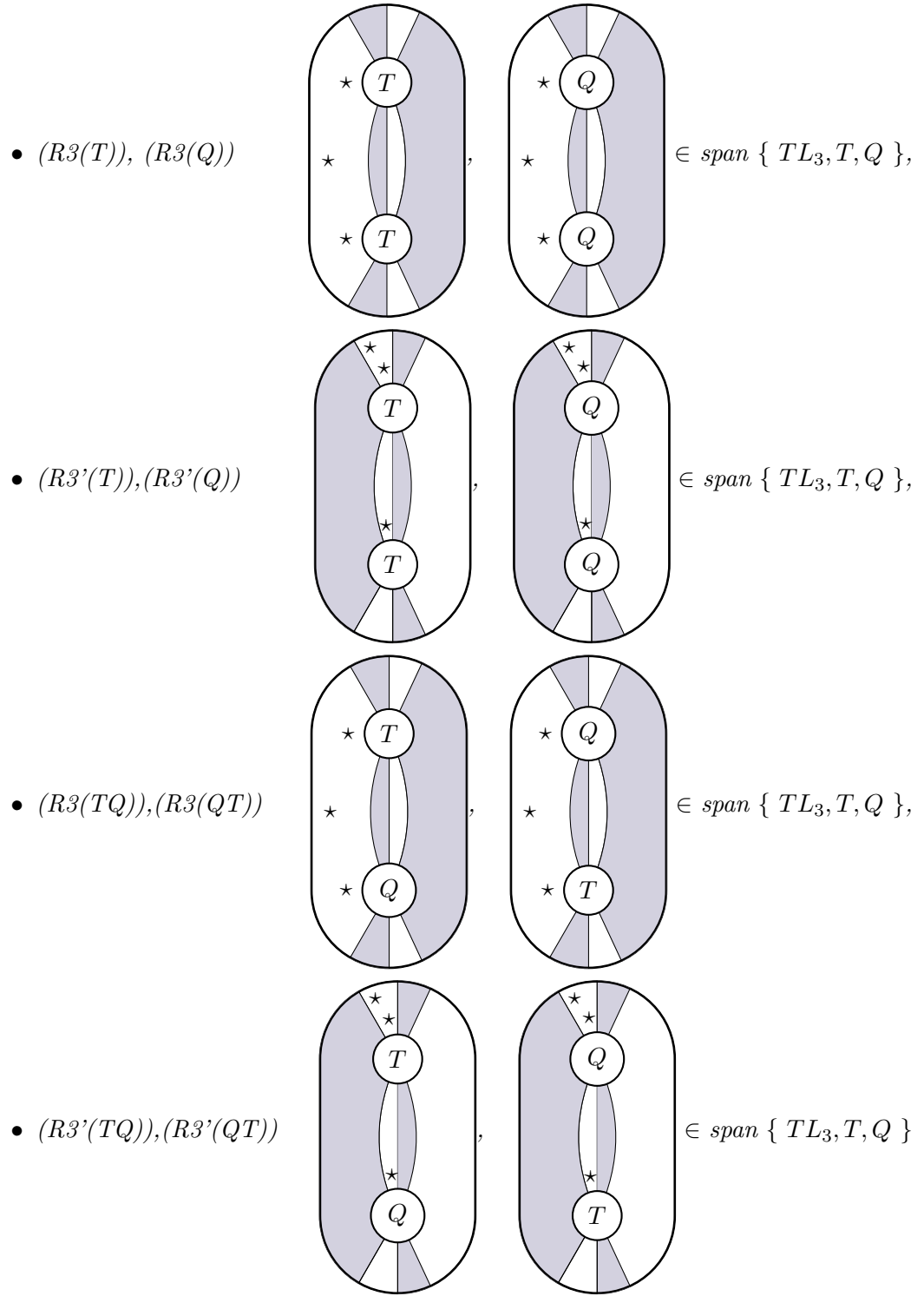
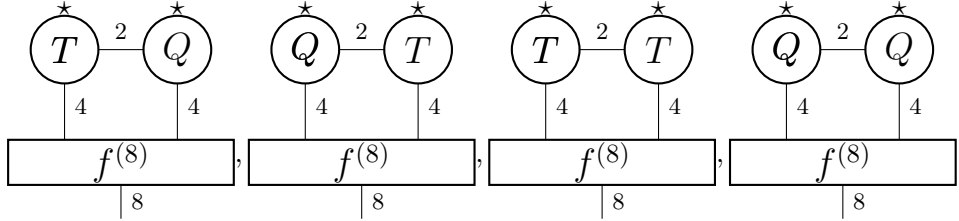
\begin{itemize}
\item (R3(T)), (R3(Q))\hspace{1cm}
$\begin{tikzpicture}[scale=.6, baseline]
	\clip [draw] (2,2) arc (0:180:2cm) -- (-2,-2) arc (-180:0:2cm) -- (2,2);
	
	\draw[shaded] (0,2) .. controls ++(-115:1.5cm) and ++(115:1.5cm) .. (0,-2);
	\draw (0,2) .. controls ++(-90:1.5cm) and ++(90:1.5cm) .. (0,-2);
	\draw (0,2) .. controls ++(-65:1.5cm) and ++(65:1.5cm) .. (0,-2);
	\draw[shaded] (0,2) -- +(90:3cm) -- +(120:3cm) -- (0,2);
	\filldraw[shaded] (0,2) ++(65:5cm) -- (0,2) .. controls ++(-65:1.5cm) and ++(65:1.5cm) .. (0,-2) -- ++(-65:5cm);
	
	\draw[shaded] (0,-2) -- +(-90:3cm) -- +(-120:3cm) -- (0,-2);
	\node at (0,2) [Tbox] (A) {$T$};
	\node at (0,-2) [Tbox] (B) {$T$};
	\node at (A.180) [left] {$\star$};
	\node at (B.180) [left] {$\star$};
	\node at (-1.5,0)  {$\star$};	
	
	\draw[ultra thick] (2,2) arc (0:180:2cm) -- (-2,-2) arc (-180:0:2cm) -- (2,2);
\end{tikzpicture}$, \quad
	$\begin{tikzpicture}[scale=.6, baseline]
	\clip [draw] (2,2) arc (0:180:2cm) -- (-2,-2) arc (-180:0:2cm) -- (2,2);
	
	\draw[shaded] (0,2) .. controls ++(-115:1.5cm) and ++(115:1.5cm) .. (0,-2);
	\draw (0,2) .. controls ++(-90:1.5cm) and ++(90:1.5cm) .. (0,-2);
	\draw (0,2) .. controls ++(-65:1.5cm) and ++(65:1.5cm) .. (0,-2);
	\draw[shaded] (0,2) -- +(90:3cm) -- +(120:3cm) -- (0,2);
	\filldraw[shaded] (0,2) ++(65:5cm) -- (0,2) .. controls ++(-65:1.5cm) and ++(65:1.5cm) .. (0,-2) -- ++(-65:5cm);
	
	\draw[shaded] (0,-2) -- +(-90:3cm) -- +(-120:3cm) -- (0,-2);
	\node at (0,2) [Tbox] (A) {$Q$};
	\node at (0,-2) [Tbox] (B) {$Q$};
	\node at (A.180) [left] {$\star$};
	\node at (B.180) [left] {$\star$};
	\node at (-1.5,0)  {$\star$};	
	
	\draw[ultra thick] (2,2) arc (0:180:2cm) -- (-2,-2) arc (-180:0:2cm) -- (2,2);
\end{tikzpicture}$
$\in$
span \{ $TL_3, T , Q$ \},
\item (R3'(T)),(R3'(Q))\hspace{1cm}
$\begin{tikzpicture}[scale=.6, baseline]
	\clip [draw] (2,2) arc (0:180:2cm) -- (-2,-2) arc (-180:0:2cm) -- (2,2);
	
	\draw (0,2) .. controls ++(-115:1.5cm) and ++(115:1.5cm) .. (0,-2);
	\draw[shaded] (0,2) .. controls ++(-90:1.5cm) and ++(90:1.5cm) .. (0,-2);
	\draw[shaded] (0,2) .. controls ++(-65:1.5cm) and ++(65:1.5cm) .. (0,-2);
	\draw[shaded] (0,-2) -- +(-90:3cm) -- +(-65:3cm) -- (0,-2);
	\draw (0,2) -- +(90:3cm) -- +(120:3cm) -- (0,2);
	\draw[shaded] (0,2) -- +(65:3cm)--+(90:3cm) -- (0,2);
	\filldraw[shaded] (0,-2) ++(-120:4cm) -- (0,-2) .. controls ++(115:1.5cm) and ++(-115:1.5cm) .. (0,2) -- ++(120:4cm);
	\draw (0,-2) -- +(-90:3cm) -- +(-120:3cm) -- (0,-2);
	\node at (0,2) [Tbox] (A) {$T$};
	\node at (0,-2) [Tbox] (B) {$T$};
	\node at (A.75) [above left] {$\star$};
	\node at (B.70) [above left] {$\star$};
	\draw (0,2) ++(115:2cm) node [below right=-.7mm] {$\star$};
	
	\draw[ultra thick] (2,2) arc (0:180:2cm) -- (-2,-2) arc (-180:0:2cm) -- (2,2);
\end{tikzpicture}$, \quad
$\begin{tikzpicture}[scale=.6, baseline]
	\clip [draw] (2,2) arc (0:180:2cm) -- (-2,-2) arc (-180:0:2cm) -- (2,2);
	
	\draw (0,2) .. controls ++(-115:1.5cm) and ++(115:1.5cm) .. (0,-2);
	\draw[shaded] (0,2) .. controls ++(-90:1.5cm) and ++(90:1.5cm) .. (0,-2);
	\draw[shaded] (0,2) .. controls ++(-65:1.5cm) and ++(65:1.5cm) .. (0,-2);
	\draw[shaded] (0,-2) -- +(-90:3cm) -- +(-65:3cm) -- (0,-2);
	\draw (0,2) -- +(90:3cm) -- +(120:3cm) -- (0,2);
	\draw[shaded] (0,2) -- +(65:3cm)--+(90:3cm) -- (0,2);
	\filldraw[shaded] (0,-2) ++(-120:4cm) -- (0,-2) .. controls ++(115:1.5cm) and ++(-115:1.5cm) .. (0,2) -- ++(120:4cm);
	\draw (0,-2) -- +(-90:3cm) -- +(-120:3cm) -- (0,-2);
	\node at (0,2) [Tbox] (A) {$Q$};
	\node at (0,-2) [Tbox] (B) {$Q$};
	\node at (A.75) [above left] {$\star$};
	\node at (B.70) [above left] {$\star$};
	\draw (0,2) ++(115:2cm) node [below right=-.7mm] {$\star$};
	
	\draw[ultra thick] (2,2) arc (0:180:2cm) -- (-2,-2) arc (-180:0:2cm) -- (2,2);
\end{tikzpicture}$
$\in$
span \{ $TL_3, T , Q$ \},
\item (R3(TQ)),(R3(QT))\hspace{.6cm}
$\begin{tikzpicture}[scale=.6, baseline]
	\clip [draw] (2,2) arc (0:180:2cm) -- (-2,-2) arc (-180:0:2cm) -- (2,2);
	
	\draw[shaded] (0,2) .. controls ++(-115:1.5cm) and ++(115:1.5cm) .. (0,-2);
	\draw (0,2) .. controls ++(-90:1.5cm) and ++(90:1.5cm) .. (0,-2);
	\draw (0,2) .. controls ++(-65:1.5cm) and ++(65:1.5cm) .. (0,-2);
	\draw[shaded] (0,2) -- +(90:3cm) -- +(120:3cm) -- (0,2);
	\filldraw[shaded] (0,2) ++(65:5cm) -- (0,2) .. controls ++(-65:1.5cm) and ++(65:1.5cm) .. (0,-2) -- ++(-65:5cm);
	
	\draw[shaded] (0,-2) -- +(-90:3cm) -- +(-120:3cm) -- (0,-2);
	\node at (0,2) [Tbox] (A) {$T$};
	\node at (0,-2) [Tbox] (B) {$Q$};
	\node at (A.180) [left] {$\star$};
	\node at (B.180) [left] {$\star$};
	\node at (-1.5,0)  {$\star$};	
	
	\draw[ultra thick] (2,2) arc (0:180:2cm) -- (-2,-2) arc (-180:0:2cm) -- (2,2);
\end{tikzpicture}$, \quad
	$\begin{tikzpicture}[scale=.6, baseline]
	\clip [draw] (2,2) arc (0:180:2cm) -- (-2,-2) arc (-180:0:2cm) -- (2,2);
	
	\draw[shaded] (0,2) .. controls ++(-115:1.5cm) and ++(115:1.5cm) .. (0,-2);
	\draw (0,2) .. controls ++(-90:1.5cm) and ++(90:1.5cm) .. (0,-2);
	\draw (0,2) .. controls ++(-65:1.5cm) and ++(65:1.5cm) .. (0,-2);
	\draw[shaded] (0,2) -- +(90:3cm) -- +(120:3cm) -- (0,2);
	\filldraw[shaded] (0,2) ++(65:5cm) -- (0,2) .. controls ++(-65:1.5cm) and ++(65:1.5cm) .. (0,-2) -- ++(-65:5cm);
	
	\draw[shaded] (0,-2) -- +(-90:3cm) -- +(-120:3cm) -- (0,-2);
	\node at (0,2) [Tbox] (A) {$Q$};
	\node at (0,-2) [Tbox] (B) {$T$};
	\node at (A.180) [left] {$\star$};
	\node at (B.180) [left] {$\star$};
	\node at (-1.5,0)  {$\star$};	
	
	\draw[ultra thick] (2,2) arc (0:180:2cm) -- (-2,-2) arc (-180:0:2cm) -- (2,2);
\end{tikzpicture}$
$\in$
span \{ $TL_3, T , Q$ \},
\item (R3'(TQ)),(R3'(QT))\hspace{.5cm}
$\begin{tikzpicture}[scale=.6, baseline]
	\clip [draw] (2,2) arc (0:180:2cm) -- (-2,-2) arc (-180:0:2cm) -- (2,2);
	
	\draw (0,2) .. controls ++(-115:1.5cm) and ++(115:1.5cm) .. (0,-2);
	\draw[shaded] (0,2) .. controls ++(-90:1.5cm) and ++(90:1.5cm) .. (0,-2);
	\draw[shaded] (0,2) .. controls ++(-65:1.5cm) and ++(65:1.5cm) .. (0,-2);
	\draw[shaded] (0,-2) -- +(-90:3cm) -- +(-65:3cm) -- (0,-2);
	\draw (0,2) -- +(90:3cm) -- +(120:3cm) -- (0,2);
	\draw[shaded] (0,2) -- +(65:3cm)--+(90:3cm) -- (0,2);
	\filldraw[shaded] (0,-2) ++(-120:4cm) -- (0,-2) .. controls ++(115:1.5cm) and ++(-115:1.5cm) .. (0,2) -- ++(120:4cm);
	\draw (0,-2) -- +(-90:3cm) -- +(-120:3cm) -- (0,-2);
	\node at (0,2) [Tbox] (A) {$T$};
	\node at (0,-2) [Tbox] (B) {$Q$};
	\node at (A.75) [above left] {$\star$};
	\node at (B.70) [above left] {$\star$};
	\draw (0,2) ++(115:2cm) node [below right=-.7mm] {$\star$};
	
	\draw[ultra thick] (2,2) arc (0:180:2cm) -- (-2,-2) arc (-180:0:2cm) -- (2,2);
\end{tikzpicture}$, \quad
$\begin{tikzpicture}[scale=.6, baseline]
	\clip [draw] (2,2) arc (0:180:2cm) -- (-2,-2) arc (-180:0:2cm) -- (2,2);
	
	\draw (0,2) .. controls ++(-115:1.5cm) and ++(115:1.5cm) .. (0,-2);
	\draw[shaded] (0,2) .. controls ++(-90:1.5cm) and ++(90:1.5cm) .. (0,-2);
	\draw[shaded] (0,2) .. controls ++(-65:1.5cm) and ++(65:1.5cm) .. (0,-2);
	\draw[shaded] (0,-2) -- +(-90:3cm) -- +(-65:3cm) -- (0,-2);
	\draw (0,2) -- +(90:3cm) -- +(120:3cm) -- (0,2);
	\draw[shaded] (0,2) -- +(65:3cm)--+(90:3cm) -- (0,2);
	\filldraw[shaded] (0,-2) ++(-120:4cm) -- (0,-2) .. controls ++(115:1.5cm) and ++(-115:1.5cm) .. (0,2) -- ++(120:4cm);
	\draw (0,-2) -- +(-90:3cm) -- +(-120:3cm) -- (0,-2);
	\node at (0,2) [Tbox] (A) {$Q$};
	\node at (0,-2) [Tbox] (B) {$T$};
	\node at (A.75) [above left] {$\star$};
	\node at (B.70) [above left] {$\star$};
	\draw (0,2) ++(115:2cm) node [below right=-.7mm] {$\star$};
	
	\draw[ultra thick] (2,2) arc (0:180:2cm) -- (-2,-2) arc (-180:0:2cm) -- (2,2);
\end{tikzpicture}$
$\in$
span \{ $TL_3, T , Q$ \}
\end{itemize}
\item [(5)]Each of 
$$\begin{tikzpicture}[baseline=0,scale=1.5]
	{%
	\draw (-0.5,1) -- (-0.5,0);
	\node[anchor=west] at (-0.5,0.5) {\footnotesize$4$};
	\draw (0.5,1) -- (0.5,0);
	\node[anchor=west] at (0.5,0.5) {\footnotesize$4$};
}
 {%
	\node[anchor=south] at (0,1) {\footnotesize$2$};
	\draw (-0.5,1) -- (0.5, 1);
	\foreach \x in {-0.5,0.5} {
		{%
	\filldraw[fill=white,thick] (\x,1) ellipse (3mm and 3mm);
	\node at (-.5,1) {\Large $T$};
	\node at (.5,1) {\Large $Q$};
	\path(\x,1) ++(90:0.37) node {$\star$};
}
	}
}
        \draw (0,0)--(0,-0.5);
        \node[anchor=west] at (0,-0.35) {\footnotesize$8$};
	{%
	\filldraw[fill=white,thick] (-1,-0.2) rectangle (1,0.2);
	\node at (0,0) {\Large$\JW{8}$};
}
\end{tikzpicture}, \begin{tikzpicture}[baseline=0,scale=1.5]
	{%
	\draw (-0.5,1) -- (-0.5,0);
	\node[anchor=west] at (-0.5,0.5) {\footnotesize$4$};
	\draw (0.5,1) -- (0.5,0);
	\node[anchor=west] at (0.5,0.5) {\footnotesize$4$};
}
 {%
	\node[anchor=south] at (0,1) {\footnotesize$2$};
	\draw (-0.5,1) -- (0.5, 1);
	\foreach \x in {-0.5,0.5} {
		{%
	\filldraw[fill=white,thick] (\x,1) ellipse (3mm and 3mm);
	\node at (-.5,1) {\Large $Q$};
	\node at (.5,1) {\Large $T$};
	\path(\x,1) ++(90:0.37) node {$\star$};
}
	}
}
        \draw (0,0)--(0,-0.5);
        \node[anchor=west] at (0,-0.35) {\footnotesize$8$};
	{%
	\filldraw[fill=white,thick] (-1,-0.2) rectangle (1,0.2);
	\node at (0,0) {\Large$\JW{8}$};
}
\end{tikzpicture}, \begin{tikzpicture}[baseline=0,scale=1.5]
	{%
	\draw (-0.5,1) -- (-0.5,0);
	\node[anchor=west] at (-0.5,0.5) {\footnotesize$4$};
	\draw (0.5,1) -- (0.5,0);
	\node[anchor=west] at (0.5,0.5) {\footnotesize$4$};
}
 {%
	\node[anchor=south] at (0,1) {\footnotesize$2$};
	\draw (-0.5,1) -- (0.5, 1);
	\foreach \x in {-0.5,0.5} {
		{%
	\filldraw[fill=white,thick] (\x,1) ellipse (3mm and 3mm);
	\node at (-.5,1) {\Large $T$};
	\node at (.5,1) {\Large $T$};
	\path(\x,1) ++(90:0.37) node {$\star$};
}
	}
}
        \draw (0,0)--(0,-0.5);
        \node[anchor=west] at (0,-0.35) {\footnotesize$8$};
	{%
	\filldraw[fill=white,thick] (-1,-0.2) rectangle (1,0.2);
	\node at (0,0) {\Large$\JW{8}$};
}
\end{tikzpicture}, \begin{tikzpicture}[baseline=0,scale=1.5]
	{%
	\draw (-0.5,1) -- (-0.5,0);
	\node[anchor=west] at (-0.5,0.5) {\footnotesize$4$};
	\draw (0.5,1) -- (0.5,0);
	\node[anchor=west] at (0.5,0.5) {\footnotesize$4$};
}
 {%
	\node[anchor=south] at (0,1) {\footnotesize$2$};
	\draw (-0.5,1) -- (0.5, 1);
	\foreach \x in {-0.5,0.5} {
		{%
	\filldraw[fill=white,thick] (\x,1) ellipse (3mm and 3mm);
	\node at (-.5,1) {\Large $Q$};
	\node at (.5,1) {\Large $Q$};
	\path(\x,1) ++(90:0.37) node {$\star$};
}
	}
}
        \draw (0,0)--(0,-0.5);
        \node[anchor=west] at (0,-0.35) {\footnotesize$8$};
	{%
	\filldraw[fill=white,thick] (-1,-0.2) rectangle (1,0.2);
	\node at (0,0) {\Large$\JW{8}$};
}
\end{tikzpicture}$$ lies in the span of 

$$\Big( 
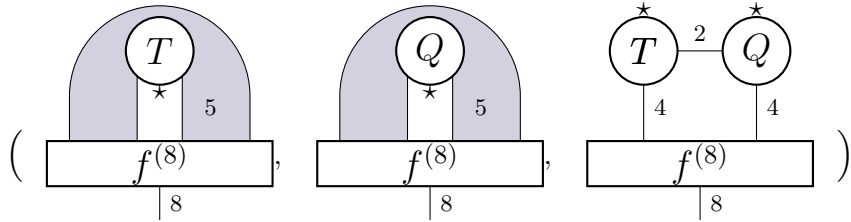
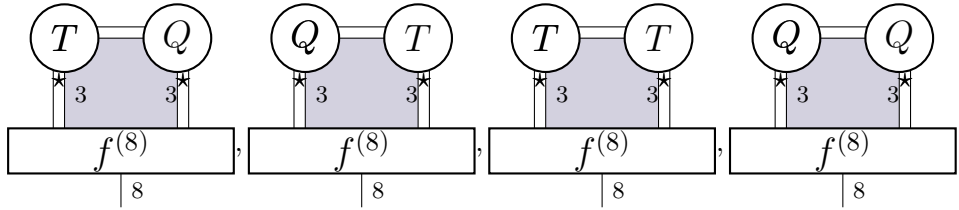
\hspace{3mm}\begin{tikzpicture}[baseline=0,scale=1.5]
	{
	\fill[shaded] (-0.8,0) -- (-0.8,0.6) arc (180:0:0.8) -- (0.8,0) -- (0.2,0) -- (0.2,1) -- (-0.2,1) -- (-0.2,0);
	\draw (-0.8,0) -- (-0.8,0.6) arc (180:0:0.8) -- (0.8,0);
	\draw (-0.2,0) -- (-0.2,1);
	\draw (0.2,0) -- (0.2,1);
	\node at (0.45,0.5) {\footnotesize$5$};
	{%
	\filldraw[fill=white,thick] (0,1) ellipse (3mm and 3mm);
	\node at (0,1) {\Large $T$};
	\path(0,1) ++(-90:0.37) node {$\star$};
}
}
;
        \draw (0,0)--(0,-0.5);
        \node[anchor=west] at (0,-0.35) {\footnotesize$8$};
	{%
	\filldraw[fill=white,thick] (-1,-0.2) rectangle (1,0.2);
	\node at (0,0) {\Large$\JW{8}$};
}\end{tikzpicture},
\quad \begin{tikzpicture}[baseline=0,scale=1.5]
	{
	\fill[shaded] (-0.8,0) -- (-0.8,0.6) arc (180:0:0.8) -- (0.8,0) -- (0.2,0) -- (0.2,1) -- (-0.2,1) -- (-0.2,0);
	\draw (-0.8,0) -- (-0.8,0.6) arc (180:0:0.8) -- (0.8,0);
	\draw (-0.2,0) -- (-0.2,1);
	\draw (0.2,0) -- (0.2,1);
	\node at (0.45,0.5) {\footnotesize$5$};
	{%
	\filldraw[fill=white,thick] (0,1) ellipse (3mm and 3mm);
	\node at (0,1) {\Large $Q$};
	\path(0,1) ++(-90:0.37) node {$\star$};
}
}
;
        \draw (0,0)--(0,-0.5);
        \node[anchor=west] at (0,-0.35) {\footnotesize$8$};
	{%
	\filldraw[fill=white,thick] (-1,-0.2) rectangle (1,0.2);
	\node at (0,0) {\Large$\JW{8}$};
}
\end{tikzpicture} ,\quad \begin{tikzpicture}[baseline=0,scale=1.5]
	{%
	\draw (-0.5,1) -- (-0.5,0);
	\node[anchor=west] at (-0.5,0.5) {\footnotesize$4$};
	\draw (0.5,1) -- (0.5,0);
	\node[anchor=west] at (0.5,0.5) {\footnotesize$4$};
}
 {%
	\node[anchor=south] at (0,1) {\footnotesize$2$};
	\draw (-0.5,1) -- (0.5, 1);
	\foreach \x in {-0.5,0.5} {
		{%
	\filldraw[fill=white,thick] (\x,1) ellipse (3mm and 3mm);
	\node at (-.5,1) {\Large $T$};
	\node at (.5,1) {\Large $Q$};
	\path(\x,1) ++(90:0.37) node {$\star$};
}
	}
}
        \draw (0,0)--(0,-0.5);
        \node[anchor=west] at (0,-0.35) {\footnotesize$8$};
	{%
	\filldraw[fill=white,thick] (-1,-0.2) rectangle (1,0.2);
	\node at (0,0) {\Large$\JW{8}$};
}
\end{tikzpicture} \hspace{3mm}\Big)$$
\item [(6)]Each of 
$$\begin{tikzpicture}[baseline=0,scale=1.5]
	{%
	\fill[shaded] (-0.5,0) rectangle (0.5,1);
	\draw (-0.5,1) -- (-0.5,0);
	\node[anchor=west] at (-0.5,0.5) {\footnotesize$3$};
	\draw (-0.6,1) -- (-0.6,0);
	\draw (0.5,1) -- (0.5,0);
	\node[anchor=west] at (0.3,0.5) {\footnotesize$3$};
	\draw (0.6,1) -- (0.6,0);
}
 {%
	
	\draw (-0.5,1) -- (0.5, 1);
	\draw (-0.5,1.1) -- (0.5,1.1);
	\foreach \x in {-0.5,0.5} {
		{%
	\filldraw[fill=white,thick] (\x,1) ellipse (3mm and 3mm);
	\node at (-.5,1) {\Large $T$};
	\node at (.5,1) {\Large $Q$};
	\path(-.55,1) ++(-90:0.37) node {$\star$};
	\path(0.55,1) ++(-90:0.37) node {$\star$};
}
	}
}
        \draw (0,0)--(0,-0.5);
        \node[anchor=west] at (0,-0.35) {\footnotesize$8$};
	{%
	\filldraw[fill=white,thick] (-1,-0.2) rectangle (1,0.2);
	\node at (0,0) {\Large$\JW{8}$};
}
\end{tikzpicture}, \begin{tikzpicture}[baseline=0,scale=1.5]
	{%
	\fill[shaded] (-0.5,0) rectangle (0.5,1);
	\draw (-0.5,1) -- (-0.5,0);
	\node[anchor=west] at (-0.5,0.5) {\footnotesize$3$};
	\draw (-0.6,1) -- (-0.6,0);
	\draw (0.5,1) -- (0.5,0);
	\node[anchor=west] at (0.3,0.5) {\footnotesize$3$};
	\draw (0.6,1) -- (0.6,0);
}
 {%
	
	\draw (-0.5,1) -- (0.5, 1);
	\draw (-0.5,1.1) -- (0.5,1.1);
	\foreach \x in {-0.5,0.5} {
		{%
	\filldraw[fill=white,thick] (\x,1) ellipse (3mm and 3mm);
	\node at (-.5,1) {\Large $Q$};
	\node at (.5,1) {\Large $T$};
	\path(-.55,1) ++(-90:0.37) node {$\star$};
	\path(0.55,1) ++(-90:0.37) node {$\star$};
}
	}
}
        \draw (0,0)--(0,-0.5);
        \node[anchor=west] at (0,-0.35) {\footnotesize$8$};
	{%
	\filldraw[fill=white,thick] (-1,-0.2) rectangle (1,0.2);
	\node at (0,0) {\Large$\JW{8}$};
}
\end{tikzpicture}, \begin{tikzpicture}[baseline=0,scale=1.5]
	{%
	\fill[shaded] (-0.5,0) rectangle (0.5,1);
	\draw (-0.5,1) -- (-0.5,0);
	\node[anchor=west] at (-0.5,0.5) {\footnotesize$3$};
	\draw (-0.6,1) -- (-0.6,0);
	\draw (0.5,1) -- (0.5,0);
	\node[anchor=west] at (0.3,0.5) {\footnotesize$3$};
	\draw (0.6,1) -- (0.6,0);
}
 {%
	
	\draw (-0.5,1) -- (0.5, 1);
	\draw (-0.5,1.1) -- (0.5,1.1);
	\foreach \x in {-0.5,0.5} {
		{%
	\filldraw[fill=white,thick] (\x,1) ellipse (3mm and 3mm);
	\node at (-.5,1) {\Large $T$};
	\node at (.5,1) {\Large $T$};
	\path(-.55,1) ++(-90:0.37) node {$\star$};
	\path(0.55,1) ++(-90:0.37) node {$\star$};
}
	}
}
        \draw (0,0)--(0,-0.5);
        \node[anchor=west] at (0,-0.35) {\footnotesize$8$};
	{%
	\filldraw[fill=white,thick] (-1,-0.2) rectangle (1,0.2);
	\node at (0,0) {\Large$\JW{8}$};
}
\end{tikzpicture}, \begin{tikzpicture}[baseline=0,scale=1.5]
	{%
	\fill[shaded] (-0.5,0) rectangle (0.5,1);
	\draw (-0.5,1) -- (-0.5,0);
	\node[anchor=west] at (-0.5,0.5) {\footnotesize$3$};
	\draw (-0.6,1) -- (-0.6,0);
	\draw (0.5,1) -- (0.5,0);
	\node[anchor=west] at (0.3,0.5) {\footnotesize$3$};
	\draw (0.6,1) -- (0.6,0);
}
 {%
	
	\draw (-0.5,1) -- (0.5, 1);
	\draw (-0.5,1.1) -- (0.5,1.1);
	\foreach \x in {-0.5,0.5} {
		{%
	\filldraw[fill=white,thick] (\x,1) ellipse (3mm and 3mm);
	\node at (-.5,1) {\Large $Q$};
	\node at (.5,1) {\Large $Q$};
	\path(-.55,1) ++(-90:0.37) node {$\star$};
	\path(0.55,1) ++(-90:0.37) node {$\star$};
}
	}
}
        \draw (0,0)--(0,-0.5);
        \node[anchor=west] at (0,-0.35) {\footnotesize$8$};
	{%
	\filldraw[fill=white,thick] (-1,-0.2) rectangle (1,0.2);
	\node at (0,0) {\Large$\JW{8}$};
}
\end{tikzpicture}$$ lies in the span of 

$$\Big( 
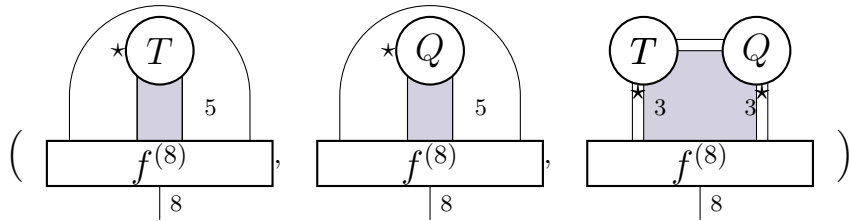
\hspace{3mm}\begin{tikzpicture}[baseline=0,scale=1.5]
	{\draw[shaded] (0.2,0) -- (0.2,1) -- (-0.2,1) -- (-0.2,0);
	\draw (-0.8,0) -- (-0.8,0.6) arc (180:0:0.8) -- (0.8,0);
	\draw(-0.2,0) -- (-0.2,1);
	\draw (0.2,0) -- (0.2,1);
	\node at (0.45,0.5) {\footnotesize$5$};
	{%
	\filldraw[fill=white,thick] (0,1) ellipse (3mm and 3mm);
	\node at (0,1) {\Large $T$};
	\path(0,1) ++(-180:0.37) node {$\star$};
}
}
;
        \draw (0,0)--(0,-0.5);
        \node[anchor=west] at (0,-0.35) {\footnotesize$8$};
	{%
	\filldraw[fill=white,thick] (-1,-0.2) rectangle (1,0.2);
	\node at (0,0) {\Large$\JW{8}$};
}\end{tikzpicture},
\quad \begin{tikzpicture}[baseline=0,scale=1.5]
	{\draw[shaded] (0.2,0) -- (0.2,1) -- (-0.2,1) -- (-0.2,0);
	\draw (-0.8,0) -- (-0.8,0.6) arc (180:0:0.8) -- (0.8,0);
	\draw(-0.2,0) -- (-0.2,1);
	\draw (0.2,0) -- (0.2,1);
	\node at (0.45,0.5) {\footnotesize$5$};
	{%
	\filldraw[fill=white,thick] (0,1) ellipse (3mm and 3mm);
	\node at (0,1) {\Large $Q$};
	\path(0,1) ++(-180:0.37) node {$\star$};
}
}
;
        \draw (0,0)--(0,-0.5);
        \node[anchor=west] at (0,-0.35) {\footnotesize$8$};
	{%
	\filldraw[fill=white,thick] (-1,-0.2) rectangle (1,0.2);
	\node at (0,0) {\Large$\JW{8}$};
}\end{tikzpicture} ,\quad \begin{tikzpicture}[baseline=0,scale=1.5]
	{%
	\fill[shaded] (-0.5,0) rectangle (0.5,1);
	\draw (-0.5,1) -- (-0.5,0);
	\node[anchor=west] at (-0.5,0.5) {\footnotesize$3$};
	\draw (-0.6,1) -- (-0.6,0);
	\draw (0.5,1) -- (0.5,0);
	\node[anchor=west] at (0.3,0.5) {\footnotesize$3$};
	\draw (0.6,1) -- (0.6,0);
}
 {%
	
	\draw (-0.5,1) -- (0.5, 1);
	\draw (-0.5,1.1) -- (0.5,1.1);
	\foreach \x in {-0.5,0.5} {
		{%
	\filldraw[fill=white,thick] (\x,1) ellipse (3mm and 3mm);
	\node at (-.5,1) {\Large $T$};
	\node at (.5,1) {\Large $Q$};
	\path(-.55,1) ++(-90:0.37) node {$\star$};
	\path(0.55,1) ++(-90:0.37) node {$\star$};
}
	}
}
        \draw (0,0)--(0,-0.5);
        \node[anchor=west] at (0,-0.35) {\footnotesize$8$};
	{%
	\filldraw[fill=white,thick] (-1,-0.2) rectangle (1,0.2);
	\node at (0,0) {\Large$\JW{8}$};
}
\end{tikzpicture} \hspace{3mm}\Big)$$
\end{enumerate}
\end{theorem}
We will prove this theorem in the remainder of this paper.  In this chapter, we identify the two generators $T$ and $Q$ satisfying (1), (2), and (3) above.  In chapter 5, we prove the quadratic relations in (4), (5), and (6).  In chapter 6, we obtain some important one-strand braiding substitutes.  In chapter 7, we show that the planar algebra has the subfactor property.  Finally, in chapter 8 we show that our planar algebra has principal graph $H$.

\indent $PABG(H)_{3,+}$ has basis consisting of all even-based loops of length 6 on the graph H=``2221''.  Let $X^{3,\zeta}$ be the subspace of $PABG(H)_{3,+}$ consisting of elements satisfying the relations
$$\epsilon_i(T)=0 \mbox{ for }i=1,\ldots,6, \quad \rho(T)=\zeta T.$$

\begin{theorem}
$X^{3,1}$ is five-dimensional and $X^{3,\omega}$ and $X^{3,\omega^2}$ are six-dimensional, where $\omega$ is the $3$-th root of unity $e^{\frac{2\pi i}{3}}$.
\end{theorem}
\begin{proof}We decompose $PABG(H)$ into irreducible annular Temperley-Lieb modules.  Let $\Lambda$ be the even-odd adjacency matrix of H.  
\newline The eigenvalues of $\Lambda\Lambda^t$ and $\Lambda^t\Lambda$ are $\{1,1,\frac{5+\sqrt{21}}{2},\frac{5-\sqrt{21}}{2}\}$ and $ \{1,1,\frac{5+\sqrt{21}}{2},\frac{5-\sqrt{21}}{2}\}$.  The eigenspace of $\Lambda\Lambda^t$ corresponding to eigenvalue $1$ is two-dimensional because the multiplicity is $2$.  There is a one-dimensional eigenspace of $\Lambda\Lambda^t$ corresponding to eigenvalue $\frac{5-\sqrt{21}}{2}$.  So far, we have a two-dimensional eigenspace and a one-dimensional eigenspace corresponding to eigenvalues $\mu^2$ ,where $\mu^2\in(0,\delta^2)$.  Finally, there is a one-dimensional eigenspace corresponding to eigenvalue $\frac{5+\sqrt{21}}{2}=\delta^2$.  So $$PABG(H)_{0,+}\supseteq V_{+}^{0,\delta}\oplus 2V_{+}^{0,1}\oplus V_{+}^{0,\frac{5-\sqrt{21}}{2}}$$
$$=V_{+}^{0,\delta}\oplus 3V_{+}^{0,\mu}$$
Since $dim PABG(H)_{0,+}=4$ and $dim V_{+}^{0,\delta}=1$ and $dim V_{+}^{0,\mu}=1$, we have that $dim (V_{+}^{0,\delta}\oplus 3V_{+}^{0,\mu})=4$ and hence $$PABG(H)_{0,+}=V_{+}^{0,\delta}\oplus 3V_{+}^{0,\mu}.$$
Similarly, $PABG(H)_{0,-}\supseteq V_{-}^{0,\delta}\oplus 3V_{-}^{0,\mu}$.  We know that $dim V_{-}^{0,\delta}=1$ and $dim V_{-}^{0,\mu}=1$, and hence $dim (V_{-}^{0,\delta}\oplus 3V_{-}^{0,\mu})=4$.  Since $dim PABG(H)_{0,-}=4$, we must have that $$PABG(H)_{0,-}=V_{-}^{0,\delta}\oplus 3V_{-}^{0,\mu}.$$ 
We then deduce that $$PABG(H)_{1}\supset V_{1}^{0,\delta}\oplus 3V_{1}^{0,\mu}.$$
Counting length-2 loops based at even vertices, we have 7 of them.\\So $dim (PABG(H)_1)=7.$ Theorem 2.3.6 implies that $dim(V_{1}^{0,\delta})=1$ and $dim(V_{1}^{0,\mu})=2$.  So $dim(V_{1}^{0,\delta}\oplus 3V_{1}^{0,\mu})=7$.  We deduce that $$PABG(H)_{1}= V_{1}^{0,\delta}\oplus 3V_{1}^{0,\mu}.$$
Next, we get that $$PABG(H)_{2}\supset V_{2}^{0,\delta}\oplus 3V_{2}^{0,\mu}.$$
Counting length-4 loops based at even vertices, we have 25 of them.  Hence, $dim PABG(H)_{2}=25$.  Theorem 2.3.6 tells us that $dim(V_{2}^{0,\delta}\oplus 3V_{2}^{0,\mu})=20$.  Thus, $PABG(H)$ must contain a copy of $V_{2}^{2,\beta}$.  In particular, there must be 5 modules of the form $V^{2,\beta}$.  
To consider how many copies of each $V_{2}^{2,\beta}$ there are, we consider the action of $\rho$ on $PABG(H)_{2}$ and on the irreducible modules:  Let $\mathbf{l}$=$l_1l_2l_3l_4$  Then $\rho(l_1l_2l_3l_4)=l_3l_4l_1l_2$.  Hence $l_1l_2l_3l_4$ is a fixed point iff $l_1=l_3$, $l_2=l_4$; i.e., iff $l_1l_2l_3l_4=l_1l_2l_1l_2$.  On the basis of $PABG(H)_{2}$, consisting of even-based loops, we have that $z_ib_iz_ib_i, cb_icb_i (i=0,1,2)$, and $cdcd$ are fixed points.  Furthermore, the pair $\{z_ib_icb_i,cb_iz_ib_i\}$ forms an orbit-pair (i=0,1,2).  
Also, $$\{cb_1cb_0,cb_0cb_1c\},\{cb_1cb_2,cb_2cb_1c\},\{cb_1cd,cdcb_1\}
,\{cb_0cd,cdcb_0\},\{cb_0cb_2,cb_2cb_0\},\{cdcb_2,cb_2cd\}$$  form orbit-pairs.  Thus, we have 7 fixed points and 9 orbit-pairs.  So $\rho$ has eigenvalue 1 for 16 elements and eigenvalue -1 for 9 elements. (Note: $PABG(H)_2$ is 25-dimensional.)
Now, consider the action of $\rho$ on $V_{2}^{0,\delta}\oplus 3V_{2}^{0,\mu}$.  For $V^{0,\delta}$, the basis consists of Temperley-Lieb pictures with four boundary points.  For $V^{0,\mu}$, the basis consists of $TL_2$ pictures with the generator T of $V^{0,\mu}$ in one of the four regions.  On the basis of $V^{0,\delta}$, $\rho$ has 2 fixed points.  On the basis of $V^{0,\mu}$, $\rho$ has 2 fixed points and 2 orbits.  Thus, there are 8 fixed points and 6 orbits total on $V_{2}^{0,\delta}\oplus 3V_{2}^{0,\mu}$, and therefore $\rho$ has eigenvalue 1 with multiplicity 14 and eigenvalue -1 with  multiplicity 6.  Since $\rho$ has eigenvalue 1 with multiplicity 16 and eigenvalue -1 with multiplicity 9, on $PABG(H)_2$, we must have 2 copies of $V_{2}^{2,1}$ and 3 copies of $V_{2}^{2,-1}$.  
Therefore, $$PABG(H)_2=V_{2}^{0,\delta}\oplus 3V_{2}^{0,\mu}\oplus 2V_{2}^{2,1}\oplus 3V_{2}^{2,-1}$$
We deduce that $$PABG(H)_3\supset V_{3}^{0,\delta}\oplus 3V_{3}^{0,\mu}\oplus 2V_{3}^{2,1}\oplus 3V_{3}^{2,-1}$$
Count the  number of length-6 loops based at even vertices.  We find that there are 112 loops of length 6.  Hence, $dim PABG(H)_3= 112$.  But $dim (V_{3}^{0,\delta}\oplus 3V_{3}^{0,\mu}\oplus 2V_{3}^{2,1}\oplus 3V_{3}^{2,-1})=95$.  This implies that there must be 17 modules of the form $V^{3,\zeta}$ in $PABG(H)_3$.  Let 1, $\omega$, $\omega^2$ be the three 3th roots of unity.  To find the  multiplicites of $V^{3,1},V^{3,\omega},V^{3,\omega^2}$ in $PABG(H)_3$, we compute the action of $\rho$ on $PABG(H)_3$ and on its known submodules.
If $\mathbf{l}$=$l_1l_2l_3l_4l_5l_6$, then $l_1l_2l_3l_4l_5l_6$ is fixed iff $l_1=l_3=l_5$ and $l_2=l_4=l_6$.  Further, $$\{l_1l_2l_3l_4l_5l_6,\sqrt{\frac{\lambda(l_1)\lambda(l_4)}{\lambda(l_5)\lambda(l_2)}}l_5l_6l_1l_2l_3l_4, \sqrt{\frac{\lambda(l_1)\lambda(l_4)}{\lambda(l_3)\lambda(l_6)}}l_3l_4l_5l_6l_1l_2\}$$ forms a 3-orbit.  We have 7 fixed points and 35 3-orbits.  If $\{\mathbf{x},\mathbf{y},\mathbf{z}\}$ is an orbit, then $\mathbf{x}+\mathbf{y}+\mathbf{z}$ is a fixed point.  Also, $\omega\mathbf{x}+\mathbf{y}+\omega^2\mathbf{z}$ has eigenvalue $\omega$ and $\omega^2\mathbf{x}+\mathbf{y}+\omega\mathbf{z}$ has eigenvalue $\omega^2$.  
Thus, $\rho$ has eigenvalue 1 with multiplicity 42 and eigenvalues $\omega$ and $\omega^2$ with multiplicity 35 each.  On $V_{3}^{0,\delta}\oplus 3V_{3}^{0,\mu}\oplus 2V_{3}^{2,1}\oplus 3V_{3}^{2,-1}$, we pick a basis on which $\rho$ acts by permutation.  $V_{3}^{0,\delta}$ has $TL_3$ pictures as basis  ($dim V_{3}^{0,\delta}=5$).  $V_{3}^{0,\mu}$ has, as basis, $TL_3$ pictures with a generator inserted in all regions ($dim V_{3}^{0,\mu}=20$).  $V_{3}^{2,-1}$ has, as basis, annular tangles with 4 strands connecting the generator to the outer boundary, and two points on the outer boundary connected to each other  ($dim V_{3}^{2,-1}=6$).  $V_{3}^{2,1}$ is similar except the generator $T_0$ has eigenvalue +1.  For $V_{3}^{0,\delta}$, we have 1 orbit and 2 fixed points.  For $V_{3}^{0,\mu}$, we have 6 orbits and 2 fixed points.  For $V_{3}^{2,-1}$, we have 2 orbits.  Similarly, $V_{3}^{2,1}$ has 2 orbits.  In total, we have for $V_{3}^{0,\delta}\oplus 3V_{3}^{0,\mu}\oplus 2V_{3}^{2,1}\oplus 3V_{3}^{2,-1}$ 29 orbits and 8 fixed points.  Thus, $\rho$ has eigenvalues $\{1,\omega,\omega^2\}$ with multiplicities \{37,29,29\}.
Recall that, on $PABG(H)_3$, $\rho$ has eigenvalues $\{1,\omega,\omega^2\}$ with multiplicities \{42,35,35\}.  Comparing the two sequences, we see that the multiplicities of $V^{3,1},V^{3,\omega},V^{3,\omega^2}$ in $PABG(H)$ are 5,6,6.  
Hence, $$PABG(H)_3=V_{3}^{0,\delta}\oplus 3V_{3}^{0,\mu}\oplus 2V_{3}^{2,1}\oplus 3V_{3}^{2,-1}\oplus5V_{3}^{3,1}\oplus6V_{3}^{3,\omega}\oplus6V_{3}^{3,\omega^2}$$
The subspace of $PABG(H)_3$ of low weight 3 elements with $\rho$-eigenvalue 1 is 5-dimensional, and the subspaces of low weight 3 elements with $\rho$-eigenvalue $\omega$ and $\omega^2$ are 6-dimensional each.
\end{proof}

We would like to have  explicit expressions for $X^{3,1}$, $X^{3,\omega}$, and $X^{3,\omega^2}$.  The symmetries of $H$ and the relations imposed on these spaces will allow us to give explicit expressions for them.  We will need the following notation.

\begin{notation}[\cite{VJ5},\cite{EP}]

\begin{itemize}
\item The rotation operator $\rho$ has already been introduced:
$$\rho=
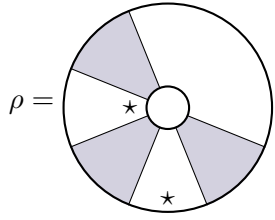
	\begin{tikzpicture}[annular]
	\clip (0,0) circle (2cm);
	\filldraw[shaded] (158:4cm)--(0,0)--(112:4cm);
	\filldraw[shaded] (-158:4cm)--(0,0)--(-112:4cm);
	\draw[shaded] (-68:4cm)--(0,0)--(-22:4cm);
	\draw[ultra thick] (0,0) circle (2cm);
	\node at (0,0)  [empty box] (T) {};
	\node at (T.180) [left] {$\star$};
	\node at (-90:2cm) [above] {$\star$};
\end{tikzpicture}$$
\item
The operator $\alpha$ is a symmetry of $H$ which permutes legs; Specifically, it sends $z_j \mapsto z_{j+1 \mod 3}$  and  $b_j \mapsto b_{j+1 \mod 3}$.  
\item
The operator $f_i$ on paths ``flips''  vertex $(i+1)$; if $v_{i}=v_{i+2}$, and $v_{i}$ is adjacent to only two vertices ($v_{i+1}$ and another, call it $w$), then 
	$$f_i (v_1 \ldots v_{i} v_{i+1} v_{i+2} \ldots v_8) = 
	- \sqrt{\frac{\lambda(v_{i+1})}{\lambda(w)}} v_1 \ldots v_{i} w v_{i+2} \ldots v_8.$$  
\end{itemize}
\end{notation}

When two paths differ at only one position and there are only two possible vertices that can go there, the capping-off relation $\epsilon_i$ forces the coefficient of one of the paths to be a certain scalar multiple of the other.  In the next theorem, we utilize the above notation to find an explicit expression for $X^{3,1}$.
\begin{theorem} Let $t_0$, $t_1$, $t_2$, $t_0'$, and $s$ be free variables.  Let R=$\sqrt{\frac{\lambda(d)}{\lambda(b_i)}}$.\\
If $X \subset PABG(H)_{3,+}$ is defined by

\noindent X= $\big(1 + \rho + \rho^2\big) \cdot \\
\Big(\big( t_0  + t_1 \cdot \alpha + t_2 \cdot \alpha^2 \big)\big(1+f_2\big) \big(cb_0cb_0cb_1 \big)\\
+  \big(t_0' + (t_0'+ t_2-t_1) \cdot \alpha + (t_0'+t_0-t_1) \cdot \alpha^2 \big) \big( 1+f_2\big) \big(cb_0cb_0cb_2 \big)\\
+ s  
	\big(cb_0cb_1cb_2\big)  \\
	-\big(2t_0+3t_0'+2t_2-t_1+s\big) \big(cb_0cb_2cb_1\big)  \\
	+\left(-\frac{t_0+t_0'}{R}-\frac{t_2+t_0'}{R}\cdot \alpha -\frac{t_0-t_1+t_2+t_0'}{R}\cdot \alpha^2\right)\big(1+f_4\big) \big(cdcb_0cb_0\big)  \\
	+\Big(-\frac{t_0-t_1+t_2+t_0'+s}{R}-\frac{t_0+t_0'+s}{R}\cdot \alpha-\frac{t_2+t_0'+s}{R}\cdot \alpha^2\Big) \big(cdcb_0cb_1\big)  \\
	+\Big(-\frac{-2(t_0+t_0')+t_1-t_2-s}{R}+\frac{2t_0'+t_2+t_0+s}{R}\cdot \alpha-\frac{-2(t_2+t_0')+t_1-t_0-s}{R}\cdot \alpha^2\Big) \big(cdcb_0cb_2\big)\Big),$\\
then $X=X^{3,1}$.
\end{theorem}
\begin{proof}  We must check that for all  $T \in X$,
$$\epsilon_i(T)=0 \mbox{ for }i=1,\ldots,6, \mbox{ and } \rho(T)=T.$$
Since the expression for $T \in X$ begins with $\big(1 + \rho + \rho^2\big)$, T has rotational eigenvalue 1.

To prove that $\epsilon_i(X)=0$ for all $i$, note that the only possible states on the diagram 
$$\epsilon_i = 
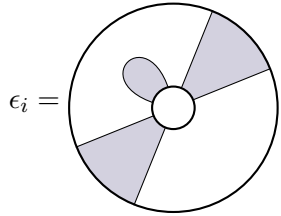
\begin{tikzpicture}[annular]
	\clip (0,0) circle (2cm);

	\filldraw[shaded] (0,0) .. controls ++(170:2cm) and ++(100:2cm) .. (0,0);
	\filldraw[shaded] (-158:4cm)--(0,0)--(-112:4cm);
	\draw[shaded] (68:4cm)--(0,0)--(22:4cm);
		
	\draw[ultra thick] (0,0) circle (2cm);
	
	\node at (0,0)  [empty box] (T) {};
\end{tikzpicture} $$ 
must have, along their inner boundary, loops $\bf l$ such that $l_i = l_{i+2}$.  
(By $l_k$ we mean the $k^{th}$ vertex of $\bf l$.)
Then $\epsilon_i(\bf l)$ is (some multiple of) the length-four loop which comes from deleting vertices $l_{i+1}$ and $l_{i+2}$:  $l_1 \ldots l_i l_{i+3} \ldots l_6$.  

The loops in $X$ that could contribute a non-zero term to  $\epsilon_i(X)$ 
are terms where the $i$ and $i+2$ vertices are the same.  
We need the length-four loops that come from applying $\epsilon_i$ to cancel each other out.
Suppose $\bf l$ is a 4-loop and consider the four possible cases.  

Case 1: $l_i$ = $z_j$ for some j.  Let $\bf l'$ be a 6-loop.  We want $\epsilon_i(\bf l')$ = (a scalar multiple of $\bf l$).  In that case, $l_i'=z_j$ and $l_{i+2}'=z_j$.  There is only one such $\bf l'$ which gives $\bf l$, under $\epsilon_i$.  If $\mathbf {l} = l_1\cdots l_il_{i+3}\cdots l_6,\mbox{ then }\mathbf{l'}=l_1\cdots l_ib_jl_il_{i+3}\cdots l_6$.  Thus, the coefficient of $\mathbf{l'}$ in $X$ is 0.

Case 2: $l_i$=$b_j$ for some j.  
If $\mathbf {l}=l_1\cdots l_il_{i+3}\cdots l_6,\mbox{then either }\\ \mathbf {l'}= l_1\cdots l_iz_jl_il_{i+3}\cdots l_6 ,\mbox{or } \mathbf{l'}=l_1\cdots l_icl_il_{i+3}\cdots l_6$.  So there are two 6-loops $\mathbf{l'}$ and $\mathbf{l''}$ such that $\epsilon_i(\mathbf{l'})$ and $\epsilon_i(\mathbf{l''})$ are scalar multiples of $\mathbf{l}$.
If we let $x=Coeff_{X}(\mathbf{l'})$ and $y=Coeff_{Y}(\mathbf{l''})$, then by imposing the conditions $$\epsilon_i(x\mathbf{l'}+y\mathbf{l''})=0,$$ we get a relation between $x$ and $y$.  \\In particular, $y=-x\sqrt{\frac{\lambda(l_{i+1}')}{\lambda(l_{i+1}'')}}$ for $i=1,2,3,4,5$ and $y=-x\frac{\lambda(l_{i+1}')}{\lambda(l_{i+1}'')}$ for $i=6$.

Case 3: $l_i$ = $c$.  There are four 6-loops $\mathbf{l'},\mathbf{l''},\mathbf{l'''},\mathbf{l^{(4)}}$ such that $\epsilon_i (\mathbf{l'}), \epsilon_i (\mathbf{l''}),\\ \epsilon_i (\mathbf{l'''}), \epsilon_i (\mathbf{l^{(4)}})$ are scalar multiples of $\mathbf{l}$.  Let $x,y,z,u$ be the coefficients, respectively, of $\mathbf {l'},\mathbf {l''},\mathbf {l'''},\mathbf {l^{(4)}}$ in $X$.  Then we want $\epsilon_i(x \cdot \mathbf{l'}+y \cdot \mathbf{l''}+z\cdot \mathbf{l'''}+u\cdot \mathbf{l^{(4)}})= 0$.  This condition implies that $x+y+z+u\cdot R=0$ for $i=1,5$ and $x+y+z+u\cdot R^2=0$ for $i=3$.  

Case 4: $l_i$=$d$.  The coefficient of $\mathbf{l'} = l_1\cdots l_icl_i\cdots l_6$ is 0 since there is only one loop $\mathbf{l'}$ such that $\epsilon_i(\mathbf{l'})$= (a scalar)$\cdot \mathbf{l}$.

One can check to see that $\epsilon_i(T) = 0$ for $i=1,\ldots,6$.  However, it is easier to see this by imposing the conditions $\epsilon_i(T) = 0$, for $i=1,\ldots,6$, on a general element of $V_3^{3,1}$ and deriving the desired expression for $T$.  Then we know that the element $T$ satisfies the desired conditions because it was by imposing those conditions that we derived $T$.  A general element of $V_3^{3,1}$ will be a linear combination of 6-loops based at even vertices $z_0,z_1,z_2,c$.  So let X be a linear combination of loops based at $z_0,z_1,z_2, c$.  We can replace a loop based at $z_1$ with $\alpha$ of some loop based at $z_0$.  Similarly, we can think of loops based at $z_2$ as loops based at $z_0$ with $\alpha^2$ applied to them.  So far $$X=\sum_k(\beta_{k,0}+\beta_{k,1}\alpha+\beta_{k,2}\alpha^2)(\mathbf{p}_k)  +\sum_j \eta_j (\mathbf{w}_j),$$ where $\mathbf{p}_k$ is a loop based at $z_0$ and $\mathbf{w}_j$  is a loop based at $c$.

To find a more precise expression for X, we impose the conditions $\epsilon_i(X)=0$ $\forall i$ and $\rho (X) = X$.  Suppose $\mathbf{p}$ is a path in the expression for $X$.  Then the coefficient of $\mathbf{p}$ in $X$ determines the coefficients of rotations of , permutations $\alpha$  and $\alpha^2$ of,  flips (through $b_i$) at one vertex (e.g., $z_0b_0cb_0\cdots$ and $z_0b_0z_0b_0\cdots$) of,  and combinations of all these of $\mathbf{p}$.  $\epsilon_i$ commutes with $\alpha$ and so $$\epsilon_i(X) = \epsilon_i ( \sum_k (\beta_{k,0}+\beta_{k,1}\alpha+\beta_{k,2}\alpha^2)(\mathbf{p}_k)  +\sum_j \eta_j (\mathbf{w}_j))$$ 
 $$= \sum_k (\beta_{k,0}+\beta_{k,1}\alpha+\beta_{k,2}\alpha^2)\epsilon_i(\mathbf{p}_k)  +\sum_j \eta_j \epsilon_i(\mathbf{w}_j)$$

Now, if $\mathbf{p}_k$ is such that $p_i \neq p_{i+2}$, then $\epsilon_i(\mathbf{p}_k)=0$.  If $p_i =p_{i+2}$, then there are the following cases for such a path $\mathbf{p}$:

Case (i): $p_i=p_{i+2}=b_j$ for some $j$.  Then, in order that $\epsilon_i(\mathbf{p})=0$ we must have the coefficient of the flip (through $b_j$) at vertex $p_{i+1}$ to be $-x\sqrt{\frac{\lambda(p_{i+1})}{\lambda(r_{i+1})}}$ for $i=1,2,3,4,5$ and $-x\frac{\lambda(p_{i+1})}{\lambda(r_{i+1})}$ for $i=6$, where $x$ is the coefficient of $\mathbf{p}$ and $r_{i+1}$ is the other vertex adjacent to $b_j$.  By convention, if $p_i=p_{i+2}=b_j$, let $p_{i+1}=c$ and relegate the path with $p_{i+1}=z_j$ as the flip of $\mathbf{p}$.  

Case (ii): $p_i=p_{i+2}=c$.  By convention, let $p_{i+1}=b_0$.  There are three flips of $\mathbf{p}$ (through $c$) at $p_{i+1}$ --- one with $b_1$, one with $b_2$, and one with $d$.  The coefficients $x,y,z,u$ of these four paths must satisfy $x+y+z+u\cdot R=0$ for $i=1,5$ and $x+y+z+u\cdot R^2=0$ for $i=3$ where $R= \sqrt{\frac{\lambda(d)}{\lambda(b_i)}}$.

Case (iii): $p_i=p_{i+2}=z_j$ for some $j$.  If $\epsilon_i(\mathbf{p})=0$, then $Coeff_X(\mathbf{p})=0$ because there are no other 6-loops such that $\epsilon_i$ of it gives us a multiple of $\epsilon_i(\mathbf{p})$ and hence no other 6-loop such that $\epsilon_i$ of it cancels out with $\epsilon_i(\mathbf{p})$.  

Case (iv):  $p_i=p_{i+2}=d$.  Then $Coeff_X(\mathbf{p})=0$ for the same reason as in Case (iii).

Any loop which contains the vertex $v_{i+1}=d$ must have $v_i=c=v_{i+2}$ and hence fall under Case (ii).  In this case, we will have a segment of the form $\ldots\left[\  \right]cdc\left[\  \right]\ldots$, where the first empty box admits $b_k$ for some k or $d$.  We can rotate the path to have it start at c so that we have: $cdc\left[\  \right]\left[\  \right]\left[\  \right]$, where the first and last empty boxes admit $b_k$ or $d$.  

Any loop which contains the vertex $v_{i+1}=z_j$ must have $v_i=b_j=v_{i+2}$ and hence fall under Case(i).  In this case, we will have a segment of the form $\ldots cb_jz_jb_jc\ldots$, since we can't have the case where any two of the same $z_j$'s are separated by one vertex according to Case (iii).  We can rotate the path to have it start at $c$ to get: $cb_jz_jb_jc\left[\  \right]$.  Furthermore, we can apply $\alpha$ to it such that $j=0$: $cb_0z_0b_0c\left[\  \right]$.  

Thus, we have loops such that $p_i=p_{i+2}=b_j$ for some $j$ and $i$ or $p_i=p_{i+2}=c$ for some $i$, where these don't contain any $z_j$ nor $d$, or loops of the form $cb_0z_0b_0c\left[\  \right]$, where these don't contain any $d$, or loops of the form $cdc\left[\  \right]\left[\  \right]\left[\  \right]$, where the first and last empty boxes admit $b_k$ or $d$.

Thus we have paths which contain either a $z_j$ or a $d$ and those that don't.  Let us classify three main types of loops.

\noindent Type 1: $\mathbf{p}$ contains a $z_j$ but no $d$.  Then $\mathbf{p}$ is $cb_0z_0b_0c\left[\  \right]$ up to a combination of rotations and $\alpha$'s.  

\noindent Type 2: $\mathbf{p}$ does not contain any $z_j$ nor $d$.\\
\indent Type 2a: $\mathbf{p}$ does not satisfy $p_i=p_{i+2}= b_j$ or $c$ for any $i,j$.\\
\indent Type 2b: $\mathbf{p}$ does satisfy $p_i=p_{i+2}=b_j$ or $c$ for some $i,j$.

\noindent Type 3: $\mathbf{p}$ contains a $d$\\
\indent Type 3a: $\mathbf{p}$ is of the form $cdcb_0\left[c\right]b_k$ for some $k$.\\
\indent Type 3b: $\mathbf{p}$ is of the form $cdcb_0\left[z_0\right]b_0$.

\noindent Consider loops of type 2a.  There are none of this type.  

\noindent Consider loops of type 2b.  We can reduce all loops of type 2b to the following five loops: $$cb_0cb_0cb_0,cb_0cb_0cb_1,cb_0cb_0cb_2,cb_0cb_1cb_2,cb_0cb_2cb_1$$

\noindent Consider loops of type 1: 
$$cb_0z_0b_0cb_0\cong f_2(cb_0cb_0cb_0)$$
$$cb_0z_0b_0cb_1\cong f_2(cb_0cb_0cb_1)$$
$$cb_0z_0b_0cb_2\cong f_2(cb_0cb_0cb_2),$$
where $\mathbf{p}\cong \mathbf{q}$ means that $\mathbf{p}$ is equal to $\mathbf{q}$ up to a scalar multiple.
The possible flips (except those which give rise to loops of type 3) are $$f_4f_2(cb_0cb_0cb_0),$$ $$f_6f_4f_2(cb_0cb_0cb_0),$$ $$f_2(cb_0cb_0cb_1),$$ $$f_2(cb_0cb_0cb_2)$$

\noindent Consider loops of type 3:  \\
\indent Type 3a loops consist of : $$cdcb_0cb_0,$$ $$cdcb_0cb_1,$$ $$cdcb_0cb_2$$
\indent The possible flips are : $$f_4(cdcb_0cb_0),$$ $$cdcb_0cb_1,$$ $$cdcb_0cb_2$$
\indent Type 3b loops consist of: $cdcb_0z_0b_0\cong f_4(cdcb_0cb_0)$

In total then, our basic loops will be the following fourteen loops: $$cb_0cb_0cb_0,cb_0cb_0cb_1,cb_0cb_0cb_2,cb_0cb_1cb_2,cb_0cb_2cb_1,$$ $$f_2(cb_0cb_0cb_0), f_2(cb_0cb_0cb_1),f_2(cb_0cb_0cb_2),$$ $$f_4f_2(cb_0cb_0cb_0), f_6f_4f_2(cb_0cb_0cb_0),$$ $$ cdcb_0cb_0,cdcb_0cb_1,cdcb_0cb_2,$$ $$f_4(cdcb_0cb_0)$$

The coefficient of the flip of $cb_0cb_0cb_1$, for instance, in $X$ must be that given by $f_2$ times the coefficient of $cb_0cb_0cb_1$.

\noindent Let $X= \big(1+\rho + \rho^2\big)\cdot \\
\Big( \big(u_0+u_1\cdot\alpha+u_2\cdot\alpha^2\big)\big(1+f_2+f_4f_2+f_6f_4f_2\big)\big(cb_0cb_0cb_0\big)\\
+\big(t_0+t_1\cdot\alpha+t_2\cdot\alpha^2\big)\big(1+f_2\big)\big(cb_0cb_0cb_1\big)\\
+\big(t_0'+t_1'\cdot\alpha+t_2'\cdot\alpha^2\big)\big(1+f_2\big)\big(cb_0cb_0cb_2\big)\\
+\big(s_0+s_1\cdot\alpha+s_2\cdot\alpha^2\big)\big(cb_0cb_1cb_2\big)\\
+\big(r_0+r_1\cdot\alpha+r_2\cdot\alpha^2\big)\big(cb_0cb_2cb_1\big)\\
+\big(q_0+q_1\cdot\alpha+q_2\cdot\alpha^2\big)\big(1+f_4\big)\big(cdcb_0cb_0\big)\\
+\big(m_0+m_1\cdot\alpha+m_2\cdot\alpha^2\big)\big(cdcb_0cb_1\big)\\
+\big(n_0+n_1\cdot\alpha+n_2\cdot\alpha^2\big)\big(cdcb_0cb_2\big)\Big)$.

\noindent Notice that $f_4f_2(cb_0cb_0cb_0)\cong f_4(cb_0z_0b_0cb_0)\cong cb_0z_0b_0z_0b_0$.\\  
But $Coeff_X(cb_0z_0b_0z_0b_0)=0$, and hence $u_0=u_1=u_2=0$.  \\
Thus, \\
$X= \big(1+\rho + \rho^2\big)\cdot\\
\Big(\big(t_0+t_1\cdot\alpha+t_2\cdot\alpha^2\big)\big(1+f_2\big)\big(cb_0cb_0cb_1\big)\\
+\big(t_0'+t_1'\cdot\alpha+t_2'\cdot\alpha^2\big)\big(1+f_2\big)\big(cb_0cb_0cb_2\big)\\
+\big(s_0+s_1\cdot\alpha+s_2\cdot\alpha^2\big)\big(cb_0cb_1cb_2\big)\\
+\big(r_0+r_1\cdot\alpha+r_2\cdot\alpha^2\big)\big(cb_0cb_2cb_1\big)\\
+\big(q_0+q_1\cdot\alpha+q_2\cdot\alpha^2\big)\big(1+f_4\big)\big(cdcb_0cb_0\big)\\
+\big(m_0+m_1\cdot\alpha+m_2\cdot\alpha^2\big)\big(cdcb_0cb_1\big)\\
+\big(n_0+n_1\cdot\alpha+n_2\cdot\alpha^2\big)\big(cdcb_0cb_2\big)\Big)$.

Since applying $\alpha$ to $cb_0cb_1cb_2$ and $cb_0cb_2cb_1$ gives loops which can be got by applying a rotation, we have that\\
$X= \big(1+\rho + \rho^2\big)\cdot\\
\Big(\big(t_0+t_1\cdot\alpha+t_2\cdot\alpha^2\big)\big(1+f_2\big)\big(cb_0cb_0cb_1\big)\\
+\big(t_0'+t_1'\cdot\alpha+t_2'\cdot\alpha^2\big)\big(1+f_2\big)\big(cb_0cb_0cb_2\big)\\
+s_0\big(cb_0cb_1cb_2\big)\\
+r_0\big(cb_0cb_2cb_1\big)\\
+\big(q_0+q_1\cdot\alpha+q_2\cdot\alpha^2\big)\big(1+f_4\big)\big(cdcb_0cb_0\big)\\
+\big(m_0+m_1\cdot\alpha+m_2\cdot\alpha^2\big)\big(cdcb_0cb_1\big)\\
+\big(n_0+n_1\cdot\alpha+n_2\cdot\alpha^2\big)\big(cdcb_0cb_2\big)\Big)$.

Imposing the condition $\epsilon_1(X)=0$ gives us the following equations:
$$m_i=-\frac{t_i+t_{i+1 mod 3}'+s_0}{R},$$ 
$$n_i=-\frac{t_i'+r_0+t_{i+2 mod 3}}{R},$$ 
$$q_i=-\frac{t_i+t_i'}{R}$$ for $i=0,1,2$

\noindent So \\
$X= \big(1+\rho + \rho^2\big)\cdot\\
\Big(\big(t_0+t_1\cdot\alpha+t_2\cdot\alpha^2\big)\big(1+f_2\big)\big(cb_0cb_0cb_1\big)\\
+\big(t_0'+t_1'\cdot\alpha+t_2'\cdot\alpha^2\big)\big(1+f_2\big)\big(cb_0cb_0cb_2\big)\\
+s\big(cb_0cb_1cb_2\big)\\
+r\big(cb_0cb_2cb_1\big)\\
+\big(-\frac{t_0+t_0'}{R}-\frac{t_1+t_1'}{R}\cdot\alpha-\frac{t_2+t_2'}{R}\cdot\alpha^2\big)\big(1+f_4\big)\big(cdcb_0cb_0\big)\\
+\big(-\frac{t_0+t_1'+s}{R}-\frac{t_1+t_2'+s}{R}\cdot\alpha-\frac{t_2+t_0'+s}{R}\cdot\alpha^2\big)\big(cdcb_0cb_1\big)\\
+\big(-\frac{t_0'+r+t_2}{R}-\frac{t_1'+r+t_0}{R}\cdot\alpha-\frac{t_2'+r+t_1}{R}\cdot\alpha^2\big)\big(cdcb_0cb_2\big)\Big)$,
where $s_0=s \mbox{ and } r_0=r$.

By imposing the condition $\epsilon_3(x)=0$, we get the equation:
$$2(t_i+t_i')+t_{i+2 mod 3}+ t_{i+1 mod 3}'+r+s=0$$ for $i=0,1,2$.  From these three equations, we get that $t_2'+t_1-t_0'-t_0=0$.  Similarly, $t_0'+t_2-t_1'-t_1=0$ and $t_1'+t_0-t_2'-t_2=0$.  So $t_2'=t_0'+t_0-t_1$ and $t_1'=t_0'+t_2-t_1$.  We now have $t_1'$ and $t_2'$ in terms of $t_0'$.  We can easily solve for $r$ to get $r=-(2t_0+3t_0'+2t_2-t_1+s)$.  Imposing the condition $\epsilon_5(X)=0$ holds and does not give us any new information. 

So far, \\
$X= \big(1+\rho + \rho^2\big)\cdot\\
\Big(\big(t_0+t_1\cdot\alpha+t_2\cdot\alpha^2\big)\big(1+f_2\big)\big(cb_0cb_0cb_1\big)\\
+\big(t_0'+t_1'\cdot\alpha+t_2'\cdot\alpha^2\big)\big(1+f_2\big)\big(cb_0cb_0cb_2\big)\\
+s\big(cb_0cb_1cb_2\big)\\
-(2t_0+3t_0'+2t_2-t_1+s)\big(cb_0cb_2cb_1\big)\\
+\big(-\frac{t_0+t_0'}{R}-\frac{t_2+t_0'}{R}\cdot\alpha-\frac{t_0-t_1+t_2+t_0'}{R}\cdot\alpha^2\big)\big(1+f_4\big)\big(cdcb_0cb_0\big)\\
+\big(-\frac{t_0-t_1+t_2+t_0'+s}{R}-\frac{t_0+t_0'+s}{R}\cdot\alpha-\frac{t_2+t_0'+s}{R}\cdot\alpha^2\big)\big(cdcb_0cb_1\big)\\
+\big(-\frac{-2(t_0+t_0')+(t_1-t_2)-s}{R}+\frac{2t_0'+t_2+t_0+s}{R}\cdot\alpha-\frac{-2(t_2+t_0')+(t_1-t_0)-s}{R}\cdot\alpha^2\big)\big(cdcb_0cb_2\big)\Big)$\\
Thus, we have that $X \subseteq X^{3,1}$.  Since $X$ is 5-dimensional, and $X^{3,1}$ is also, we have $X=X^{3,1}$.

\end{proof}

If we want $T^*=T$ for $T \in X$, then we must have $\bar{t_0}=t_0$.  Similarly, $\bar{t_1}=t_1$ and $\bar{t_2}=t_2$.  This implies that $t_0,t_1,t_2 \in \mathbb{R}$.  Also, $\bar{t_0'}=t_0'$ implies $t_0' \in \mathbb{R}$.  Furthermore, $\bar{s}=-(2t_0+3t_0'+2t_2-t_1+s)$

We have a similar explicit expression for $X^{3,\omega}$.  
\begin{theorem}  Let $t_0$, $t_1$, $t_2$, $t_0'$, and $s$ be free variables.  Let R=$\sqrt{\frac{\lambda(d)}{\lambda(b_i)}}$ and let $\eta = 1+i\sqrt{3}$.\\
If $X \subset PABG(H)_{3,+}$ is defined by

\noindent X= $\big(1 + \omega^2 \cdot \rho + \omega \cdot \rho^2 \big) \cdot \\
\Big(\eta\big( t_0  + t_1 \cdot \alpha + t_2 \cdot \alpha^2 \big)\big(1+f_2\big) \big(cb_0cb_0cb_1 \big)\\
+  \Big(\eta t_0' + \big(\eta (t_0+ t_0'-t_2)-\bar{s}-\omega s \big)\cdot \alpha + \big(\eta(t_1+t_0'-t_2)+ \omega\bar{s}+s\big) \cdot \alpha^2 \Big) \big( 1+f_2\big) \big(cb_0cb_0cb_2 \big)\\
+ s  
	\big(cb_0cb_1cb_2\big)  \\
	+\omega^2\bar{s}\big(cb_0cb_2cb_1\big)  \\
	-\frac{\omega^2}{R}\Big(\eta(t_0+t_0')+\big(\eta(t_1+t_0+t_0'-t_2)-\bar{s}-\omega s\big)\cdot \alpha +\big(\eta(t_1+t_0')+\omega\bar{s}+s\big)\cdot \alpha^2\Big)\big(1+f_4\big)\big(cdcb_0cb_0\big)  \\
	-\frac{1}{R}\Big(\big(\eta(-\omega^2t_0+\omega t_0'-\omega t_2)-\omega\bar{s}\big)+\big(\eta(-\omega^2 t_1+\omega t_0'-\omega t_2)+\omega^2\bar{s}-\omega^2 s\big)\cdot \alpha + \big(\eta(t_2+\omega t_0')+\omega s\big)\cdot \alpha^2\Big)\big(cdcb_0cb_1\big)  \\
	-\frac{1}{R}\Big( \big(\eta(t_0'+\omega t_2)+\omega \bar{s}\big)+\big(\eta(-\omega^2t_0+t_0'-t_2)-\omega s\big)\cdot \alpha +\big( \eta(t_0'-\omega^2 t_1-t_2)-\bar{s}+s\big)\cdot \alpha^2\Big)\big(cdcb_0cb_2\big)\Big)$,

\noindent then $X=X^{3,\omega}$.
  Furthermore, $t_0,t_1,t_2,t_0' \in \mathbb{R}$.
\end{theorem}

\begin{proof}  The proof is very similar to the proof of the above theorem.  
\end{proof}
 
We can certainly find an explicit expression for $X^{3,\omega^2}$, but we need not consider this space (for now) because our generators $T$ and $Q$ will come from $X^{3,1}$ and $X^{3,\omega}$.  Note for now that there is a generator coming from the space $X^{3,\omega^2}$ which, along with the conjugate of $T$, generates a ``2221'' subfactor planar algebra.  In fact, the generator with rotational eigenvalue $\omega^2$ is just $Q$, with every $\eta$ and $\omega$ conjugated.  We will explore this second pair of generators in the chapter on uniqueness.

\begin{theorem}  The only elements of $X^{3,1} \subset PABGH_3(H)$ which could generate a ``2221'' planar algebra are multiples (or multiples of conjugates) of 

\noindent $T= \big(1+\rho + \rho^2\big)\cdot\\
\Big(t\big(1+\alpha+\alpha^2\big)\big(1+f_2\big)\big(cb_0cb_0cb_1\big)\\
+t_0'\big(1+\alpha+\alpha^2\big)\big(1+f_2\big)\big(cb_0cb_0cb_2\big)\\
+s\big(cb_0cb_1cb_2\big)\\
+\bar{s}\big(cb_0cb_2cb_1\big)\\
-\big(\frac{t_0+t_0'}{R}\big)\big(1+\alpha+\alpha^2\big)\big(1+f_4\big)\big(cdcb_0cb_0\big)\\
-\big(\frac{t+t_0'+s}{R}\big)\big(1+\alpha+\alpha^2\big)\big(cdcb_0cb_1\big)\\
+\big(\frac{2(t+t_0')+s}{R}\big)\big(1+\alpha+\alpha^2\big)\big(cdcb_0cb_2\big)\Big)$

\noindent where $t=\frac{1}{12}\big(-3+\sqrt{21}+\sqrt{-114+26\sqrt{21}}\big)$, $t_0'=\frac{1}{4}\big(-7-\sqrt{21}+\sqrt{54+14\sqrt{21}}\big)t$, and $s=\frac{1}{4}\big(3-\sqrt{21}\big)+\frac{1}{2}\sqrt{\frac{-57+13\sqrt{21}}{6}}\cdot i$, \\
or of $P$, which is the same expression as $T$ except with $t=\frac{1}{12}\big(-3+\sqrt{21}-\sqrt{-114+26\sqrt{21}}\big)$, $t_0'=\frac{1}{4}\big(-7-\sqrt{21}-\sqrt{54+14\sqrt{21}}\big)t$, and $s=\frac{1}{4}\big(3-\sqrt{21}\big)-\frac{1}{2}\sqrt{\frac{-57+13\sqrt{21}}{6}}\cdot i$ .
\end{theorem}
\begin{proof}Since $\dim (P_{0,+})=1 $ and $\dim (TL_{0,+})=1$ for a (still hypothetical) ``2221'' planar algebra $P$,  we look for $T \in X$ such that $tr(T^2) \in TL_{0,+} \subset PABG(H)_{0,+}$.  The single basis element of $TL_+$, when interpreted as an element of $PABG(H)_+$, is $z_0+z_1+z_2+c$.  Thus,
 $tr(T^2)$ will  be a scalar times $z_0+z_1+z_2+c$.  This implies that the coefficients in $tr(T^2)$ of $z_0$, $z_1$,  $z_2$, and $c$ must all be equal.  In fact, $tr(T^n)$ will be a scalar multiple of $z_0+z_1+z_2+c$ for all $n \geq 1$.  By setting $Coeff_{tr(T^n)}(z_0)=Coeff_{tr(T^n)}(z_1)= Coeff_{tr(T^n)}(z_2)$ for all  $n \geq 1$, we see that it must be the case that $t_0=t_1=t_2$.(see the section on Calculating Traces for $T$ and $Q$)  So let $t_0=t_1=t_2=t$.  By setting $Coeff_{tr(T^2)}(c)=Coeff_{tr(T^2)}(z_0)$, we get that $\vert s\vert ^2=\frac{-3+\sqrt{21}}{2}(t^2+(t_0')^2)+(-2+\sqrt{21})(t+t_0')^2$.  It follows that $$Im(s) = \pm \sqrt{\Big(\frac{-3+\sqrt{21}}{2}\Big)(t^2+(t_0')^2)+\Big(\frac{-17}{4}+\sqrt{21}\Big)(t+t_0')^2}.$$  We already have $Re(s)=-\frac{3}{2}(t+t_0')$.  Take the positive square root for $Im(s)$ and let $s=Re(s) + i\cdot Im(s)$ where $Re(s)$ and $Im(s)$ are given as above.  
 
By setting $Coeff_{tr(T^3)}(c)=Coeff_{tr(T^3)}(z_0)$, we get the following quadratic equation in the variables $t$ and $t_0'$:
$$-\frac{3}{2}(t+t_0')\big((21\sqrt{3}+13\sqrt{7})t^2+7(17\sqrt{3}+11\sqrt{7})t\cdot t_0'+(21\sqrt{3}+13\sqrt{7})(t_0')^2\big)=0,$$
which implies that either \[t_0'=-t \mbox{ or } t_0'=\frac{1}{4}\big(-7-\sqrt{21}\pm\sqrt{54+14\sqrt{21}}\big)\cdot t.\]\\
It can't be the case that $t_0'=-t_0$ since setting  $Coeff_{tr(T^4)}(c)=Coeff_{tr(T^4)}(z_0)$ implies $t_0=0$ and hence gives the null solution.  Therefore, $t_0'=\frac{1}{4}\big(-7-\sqrt{21}\pm\sqrt{54+14\sqrt{21}}\big)\cdot t$. Let $t_0'=\frac{1}{4}\big(-7-\sqrt{21}+\sqrt{54+14\sqrt{21}}\big)t$ and denote this generator by $T$.  Denote by $P$ the generator with $t_0'=\frac{1}{4}\big(-7-\sqrt{21}-\sqrt{54+14\sqrt{21}}\big)\cdot t$.  
We also want $\left\langle T,T\right\rangle=[4]$, i.e., $Z(T^2)=[4]$.

\[ Z(T^2)=3\cdot Coeff_{tr(T^2)}(z_0)\cdot (\lambda(z))^2+Coeff_{tr(T^2)}(c)\cdot(\lambda(c))^2 \]
 \[=\big((t^2+(t_0')^2)\frac{\theta^2}{\sqrt{3}}+\frac{(t+t_0')^2}{R^4}\sqrt{3}\big)\theta^4\cdot3(\lambda(z))^2\]
 \[+\big((9+6\theta^2)(t^2+(t_0')^2)+\theta^2(3+2\theta^2)(t+t_0')^2+6\vert s\vert^2+6\theta^2\vert s\vert^2 \]  \[-12\theta^2(t+t_0')^2\big)\frac{1}{\sqrt{3}}(\lambda(c))^2  \]
 \[= \big((t^2+(t_0')^2)\frac{\theta^2}{\sqrt{3}}+(t+t_0')^2\cdot \frac{9\sqrt{3}}{\theta^4}\big)\theta^4\cdot 3(\lambda(z))^2 \]
 \[+ \big( (9+6\theta^2)(t^2+(t_0')^2)+(-9\theta^2+2\theta^4)(t+t_0')^2\]
 \[+(6+6\theta^2)\cdot\vert s\vert^2\big)\cdot\frac{1}{\sqrt{3}}(\lambda(c))^2 .\]
Setting $Z(T^2)=[4]$, we get
 \[\frac{3}{4}t\big(\sqrt{42(19673+4293\sqrt{21})}\cdot t-\sqrt{727722+158802\sqrt{21}}\cdot t\big)\] 
 \[=\sqrt{7}+2\sqrt{3},\] which implies 
  \[ \frac{3}{4}\big(\sqrt{42(19673+4293\sqrt{21}}-\sqrt{727722+158802\sqrt{21}}\big)\cdot t^2=\sqrt{7}+2\sqrt{3}.\]
This implies \[t^2= \frac{\frac{4}{3}(\sqrt{7}+2\sqrt{3})}{\sqrt{42(19673+4293\sqrt{21})}-\sqrt{727722+158802\sqrt{21}}}\]
Hence, $t=\frac{1}{12}\big(-3+\sqrt{21}+\sqrt{-114+26\sqrt{21}}\big)$.
With the given value for $t$, we have that $s=\frac{1}{4}(3-\sqrt{21})+\frac{1}{2}\sqrt{\frac{-57+13\sqrt{21}}{6}}\cdot i$.  Thus, we have our desired expression for $T$.  $P$ is derived similarly, except with $t_0'=\frac{1}{4}\big(-7-\sqrt{21}-\sqrt{54+14\sqrt{21}}\big)t$.
\end{proof}

\begin{theorem}The only elements of $X^{3,\omega} \subset PABGH_3(H)$ which could generate a ``2221'' planar algebra are multiples of \\
Q= $\eta \big(1 + \omega^2 \cdot \rho + \omega \cdot \rho^2 \big) \cdot \\
\Big(t\big(1+\alpha +  \alpha^2 \big)\big(1+f_2\big) \big(cb_0cb_0cb_1 \big)\\
+  \Big(t_0'\big(1 + \alpha +\alpha^2 \Big) \big( 1+f_2\big) \big(cb_0cb_0cb_2 \big)\\
	-\frac{\omega^2}{R}\Big((t+t_0')\big(1+\alpha +\alpha^2\big)\Big)\big(1+f_4\big)\big(cdcb_0cb_0\big)  \\
	-\frac{1}{R}\Big(\big(t+ \omega t_0'\big)\big(1+\alpha +\alpha^2\big)\Big)\big(cdcb_0cb_1\big)  \\
	-\frac{1}{R}\Big( \big(t_0'+ \omega t\big)\big(1+\alpha +\alpha^2\big)\Big)\big(cdcb_0cb_2\big)\Big)$,\\
	where $t=\sqrt{\frac{[4]}{1+\Psi^2}\frac{5}{(246+54\sqrt{21})\sqrt{3}}}$, $\Psi=\frac{1-\theta^2-\theta}{2}$, $t_0'=\Psi t$,\\
or of $S$, where $S$ is just $Q$ except with  $t=\sqrt{\frac{[4]}{1+\Omega^2}\frac{5}{(246+54\sqrt{21})\sqrt{3}}}$, $\Omega=\frac{1-\theta^2+\theta}{2}$, $t_0'=\Omega t.$

\end{theorem}
\begin{proof} The proof is similar to that of the case for $X^{3,1}$.  
By setting \[Coeff_{tr(Q^n)}(z_0)= Coeff_{tr(Q^n)}(z_1)= Coeff_{tr(Q^n)}(z_2)= Coeff_{tr(Q^n)}(c),\] for all  $n \geq 1$, we see that it must be the case that $t_0=t_1=t_2$ and $s=0$. (see the appendix on Calculating Traces for $T$ and $Q$.)  So let $t_0=t_1=t_2=t$.
By setting \[Coeff_{tr(Q^2)}(c)=Coeff_{tr(Q^2)}(z_0),\] we get the following quadratic equation in the variables $t$ and $t_0'$.
\[t^2+\Big(\frac{1+\sqrt{21}}{2}\Big)(t\cdot t_0')+(t_0')^2 =0.\]
It follows that $t_0'=\frac{1-\theta^2\pm \theta}{2}\cdot t$.  Let $t_0'=\Omega \cdot t$, where $\Omega=\frac{1-\theta^2+\theta}{2}$ and denote the generator by $S$.  Denote by $Q$ the generator with $t_0' =\Psi \cdot t$, where $\Psi=\frac{1-\theta^2-\theta}{2}$.

We want $\left\langle S,S\right\rangle=[4]$ and $\left\langle Q,Q\right\rangle=[4]$, i.e., $Z(S^2)=[4]$ and $Z(Q^2)=[4]$.  \[Z(Q^2)=3Coeff_{tr(Q^2)}(z_0)\cdot (\lambda(z))^2+Coeff_{tr(Q^2)}(c)\cdot (\lambda(c))^2\]
\[=\Big(\frac{4}{\sqrt{3}}\theta^6+\frac{4}{9}\theta^8\sqrt{3}\Big(\frac{-6+\sqrt{21}}{-5}\Big)\Big)(t^2+(t_0')^2)\]\[=\Big(\frac{(246+54\sqrt{21})\sqrt{3}}{5}\Big)(1+\Psi^2)t^2,\]
where $t_0'=\Psi t \mbox{ and } \Psi=\frac{1-\theta^2 - \theta}{2}$.  
Setting the above expression equal to $[4]$, we get
\[t^2=\frac{[4]}{1+\Psi^2}\Big(\frac{5}{(246+54\sqrt{21})\sqrt{3}}\Big),\mbox{for Q}.\] Using the same equation, except with $t_0'=\Omega t$, where $\Omega = \frac{1-\theta^2+\theta}{2}$, we get
 \[t^2=\frac{[4]}{1+\Omega^2}\Big(\frac{5}{(246+54\sqrt{21})\sqrt{3}}\Big),\mbox{for S}\] 
Solving for $t$, we get \[t=\sqrt{\frac{[4]}{1+\Omega^2}\Big(\frac{5}{(246+54\sqrt{21})\sqrt{3}}\Big)}, \mbox{for S}\] and \[t=\sqrt{\frac{[4]}{1+\Psi^2}\Big(\frac{5}{(246+54\sqrt{21})\sqrt{3}}\Big)}, \mbox{for Q}.\] 
Thus, we have our desired expression for $Q$.
\end{proof}
\indent Similarly, the only elements of $X^{3,\omega^2} \subset PABGH_3(H)$ which could generate a ``2221'' planar algebra are multiples of the conjugate of $Q$ or of the conjugate of $S$.  Denote the conjugate of $Q$ by $V$ and that of $S$ by $U$.  All of the equations used to find the potential generators in $X^{3,\omega^2}$ are the same as those used to find $Q$ and $S$, except with every $\eta$ and $\omega$ conjugated.  Hence, the only elements of $X^{3,\omega^2} \subset PABGH_3(H)$ which could generate a ``2221'' planar algebra are going to have the same matrices as those for the generators in $X^{3,\omega}$, except with every $\eta$ and $\omega$ conjugated.  Then, the coefficients of $z_0,z_1,z_2,c$ in $tr(V^n)$, say, are the same as the coefficients of $z_0,z_1,z_2,c$ in $tr(Q^n)$, respectively, except with $s$ and $\bar{s}$ switched.  The proof in the appendix that $t_0=t_1=t_2$ and $s=0$ for the potential generators in $X^{3,\omega}$ works in analogous fashion with $s$ and $\bar{s}$ switched.  With $s=0$, the matrices of the potential generators in $X^{3,\omega^2}$ reduce to being just the transposes of the matrices of the potential generators in $X^{3,\omega}$.  Setting $Coeff_{tr(V^2)}(c)=Coeff_{tr(V^2)}(z_0)$ gives $t_0'=\frac{1-\theta^2\pm \theta}{2}t$ and we can choose our potential generators in $X^{3,\omega^2}$ to be the conjugates of $Q$ and $S$.  
\chapter{Quadratic Relations on T and Q}
In this section, we prove some quadratic relations involving $T$ and $Q$.  In order to do this, we need the following trace values.
\begin{lem}
Recall that $\rho^{1/2}$ is the shading-changing rotation morphism 
$$PABG(H)_{3,+} \rightarrow PABG(H)_{3,-}.$$
We have the following trace values for $T$:
\begin{itemize}
\item $Z(T)=0$
\item $Z(T^2)=[4]$
\item $Z(T^3)=\frac{-3\sqrt{3}+\sqrt{7}}{6}$
\item $Z(T^4)=\frac{5}{3}\sqrt{\frac{2}{3}(23+5\sqrt{21})},$
\end{itemize}
and the following trace values for $Q$:
\begin{itemize}
\item $Z(Q)=0$
\item $Z(Q^2)=[4]$
\item $Z(Q^3)=-\frac{2}{3}\sqrt{8+3\sqrt{21}}$
\item $Z(Q^4)=\frac{2}{3}\sqrt{\frac{253}{3}+16\sqrt{21}},$
\end{itemize}
and the following trace values for mixtures of $T$ and $Q$:
\begin{itemize}
\item $Z(TQ)=0$
\item $Z(T^2Q)=\sqrt{\frac{151}{18}+\frac{41}{6}\sqrt{\frac{7}{3}}}$
\item $Z(TQ^2)=\frac{1}{3}(2\sqrt{3}+\sqrt{7})$
\item $Z(T^2Q^2)=\frac{7}{3}\sqrt{\frac{1}{6}(11+\sqrt{21})}.$
\end{itemize}
We also have the following trace values for $\rho ^{\frac{1}{2}}(T)$:
\begin{itemize}
\item $Z(\rho ^{\frac{1}{2}}(T))=0$
\item $Z((\rho ^{\frac{1}{2}}(T))^2)=[4]$
\item $Z((\rho ^{\frac{1}{2}}(T))^3)=\frac{1}{6}(-3\sqrt{3}+\sqrt{7})=Z(T^3)$
\item $Z((\rho ^{\frac{1}{2}}(T))^4)=\frac{5}{3}\sqrt{\frac{2}{3}(23+5\sqrt{21})}=Z(T^4),$
\end{itemize}
and the following trace values for $\rho ^{\frac{1}{2}}(Q)$:
\begin{itemize}
\item $Z(\rho ^{\frac{1}{2}}(Q))=0$
\item $Z((\rho ^{\frac{1}{2}}(Q))^2)=\omega [4]$
\item $Z((\rho ^{\frac{1}{2}}(Q))^3)=-\frac{2}{3}\sqrt{8+3\sqrt{21}}=Z(Q^3)$
\item $Z((\rho ^{\frac{1}{2}}(Q))^4)=-\frac{i(-i+\sqrt{3})(8+3\sqrt{21})}{3\sqrt{3}}=\omega^2 Z(Q^4),$
\end{itemize}
and the following trace values for mixtures of $\rho ^{\frac{1}{2}}(T)$ and $\rho^{\frac{1}{2}}(Q)$:
\begin{itemize}
\item $Z(\rho ^{\frac{1}{2}}(T)\rho^{\frac{1}{2}}(Q))=0$
\item $Z((\rho ^{\frac{1}{2}}(T)(\rho^{\frac{1}{2}}(Q))^2)=\frac{i(i+\sqrt{3})(697+152\sqrt{21})^{\frac{1}{4}}}{6}=\omega Z(TQ^2)$
\item $Z((\rho ^{\frac{1}{2}}(Q)(\rho^{\frac{1}{2}}(T))^2)=-\frac{1}{12}\sqrt{302+82\sqrt{21}}-\frac{1}{12}\sqrt{906+246\sqrt{21}}\cdot i=\omega^2 Z(QT^2)$
\item $Z((\rho ^{\frac{1}{2}}(T))^2(\rho^{\frac{1}{2}}(Q))^2)=\frac{7i(i+\sqrt{3})(1+\sqrt{21})}{12\sqrt{3}}=\omega Z(T^2Q^2).$
\end{itemize}
\end{lem}

\begin{proof}
These trace values are calculated in the Appendix using matrix representations.
\end{proof}

\begin{theorem}The following quadratic skein relations hold in $PABG(H)_{3}$:
\begin{itemize}
\item (R3(T)), (R3(Q))\hspace{.5cm}
$\begin{tikzpicture}[scale=.6, baseline]
	\clip [draw] (2,2) arc (0:180:2cm) -- (-2,-2) arc (-180:0:2cm) -- (2,2);
	
	\draw[shaded] (0,2) .. controls ++(-115:1.5cm) and ++(115:1.5cm) .. (0,-2);
	\draw (0,2) .. controls ++(-90:1.5cm) and ++(90:1.5cm) .. (0,-2);
	\draw (0,2) .. controls ++(-65:1.5cm) and ++(65:1.5cm) .. (0,-2);
	\draw[shaded] (0,2) -- +(90:3cm) -- +(120:3cm) -- (0,2);
	\filldraw[shaded] (0,2) ++(65:5cm) -- (0,2) .. controls ++(-65:1.5cm) and ++(65:1.5cm) .. (0,-2) -- ++(-65:5cm);
	
	\draw[shaded] (0,-2) -- +(-90:3cm) -- +(-120:3cm) -- (0,-2);
	\node at (0,2) [Tbox] (A) {$T$};
	\node at (0,-2) [Tbox] (B) {$T$};
	\node at (A.180) [left] {$\star$};
	\node at (B.180) [left] {$\star$};
	\node at (-1.5,0)  {$\star$};	
	
	\draw[ultra thick] (2,2) arc (0:180:2cm) -- (-2,-2) arc (-180:0:2cm) -- (2,2);
\end{tikzpicture}$, \quad
	$\begin{tikzpicture}[scale=.6, baseline]
	\clip [draw] (2,2) arc (0:180:2cm) -- (-2,-2) arc (-180:0:2cm) -- (2,2);
	
	\draw[shaded] (0,2) .. controls ++(-115:1.5cm) and ++(115:1.5cm) .. (0,-2);
	\draw (0,2) .. controls ++(-90:1.5cm) and ++(90:1.5cm) .. (0,-2);
	\draw (0,2) .. controls ++(-65:1.5cm) and ++(65:1.5cm) .. (0,-2);
	\draw[shaded] (0,2) -- +(90:3cm) -- +(120:3cm) -- (0,2);
	\filldraw[shaded] (0,2) ++(65:5cm) -- (0,2) .. controls ++(-65:1.5cm) and ++(65:1.5cm) .. (0,-2) -- ++(-65:5cm);
	
	\draw[shaded] (0,-2) -- +(-90:3cm) -- +(-120:3cm) -- (0,-2);
	\node at (0,2) [Tbox] (A) {$Q$};
	\node at (0,-2) [Tbox] (B) {$Q$};
	\node at (A.180) [left] {$\star$};
	\node at (B.180) [left] {$\star$};
	\node at (-1.5,0)  {$\star$};	
	
	\draw[ultra thick] (2,2) arc (0:180:2cm) -- (-2,-2) arc (-180:0:2cm) -- (2,2);
\end{tikzpicture}$
$\in$
span \{ $TL_3, T , Q$ \},
\item  (R3'(T)),(R3'(Q))\hspace{.5cm} 
$\begin{tikzpicture}[scale=.6, baseline]
	\clip [draw] (2,2) arc (0:180:2cm) -- (-2,-2) arc (-180:0:2cm) -- (2,2);
	
	\draw (0,2) .. controls ++(-115:1.5cm) and ++(115:1.5cm) .. (0,-2);
	\draw[shaded] (0,2) .. controls ++(-90:1.5cm) and ++(90:1.5cm) .. (0,-2);
	\draw[shaded] (0,2) .. controls ++(-65:1.5cm) and ++(65:1.5cm) .. (0,-2);
	\draw[shaded] (0,-2) -- +(-90:3cm) -- +(-65:3cm) -- (0,-2);
	\draw (0,2) -- +(90:3cm) -- +(120:3cm) -- (0,2);
	\draw[shaded] (0,2) -- +(65:3cm)--+(90:3cm) -- (0,2);
	\filldraw[shaded] (0,-2) ++(-120:4cm) -- (0,-2) .. controls ++(115:1.5cm) and ++(-115:1.5cm) .. (0,2) -- ++(120:4cm);
	\draw (0,-2) -- +(-90:3cm) -- +(-120:3cm) -- (0,-2);
	\node at (0,2) [Tbox] (A) {$T$};
	\node at (0,-2) [Tbox] (B) {$T$};
	\node at (A.75) [above left] {$\star$};
	\node at (B.70) [above left] {$\star$};
	\draw (0,2) ++(115:2cm) node [below right=-.7mm] {$\star$};
	
	\draw[ultra thick] (2,2) arc (0:180:2cm) -- (-2,-2) arc (-180:0:2cm) -- (2,2);
\end{tikzpicture}$, \quad
$\begin{tikzpicture}[scale=.6, baseline]
	\clip [draw] (2,2) arc (0:180:2cm) -- (-2,-2) arc (-180:0:2cm) -- (2,2);
	
	\draw (0,2) .. controls ++(-115:1.5cm) and ++(115:1.5cm) .. (0,-2);
	\draw[shaded] (0,2) .. controls ++(-90:1.5cm) and ++(90:1.5cm) .. (0,-2);
	\draw[shaded] (0,2) .. controls ++(-65:1.5cm) and ++(65:1.5cm) .. (0,-2);
	\draw[shaded] (0,-2) -- +(-90:3cm) -- +(-65:3cm) -- (0,-2);
	\draw (0,2) -- +(90:3cm) -- +(120:3cm) -- (0,2);
	\draw[shaded] (0,2) -- +(65:3cm)--+(90:3cm) -- (0,2);
	\filldraw[shaded] (0,-2) ++(-120:4cm) -- (0,-2) .. controls ++(115:1.5cm) and ++(-115:1.5cm) .. (0,2) -- ++(120:4cm);
	\draw (0,-2) -- +(-90:3cm) -- +(-120:3cm) -- (0,-2);
	\node at (0,2) [Tbox] (A) {$Q$};
	\node at (0,-2) [Tbox] (B) {$Q$};
	\node at (A.75) [above left] {$\star$};
	\node at (B.70) [above left] {$\star$};
	\draw (0,2) ++(115:2cm) node [below right=-.7mm] {$\star$};
	
	\draw[ultra thick] (2,2) arc (0:180:2cm) -- (-2,-2) arc (-180:0:2cm) -- (2,2);
\end{tikzpicture}$
$\in$
span \{ $TL_3, T , Q$ \},
\item  (R3(TQ)), (R3(QT))\hspace{.5cm} 
$\begin{tikzpicture}[scale=.6, baseline]
	\clip [draw] (2,2) arc (0:180:2cm) -- (-2,-2) arc (-180:0:2cm) -- (2,2);
	
	\draw[shaded] (0,2) .. controls ++(-115:1.5cm) and ++(115:1.5cm) .. (0,-2);
	\draw (0,2) .. controls ++(-90:1.5cm) and ++(90:1.5cm) .. (0,-2);
	\draw (0,2) .. controls ++(-65:1.5cm) and ++(65:1.5cm) .. (0,-2);
	\draw[shaded] (0,2) -- +(90:3cm) -- +(120:3cm) -- (0,2);
	\filldraw[shaded] (0,2) ++(65:5cm) -- (0,2) .. controls ++(-65:1.5cm) and ++(65:1.5cm) .. (0,-2) -- ++(-65:5cm);
	
	\draw[shaded] (0,-2) -- +(-90:3cm) -- +(-120:3cm) -- (0,-2);
	\node at (0,2) [Tbox] (A) {$T$};
	\node at (0,-2) [Tbox] (B) {$Q$};
	\node at (A.180) [left] {$\star$};
	\node at (B.180) [left] {$\star$};
	\node at (-1.5,0)  {$\star$};	
	
	\draw[ultra thick] (2,2) arc (0:180:2cm) -- (-2,-2) arc (-180:0:2cm) -- (2,2);
\end{tikzpicture}$, \quad
	$\begin{tikzpicture}[scale=.6, baseline]
	\clip [draw] (2,2) arc (0:180:2cm) -- (-2,-2) arc (-180:0:2cm) -- (2,2);
	
	\draw[shaded] (0,2) .. controls ++(-115:1.5cm) and ++(115:1.5cm) .. (0,-2);
	\draw (0,2) .. controls ++(-90:1.5cm) and ++(90:1.5cm) .. (0,-2);
	\draw (0,2) .. controls ++(-65:1.5cm) and ++(65:1.5cm) .. (0,-2);
	\draw[shaded] (0,2) -- +(90:3cm) -- +(120:3cm) -- (0,2);
	\filldraw[shaded] (0,2) ++(65:5cm) -- (0,2) .. controls ++(-65:1.5cm) and ++(65:1.5cm) .. (0,-2) -- ++(-65:5cm);
	
	\draw[shaded] (0,-2) -- +(-90:3cm) -- +(-120:3cm) -- (0,-2);
	\node at (0,2) [Tbox] (A) {$Q$};
	\node at (0,-2) [Tbox] (B) {$T$};
	\node at (A.180) [left] {$\star$};
	\node at (B.180) [left] {$\star$};
	\node at (-1.5,0)  {$\star$};	
	
	\draw[ultra thick] (2,2) arc (0:180:2cm) -- (-2,-2) arc (-180:0:2cm) -- (2,2);
\end{tikzpicture}$
$\in$
span \{ $TL_3, T , Q$ \},
\item  (R3'(TQ)),(R3'(QT))\hspace{.5cm} 
$\begin{tikzpicture}[scale=.6, baseline]
	\clip [draw] (2,2) arc (0:180:2cm) -- (-2,-2) arc (-180:0:2cm) -- (2,2);
	
	\draw (0,2) .. controls ++(-115:1.5cm) and ++(115:1.5cm) .. (0,-2);
	\draw[shaded] (0,2) .. controls ++(-90:1.5cm) and ++(90:1.5cm) .. (0,-2);
	\draw[shaded] (0,2) .. controls ++(-65:1.5cm) and ++(65:1.5cm) .. (0,-2);
	\draw[shaded] (0,-2) -- +(-90:3cm) -- +(-65:3cm) -- (0,-2);
	\draw (0,2) -- +(90:3cm) -- +(120:3cm) -- (0,2);
	\draw[shaded] (0,2) -- +(65:3cm)--+(90:3cm) -- (0,2);
	\filldraw[shaded] (0,-2) ++(-120:4cm) -- (0,-2) .. controls ++(115:1.5cm) and ++(-115:1.5cm) .. (0,2) -- ++(120:4cm);
	\draw (0,-2) -- +(-90:3cm) -- +(-120:3cm) -- (0,-2);
	\node at (0,2) [Tbox] (A) {$T$};
	\node at (0,-2) [Tbox] (B) {$Q$};
	\node at (A.75) [above left] {$\star$};
	\node at (B.70) [above left] {$\star$};
	\draw (0,2) ++(115:2cm) node [below right=-.7mm] {$\star$};
	
	\draw[ultra thick] (2,2) arc (0:180:2cm) -- (-2,-2) arc (-180:0:2cm) -- (2,2);
\end{tikzpicture}$, \quad
$\begin{tikzpicture}[scale=.6, baseline]
	\clip [draw] (2,2) arc (0:180:2cm) -- (-2,-2) arc (-180:0:2cm) -- (2,2);
	
	\draw (0,2) .. controls ++(-115:1.5cm) and ++(115:1.5cm) .. (0,-2);
	\draw[shaded] (0,2) .. controls ++(-90:1.5cm) and ++(90:1.5cm) .. (0,-2);
	\draw[shaded] (0,2) .. controls ++(-65:1.5cm) and ++(65:1.5cm) .. (0,-2);
	\draw[shaded] (0,-2) -- +(-90:3cm) -- +(-65:3cm) -- (0,-2);
	\draw (0,2) -- +(90:3cm) -- +(120:3cm) -- (0,2);
	\draw[shaded] (0,2) -- +(65:3cm)--+(90:3cm) -- (0,2);
	\filldraw[shaded] (0,-2) ++(-120:4cm) -- (0,-2) .. controls ++(115:1.5cm) and ++(-115:1.5cm) .. (0,2) -- ++(120:4cm);
	\draw (0,-2) -- +(-90:3cm) -- +(-120:3cm) -- (0,-2);
	\node at (0,2) [Tbox] (A) {$Q$};
	\node at (0,-2) [Tbox] (B) {$T$};
	\node at (A.75) [above left] {$\star$};
	\node at (B.70) [above left] {$\star$};
	\draw (0,2) ++(115:2cm) node [below right=-.7mm] {$\star$};
	
	\draw[ultra thick] (2,2) arc (0:180:2cm) -- (-2,-2) arc (-180:0:2cm) -- (2,2);
\end{tikzpicture}$
$\in$
span \{ $TL_3, T , Q$ \}
\end{itemize}
\end{theorem}

We recast the above theorem in more explicit form, as follows.  
\begin{theorem}	We have the following 3-box relations:
	\begin{enumerate}
	\item(R3(T)) $T^2= f^{(3)}+\frac{Z(T^3)}{[4]}T+\frac{Z(QT^2)}{[4]}Q$
	\item(R3(Q)) $Q^2= f^{(3)}+\frac{Z(Q^3)}{[4]}Q+\frac{Z(TQ^2)}{[4]}T$
	\item(R3'(T)) 		$Join_3'(T,T)=\rho^{-1/2}(f^{(3)})+\frac{Z((\rho^{1/2}(T))^3)}{[4]}T+\omega^2\frac{Z(\rho^{1/2}(Q)(\rho^{1/2}(T))^2)}{[4]}Q$
	\item(R3'(Q)) $Join_3'(Q,Q)=\omega\rho^{-1/2}(f^{(3)})+\omega^2\frac{Z((\rho^{1/2}(Q))^3)}{[4]}Q+\frac{Z(\rho^{1/2}(T)(\rho^{1/2}(Q))^2)}{[4]}T$
	\item (R3(TQ)) $TQ=\frac{Z(T^2Q)}{[4]}T+\frac{Z(TQ^2)}{[4]}Q$
	\item (R3(QT)) $QT=\frac{Z(Q^2T)}{[4]}Q+\frac{Z(QT^2)}{[4]}T$
	\item (R3'(TQ)) $Join_3'(T,Q)=\frac{Z((\rho^{1/2}(T))^2\rho^{1/2}(Q))}{[4]}T+\omega^2\frac{Z(\rho^{1/2}(T)(\rho^{1/2}(Q))^2)}{[4]}Q$
	\item (R3'(QT)) $Join_3'(Q,T)=\frac{Z((\rho^{1/2}(T))^2\rho^{1/2}(Q))}{[4]}T+\omega^2\frac{Z(\rho^{1/2}(T)(\rho^{1/2}(Q))^2)}{[4]}Q$
	\end{enumerate}
\end{theorem}
Here is the relevant notation.
\begin{notation}\cite{EP}
Let $Join_{n}(A,B)$ be the tangle with an $A$ box joined to a $B$ box by the $n$ strings counterclockwise of $A$'s star and the $n$ strings clockwise of $B$'s star.  There is also a `twisted' version of this, $Join_{n}'(A,B)=\rho^{-1/2}Join_{n}(\rho^{1/2}A,\rho^{1/2}B)$:

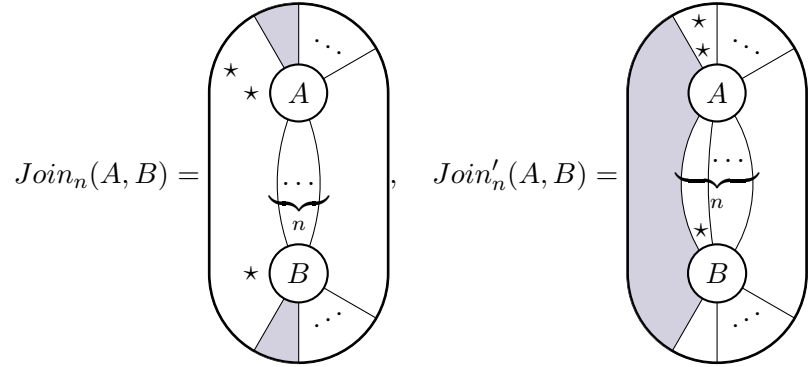
$$ Join_{n}(A,B) = 
\begin{tikzpicture}[scale=.6, baseline]
	\clip [draw] (2,2) arc (0:180:2cm) -- (-2,-2) arc (-180:0:2cm) -- (2,2);
	
	\draw (0,2) .. controls ++(-115:1.5cm) and ++(115:1.5cm) .. (0,-2);
	\draw (0,2) .. controls ++(-65:1.5cm) and ++(65:1.5cm) .. (0,-2);
	
	\draw[shaded] (0,2) -- +(90:3cm) -- +(120:3cm) -- (0,2);
	\draw (0,2) -- ++(30:3cm);
		
	\draw (0,-2) -- ++(-30:3cm);
	\draw[shaded] (0,-2) -- +(-90:3cm) -- +(-120:3cm) -- (0,-2);
	
	\node at (0,0) {$\ldots$};
	\node at (0,0) [below] {$\underbrace{\qquad}_{n}$};
	\node at (.7,3) [rotate=-30] {$\dots$};
	\node at (.7,-3) [rotate=30] {$\dots$};

	\node at (0,2) [Tbox] (A) {$A$};
	\node at (0,-2) [Tbox] (B) {$B$};
	\node at (A.180) [left] {$\star$};
	\node at (B.180) [left] {$\star$};
	\node at (-1.5,2.5)  {$\star$};	
	
	\draw[ultra thick] (2,2) arc (0:180:2cm) -- (-2,-2) arc (-180:0:2cm) -- (2,2);
\end{tikzpicture},  \quad
Join_{n}'(A,B) =
\begin{tikzpicture}[scale=.6, baseline]
	\clip [draw] (2,2) arc (0:180:2cm) -- (-2,-2) arc (-180:0:2cm) -- (2,2);
	
	\draw (0,2) .. controls ++(-100:1.5cm) and ++(100:1.5cm) .. (0,-2);
	\draw (0,2) .. controls ++(-45:1.5cm) and ++(45:1.5cm) .. (0,-2);
	
	\draw[shaded] (0,2)--++(120:3cm) arc (120:180:3cm) -- ++(0,-4) arc (-180:-120:3cm) -- (0,-2) .. controls ++(135:1.5cm) and ++(-135:1.5cm) .. (0,2);
	
	\draw (0,4) -- (0,2);
	\draw (0,2) -- ++(30:3cm);
		
	\draw (0,-2) -- ++(-30:3cm);
	\draw (0,-2) -- (0,-4);
	
	\node at (.25,.5) {$\ldots$};
	\node at (0,.5) [below] {$\underbrace{\qquad \; \; \; }_{n}$};
	\node at (.7,3) [rotate=-30] {$\dots$};
	\node at (.7,-3) [rotate=30] {$\dots$};

	\node at (0,2) [Tbox] (A) {$A$};
	\node at (0,-2) [Tbox] (B) {$B$};
	\node at (A.90) [above left=-.5mm] {$\star$};
	\node at (B.122) [above] {$\star$};
	\node at (-.4,3.6)  {$\star$};	
	
	\draw[ultra thick] (2,2) arc (0:180:2cm) -- (-2,-2) arc (-180:0:2cm) -- (2,2);
\end{tikzpicture}
$$
\end{notation}
\section{Relations among 3-boxes}
To prove all these relations we use Bessel's inequality which states: For any $v\in V$, $W\subset V$, we have $\left\Vert v \right\Vert^2 \geq \left\Vert \proj{W}{v} \right\Vert^2$, and  $\left\Vert v \right\Vert^2 = \left\Vert \proj{W}{v} \right\Vert^2$ if and only if $v \in W$.
\begin{proof}[Proof of (R3(T))] .  We will show $T^2= f^{(3)}+\frac{Z(T^3)}{[4]}T+\frac{Z(QT^2)}{[4]}Q$ by showing that $$\Norm{T^2}^2=\Norm{\proj{span(TL,T,Q)}{T^2}}^2.$$
On one hand, 
\begin{align*}
\Norm{T^2}^2 & =\ip{T^2,T^2} \\
& = Z(T^4)\\
& = \frac{5}{3}\sqrt{\frac{2}{3}(23+5\sqrt{21})},    \mbox{\indent by Lemma 5.0.14}
\end{align*}  

On the other hand, 
\begin{align*}
\Norm{\proj{span(TL_3,T,Q)}{T^2}}^2	& = \Norm{\proj{TL_3}{T^2}+\proj{T}{T^2}+\proj{Q}{T^2}}^2\\
	& = \Norm{\sum_{\beta \in B(TL_3)}\ip{T^2,\beta}\hat{\beta}+\ip{T^2,T}\hat{T}+\ip{T^2,Q}\hat{Q}}^2\\
	& = \Norm{\ip{T^2,\id}\frac{f^{(3)}}{[4]}+Z(T^3)\frac{T}{Z(T^2)}+Z(QT^2)\frac{Q}{Z(Q^2)}}^2\\
	& = \Norm{(Z(T^2)\frac{f^{(3)}}{[4]}}^2+\Norm{\frac{Z(T^3)}{Z(T^2)}T+\frac{Z(QT^2)}{Z(Q^2)}Q}^2\\
	&   \mbox{\indent because $\ip{f^{(3)},T}=0, \ip{f^{(3)},Q}=0$ }\\
	& = \Norm{f^{(3)}}^2+\Norm{\frac{Z(T^3)}{Z(T^2)}T}^2+\Norm{\frac{Z(QT^2)}{Z(Q^2)}Q}^2, \\
	&   \mbox{\indent because $\ip{Q,T}=0$}\\
	& = [4]+\frac{\norm{Z(T^3)}^2}{[4]}+\frac{\norm{Z(QT^2)}^2}{[4]}\\
	& = \frac{5}{3}\sqrt{\frac{2}{3}(23+5\sqrt{21})},\mbox{\indent by Lemma 5.0.14}
\end{align*}
So $\Norm{T^2}^2=\Norm{\proj{span(TL_3,T,Q}{T^2}}^2,$ which implies that $T^2 \in span(TL_3,T,Q)$.  Hence, $T^2=f^{(3)}+\frac{Z(T^3)}{[4]}T+\frac{Z(QT^2)}{[4]}Q$.

\end{proof}

The proof of (R3(Q)) is analogous.  
\begin{proof}[Proof of (R3(Q))]   
\begin{align*}
\Norm{\proj{span(TL_3,T,Q)}{Q^2}}^2	& = \Norm{\proj{TL_3}{Q^2}+\proj{T}{Q^2}+\proj{Q}{Q^2}}^2\\
	& = \Norm{\sum_{\beta \in B(TL_3)}\ip{Q^2,\beta}\hat{\beta}+\ip{Q^2,T}\hat{T}+\ip{Q^2,Q}\hat{Q}}^2\\
	& = \Norm{\ip{Q^2,\id}\frac{f^{(3)}}{[4]}+Z(TQ^2)\frac{T}{Z(T^2)}+Z(Q^3)\frac{Q}{Z(Q^2)}}^2\\
	& = \Norm{(Z(Q^2)\frac{f^{(3)}}{[4]}}^2+\Norm{\frac{Z(TQ^2)}{Z(T^2)}T+\frac{Z(Q^3)}{Z(Q^2)}Q}^2\\
	&   \mbox{\indent because $\ip{f^{(3)},T}=0, \ip{f^{(3)},Q}=0$ }\\
	& = \Norm{f^{(3)}}^2+\Norm{\frac{Z(TQ^2)}{Z(T^2)}T}^2+\Norm{\frac{Z(Q^3)}{Z(Q^2)}Q}^2, \\
	&   \mbox{\indent because $\ip{Q,T}=0$}\\
	& = [4]+\frac{\norm{Z(TQ^2)}^2}{[4]}+\frac{\norm{Z(Q^3)}^2}{[4]}\\
	& = Z(Q^4),\mbox{\indent by Lemma 5.0.14}
\end{align*}
So $\Norm{Q^2}^2=\Norm{\proj{span(TL_3,T,Q}{Q^2}}^2,$ which implies that $Q^2 \in span(TL_3,T,Q)$.  Hence, $Q^2=f^{(3)}+\frac{Z(TQ^2)}{[4]}T+\frac{Z(Q^3)}{[4]}Q$.

\end{proof}

We now prove the twisted versions of (R3(T)) and (R3(Q)).
\begin{proof}[Proof of (R3'(T))] 
Claim:		$$Join_3'(T,T)=\rho^{-1/2}(f^{(3)})+\frac{Z((\rho^{1/2}(T))^3)}{[4]}T+\omega^2\frac{Z(\rho^{1/2}(Q)(\rho^{1/2}(T))^2)}{[4]}Q$$\\
On the one hand, we have that 
\begin{align*}
\Norm{Join_3'(T,T)}^2 
	& = \ip{Join_3'(T,T),Join_3'(T,T)}\\
	& = Z((\rho^{1/2}(T))^4)\\
	& = \frac{5}{3}\sqrt{\frac{2}{3}(23+5\sqrt{21})}
\end{align*}

On the other hand, 
\begin{align*}
&\Norm{\proj{span(TL_3,T,Q)}{Join_3'(T,T)}}^2 \\
& = \Norm{\proj{TL_3}{Join_3'(T,T)}+\proj{T}{Join_3'(T,T)}+\proj{Q}{Join_3'(T,T)}}^2\\
& = \Norm{\sum_{\beta \in B(TL_3)}\ip{Join_3'(T,T),\beta}\hat{\beta}+\ip{Join_3'(T,T),T}\hat{T}+\ip{Join_3'(T,T),Q}\hat{Q}}^2\\
& = \Big\Vert\ip{Join_3'(T,T),\id}\hat{\id}+\ip{Join_3'(T,T),\begin{tikzpicture}[scale=.2]
 	\draw (1,1)--(1,-1);
	\draw (2,1)arc (-180:0:.5cm);
	\draw (2,-1) arc (180:0:.5cm);
 \end{tikzpicture}}\widehat{\begin{tikzpicture}[scale=.2]
 	\draw (1,1)--(1,-1);
	\draw (2,1)arc (-180:0:.5cm);
	\draw (2,-1) arc (180:0:.5cm);
 \end{tikzpicture}}+\ip{Join_3'(T,T),\big(\begin{tikzpicture}[scale=.2]
 	\draw (3,1)--(1,-1);
	\draw (1,1)arc (-180:0:.5cm);
	\draw (2,-1) arc (180:0:.5cm);
 \end{tikzpicture}\big)^{*}}\widehat{\big(\begin{tikzpicture}[scale=.2]
 	\draw (3,1)--(1,-1);
	\draw (1,1)arc (-180:0:.5cm);
	\draw (2,-1) arc (180:0:.5cm);
 \end{tikzpicture}\big)^{*}} \\
& + \ip{Join_3'(T,T),T}\frac{T}{Z(T^2)}+\ip{Join_3'(T,T),Q}\frac{Q}{Z(Q^2)}\Big\Vert^2\\
& = \Norm{[4]\widehat{\big(\begin{tikzpicture}[scale=.2]
 	\draw (3,1)--(1,-1);
	\draw (1,1)arc (-180:0:.5cm);
	\draw (2,-1) arc (180:0:.5cm);
 \end{tikzpicture}\big)^{*}}+\ip{Join_3'(T,T),T}\frac{T}{Z(T^2)}+\ip{Join_3'(T,T),Q}\frac{Q}{Z(Q^2)}}^2\\
& =  \Norm{\rho^{-1/2}f^{(3)}+\ip{Join_3'(T,T),T}\frac{T}{Z(T^2)}+\ip{Join_3'(T,T),Q}\frac{Q}{Z(Q^2)}}^2\\
&   \mbox{\indent because $\widehat{\big(\begin{tikzpicture}[scale=.2]
 	\draw (3,1)--(1,-1);
	\draw (1,1)arc (-180:0:.5cm);
	\draw (2,-1) arc (180:0:.5cm);
 \end{tikzpicture}\big)^{*}}= \frac{\rho^{-1/2}f^{(3)}}{[4]}$}\\
& = \Norm{\rho^{-1/2}f^{(3)}+Z((\rho^{1/2}(T))^3)\frac{T}{Z(T^2)}+\omega^2Z(\rho^{1/2}(Q)(\rho^{1/2}(T))^2)\frac{Q}{Z(Q^2)}}^2\\
& = \Norm{\rho^{-1/2}f^{(3)}}^2+\Norm{Z((\rho^{1/2}(T))^3)\frac{T}{Z(T^2)}}^2+\Norm{\omega^2Z(\rho^{1/2}(Q)(\rho^{1/2}(T))^2)\frac{Q}{Z(Q^2)}}^2\\
& \mbox{\noindent because $\ip{\rho^{-1/2}f^{(3)},T}=0,\ip{\rho^{-1/2}f^{(3)},Q}=0,\ip{T,Q}=0$}\\
& = \frac{5}{3}\sqrt{\frac{2}{3}(23+5\sqrt{21})}\mbox{\indent by Lemma 5.0.14}
\end{align*}
\end{proof}
We have an analogous situation by switching $T$ and $Q$. 
\begin{proof}[Proof of (R3'(Q))] 
Claim:	$$Join_3'(Q,Q)=\omega\rho^{-1/2}(f^{(3)})+\omega^2\frac{Z((\rho^{1/2}(Q))^3)}{[4]}Q+\frac{Z(\rho^{1/2}(T)(\rho^{1/2}(Q))^2)}{[4]}T$$
On the one hand, we have that 
\begin{align*}
\Norm{Join_3'(Q,Q}^2 &= \ip{Join_3'(Q,Q),Join_3'(Q,Q}\\
& = \omega Z((\rho^{1/2}(Q))^4)\\
& = \frac{2}{9}\sqrt{8\sqrt{3}+9\sqrt{7}}
\end{align*}
On the other hand, 
\begin{align*}
&\Norm{\proj{span(TL_3,T,Q)}{Join_3'(Q,Q)}}^2 \\
& = \Norm{\proj{TL_3}{Join_3'(Q,Q)}+\proj{T}{Join_3'(Q,Q)}+\proj{Q}{Join_3'(Q,Q)}}^2\\
& = \Norm{\sum_{\beta \in B(TL_3)}\ip{Join_3'(Q,Q),\beta}\hat{\beta}+\ip{Join_3'(Q,Q),T}\hat{T}+\ip{Join_3'(Q,Q),Q}\hat{Q}}^2\\
& = \Big\Vert\ip{Join_3'(Q,Q),\id}\hat{\id}+\ip{Join_3'(Q,Q),\begin{tikzpicture}[scale=.2]
 	\draw (1,1)--(1,-1);
	\draw (2,1)arc (-180:0:.5cm);
	\draw (2,-1) arc (180:0:.5cm);
 \end{tikzpicture}}\widehat{\begin{tikzpicture}[scale=.2]
 	\draw (1,1)--(1,-1);
	\draw (2,1)arc (-180:0:.5cm);
	\draw (2,-1) arc (180:0:.5cm);
 \end{tikzpicture}}+\ip{Join_3'(Q,Q),\big(\begin{tikzpicture}[scale=.2]
 	\draw (3,1)--(1,-1);
	\draw (1,1)arc (-180:0:.5cm);
	\draw (2,-1) arc (180:0:.5cm);
 \end{tikzpicture}\big)^{*}}\widehat{\big(\begin{tikzpicture}[scale=.2]
 	\draw (3,1)--(1,-1);
	\draw (1,1)arc (-180:0:.5cm);
	\draw (2,-1) arc (180:0:.5cm);
 \end{tikzpicture}\big)^{*}} \\
& + \ip{Join_3'(Q,Q),T}\frac{T}{Z(T^2)}+\ip{Join_3'(Q,Q),Q}\frac{Q}{Z(Q^2)}\Big\Vert^2\\
& = \Norm{\omega[4]\widehat{\big(\begin{tikzpicture}[scale=.2]
 	\draw (3,1)--(1,-1);
	\draw (1,1)arc (-180:0:.5cm);
	\draw (2,-1) arc (180:0:.5cm);
 \end{tikzpicture}\big)^{*}}+\ip{Join_3'(Q,Q),T}\frac{T}{Z(T^2)}+\ip{Join_3'(Q,Q),Q}\frac{Q}{Z(Q^2)}}^2\\
& =  \Norm{\omega\rho^{-1/2}f^{(3)}+\ip{Join_3'(Q,Q),T}\frac{T}{Z(T^2)}+\ip{Join_3'(Q,Q),Q}\frac{Q}{Z(Q^2)}}^2\\
&   \mbox{\indent because $\widehat{\big(\begin{tikzpicture}[scale=.2]
 	\draw (3,1)--(1,-1);
	\draw (1,1)arc (-180:0:.5cm);
	\draw (2,-1) arc (180:0:.5cm);
 \end{tikzpicture}\big)^{*}}= \frac{\rho^{-1/2}f^{(3)}}{[4]}$}\\
& = \Norm{\omega\rho^{-1/2}f^{(3)}+Z((\rho^{1/2}(T)(\rho^{1/2}(Q))^2))\frac{T}{Z(T^2)}+\omega^2Z((\rho^{1/2}(Q))^3)\frac{Q}{Z(Q^2)}}^2\\
& = \Norm{\rho^{-1/2}f^{(3)}}^2+\Norm{Z((\rho^{1/2}(T)(\rho^{1/2}(Q))^2))\frac{T}{Z(T^2)}}^2+\Norm{\omega^2Z((\rho^{1/2}(Q))^3)\frac{Q}{Z(Q^2)}}^2\\
& \mbox{\noindent because $\ip{\rho^{-1/2}f^{(3)},T}=0,\ip{\rho^{-1/2}f^{(3)},Q}=0,\ip{T,Q}=0$}\\
& = \frac{2}{9}\sqrt{8\sqrt{3}+9\sqrt{7}}\mbox{\indent by Lemma 5.0.14}
\end{align*}
\end{proof}

Next, we want to see how $T$ and $Q$ interact.  
\begin{proof}[Proof of (R3(TQ))] 
Claim:  $$TQ=\frac{Z(T^2Q)}{[4]}T+\frac{Z(TQ^2)}{[4]}Q$$
On the one hand, 
\begin{align*}
\Norm{TQ}^2 & =\ip{TQ, TQ}\\
& = Z((TQ)^{*}TQ)=Z(Q^{*}T^{*}TQ)\\
& = Z(QT^2Q)\\
& =Z(Q^2T^2)\\
& = \frac{7}{3}\sqrt{\frac{11+\sqrt{21}}{6}}\mbox{\indent by Lemma 5.0.14}
\end{align*}
On the other hand, 
\begin{align*}
\Norm{\proj{span(TL_3,T,Q)}{TQ}}^2 & = \Norm{\sum_{\beta \in B(TL_3)}\ip{TQ,\beta}\hat{\beta}+\ip{TQ,T}\hat{T}+\ip{TQ,Q}\hat{Q}}^2\\
& = \Norm{\ip{TQ,\id}\hat{\id}+\ip{TQ,T}\hat{T}+\ip{TQ,Q}\hat{Q}}^2\\
& = \Norm{Z(TQ)\hat{\id}+Z(T^2Q)\frac{T}{[4]}+Z(Q^2T)\frac{Q}{[4]}}^2\\
& = \Norm{\frac{Z(T^2Q)}{[4]}T}^2+\Norm{\frac{Z(Q^2T)}{[4]}Q}^2\mbox{\indent because $\ip{T,Q}=0$}\\
& = \frac{7}{3}\sqrt{\frac{11+\sqrt{21}}{6}}\mbox{\indent by Lemma 5.0.14}
\end{align*}
\end{proof}
\begin{proof}[Proof of (R3(QT))] 
Claim: $$QT=\frac{Z(Q^2T)}{[4]}Q+\frac{Z(QT^2)}{[4]}T$$
The proof of (R3(QT)) follows by switching $T$ and $Q$ in the requisite equality in the proof of (R3(TQ)), that is, by checking that the following equality holds:
$$\Norm{\frac{Z(Q^2T)}{[4]}Q}^2+\Norm{\frac{Z(QT^2)}{[4]}T}^2=Z(T^2Q^2)$$
This equality holds because we already saw from the proof of (R3(TQ)) that the following holds: $$\frac{\norm{Z(T^2Q)}^2}{[4]}+\frac{\norm{Z(Q^2T)}^2}{[4]}=Z(Q^2T^2)$$
\end{proof}

We now prove the twisted versions of (R3(TQ)) and (R3(QT)).
\begin{proof}[Proof of (R3'(TQ))]
Claim:  $$Join_3'(T,Q)=\frac{Z((\rho^{1/2}(T))^2\rho^{1/2}(Q))}{[4]}T+\omega^2\frac{Z(\rho^{1/2}(T)(\rho^{1/2}(Q))^2)}{[4]}Q$$
On the one hand, 
\begin{align*}
\Norm{Join_3'(T,Q)}^2& = \ip{Join_3'(T,Q),Join_3'(T,Q)}\\
& = \omega^2Z((\rho^{1/2}(Q))^2(\rho^{1/2}(T))^2)\\
& = \frac{7(1+\sqrt{21}}{6\sqrt{3}}
\end{align*}
On the other hand, 
\begin{align*}
& \Norm{\proj{span(TL_3,T,Q)}{Join_3'(T,Q)}}^2\\ 
& =\Norm{\proj{TL_3}{Join_3'(T,Q)}+\proj{T}{Join_3'(T,Q)}+\proj{Q}{Join_3'(T,Q)}}^2\\
& = \Norm{\sum_{\beta \in B(TL_3)}\ip {Join_3'(T,Q),\beta}\hat{\beta}+\ip{Join_3'(T,Q),T}\hat{T}+\ip{Join_3'(T,Q),Q}\hat{Q}}^2\\
& = \Norm{ \ip{Join_3'(T,Q),\big(\begin{tikzpicture}[scale=.2]
 	\draw (3,1)--(1,-1);
	\draw (1,1)arc (-180:0:.5cm);
	\draw (2,-1) arc (180:0:.5cm);
 \end{tikzpicture}\big)^{*}}\widehat{\big(\begin{tikzpicture}[scale=.2]
 	\draw (3,1)--(1,-1);
	\draw (1,1)arc (-180:0:.5cm);
	\draw (2,-1) arc (180:0:.5cm);
 \end{tikzpicture}\big)^{*}} +\ip{Join_3'(T,Q),T}\frac{T}{[4]}+\ip{Join_3'(T,Q),Q}\frac{Q}{[4]}}^2\\
 & = \Norm{\omega Z(TQ)\widehat{\big(\begin{tikzpicture}[scale=.2]
 	\draw (3,1)--(1,-1);
	\draw (1,1)arc (-180:0:.5cm);
	\draw (2,-1) arc (180:0:.5cm);
 \end{tikzpicture}\big)^{*}}+ \ip{Join_3'(T,Q),T}\frac{T}{[4]}+\ip{Join_3'(T,Q),Q}\frac{Q}{[4]}}^2\\
 & = \Norm{\ip{Join_3'(T,Q),T}\frac{T}{[4]}+\ip{Join_3'(T,Q),Q}\frac{Q}{[4]}}^2\\
 & = \Norm{Z((\rho^{1/2}(T))^2\rho^{1/2}(Q))\frac{T}{[4]}+\omega^2Z((\rho^{1/2}(Q))^2\rho^{1/2}(T))\frac{Q}{[4]}}^2\\
 & =\Norm{Z((\rho^{1/2}(T))^2\rho^{1/2}(Q))\frac{T}{[4]}}^2+\Norm{Z((\rho^{1/2}(Q))^2\rho^{1/2}(T))\frac{Q}{[4]}}^2\\
 & = \frac{\norm{Z((\rho^{1/2}(T))^2\rho^{1/2}(Q))}^2}{[4]}+\frac{\norm{Z((\rho^{1/2}(Q))^2\rho^{1/2}(T))}^2}{[4]}\\
 & = \frac{7(1+\sqrt{21}}{6\sqrt{3}}
\end{align*}
\end{proof}

The proof of the twisted version of (R3(QT)) is very similar.
\begin{proof}[Proof of (R3'(QT))]
Claim:  $$Join_3'(Q,T)=\frac{Z((\rho^{1/2}(T))^2\rho^{1/2}(Q))}{[4]}T+\omega^2\frac{Z(\rho^{1/2}(T)(\rho^{1/2}(Q))^2)}{[4]}Q$$
On the one hand, 
\begin{align*}
\Norm{Join_3'(Q,T)}^2& = \ip{Join_3'(Q,T),Join_3'(Q,T)}\\
& = \omega^2Z((\rho^{1/2}(Q))^2(\rho^{1/2}(T))^2)\\
& = \frac{7(1+\sqrt{21}}{6\sqrt{3}}
\end{align*}
On the other hand, 
\begin{align*}
& \Norm{\proj{span(TL_3,T,Q)}{Join_3'(Q,T)}}^2\\ 
& =\Norm{\proj{TL_3}{Join_3'(Q,T)}+\proj{T}{Join_3'(Q,T)}+\proj{Q}{Join_3'(Q,T)}}^2\\
& = \Norm{\sum_{\beta \in B(TL_3)}\ip {Join_3'(Q,T),\beta}\hat{\beta}+\ip{Join_3'(Q,T),T}\hat{T}+\ip{Join_3'(Q,T),Q}\hat{Q}}^2\\
& = \Norm{ \ip{Join_3'(Q,T),\big(\begin{tikzpicture}[scale=.2]
 	\draw (3,1)--(1,-1);
	\draw (1,1)arc (-180:0:.5cm);
	\draw (2,-1) arc (180:0:.5cm);
 \end{tikzpicture}\big)^{*}}\widehat{\big(\begin{tikzpicture}[scale=.2]
 	\draw (3,1)--(1,-1);
	\draw (1,1)arc (-180:0:.5cm);
	\draw (2,-1) arc (180:0:.5cm);
 \end{tikzpicture}\big)^{*}} +\ip{Join_3'(Q,T),T}\frac{T}{[4]}+\ip{Join_3'(Q,T),Q}\frac{Q}{[4]}}^2\\
 & = \Norm{Z(TQ)\widehat{\big(\begin{tikzpicture}[scale=.2]
 	\draw (3,1)--(1,-1);
	\draw (1,1)arc (-180:0:.5cm);
	\draw (2,-1) arc (180:0:.5cm);
 \end{tikzpicture}\big)^{*}}+ \ip{Join_3'(Q,T),T}\frac{T}{[4]}+\ip{Join_3'(Q,T),Q}\frac{Q}{[4]}}^2\\
 & = \Norm{\ip{Join_3'(Q,T),T}\frac{T}{[4]}+\ip{Join_3'(Q,T),Q}\frac{Q}{[4]}}^2\\
 & = \Norm{Z((\rho^{1/2}(T))^2\rho^{1/2}(Q))\frac{T}{[4]}+\omega^2Z((\rho^{1/2}(Q))^2\rho^{1/2}(T))\frac{Q}{[4]}}^2\\
 & =\Norm{Z((\rho^{1/2}(T))^2\rho^{1/2}(Q))\frac{T}{[4]}}^2+\Norm{Z((\rho^{1/2}(Q))^2\rho^{1/2}(T))\frac{Q}{[4]}}^2\\
 & = \frac{\norm{Z((\rho^{1/2}(T))^2\rho^{1/2}(Q))}^2}{[4]}+\frac{\norm{Z((\rho^{1/2}(Q))^2\rho^{1/2}(T))}^2}{[4]}\\
 & = \frac{7(1+\sqrt{21}}{6\sqrt{3}}
\end{align*}
\end{proof}
\section{Relations among 4-boxes}
In this section, we want to see what happens when we take two generators and connect them up by two strands only.  Recall that we have the following decomposition:$$\mathcal{P}_4=TL_4\oplus 2V_4^{3,\zeta}\oplus V_4^{4,\beta}$$  Since there is only one new element $M$ at level 4, we expect our quadratic relations at level 4 to involve the new element $M$.  We will prove such quadratic relations in this section.  First, we need some definitions and notation.
\begin{Defn}[\cite{ExtH}]
For $m \ge 0$, let
$$W_m = q^m+q^{-m}-\omega-\omega^{-1},$$ 
as on page 30 of \cite{VJ3}.
\end{Defn}
We will let $\omega$ in the above notation depend on the generator in question.  So $T$ will have $\omega=1$ and $Q$ will have $\omega=e^{2\pi i/3}$. 
We begin by looking for some unshaded 4-box relations.
\newpage
\subsection{Unshaded 4-box Relations}
Let
\begin{figure}[!htb]
$$u=\begin{tikzpicture}[baseline=0,scale=2]
	{
	\fill[shaded] (-0.8,0) -- (-0.8,0.6) arc (180:0:0.8) -- (0.8,0) -- (0.2,0) -- (0.2,1) -- (-0.2,1) -- (-0.2,0);
	\draw (-0.8,0) -- (-0.8,0.6) arc (180:0:0.8) -- (0.8,0);
	\draw (-0.2,0) -- (-0.2,1);
	\draw (0.2,0) -- (0.2,1);
	\node at (0.45,0.5) {\footnotesize$5$};
	{%
	\filldraw[fill=white,thick] (0,1) ellipse (3mm and 3mm);
	\node at (0,1) {\Large $T$};
	\path(0,1) ++(-90:0.37) node {$\star$};
}
}
;
        \draw (0,0)--(0,-0.5);
        \node[anchor=west] at (0,-0.35) {\footnotesize$8$};
	{%
	\filldraw[fill=white,thick] (-1,-0.2) rectangle (1,0.2);
	\node at (0,0) {\Large$\JW{8}$};
}\end{tikzpicture},
\quad v=\begin{tikzpicture}[baseline=0,scale=2]
	{
	\fill[shaded] (-0.8,0) -- (-0.8,0.6) arc (180:0:0.8) -- (0.8,0) -- (0.2,0) -- (0.2,1) -- (-0.2,1) -- (-0.2,0);
	\draw (-0.8,0) -- (-0.8,0.6) arc (180:0:0.8) -- (0.8,0);
	\draw (-0.2,0) -- (-0.2,1);
	\draw (0.2,0) -- (0.2,1);
	\node at (0.45,0.5) {\footnotesize$5$};
	{%
	\filldraw[fill=white,thick] (0,1) ellipse (3mm and 3mm);
	\node at (0,1) {\Large $Q$};
	\path(0,1) ++(-90:0.37) node {$\star$};
}
}
;
        \draw (0,0)--(0,-0.5);
        \node[anchor=west] at (0,-0.35) {\footnotesize$8$};
	{%
	\filldraw[fill=white,thick] (-1,-0.2) rectangle (1,0.2);
	\node at (0,0) {\Large$\JW{8}$};
}
\end{tikzpicture} 
.$$
\caption{$u$ and $v$}
\label{fig:AB}
\end{figure}

\begin{lem}
\begin{enumerate}
\item [(1)] $\ip{u,u}=\frac{1}{[8]}W_8\cdot Z(T^2)$, where $\omega$ is taken to be $1$ in the expression for $W_8$. More specifically,$\ip{u,u}= \frac{3}{2}+\frac{11\sqrt{\frac{3}{7}}}{2}$
\item [(2)] $\ip{v,v}=\frac{1}{[8]}W_8\cdot Z(Q^2)$, where $\omega$ is taken to be $e^{2\pi i/3}$ in the expression for $W_8$. More specifically, $\ip{v,v}=3+4\sqrt{\frac{3}{7}}$
\end{enumerate}
The same holds with the reverse shading.
\end{lem}
\begin{proof}[Proof of (1)]
$$\ip{u,u}=
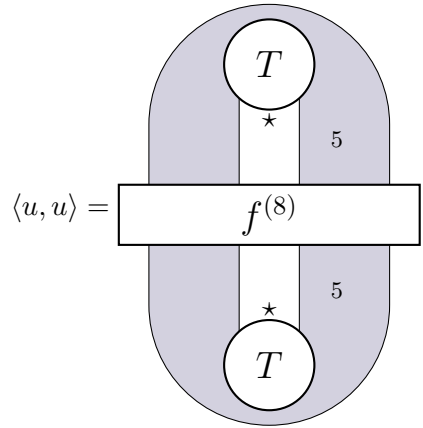
\begin{tikzpicture}[baseline=0,scale=2]
	{
	\fill[shaded] (-0.8,0) -- (-0.8,0.6) arc (180:0:0.8) -- (0.8,0) -- (0.2,0) -- (0.2,1) -- (-0.2,1) -- (-0.2,0);
	\draw (-0.8,0) -- (-0.8,0.6) arc (180:0:0.8) -- (0.8,0);
	\draw (-0.2,0) -- (-0.2,1);
	\draw (0.2,0) -- (0.2,1);
	\node at (0.45,0.5) {\footnotesize$5$};
	{%
	\filldraw[fill=white,thick] (0,1) ellipse (3mm and 3mm);
	\node at (0,1) {\Large $T$};
	\path(0,1) ++(-90:0.37) node {$\star$};
}
}
	{\begin{scope}[y=-1cm] {{
	\fill[shaded] (-0.8,0) -- (-0.8,0.6) arc (180:0:0.8) -- (0.8,0) -- (0.2,0) -- (0.2,1) -- (-0.2,1) -- (-0.2,0);
	\draw (-0.8,0) -- (-0.8,0.6) arc (180:0:0.8) -- (0.8,0);
	\draw (-0.2,0) -- (-0.2,1);
	\draw (0.2,0) -- (0.2,1);
	\node at (0.45,0.5) {\footnotesize$5$};
	{%
	\filldraw[fill=white,thick] (0,1) ellipse (3mm and 3mm);
	\node at (0,1) {\Large $T$};
	\path(0,1) ++(-90:0.37) node {$\star$};
}
}}\end{scope}}
	{%
	\filldraw[fill=white,thick] (-1,-0.2) rectangle (1,0.2);
	\node at (0,0) {\Large$\JW{8}$};
}
\end{tikzpicture}$$
Apply Lemma 2.1.9 to $f^{(8)}$ in the above diagram.  Consider the first term in the equation for Lemma 2.1.9.  Using sphericality, we get a partial trace of $f^{(7)}$, which we can replace with $\frac{[8]}{[7]}f^{(6)}$.  Thus, we have $\frac{[8]}{[7]}Z(T^2)$.  Now, consider the remaining terms in the sum:
$$\frac{1}{[8][7]}
\sum_{a,b = 1}^{7} (-1)^{a+b+1} [a][b] \;
\begin{tikzpicture}[baseline=0,scale=0.3]
\draw[decorate,decoration={brace,raise=2pt}] (-4,3.5) -- (-2,3.5);
\node at (-3,4.6) {$a$};
\draw[decorate,decoration={brace,mirror,raise=2pt}] (-4,-3.5) -- (-1,-3.5);
\node at (-2.5,-4.85) {$b$};
\foreach \x in {-4,-3} { \draw (\x,-3.5) -- (\x+1,0) -- (\x,3.5); }
\draw (-2,-3.5) -- (-1,0);
\draw (-1,-3.5) arc (180:0:0.5);
\draw (-1,3.5) arc (360:180:0.5);
\draw (0,3.5) -- (-1,0);
\foreach \x in {1,2} { \draw (\x,-3.5) -- (\x-1,0) -- (\x,3.5); }
\filldraw[thick,fill=white] (-4,-1.5) rectangle (2,1.5);
\node at (-1,0) {$\JW{6}$};
\end{tikzpicture}$$
Every term in the sum gives $0$ except when $a,b \in \left\{1,7\right\}$.  
\begin{table}[ht]
\center
\begin{tabular}{ c | c  }
$a,b$  & value of diagram \\
\hline
      $a=b=1$
     & $Z(T^2)$ \\
     $a=1, b=7$
     & $\omega Z(T^2)$ \\
     $a=7,b=7$ 
     & $Z(T^2)$\\
     $a=7,b=1$
     & $\omega^2 Z(T^2)$\\
\end{tabular}
\vspace{6pt}
\caption{The remaining terms that make a contribution. }
\end{table}

 So 
\begin{align*} \ip{u,u} & =\frac{[8]}{[7]}Z(T^2)-\frac{1}{[8][7]}(1+\omega^2[7])+[7]^2+\omega[7])Z(T^2)\\
& = \frac{[8]^2-(1+[7]^2+[7](\omega+\omega^{-1})}{[8][7]}Z(T^2)\\
& = \frac{W_8}{[8]}Z(T^2)\\
& = \frac{3}{2}+\frac{11\sqrt{\frac{3}{7}}}{2}.
\end{align*}
\end{proof}

\begin{proof}[Proof of (2)]  The proof of (2) is exactly the same as in (1), except $\omega$ is replaced with $e^{2\pi i/3}$.
For the reverse shading, one can easily check to see that 
$$
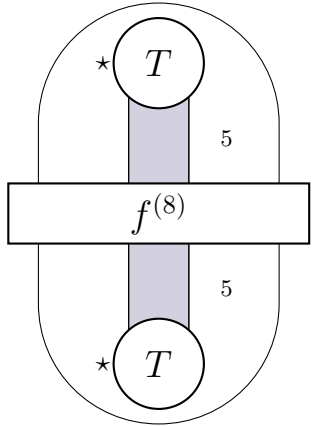
\begin{tikzpicture}[baseline=0,scale=2]
	{\draw[shaded] (0.2,0) -- (0.2,1) -- (-0.2,1) -- (-0.2,0);
	\draw (-0.8,0) -- (-0.8,0.6) arc (180:0:0.8) -- (0.8,0);
	\draw(-0.2,0) -- (-0.2,1);
	\draw (0.2,0) -- (0.2,1);
	\node at (0.45,0.5) {\footnotesize$5$};
	}	{%
	\filldraw[fill=white,thick] (0,1) ellipse (3mm and 3mm);
	\node at (0,1) {\Large $T$};
	\path(0,1) ++(-180:0.37) node {$\star$};
}
	\upsidedown{	{\draw[shaded] (0.2,0) -- (0.2,1) -- (-0.2,1) -- (-0.2,0);
	\draw (-0.8,0) -- (-0.8,0.6) arc (180:0:0.8) -- (0.8,0);
	\draw(-0.2,0) -- (-0.2,1);
	\draw (0.2,0) -- (0.2,1);
	\node at (0.45,0.5) {\footnotesize$5$};
	}	{%
	\filldraw[fill=white,thick] (0,1) ellipse (3mm and 3mm);
	\node at (0,1) {\Large $T$};
	\path(0,1) ++(-180:0.37) node {$\star$};
}}
	{%
	\filldraw[fill=white,thick] (-1,-0.2) rectangle (1,0.2);
	\node at (0,0) {\Large$\JW{8}$};
}
\end{tikzpicture}$$ is equal to $\frac{W_8}{[8]}Z(T^2)$.  Similarly, with $T$ replaced by $Q$.
\end{proof}

\begin{notation} If $E$ $\in \left\{T,Q\right\}$, then we will denote by $(E)$ the diagram which has the generator $E$ standing on a Jones-Wenzl idempotent $f^{(8)}$ as follows:  
$$(E) = 
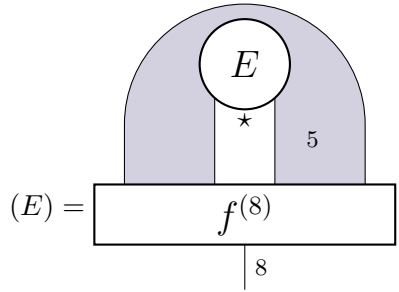
\begin{tikzpicture}[baseline=0,scale=2]
	{
	\fill[shaded] (-0.8,0) -- (-0.8,0.6) arc (180:0:0.8) -- (0.8,0) -- (0.2,0) -- (0.2,1) -- (-0.2,1) -- (-0.2,0);
	\draw (-0.8,0) -- (-0.8,0.6) arc (180:0:0.8) -- (0.8,0);
	\draw (-0.2,0) -- (-0.2,1);
	\draw (0.2,0) -- (0.2,1);
	\node at (0.45,0.5) {\footnotesize$5$};
	{%
	\filldraw[fill=white,thick] (0,1) ellipse (3mm and 3mm);
	\node at (0,1) {\Large $E$};
	\path(0,1) ++(-90:0.37) node {$\star$};
}
}
;
        \draw (0,0)--(0,-0.5);
        \node[anchor=west] at (0,-0.35) {\footnotesize$8$};
	{%
	\filldraw[fill=white,thick] (-1,-0.2) rectangle (1,0.2);
	\node at (0,0) {\Large$\JW{8}$};
}\end{tikzpicture}$$

  If $F$,$G$ $\in \left\{T,Q\right\}$, then we will denote by $(F,G)$ the diagram which has the generators $F$ and $G$ standing on the Jones-Wenzl idempotent $f^{(8)}$ and with two strings connecting them as follows:
$$(F,G)= \begin{tikzpicture}[baseline=0,scale=2]
	{%
	\draw (-0.5,1) -- (-0.5,0);
	\node[anchor=west] at (-0.5,0.5) {\footnotesize$4$};
	\draw (0.5,1) -- (0.5,0);
	\node[anchor=west] at (0.5,0.5) {\footnotesize$4$};
}
 {%
	\node[anchor=south] at (0,1) {\footnotesize$2$};
	\draw (-0.5,1) -- (0.5, 1);
	\foreach \x in {-0.5,0.5} {
		{%
	\filldraw[fill=white,thick] (\x,1) ellipse (3mm and 3mm);
	\node at (-.5,1) {\Large $F$};
	\node at (.5,1) {\Large $G$};
	\path(\x,1) ++(90:0.37) node {$\star$};
}
	}
}
        \draw (0,0)--(0,-0.5);
        \node[anchor=west] at (0,-0.35) {\footnotesize$8$};
	{%
	\filldraw[fill=white,thick] (-1,-0.2) rectangle (1,0.2);
	\node at (0,0) {\Large$\JW{8}$};
}
\end{tikzpicture}$$
\end{notation}
\begin{lem}Let $A=u$ and $B=(T,T)$.  Then $$\ip{A,B}=\omega^2Z((\rho^{1/2}(T))^3)+\frac{[4]}{[8]}(1+\omega)Z(T^3)$$, where   $\omega=1$.  More specifically, $\ip{A,B}=-\frac{1}{3}\sqrt{\frac{779}{14}-\frac{23}{2}\sqrt{21}}.$
\end{lem}
\begin{proof}We need to evaluate the diagram 
$$<A,B> =
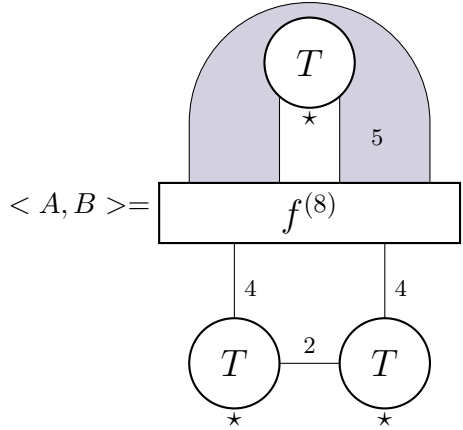
\begin{tikzpicture}[baseline=0,scale=2]
       {
	\fill[shaded] (-0.8,0) -- (-0.8,0.6) arc (180:0:0.8) -- (0.8,0) -- (0.2,0) -- (0.2,1) -- (-0.2,1) -- (-0.2,0);
	\draw (-0.8,0) -- (-0.8,0.6) arc (180:0:0.8) -- (0.8,0);
	\draw (-0.2,0) -- (-0.2,1);
	\draw (0.2,0) -- (0.2,1);
	\node at (0.45,0.5) {\footnotesize$5$};
	{%
	\filldraw[fill=white,thick] (0,1) ellipse (3mm and 3mm);
	\node at (0,1) {\Large $T$};
	\path(0,1) ++(-90:0.37) node {$\star$};
}
}
       \upsidedown{{%
	\draw (-0.5,1) -- (-0.5,0);
	\node[anchor=west] at (-0.5,0.5) {\footnotesize$4$};
	\draw (0.5,1) -- (0.5,0);
	\node[anchor=west] at (0.5,0.5) {\footnotesize$4$};
}{%
	\node[anchor=south] at (0,1) {\footnotesize$2$};
	\draw (-0.5,1) -- (0.5, 1);
	\foreach \x in {-0.5,0.5} {
		{%
	\filldraw[fill=white,thick] (\x,1) ellipse (3mm and 3mm);
	\node at (\x,1) {\Large $T$};
	\path(\x,1) ++(90:0.37) node {$\star$};
}
	}
}}
       {%
	\filldraw[fill=white,thick] (-1,-0.2) rectangle (1,0.2);
	\node at (0,0) {\Large$\JW{8}$};
}

\end{tikzpicture}$$
Consider the expansion of $f^{(8)}$ into TL elements.  There are only three diagrams that make a contribution.

\begin{table}[ht]
\center
\begin{tabular}{ c | c | c }
$\beta$ & $\mathrm{Coeff}_{f^{(8)}}(\beta)$ & value of diagram \\
\hline
\begin{tikzpicture}[scale=.15,baseline=-0.5ex]
    \foreach \x in {1,2,3,4,5,6,7,8} \draw (\x cm , -1.5cm)--(\x cm, 1.5cm);
    
\end{tikzpicture} 
     & $1$
     & $\omega Z(\rho^{1/2}(T)^3)$ \\
\begin{tikzpicture}[scale=.15,baseline=-0.5ex]
        \draw (1,1.5) arc (-180:0:0.5);
        \draw (1,-1.5) -- (3,1.5);
        \draw (2,-1.5) -- (4,1.5);
        \draw (3,-1.5) -- (5,1.5);
        \draw (4,-1.5) arc (180:0:0.5);
        \draw (6,-1.5) -- (6,1.5);
        \draw (7,-1.5) -- (7,1.5);
        \draw (8,-1.5) -- (8,1.5);
\end{tikzpicture}
     & $\frac{[4]}{[8]}$
     & $Z(T^3)$ \\
\begin{tikzpicture}[scale=.15,baseline=-0.5ex]
        \draw (1,-1.5) -- (1,1.5);
        \draw (2,-1.5) -- (2,1.5);
        \draw (3,-1.5) -- (3,1.5);
        \draw (4,-1.5) arc (180:0:0.5);
        \draw (6,-1.5) -- (4,1.5);
        \draw (7,-1.5) -- (5,1.5);
        \draw (8,-1.5) -- (6,1.5);
        \draw (7,1.5) arc (-180:0:0.5);
\end{tikzpicture}
     & $\frac{[4]}{[8]}$ 
     & $\omega^2 Z(T^3)$
\end{tabular}
\vspace{6pt}
\caption{The terms of $f^{(8)}$ that contribute to $\langle A,B \rangle$. }
\label{table:AB}
\end{table}
So,
\begin{align*}
\ip{A,B} & =\omega^2 Z((\rho^{1/2}(T))^3)+\frac{[4]}{[8]}Z(T^3)+\frac{[4]}{[8]}\omega Z(T^3)\\
& = \omega^2 Z((\rho^{1/2}(T))^3)+\frac{[4]}{[8]}(1+\omega)Z(T^3)\\
& = -\frac{1}{3}\sqrt{\frac{779}{14}-\frac{23}{2}\sqrt{21}}, \mbox{ \indent because $\omega=1$}
\end{align*}
\end{proof}

\begin{notation} We will denote by $(E,F,G)$ the inner product $\ip{(E),(F,G)}$.
\end{notation}

\begin{lem}In this lemma we calculate several other inner products of the form $(E,F,G)$, where $E,F,G$ $\in \left\{T,Q\right\}$.  We also list the tables for them.  In this lemma, we assume that $\omega$ is the rotational eigenvalue for $Q$, namely, $e^{2\pi i/3}$.  
\begin{itemize}
\item $(Q,T,T) = -\frac{1}{6}\sqrt{\frac{404}{7}+89\sqrt{\frac{3}{7}}}+\sqrt{\frac{101}{21}+\frac{89}{4\sqrt{21}}}\cdot i$
\item $(T,Q,Q) = \sqrt{\frac{253}{63}+\frac{16}{\sqrt{21}}}$
\item $(T,T,Q) = -\frac{5}{6}\sqrt{\frac{1}{14}(-23+7\sqrt{21})}-\frac{1}{2}\sqrt{\frac{1}{6}(151+41\sqrt{21})}\cdot i$
\item $(T,Q,T) = -\frac{5}{6}\sqrt{\frac{1}{14}(-23+7\sqrt{21})}+\frac{1}{2}\sqrt{\frac{1}{6}(151+41\sqrt{21})}\cdot i$
\item $(Q,Q,Q) = \sqrt{\frac{103}{126}+\frac{23}{6\sqrt{21}}}-\sqrt{\frac{1}{42}(103+23\sqrt{21})}\cdot i$
\item $(Q,T,Q) = \sqrt{\frac{1271}{504}+\frac{29}{8}\sqrt{\frac{3}{7}}}-(\frac{1}{4}+\frac{1}{4\sqrt{21}})i$
\item $(Q,Q,T) = -\frac{1}{6}\sqrt{\frac{1}{14}(179+39\sqrt{21})}-\frac{5}{28}(7+\sqrt{21})i$
\end{itemize}
\end{lem}
\begin{proof}The proofs are analogous to Lemma 5.2.4., paying attention to the rotational eigenvalue $\omega$ of $Q$.
First, we have the table for $(Q,T,T)$:
\begin{table}[ht]
\center
\begin{tabular}{ c | c | c }
$\beta$ & $\mathrm{Coeff}_{f^{(8)}}(\beta)$ & (Q,T,T) \\
\hline
\begin{tikzpicture}[scale=.15,baseline=-0.5ex]
    \foreach \x in {1,2,3,4,5,6,7,8} \draw (\x cm , -1.5cm)--(\x cm, 1.5cm);
    
\end{tikzpicture} 
     & $1$
     & $\omega^2 Z(\rho^{1/2}(Q)(\rho^{1/2}(T))^2)$ \\
\begin{tikzpicture}[scale=.15,baseline=-0.5ex]
        \draw (1,1.5) arc (-180:0:0.5);
        \draw (1,-1.5) -- (3,1.5);
        \draw (2,-1.5) -- (4,1.5);
        \draw (3,-1.5) -- (5,1.5);
        \draw (4,-1.5) arc (180:0:0.5);
        \draw (6,-1.5) -- (6,1.5);
        \draw (7,-1.5) -- (7,1.5);
        \draw (8,-1.5) -- (8,1.5);
\end{tikzpicture}
     & $\frac{[4]}{[8]}$
     & $Z(QT^2)$ \\
\begin{tikzpicture}[scale=.15,baseline=-0.5ex]
        \draw (1,-1.5) -- (1,1.5);
        \draw (2,-1.5) -- (2,1.5);
        \draw (3,-1.5) -- (3,1.5);
        \draw (4,-1.5) arc (180:0:0.5);
        \draw (6,-1.5) -- (4,1.5);
        \draw (7,-1.5) -- (5,1.5);
        \draw (8,-1.5) -- (6,1.5);
        \draw (7,1.5) arc (-180:0:0.5);
\end{tikzpicture}
     & $\frac{[4]}{[8]}$ 
     & $\omega^2 Z(T^2Q)$
\end{tabular}
\vspace{6pt}
\caption{The terms of $f^{(8)}$ that contribute to $(E,F,G)$. }
\end{table}
\newpage
We now have the tables (with the $\beta$'s and their coefficients suppressed) for \\
$(T,Q,Q), (T,T,Q),\mbox { and }  (T,Q,T)$:
\begin{table}[ht]
\center
\begin{tabular}{ c | c | c }
$(T,Q,Q)$ & $(T,T,Q)$ & $(T,Q,T)$ \\
\hline
$\omega^2Z(\rho^{1/2}(T)(\rho^{1/2}(Q))^2)$     & $Z((\rho^{1/2}(T))^2\rho^{1/2}(Q))$
     & $\omega^2 Z(\rho^{1/2}(Q)(\rho^{1/2}(T))^2)$ \\
			$Z(TQ^2)$     
			& $Z(T^2Q)$
     & $Z(T^2Q)$ \\
		$Z(TQ^2)$     
		& $Z(T^2Q)$ 
     & $Z(T^2Q)$
\end{tabular}
\vspace{6pt}
\caption{The terms of $f^{(8)}$ that contribute to $(E,F,G)$. }
\end{table}

Finally, we have the tables for $(Q,Q,Q),(Q,T,Q)$, and $(Q,Q,T)$:
\begin{table}[ht]
\center
\begin{tabular}{ c | c | c }
$(Q,Q,Q)$ & $(Q,T,Q)$ & $(Q,Q,T)$ \\
\hline
$\omega Z((\rho^{1/2}(Q))^3)$     & $\omega^2Z((\rho^{1/2}(Q))^2\rho^{1/2}(T))$
     & $\omega Z(\rho^{1/2}(T)(\rho^{1/2}(Q))^2)$ \\
			$Z(Q^3)$     
			& $Z(Q^2T)$
     & $Z(Q^2T)$ \\
		$\omega^2Z(Q^3)$     
		& $\omega^2Z(Q^2T)$ 
     & $\omega^2Z(Q^2T)$
\end{tabular}
\vspace{6pt}
\caption{The terms of $f^{(8)}$ that contribute to $(E,F,G)$. }
\end{table}

Using these tables, we can easily calculate the inner products.
\end{proof}

\begin{lem} Let $B=(T,T)$.  Then $\ip{B,B} =\frac{9+19\sqrt{21}}{36}$.
\end{lem}
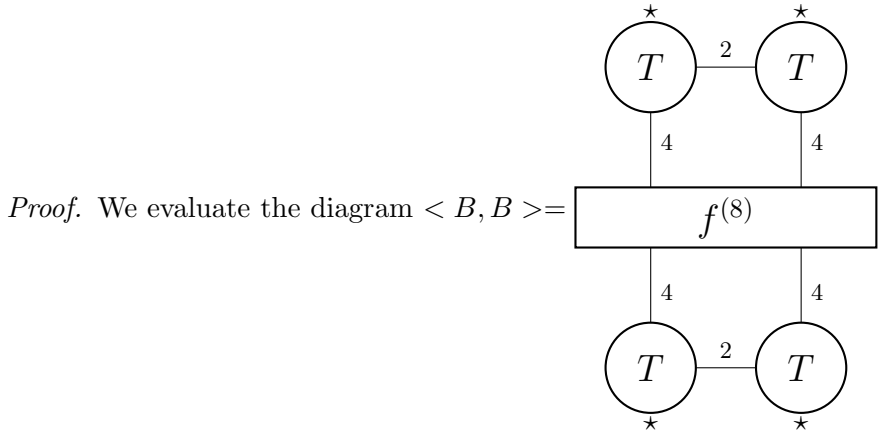
\begin{proof}  We evaluate the diagram 
$<B,B> =
\begin{tikzpicture}[baseline=0,scale=2]
       {%
	\draw (-0.5,1) -- (-0.5,0);
	\node[anchor=west] at (-0.5,0.5) {\footnotesize$4$};
	\draw (0.5,1) -- (0.5,0);
	\node[anchor=west] at (0.5,0.5) {\footnotesize$4$};
}
 {%
	\node[anchor=south] at (0,1) {\footnotesize$2$};
	\draw (-0.5,1) -- (0.5, 1);
	\foreach \x in {-0.5,0.5} {
		{%
	\filldraw[fill=white,thick] (\x,1) ellipse (3mm and 3mm);
	\node at (\x,1) {\Large $T$};
	\path(\x,1) ++(90:0.37) node {$\star$};
}
	}
}
       \upsidedown{{%
	\draw (-0.5,1) -- (-0.5,0);
	\node[anchor=west] at (-0.5,0.5) {\footnotesize$4$};
	\draw (0.5,1) -- (0.5,0);
	\node[anchor=west] at (0.5,0.5) {\footnotesize$4$};
} {%
	\node[anchor=south] at (0,1) {\footnotesize$2$};
	\draw (-0.5,1) -- (0.5, 1);
	\foreach \x in {-0.5,0.5} {
		{%
	\filldraw[fill=white,thick] (\x,1) ellipse (3mm and 3mm);
	\node at (\x,1) {\Large $T$};
	\path(\x,1) ++(90:0.37) node {$\star$};
}
	}
}}
      {%
	\filldraw[fill=white,thick] (-1,-0.2) rectangle (1,0.2);
	\node at (0,0) {\Large$\JW{8}$};
}

\end{tikzpicture}$

Expand $f^{(8)}$ in $\ip{B,B}$ into TL diagrams.  If $\beta$ is a basis element with a cup at position $i\neq4$, then the diagram is 0.  Otherwise, $\beta$ has a cup at position $i=4$.  There are four such diagrams which give non-zero values:
\begin{table}[ht]
\center
\begin{tabular}{ c | c | c }
$\beta$ & $\mathrm{Coeff}_{f^{(8)}}(\beta)$ & value of diagram \\
\hline
\begin{tikzpicture}[scale=.15,baseline=-0.5ex]
    \foreach \x in {1,2,3,4,5,6,7,8} \draw (\x cm , -1.5cm)--(\x cm, 1.5cm);
    
\end{tikzpicture} 
     & $1$
     & $\frac{[4]}{[3]}Z(T^2)$ \\
\begin{tikzpicture}[scale=.15,baseline=-0.5ex]
        \draw (4,1.5) arc (-180:0:0.5);
        \draw (1,-1.5) -- (1,1.5);
        \draw (2,-1.5) -- (2,1.5);
        \draw (3,-1.5) -- (3,1.5);
        \draw (4,-1.5) arc (180:0:0.5);
        \draw (6,-1.5) -- (6,1.5);
        \draw (7,-1.5) -- (7,1.5);
        \draw (8,-1.5) -- (8,1.5);
\end{tikzpicture}
     & $-[4]^2(\frac{1}{[8][7]}+\frac{1}{[7][6]}+\frac{1}{[6][5]}+\frac{1}{[5][4]})$
     & $Z(T^4)$ \\
\begin{tikzpicture}[scale=.15,baseline=-0.5ex]
        \draw (1,-1.5) -- (1,1.5);
        \draw (2,-1.5) -- (2,1.5);
        \draw (3,-1.5) arc (180:0:1.5 and 1.1);
        \draw (4,-1.5) arc (180:0:0.5);
        \draw (7,-1.5) -- (7,1.5);
        \draw (8,-1.5) -- (8,1.5);
        \draw (3,1.5) arc (-180:0:1.5 and 1.1);
        \draw (4,1.5) arc (-180:0:0.5);
\end{tikzpicture}
     & $\frac{[4]^2[3]^2}{[8][7][6][5][4][3]}([8][7]+[8][5]+[8][3]+[6][5]+[6][3]+[4][3])$ 
     & $\frac{[4]}{[3]}Z(T^2)$\\
\begin{tikzpicture}[scale=.15,baseline=-0.5ex]
        \draw (1,-1.5) -- (1,1.5);
        \draw (2,-1.5) arc (180:0:2.5 and 1.5);
        \draw (3,-1.5) arc (180:0:1.5 and 1.0);
        \draw (4,-1.5) arc (180:0:0.5 and 0.5);
        \draw (8,-1.5) -- (8,1.5);
        \draw (2,1.5) arc (-180:0:2.5 and 1.5);
        \draw (3,1.5) arc (-180:0:1.5 and 1.0);
        \draw (4,1.5) arc (-180:0:0.5 and 0.5);
\end{tikzpicture}
     & $\frac{-[4][3][2]}{[8][7][6][5]}([8]+[6]+[4]+[2])$ 
     & $\frac{[4]}{[2]}Z(T^2)$  \\
 \begin{tikzpicture}[scale=.15,baseline=-0.5ex]
        \draw (1,-1.5) arc (180:0:3.5 and 1.5);
        \draw (2,-1.5) arc (180:0:2.5 and 1.1);
        \draw (3,-1.5) arc (180:0:1.5 and 0.7);
        \draw (4,-1.5) arc (180:0:0.5 and 0.3);
        \draw (1,1.5) arc (-180:0:3.5 and 1.5);
        \draw (2,1.5) arc (-180:0:2.5 and 1.1);
        \draw (3,1.5) arc (-180:0:1.5 and 0.7);
        \draw (4,1.5) arc (-180:0:0.5 and 0.3);
\end{tikzpicture}
     & $\frac{[4][3][2]}{[8][7][6][5]}$ 
     & $(Z(T^2))^2$            
\end{tabular}
\vspace{6pt}
\caption{The terms of $f^{(8)}$ that contribute to $\ip{B,B}$. }
\end{table}

Adding up the contributions, we get $\ip{B,B}=\frac{9+19\sqrt{21}}{36}$.
\end{proof}

\begin{lem}Here we calculate some further inner products of the form $\ip{(E,F),(G,H)}$, where $E,F,G,H$ $\in \left\{T,Q \right\}$. We list the relevant tables as well.  
\begin{itemize}
\item $\ip{(Q,Q),(Q,Q)}= \frac{1}{252}(399+109\sqrt{21})$
\item $\ip{(T,T),(Q,Q)}=-\frac{5}{252}(21+11\sqrt{21})$
\item $\ip{(T,Q),(T,Q)}=\frac{5}{9}(6+\sqrt{21})$
\item $\ip{(Q,T),(Q,T)}=\frac{5}{9}(6+\sqrt{21})$
\item $\ip{(T,Q),(Q,T)}=-\frac{7}{6}-\frac{11}{6}\sqrt{\frac{7}{3}}$
\item $\ip{(T,Q),(T,T)}=-\frac{1}{3}\sqrt{-\frac{4}{21}+\frac{61}{\sqrt{21}}}$
\end{itemize}
\end{lem}
\newpage
\begin{proof}
To calculate $\ip{(Q,Q),(Q,Q)}$, we need the following table:
\begin{table}[htb]
\center
\begin{tabular}{ c | c | c }
$\beta$ & $\mathrm{Coeff}_{f^{(8)}}(\beta)$ & value of diagram \\
\hline
\begin{tikzpicture}[scale=.15,baseline=-0.5ex]
    \foreach \x in {1,2,3,4,5,6,7,8} \draw (\x cm , -1.5cm)--(\x cm, 1.5cm);
    
\end{tikzpicture} 
     & $1$
     & $\frac{[4]}{[3]}Z(Q^2)$ \\
\begin{tikzpicture}[scale=.15,baseline=-0.5ex]
        \draw (4,1.5) arc (-180:0:0.5);
        \draw (1,-1.5) -- (1,1.5);
        \draw (2,-1.5) -- (2,1.5);
        \draw (3,-1.5) -- (3,1.5);
        \draw (4,-1.5) arc (180:0:0.5);
        \draw (6,-1.5) -- (6,1.5);
        \draw (7,-1.5) -- (7,1.5);
        \draw (8,-1.5) -- (8,1.5);
\end{tikzpicture}
     & $-[4]^2(\frac{1}{[8][7]}+\frac{1}{[7][6]}+\frac{1}{[6][5]}+\frac{1}{[5][4]})$
     & $Z(Q^4)$ \\
\begin{tikzpicture}[scale=.15,baseline=-0.5ex]
        \draw (1,-1.5) -- (1,1.5);
        \draw (2,-1.5) -- (2,1.5);
        \draw (3,-1.5) arc (180:0:1.5 and 1.1);
        \draw (4,-1.5) arc (180:0:0.5);
        \draw (7,-1.5) -- (7,1.5);
        \draw (8,-1.5) -- (8,1.5);
        \draw (3,1.5) arc (-180:0:1.5 and 1.1);
        \draw (4,1.5) arc (-180:0:0.5);
\end{tikzpicture}
     & $\frac{[4]^2[3]^2}{[8][7][6][5][4][3]}([8][7]+[8][5]+[8][3]+[6][5]+[6][3]+[4][3])$ 
     & $\frac{[4]}{[3]}Z(Q^2)$\\
\begin{tikzpicture}[scale=.15,baseline=-0.5ex]
        \draw (1,-1.5) -- (1,1.5);
        \draw (2,-1.5) arc (180:0:2.5 and 1.5);
        \draw (3,-1.5) arc (180:0:1.5 and 1.0);
        \draw (4,-1.5) arc (180:0:0.5 and 0.5);
        \draw (8,-1.5) -- (8,1.5);
        \draw (2,1.5) arc (-180:0:2.5 and 1.5);
        \draw (3,1.5) arc (-180:0:1.5 and 1.0);
        \draw (4,1.5) arc (-180:0:0.5 and 0.5);
\end{tikzpicture}
     & $\frac{-[4][3][2]}{[8][7][6][5]}([8]+[6]+[4]+[2])$ 
     & $\frac{[4]}{[2]}Z(Q^2)$  \\
 \begin{tikzpicture}[scale=.15,baseline=-0.5ex]
        \draw (1,-1.5) arc (180:0:3.5 and 1.5);
        \draw (2,-1.5) arc (180:0:2.5 and 1.1);
        \draw (3,-1.5) arc (180:0:1.5 and 0.7);
        \draw (4,-1.5) arc (180:0:0.5 and 0.3);
        \draw (1,1.5) arc (-180:0:3.5 and 1.5);
        \draw (2,1.5) arc (-180:0:2.5 and 1.1);
        \draw (3,1.5) arc (-180:0:1.5 and 0.7);
        \draw (4,1.5) arc (-180:0:0.5 and 0.3);
\end{tikzpicture}
     & $\frac{[4][3][2]}{[8][7][6][5]}$ 
     & $(Z(Q^2))^2$            
\end{tabular}
\vspace{6pt}
\caption{The terms of $f^{(8)}$ that contribute to $\ip{(Q,Q),(Q,Q)}$. }
\end{table}

We have the following table for the inner product $\ip{(T,T),(Q,Q)}$:
\begin{table}[htb]
\center
\begin{tabular}{ c | c | c }
$\beta$ & $\mathrm{Coeff}_{f^{(8)}}(\beta)$ & value of diagram \\
\hline
\begin{tikzpicture}[scale=.15,baseline=-0.5ex]
    \foreach \x in {1,2,3,4,5,6,7,8} \draw (\x cm , -1.5cm)--(\x cm, 1.5cm);
    
\end{tikzpicture} 
     & $1$
     & $0$ \\
\begin{tikzpicture}[scale=.15,baseline=-0.5ex]
        \draw (4,1.5) arc (-180:0:0.5);
        \draw (1,-1.5) -- (1,1.5);
        \draw (2,-1.5) -- (2,1.5);
        \draw (3,-1.5) -- (3,1.5);
        \draw (4,-1.5) arc (180:0:0.5);
        \draw (6,-1.5) -- (6,1.5);
        \draw (7,-1.5) -- (7,1.5);
        \draw (8,-1.5) -- (8,1.5);
\end{tikzpicture}
     & $-[4]^2(\frac{1}{[8][7]}+\frac{1}{[7][6]}+\frac{1}{[6][5]}+\frac{1}{[5][4]})$
     & $Z(T^2Q^2)$ \\
\begin{tikzpicture}[scale=.15,baseline=-0.5ex]
        \draw (1,-1.5) -- (1,1.5);
        \draw (2,-1.5) -- (2,1.5);
        \draw (3,-1.5) arc (180:0:1.5 and 1.1);
        \draw (4,-1.5) arc (180:0:0.5);
        \draw (7,-1.5) -- (7,1.5);
        \draw (8,-1.5) -- (8,1.5);
        \draw (3,1.5) arc (-180:0:1.5 and 1.1);
        \draw (4,1.5) arc (-180:0:0.5);
\end{tikzpicture}
     & $\frac{[4]^2[3]^2}{[8][7][6][5][4][3]}([8][7]+[8][5]+[8][3]+[6][5]+[6][3]+[4][3])$ 
     & $\frac{[4]}{[3]}Z(Q^2)$\\
\begin{tikzpicture}[scale=.15,baseline=-0.5ex]
        \draw (1,-1.5) -- (1,1.5);
        \draw (2,-1.5) arc (180:0:2.5 and 1.5);
        \draw (3,-1.5) arc (180:0:1.5 and 1.0);
        \draw (4,-1.5) arc (180:0:0.5 and 0.5);
        \draw (8,-1.5) -- (8,1.5);
        \draw (2,1.5) arc (-180:0:2.5 and 1.5);
        \draw (3,1.5) arc (-180:0:1.5 and 1.0);
        \draw (4,1.5) arc (-180:0:0.5 and 0.5);
\end{tikzpicture}
     & $\frac{-[4][3][2]}{[8][7][6][5]}([8]+[6]+[4]+[2])$ 
     & $\frac{[4]}{[2]}Z(Q^2)$  \\
 \begin{tikzpicture}[scale=.15,baseline=-0.5ex]
        \draw (1,-1.5) arc (180:0:3.5 and 1.5);
        \draw (2,-1.5) arc (180:0:2.5 and 1.1);
        \draw (3,-1.5) arc (180:0:1.5 and 0.7);
        \draw (4,-1.5) arc (180:0:0.5 and 0.3);
        \draw (1,1.5) arc (-180:0:3.5 and 1.5);
        \draw (2,1.5) arc (-180:0:2.5 and 1.1);
        \draw (3,1.5) arc (-180:0:1.5 and 0.7);
        \draw (4,1.5) arc (-180:0:0.5 and 0.3);
\end{tikzpicture}
     & $\frac{[4][3][2]}{[8][7][6][5]}$ 
     & $Z(T^2)Z(Q^2)$            
\end{tabular}
\vspace{6pt}
\caption{The terms of $f^{(8)}$ that contribute to $\ip{(T,T),(Q,Q)}$. }
\end{table}
\newpage
The remaining tables are for the inner products\\
$\ip{(T,Q),(T,Q)},\ip{(Q,T),(Q,T)},\ip{(T,Q),(Q,T)},\ip{(T,Q),(T,T)}$:
\begin{table}[htb]
\center
\begin{tabular}{ c | c | c }
$\beta$ & $\mathrm{Coeff}_{f^{(8)}}(\beta)$ & value of diagram \\
\hline
\begin{tikzpicture}[scale=.15,baseline=-0.5ex]
    \foreach \x in {1,2,3,4,5,6,7,8} \draw (\x cm , -1.5cm)--(\x cm, 1.5cm);
    
\end{tikzpicture} 
     & $1$
     & $\frac{[4]}{[3]}Z(Q^2)$ \\
\begin{tikzpicture}[scale=.15,baseline=-0.5ex]
        \draw (4,1.5) arc (-180:0:0.5);
        \draw (1,-1.5) -- (1,1.5);
        \draw (2,-1.5) -- (2,1.5);
        \draw (3,-1.5) -- (3,1.5);
        \draw (4,-1.5) arc (180:0:0.5);
        \draw (6,-1.5) -- (6,1.5);
        \draw (7,-1.5) -- (7,1.5);
        \draw (8,-1.5) -- (8,1.5);
\end{tikzpicture}
     & $-[4]^2(\frac{1}{[8][7]}+\frac{1}{[7][6]}+\frac{1}{[6][5]}+\frac{1}{[5][4]})$
     & $Z(T^2Q^2)$ \\
\begin{tikzpicture}[scale=.15,baseline=-0.5ex]
        \draw (1,-1.5) -- (1,1.5);
        \draw (2,-1.5) -- (2,1.5);
        \draw (3,-1.5) arc (180:0:1.5 and 1.1);
        \draw (4,-1.5) arc (180:0:0.5);
        \draw (7,-1.5) -- (7,1.5);
        \draw (8,-1.5) -- (8,1.5);
        \draw (3,1.5) arc (-180:0:1.5 and 1.1);
        \draw (4,1.5) arc (-180:0:0.5);
\end{tikzpicture}
     & $\frac{[4]^2[3]^2}{[8][7][6][5][4][3]}([8][7]+[8][5]+[8][3]+[6][5]+[6][3]+[4][3])$ 
     & $0$\\
\begin{tikzpicture}[scale=.15,baseline=-0.5ex]
        \draw (1,-1.5) -- (1,1.5);
        \draw (2,-1.5) arc (180:0:2.5 and 1.5);
        \draw (3,-1.5) arc (180:0:1.5 and 1.0);
        \draw (4,-1.5) arc (180:0:0.5 and 0.5);
        \draw (8,-1.5) -- (8,1.5);
        \draw (2,1.5) arc (-180:0:2.5 and 1.5);
        \draw (3,1.5) arc (-180:0:1.5 and 1.0);
        \draw (4,1.5) arc (-180:0:0.5 and 0.5);
\end{tikzpicture}
     & $\frac{-[4][3][2]}{[8][7][6][5]}([8]+[6]+[4]+[2])$ 
     & $0$  \\
 \begin{tikzpicture}[scale=.15,baseline=-0.5ex]
        \draw (1,-1.5) arc (180:0:3.5 and 1.5);
        \draw (2,-1.5) arc (180:0:2.5 and 1.1);
        \draw (3,-1.5) arc (180:0:1.5 and 0.7);
        \draw (4,-1.5) arc (180:0:0.5 and 0.3);
        \draw (1,1.5) arc (-180:0:3.5 and 1.5);
        \draw (2,1.5) arc (-180:0:2.5 and 1.1);
        \draw (3,1.5) arc (-180:0:1.5 and 0.7);
        \draw (4,1.5) arc (-180:0:0.5 and 0.3);
\end{tikzpicture}
     & $\frac{[4][3][2]}{[8][7][6][5]}$ 
     & $0$            
\end{tabular}
\vspace{6pt}
\caption{The terms of $f^{(8)}$ that contribute to $\ip{(T,Q),(T,Q)}$. }
\end{table}

\begin{table}[!htb]
\center
\begin{tabular}{ c | c | c }
$\beta$ & $\mathrm{Coeff}_{f^{(8)}}(\beta)$ & value of diagram \\
\hline
\begin{tikzpicture}[scale=.15,baseline=-0.5ex]
    \foreach \x in {1,2,3,4,5,6,7,8} \draw (\x cm , -1.5cm)--(\x cm, 1.5cm);
    
\end{tikzpicture} 
     & $1$
     & $\frac{[4]}{[3]}Z(T^2)$ \\
\begin{tikzpicture}[scale=.15,baseline=-0.5ex]
        \draw (4,1.5) arc (-180:0:0.5);
        \draw (1,-1.5) -- (1,1.5);
        \draw (2,-1.5) -- (2,1.5);
        \draw (3,-1.5) -- (3,1.5);
        \draw (4,-1.5) arc (180:0:0.5);
        \draw (6,-1.5) -- (6,1.5);
        \draw (7,-1.5) -- (7,1.5);
        \draw (8,-1.5) -- (8,1.5);
\end{tikzpicture}
     & $-[4]^2(\frac{1}{[8][7]}+\frac{1}{[7][6]}+\frac{1}{[6][5]}+\frac{1}{[5][4]})$
     & $Z(T^2Q^2)$ \\
\begin{tikzpicture}[scale=.15,baseline=-0.5ex]
        \draw (1,-1.5) -- (1,1.5);
        \draw (2,-1.5) -- (2,1.5);
        \draw (3,-1.5) arc (180:0:1.5 and 1.1);
        \draw (4,-1.5) arc (180:0:0.5);
        \draw (7,-1.5) -- (7,1.5);
        \draw (8,-1.5) -- (8,1.5);
        \draw (3,1.5) arc (-180:0:1.5 and 1.1);
        \draw (4,1.5) arc (-180:0:0.5);
\end{tikzpicture}
     & $\frac{[4]^2[3]^2}{[8][7][6][5][4][3]}([8][7]+[8][5]+[8][3]+[6][5]+[6][3]+[4][3])$ 
     & $0$\\
\begin{tikzpicture}[scale=.15,baseline=-0.5ex]
        \draw (1,-1.5) -- (1,1.5);
        \draw (2,-1.5) arc (180:0:2.5 and 1.5);
        \draw (3,-1.5) arc (180:0:1.5 and 1.0);
        \draw (4,-1.5) arc (180:0:0.5 and 0.5);
        \draw (8,-1.5) -- (8,1.5);
        \draw (2,1.5) arc (-180:0:2.5 and 1.5);
        \draw (3,1.5) arc (-180:0:1.5 and 1.0);
        \draw (4,1.5) arc (-180:0:0.5 and 0.5);
\end{tikzpicture}
     & $\frac{-[4][3][2]}{[8][7][6][5]}([8]+[6]+[4]+[2])$ 
     & $0$  \\
 \begin{tikzpicture}[scale=.15,baseline=-0.5ex]
        \draw (1,-1.5) arc (180:0:3.5 and 1.5);
        \draw (2,-1.5) arc (180:0:2.5 and 1.1);
        \draw (3,-1.5) arc (180:0:1.5 and 0.7);
        \draw (4,-1.5) arc (180:0:0.5 and 0.3);
        \draw (1,1.5) arc (-180:0:3.5 and 1.5);
        \draw (2,1.5) arc (-180:0:2.5 and 1.1);
        \draw (3,1.5) arc (-180:0:1.5 and 0.7);
        \draw (4,1.5) arc (-180:0:0.5 and 0.3);
\end{tikzpicture}
     & $\frac{[4][3][2]}{[8][7][6][5]}$ 
     & $0$            
\end{tabular}
\vspace{6pt}
\caption{The terms of $f^{(8)}$ that contribute to $\ip{(Q,T),(Q,T)}$. }
\end{table}

\begin{table}[!htb]
\center
\begin{tabular}{ c | c | c }
$\beta$ & $\mathrm{Coeff}_{f^{(8)}}(\beta)$ & value of diagram \\
\hline
\begin{tikzpicture}[scale=.15,baseline=-0.5ex]
    \foreach \x in {1,2,3,4,5,6,7,8} \draw (\x cm , -1.5cm)--(\x cm, 1.5cm);
    
\end{tikzpicture} 
     & $1$
     & $0$ \\
\begin{tikzpicture}[scale=.15,baseline=-0.5ex]
        \draw (4,1.5) arc (-180:0:0.5);
        \draw (1,-1.5) -- (1,1.5);
        \draw (2,-1.5) -- (2,1.5);
        \draw (3,-1.5) -- (3,1.5);
        \draw (4,-1.5) arc (180:0:0.5);
        \draw (6,-1.5) -- (6,1.5);
        \draw (7,-1.5) -- (7,1.5);
        \draw (8,-1.5) -- (8,1.5);
\end{tikzpicture}
     & $-[4]^2(\frac{1}{[8][7]}+\frac{1}{[7][6]}+\frac{1}{[6][5]}+\frac{1}{[5][4]})$
     & $Z(T^2Q^2)$ \\
\begin{tikzpicture}[scale=.15,baseline=-0.5ex]
        \draw (1,-1.5) -- (1,1.5);
        \draw (2,-1.5) -- (2,1.5);
        \draw (3,-1.5) arc (180:0:1.5 and 1.1);
        \draw (4,-1.5) arc (180:0:0.5);
        \draw (7,-1.5) -- (7,1.5);
        \draw (8,-1.5) -- (8,1.5);
        \draw (3,1.5) arc (-180:0:1.5 and 1.1);
        \draw (4,1.5) arc (-180:0:0.5);
\end{tikzpicture}
     & $\frac{[4]^2[3]^2}{[8][7][6][5][4][3]}([8][7]+[8][5]+[8][3]+[6][5]+[6][3]+[4][3])$ 
     & $0$\\
\begin{tikzpicture}[scale=.15,baseline=-0.5ex]
        \draw (1,-1.5) -- (1,1.5);
        \draw (2,-1.5) arc (180:0:2.5 and 1.5);
        \draw (3,-1.5) arc (180:0:1.5 and 1.0);
        \draw (4,-1.5) arc (180:0:0.5 and 0.5);
        \draw (8,-1.5) -- (8,1.5);
        \draw (2,1.5) arc (-180:0:2.5 and 1.5);
        \draw (3,1.5) arc (-180:0:1.5 and 1.0);
        \draw (4,1.5) arc (-180:0:0.5 and 0.5);
\end{tikzpicture}
     & $\frac{-[4][3][2]}{[8][7][6][5]}([8]+[6]+[4]+[2])$ 
     & $0$  \\
 \begin{tikzpicture}[scale=.15,baseline=-0.5ex]
        \draw (1,-1.5) arc (180:0:3.5 and 1.5);
        \draw (2,-1.5) arc (180:0:2.5 and 1.1);
        \draw (3,-1.5) arc (180:0:1.5 and 0.7);
        \draw (4,-1.5) arc (180:0:0.5 and 0.3);
        \draw (1,1.5) arc (-180:0:3.5 and 1.5);
        \draw (2,1.5) arc (-180:0:2.5 and 1.1);
        \draw (3,1.5) arc (-180:0:1.5 and 0.7);
        \draw (4,1.5) arc (-180:0:0.5 and 0.3);
\end{tikzpicture}
     & $\frac{[4][3][2]}{[8][7][6][5]}$ 
     & $0$            
\end{tabular}
\vspace{6pt}
\caption{The terms of $f^{(8)}$ that contribute to $\ip{(T,Q),(Q,T)}$. }
\end{table}
\newpage
\begin{table}[!htb]
\center
\begin{tabular}{ c | c | c }
$\beta$ & $\mathrm{Coeff}_{f^{(8)}}(\beta)$ & value of diagram \\
\hline
\begin{tikzpicture}[scale=.15,baseline=-0.5ex]
    \foreach \x in {1,2,3,4,5,6,7,8} \draw (\x cm , -1.5cm)--(\x cm, 1.5cm);
    
\end{tikzpicture} 
     & $1$
     & $0$ \\
\begin{tikzpicture}[scale=.15,baseline=-0.5ex]
        \draw (4,1.5) arc (-180:0:0.5);
        \draw (1,-1.5) -- (1,1.5);
        \draw (2,-1.5) -- (2,1.5);
        \draw (3,-1.5) -- (3,1.5);
        \draw (4,-1.5) arc (180:0:0.5);
        \draw (6,-1.5) -- (6,1.5);
        \draw (7,-1.5) -- (7,1.5);
        \draw (8,-1.5) -- (8,1.5);
\end{tikzpicture}
     & \scriptsize$-[4]^2(\frac{1}{[8][7]}+\frac{1}{[7][6]}+\frac{1}{[6][5]}+\frac{1}{[5][4]})$
     & \scriptsize$\frac{Z(T^2Q)}{[4]}Z(T^3)+\frac{Z(TQ^2)}{[4]}Z(QT^2)$\\
\begin{tikzpicture}[scale=.15,baseline=-0.5ex]
        \draw (1,-1.5) -- (1,1.5);
        \draw (2,-1.5) -- (2,1.5);
        \draw (3,-1.5) arc (180:0:1.5 and 1.1);
        \draw (4,-1.5) arc (180:0:0.5);
        \draw (7,-1.5) -- (7,1.5);
        \draw (8,-1.5) -- (8,1.5);
        \draw (3,1.5) arc (-180:0:1.5 and 1.1);
        \draw (4,1.5) arc (-180:0:0.5);
\end{tikzpicture}
     & \scriptsize$\frac{[4]^2[3]^2}{[8][7][6][5][4][3]}([8][7]+[8][5]+[8][3]+[6][5]+[6][3]+[4][3])$ 
     & $0$\\
\begin{tikzpicture}[scale=.15,baseline=-0.5ex]
        \draw (1,-1.5) -- (1,1.5);
        \draw (2,-1.5) arc (180:0:2.5 and 1.5);
        \draw (3,-1.5) arc (180:0:1.5 and 1.0);
        \draw (4,-1.5) arc (180:0:0.5 and 0.5);
        \draw (8,-1.5) -- (8,1.5);
        \draw (2,1.5) arc (-180:0:2.5 and 1.5);
        \draw (3,1.5) arc (-180:0:1.5 and 1.0);
        \draw (4,1.5) arc (-180:0:0.5 and 0.5);
\end{tikzpicture}
     & \scriptsize$\frac{-[4][3][2]}{[8][7][6][5]}([8]+[6]+[4]+[2])$ 
     & $0$  \\
 \begin{tikzpicture}[scale=.15,baseline=-0.5ex]
        \draw (1,-1.5) arc (180:0:3.5 and 1.5);
        \draw (2,-1.5) arc (180:0:2.5 and 1.1);
        \draw (3,-1.5) arc (180:0:1.5 and 0.7);
        \draw (4,-1.5) arc (180:0:0.5 and 0.3);
        \draw (1,1.5) arc (-180:0:3.5 and 1.5);
        \draw (2,1.5) arc (-180:0:2.5 and 1.1);
        \draw (3,1.5) arc (-180:0:1.5 and 0.7);
        \draw (4,1.5) arc (-180:0:0.5 and 0.3);
\end{tikzpicture}
     & \scriptsize$\frac{[4][3][2]}{[8][7][6][5]}$ 
     & $0$            
\end{tabular}
\vspace{6pt}
\caption{The terms of $f^{(8)}$ that contribute to $\ip{(T,Q),(T,T)}$. }
\end{table}
\end{proof}

We are now ready to prove the unshaded 4-box relations.
Since there is one new element $M$ at level 4, we expect to see that $(T,Q)$ can be written as a linear combination of elements in $ATL(T)$, $ATL(Q)$, and the new element $M$.  More precisely, we expect to see $(T,Q)$ as a linear combination of $u$, $v$, and $M$.
Let $$M=(T,Q)-\proj{u}{(T,Q)}-\proj{v}{(T,Q)}$$
We want to show that $(Q,T)\in span\left\{u,v,M\right\}$.  Similarly, for $(T,T)$ and $(Q,Q)$.  We use Bessel's inequality to prove these relations.
\begin{theorem}We have the following 4-box relations in the precision sufficient for our needs.

\begin{itemize}
\item (R4(T,Q)) \hspace{2cm}$(T,Q) \in  span\left\{u,v,M\right\}$
\item (R4(Q,T)) \hspace{2cm}$(Q,T) \in  span\left\{u,v,M\right\}$
\item (R4(T,T)) \hspace{2cm}$(T,T) \in  span\left\{u,v,M\right\}$
\item (R4(Q,Q)) \hspace{2cm}$(Q,Q) \in  span\left\{u,v,M\right\}$
\end{itemize}
In terms of diagrams, this means:
Each of 
$$
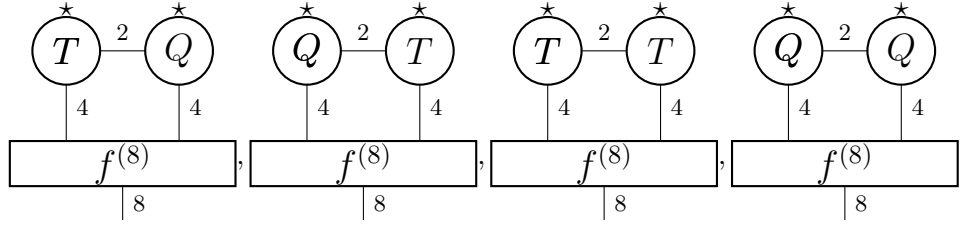
\begin{tikzpicture}[baseline=0,scale=1.5]
	{%
	\draw (-0.5,1) -- (-0.5,0);
	\node[anchor=west] at (-0.5,0.5) {\footnotesize$4$};
	\draw (0.5,1) -- (0.5,0);
	\node[anchor=west] at (0.5,0.5) {\footnotesize$4$};
}
 {%
	\node[anchor=south] at (0,1) {\footnotesize$2$};
	\draw (-0.5,1) -- (0.5, 1);
	\foreach \x in {-0.5,0.5} {
		{%
	\filldraw[fill=white,thick] (\x,1) ellipse (3mm and 3mm);
	\node at (-.5,1) {\Large $T$};
	\node at (.5,1) {\Large $Q$};
	\path(\x,1) ++(90:0.37) node {$\star$};
}
	}
}
        \draw (0,0)--(0,-0.5);
        \node[anchor=west] at (0,-0.35) {\footnotesize$8$};
	{%
	\filldraw[fill=white,thick] (-1,-0.2) rectangle (1,0.2);
	\node at (0,0) {\Large$\JW{8}$};
}
\end{tikzpicture}, \begin{tikzpicture}[baseline=0,scale=1.5]
	{%
	\draw (-0.5,1) -- (-0.5,0);
	\node[anchor=west] at (-0.5,0.5) {\footnotesize$4$};
	\draw (0.5,1) -- (0.5,0);
	\node[anchor=west] at (0.5,0.5) {\footnotesize$4$};
}
 {%
	\node[anchor=south] at (0,1) {\footnotesize$2$};
	\draw (-0.5,1) -- (0.5, 1);
	\foreach \x in {-0.5,0.5} {
		{%
	\filldraw[fill=white,thick] (\x,1) ellipse (3mm and 3mm);
	\node at (-.5,1) {\Large $Q$};
	\node at (.5,1) {\Large $T$};
	\path(\x,1) ++(90:0.37) node {$\star$};
}
	}
}
        \draw (0,0)--(0,-0.5);
        \node[anchor=west] at (0,-0.35) {\footnotesize$8$};
	{%
	\filldraw[fill=white,thick] (-1,-0.2) rectangle (1,0.2);
	\node at (0,0) {\Large$\JW{8}$};
}
\end{tikzpicture}, \begin{tikzpicture}[baseline=0,scale=1.5]
	{%
	\draw (-0.5,1) -- (-0.5,0);
	\node[anchor=west] at (-0.5,0.5) {\footnotesize$4$};
	\draw (0.5,1) -- (0.5,0);
	\node[anchor=west] at (0.5,0.5) {\footnotesize$4$};
}
 {%
	\node[anchor=south] at (0,1) {\footnotesize$2$};
	\draw (-0.5,1) -- (0.5, 1);
	\foreach \x in {-0.5,0.5} {
		{%
	\filldraw[fill=white,thick] (\x,1) ellipse (3mm and 3mm);
	\node at (-.5,1) {\Large $T$};
	\node at (.5,1) {\Large $T$};
	\path(\x,1) ++(90:0.37) node {$\star$};
}
	}
}
        \draw (0,0)--(0,-0.5);
        \node[anchor=west] at (0,-0.35) {\footnotesize$8$};
	{%
	\filldraw[fill=white,thick] (-1,-0.2) rectangle (1,0.2);
	\node at (0,0) {\Large$\JW{8}$};
}
\end{tikzpicture}, \begin{tikzpicture}[baseline=0,scale=1.5]
	{%
	\draw (-0.5,1) -- (-0.5,0);
	\node[anchor=west] at (-0.5,0.5) {\footnotesize$4$};
	\draw (0.5,1) -- (0.5,0);
	\node[anchor=west] at (0.5,0.5) {\footnotesize$4$};
}
 {%
	\node[anchor=south] at (0,1) {\footnotesize$2$};
	\draw (-0.5,1) -- (0.5, 1);
	\foreach \x in {-0.5,0.5} {
		{%
	\filldraw[fill=white,thick] (\x,1) ellipse (3mm and 3mm);
	\node at (-.5,1) {\Large $Q$};
	\node at (.5,1) {\Large $Q$};
	\path(\x,1) ++(90:0.37) node {$\star$};
}
	}
}
        \draw (0,0)--(0,-0.5);
        \node[anchor=west] at (0,-0.35) {\footnotesize$8$};
	{%
	\filldraw[fill=white,thick] (-1,-0.2) rectangle (1,0.2);
	\node at (0,0) {\Large$\JW{8}$};
}
\end{tikzpicture}$$ lies in the span of 

$$\Big( 
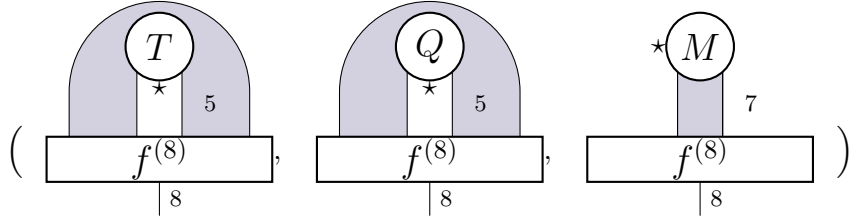
\hspace{3mm}\begin{tikzpicture}[baseline=0,scale=1.5]
	{
	\fill[shaded] (-0.8,0) -- (-0.8,0.6) arc (180:0:0.8) -- (0.8,0) -- (0.2,0) -- (0.2,1) -- (-0.2,1) -- (-0.2,0);
	\draw (-0.8,0) -- (-0.8,0.6) arc (180:0:0.8) -- (0.8,0);
	\draw (-0.2,0) -- (-0.2,1);
	\draw (0.2,0) -- (0.2,1);
	\node at (0.45,0.5) {\footnotesize$5$};
	{%
	\filldraw[fill=white,thick] (0,1) ellipse (3mm and 3mm);
	\node at (0,1) {\Large $T$};
	\path(0,1) ++(-90:0.37) node {$\star$};
}
}
;
        \draw (0,0)--(0,-0.5);
        \node[anchor=west] at (0,-0.35) {\footnotesize$8$};
	{%
	\filldraw[fill=white,thick] (-1,-0.2) rectangle (1,0.2);
	\node at (0,0) {\Large$\JW{8}$};
}\end{tikzpicture},
\quad \begin{tikzpicture}[baseline=0,scale=1.5]
	{
	\fill[shaded] (-0.8,0) -- (-0.8,0.6) arc (180:0:0.8) -- (0.8,0) -- (0.2,0) -- (0.2,1) -- (-0.2,1) -- (-0.2,0);
	\draw (-0.8,0) -- (-0.8,0.6) arc (180:0:0.8) -- (0.8,0);
	\draw (-0.2,0) -- (-0.2,1);
	\draw (0.2,0) -- (0.2,1);
	\node at (0.45,0.5) {\footnotesize$5$};
	{%
	\filldraw[fill=white,thick] (0,1) ellipse (3mm and 3mm);
	\node at (0,1) {\Large $Q$};
	\path(0,1) ++(-90:0.37) node {$\star$};
}
}
;
        \draw (0,0)--(0,-0.5);
        \node[anchor=west] at (0,-0.35) {\footnotesize$8$};
	{%
	\filldraw[fill=white,thick] (-1,-0.2) rectangle (1,0.2);
	\node at (0,0) {\Large$\JW{8}$};
}
\end{tikzpicture} ,\quad \begin{tikzpicture}[baseline=0,scale=1.5]
	{\draw[shaded] (0.2,0) -- (0.2,1) -- (-0.2,1) -- (-0.2,0);
	
	\draw(-0.2,0) -- (-0.2,1);
	\draw (0.2,0) -- (0.2,1);
	\node at (0.45,0.5) {\footnotesize$7$};
	{%
	\filldraw[fill=white,thick] (0,1) ellipse (3mm and 3mm);
	\node at (0,1) {\Large $M$};
	\path(0,1) ++(-180:0.37) node {$\star$};
}
}
;
        \draw (0,0)--(0,-0.5);
        \node[anchor=west] at (0,-0.35) {\footnotesize$8$};
	{%
	\filldraw[fill=white,thick] (-1,-0.2) rectangle (1,0.2);
	\node at (0,0) {\Large$\JW{8}$};
}\end{tikzpicture} \hspace{3mm}\Big)$$

\end{theorem}
\begin{proof}[Proof of (R4(T,Q))] This is true by definition of $M$.
\end{proof}
We use Lemmas 5.2.2, 5.2.4, 5.2.6, 5.2.7, and 5.2.8 for the values of the relevant inner products to prove the remaining relations.
\begin{proof}[Proof of (R4(Q,T))] Let $w =(T,Q)$ and $w_0=(Q,T)$.  We need to show that $$\ip{w_0,w_0} -\frac{\norm{\ip{w_0,u}}^2}{\ip{u,u}}-\frac{\norm{\ip{w_0,v}}^2}{\ip{v,v}}=\frac{\norm{\ip{w_0,w}-\frac{\ip{w_0,u}\overline{\ip{w,u}}}{\ip{u,u}}-\frac{\ip{w_0,v}\overline{\ip{w,v}}}{\ip{v,v}}}^2}{\Norm{M}^2}$$
On the left hand side, we get $\frac{9+19\sqrt{21}}{45}$.  On the right hand side, we get $\frac{\frac{2}{675}(1277+57\sqrt{21}}{\Norm{M}^2}$.
Since $\Norm{M}^2=\frac{9+19\sqrt{21}}{45}$, we have equality.  Thus, $(Q,T) \in  span\left\{u,v,M\right\}$
\end{proof}
\begin{proof}[Proof of (R4(T,T))] We need to show that \footnotesize$$\ip{(T,T),(T,T)} =\frac{\norm{\ip{(T,T),u}}^2}{\ip{u,u}}+\frac{\norm{\ip{(T,T),v}}^2}{\ip{v,v}}+\frac{\norm{\ip{(T,T),(T,Q)}-\frac{\ip{(T,T),u}\overline{\ip{(T,Q),u}}}{\ip{u,u}}-\frac{\ip{(T,T),v}\overline{\ip{(T,Q),v}}}{\ip{v,v}}}^2}{\Norm{M}^2}$$
\normalsize
That is, we need to show \footnotesize$$\ip{(T,T),(T,T)} -\frac{\norm{\ip{(T,T),u}}^2}{\ip{u,u}}-\frac{\norm{\ip{(T,T),v}}^2}{\ip{v,v}}=\frac{\norm{\ip{(T,T),(T,Q)}-\frac{\ip{(T,T),u}\overline{\ip{(T,Q),u}}}{\ip{u,u}}-\frac{\ip{(T,T),v}\overline{\ip{(T,Q),v}}}{\ip{v,v}}}^2}{\Norm{M}^2}$$
\normalsize
On the left hand side, we get $\frac{363-67\sqrt{21}}{180}$.  On the right hand side, we get $\frac{\frac{-3911+1049\sqrt{21}}{1350}}{\Norm{M}^2}$.
Since $\Norm{M}^2=\frac{9+19\sqrt{21}}{45}$, we have equality.  Thus, $(T,T) \in  span\left\{u,v,M\right\}$
\end{proof}
\begin{proof}[Proof of (R4(Q,Q))] We need to show that \footnotesize$$\ip{(Q,Q),(Q,Q)} =\frac{\norm{\ip{(Q,Q),u}}^2}{\ip{u,u}}+\frac{\norm{\ip{(Q,Q),v}}^2}{\ip{v,v}}+\frac{\norm{\ip{(Q,Q),(T,Q)}-\frac{\ip{(Q,Q),u}\overline{\ip{(T,Q),u}}}{\ip{u,u}}-\frac{\ip{(Q,Q),v}\overline{\ip{(T,Q),v}}}{\ip{v,v}}}^2}{\Norm{M}^2}$$
\normalsize
That is, we need to show \footnotesize$$\ip{(Q,Q),(Q,Q)} -\frac{\norm{\ip{(Q,Q),u}}^2}{\ip{u,u}}-\frac{\norm{\ip{(Q,Q),v}}^2}{\ip{v,v}}=\frac{\norm{\ip{(Q,Q),(T,Q)}-\frac{\ip{(Q,Q),u}\overline{\ip{(T,Q),u}}}{\ip{u,u}}-\frac{\ip{(Q,Q),v}\overline{\ip{(T,Q),v}}}{\ip{v,v}}}^2}{\Norm{M}^2}$$
\normalsize
On the left hand side, we get $\frac{87+17\sqrt{21}}{180}$.  On the right hand side, we get $\frac{\frac{1261+301\sqrt{21}}{1350}}{\Norm{M}^2}$.
Since $\Norm{M}^2=\frac{9+19\sqrt{21}}{45}$, we have equality.  Thus, $(Q,Q) \in  span\left\{u,v,M\right\}$
\end{proof}
\subsection{Shaded 4-box Relations}
In this subsection we prove shaded versions of the 4-box relations we've already seen.  The overall procedure and proofs will be completely analogous.  First, we give the requisite definitions and notations.
Let $$u_0=\begin{tikzpicture}[baseline=0,scale=2]
	{\draw[shaded] (0.2,0) -- (0.2,1) -- (-0.2,1) -- (-0.2,0);
	\draw (-0.8,0) -- (-0.8,0.6) arc (180:0:0.8) -- (0.8,0);
	\draw(-0.2,0) -- (-0.2,1);
	\draw (0.2,0) -- (0.2,1);
	\node at (0.45,0.5) {\footnotesize$5$};
	{%
	\filldraw[fill=white,thick] (0,1) ellipse (3mm and 3mm);
	\node at (0,1) {\Large $T$};
	\path(0,1) ++(-180:0.37) node {$\star$};
}
}
;
        \draw (0,0)--(0,-0.5);
        \node[anchor=west] at (0,-0.35) {\footnotesize$8$};
	{%
	\filldraw[fill=white,thick] (-1,-0.2) rectangle (1,0.2);
	\node at (0,0) {\Large$\JW{8}$};
}\end{tikzpicture}, \quad v_0=\begin{tikzpicture}[baseline=0,scale=2]
	{\draw[shaded] (0.2,0) -- (0.2,1) -- (-0.2,1) -- (-0.2,0);
	\draw (-0.8,0) -- (-0.8,0.6) arc (180:0:0.8) -- (0.8,0);
	\draw(-0.2,0) -- (-0.2,1);
	\draw (0.2,0) -- (0.2,1);
	\node at (0.45,0.5) {\footnotesize$5$};
	{%
	\filldraw[fill=white,thick] (0,1) ellipse (3mm and 3mm);
	\node at (0,1) {\Large $Q$};
	\path(0,1) ++(-180:0.37) node {$\star$};
}
}
;
        \draw (0,0)--(0,-0.5);
        \node[anchor=west] at (0,-0.35) {\footnotesize$8$};
	{%
	\filldraw[fill=white,thick] (-1,-0.2) rectangle (1,0.2);
	\node at (0,0) {\Large$\JW{8}$};
}\end{tikzpicture}.
$$
\begin{notation} If $E$ $\in \left\{T,Q\right\}$, then we will denote by $(E)$ the diagram which has the generator $E$ standing on a Jones-Wenzl idempotent $f^{(8)}$ as follows:  
$$(E) = 
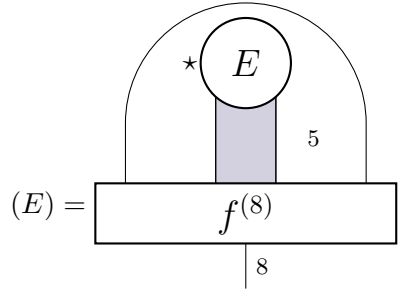
\begin{tikzpicture}[baseline=0,scale=2]
	{\draw[shaded] (0.2,0) -- (0.2,1) -- (-0.2,1) -- (-0.2,0);
	\draw (-0.8,0) -- (-0.8,0.6) arc (180:0:0.8) -- (0.8,0);
	\draw(-0.2,0) -- (-0.2,1);
	\draw (0.2,0) -- (0.2,1);
	\node at (0.45,0.5) {\footnotesize$5$};
	{%
	\filldraw[fill=white,thick] (0,1) ellipse (3mm and 3mm);
	\node at (0,1) {\Large $E$};
	\path(0,1) ++(-180:0.37) node {$\star$};
}
}
;
        \draw (0,0)--(0,-0.5);
        \node[anchor=west] at (0,-0.35) {\footnotesize$8$};
	{%
	\filldraw[fill=white,thick] (-1,-0.2) rectangle (1,0.2);
	\node at (0,0) {\Large$\JW{8}$};
}\end{tikzpicture}$$

If $F$,$G$ $\in \left\{T,Q\right\}$, then we will denote by $(F,G)$ the diagram which has the generators $F$ and $G$ standing on the Jones-Wenzl idempotent $f^{(8)}$ and with two strings connecting them as follows:
$$(F,G)= \begin{tikzpicture}[baseline=0,scale=2]
	{%
	\fill[shaded] (-0.5,0) rectangle (0.5,1);
	\draw (-0.5,1) -- (-0.5,0);
	\node[anchor=west] at (-0.5,0.5) {\footnotesize$3$};
	\draw (-0.6,1) -- (-0.6,0);
	\draw (0.5,1) -- (0.5,0);
	\node[anchor=west] at (0.3,0.5) {\footnotesize$3$};
	\draw (0.6,1) -- (0.6,0);
}
 {%
	
	\draw (-0.5,1) -- (0.5, 1);
	\draw (-0.5,1.1) -- (0.5,1.1);
	\foreach \x in {-0.5,0.5} {
		{%
	\filldraw[fill=white,thick] (\x,1) ellipse (3mm and 3mm);
	\node at (-.5,1) {\Large $F$};
	\node at (.5,1) {\Large $G$};
	\path(-.55,1) ++(-90:0.37) node {$\star$};
	\path(0.55,1) ++(-90:0.37) node {$\star$};
}
	}
}
        \draw (0,0)--(0,-0.5);
        \node[anchor=west] at (0,-0.35) {\footnotesize$8$};
	{%
	\filldraw[fill=white,thick] (-1,-0.2) rectangle (1,0.2);
	\node at (0,0) {\Large$\JW{8}$};
}
\end{tikzpicture}$$
\end{notation}

\begin{lem}Let $A=u_0$ and $B=(T,T)$.  Then $$\ip{A,B}=\big(1+\frac{[4]}{[8]}(\omega+\omega^2)\big)Z(T^3)$$, where   $\omega=1$.  More specifically, $\ip{A,B}=-\frac{1}{3}\sqrt{\frac{779}{14}-\frac{23}{2}\sqrt{21}}.$
\end{lem}
\begin{proof}We need to evaluate the diagram 
$$<A,B> =
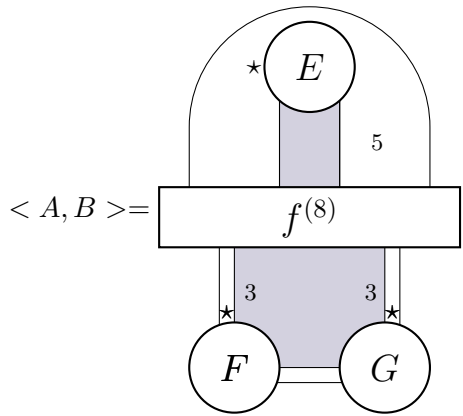
\begin{tikzpicture}[baseline=0,scale=2]
	{\draw[shaded] (0.2,0) -- (0.2,1) -- (-0.2,1) -- (-0.2,0);
	\draw (-0.8,0) -- (-0.8,0.6) arc (180:0:0.8) -- (0.8,0);
	\draw(-0.2,0) -- (-0.2,1);
	\draw (0.2,0) -- (0.2,1);
	\node at (0.45,0.5) {\footnotesize$5$};
		}	{%
	\filldraw[fill=white,thick] (0,1) ellipse (3mm and 3mm);
	\node at (0,1) {\Large $E$};
	\path(0,1) ++(-180:0.37) node {$\star$};
}
	\upsidedown{{%
	\fill[shaded] (-0.5,0) rectangle (0.5,1);
	\draw (-0.5,1) -- (-0.5,0);
	\node[anchor=west] at (-0.5,0.5) {\footnotesize$3$};
	\draw (-0.6,1) -- (-0.6,0);
	\draw (0.5,1) -- (0.5,0);
	\node[anchor=west] at (0.3,0.5) {\footnotesize$3$};
	\draw (0.6,1) -- (0.6,0);
} {%
	
	\draw (-0.5,1) -- (0.5, 1);
	\draw (-0.5,1.1) -- (0.5,1.1);
	\foreach \x in {-0.5,0.5} {
		{%
	\filldraw[fill=white,thick] (\x,1) ellipse (3mm and 3mm);
	\node at (-.5,1) {\Large $F$};
	\node at (.5,1) {\Large $G$};
	\path(-.55,1) ++(-90:0.37) node {$\star$};
	\path(0.55,1) ++(-90:0.37) node {$\star$};
}}
}
        
	{%
	\filldraw[fill=white,thick] (-1,-0.2) rectangle (1,0.2);
	\node at (0,0) {\Large$\JW{8}$};
}}
\end{tikzpicture}$$ 
Consider the expansion of $f^{(8)}$ into TL elements.  There are only three diagrams that make a contribution.

\begin{table}[ht]
\center
\begin{tabular}{ c | c | c }
$\beta$ & $\mathrm{Coeff}_{f^{(8)}}(\beta)$ & value of diagram \\
\hline
\begin{tikzpicture}[scale=.15,baseline=-0.5ex]
    \foreach \x in {1,2,3,4,5,6,7,8} \draw (\x cm , -1.5cm)--(\x cm, 1.5cm);
    
\end{tikzpicture} 
     & $1$
     & $Z(T^3)$ \\
\begin{tikzpicture}[scale=.15,baseline=-0.5ex]
        \draw (1,1.5) arc (-180:0:0.5);
        \draw (1,-1.5) -- (3,1.5);
        \draw (2,-1.5) -- (4,1.5);
        \draw (3,-1.5) -- (5,1.5);
        \draw (4,-1.5) arc (180:0:0.5);
        \draw (6,-1.5) -- (6,1.5);
        \draw (7,-1.5) -- (7,1.5);
        \draw (8,-1.5) -- (8,1.5);
\end{tikzpicture}
     & $\frac{[4]}{[8]}$
     & $\omega^2Z((\rho^{1/2}(T))^3)$ \\
\begin{tikzpicture}[scale=.15,baseline=-0.5ex]
        \draw (1,-1.5) -- (1,1.5);
        \draw (2,-1.5) -- (2,1.5);
        \draw (3,-1.5) -- (3,1.5);
        \draw (4,-1.5) arc (180:0:0.5);
        \draw (6,-1.5) -- (4,1.5);
        \draw (7,-1.5) -- (5,1.5);
        \draw (8,-1.5) -- (6,1.5);
        \draw (7,1.5) arc (-180:0:0.5);
\end{tikzpicture}
     & $\frac{[4]}{[8]}$ 
     & $\omega Z((\rho^{1/2}(T))^3)$
\end{tabular}
\vspace{6pt}
\caption{The terms of $f^{(8)}$ that contribute to $\langle A,B \rangle$. }
\label{table:AB}
\end{table}

So,
\begin{align*}
\ip{A,B} & = Z(T^3)+\frac{[4]}{[8]}\omega^2Z((\rho^{1/2}(T))^3) +\frac{[4]}{[8]}\omega Z((\rho^{1/2}(T))^3)\\
& =(1+\frac{[4]}{[8]}(\omega+\omega^2))Z(T^3), \mbox{ because $Z(T^3)= Z((\rho^{1/2}(T))^3)$ (also for $Q$ as well)}
\end{align*}
\end{proof}

\begin{notation} We will denote by $(E,F,G)$ the inner product $\ip{(E),(F,G)}$.
\end{notation}

\begin{lem}In this lemma we calculate several other inner products of the form $(E,F,G)$, where $E,F,G$ $\in \left\{T,Q\right\}$.  We also list the tables for them.  In this lemma, we assume that $\omega$ is the rotational eigenvalue for $Q$, namely, $e^{2\pi i/3}$.  
\begin{itemize}
\item $(Q,T,T) = \sqrt{\frac{404}{63}+\frac{89}{3\sqrt{21}}}$
\item $(T,Q,Q) = \sqrt{\frac{253}{63}+\frac{16}{\sqrt{21}}}$
\item $(T,T,Q) = \sqrt{\frac{404}{63}+\frac{89}{3\sqrt{21}}}+\sqrt{-\frac{53}{42}+\frac{9}{2}\sqrt{\frac{3}{7}}}\cdot i$
\item $(T,Q,T) = \sqrt{\frac{404}{63}+\frac{89}{3\sqrt{21}}}-\sqrt{-\frac{53}{42}+\frac{9}{2}\sqrt{\frac{3}{7}}}\cdot i$
\item $(Q,Q,Q) = -\frac{1}{3}\sqrt{\frac{206}{7}+46\sqrt{\frac{3}{7}}}$
\item $(Q,T,Q) = \sqrt{\frac{1271}{504}+\frac{29}{8}\sqrt{\frac{3}{7}}}-(\frac{1}{4}+\frac{1}{4\sqrt{21}})i$
\item $(Q,Q,T) = \sqrt{\frac{1271}{504}+\frac{29}{8}\sqrt{\frac{3}{7}}}+(\frac{1}{4}+\frac{1}{4\sqrt{21}})i$
\end{itemize}
\end{lem}
\begin{proof}The proofs are analogous to Lemma 5.2.11., paying attention to the rotational eigenvalue $\omega$ of $Q$.
First, we have the table for $(Q,T,T)$:
\begin{table}[ht]
\center
\begin{tabular}{ c | c | c }
$\beta$ & $\mathrm{Coeff}_{f^{(8)}}(\beta)$ & (Q,T,T) \\
\hline
\begin{tikzpicture}[scale=.15,baseline=-0.5ex]
    \foreach \x in {1,2,3,4,5,6,7,8} \draw (\x cm , -1.5cm)--(\x cm, 1.5cm);
    
\end{tikzpicture} 
     & $1$
     & $Z(T^2Q)$ \\
\begin{tikzpicture}[scale=.15,baseline=-0.5ex]
        \draw (1,1.5) arc (-180:0:0.5);
        \draw (1,-1.5) -- (3,1.5);
        \draw (2,-1.5) -- (4,1.5);
        \draw (3,-1.5) -- (5,1.5);
        \draw (4,-1.5) arc (180:0:0.5);
        \draw (6,-1.5) -- (6,1.5);
        \draw (7,-1.5) -- (7,1.5);
        \draw (8,-1.5) -- (8,1.5);
\end{tikzpicture}
     & $\frac{[4]}{[8]}$
     & $Z((\rho^{1/2}(T))^2\rho^{1/2}(Q))$ \\
\begin{tikzpicture}[scale=.15,baseline=-0.5ex]
        \draw (1,-1.5) -- (1,1.5);
        \draw (2,-1.5) -- (2,1.5);
        \draw (3,-1.5) -- (3,1.5);
        \draw (4,-1.5) arc (180:0:0.5);
        \draw (6,-1.5) -- (4,1.5);
        \draw (7,-1.5) -- (5,1.5);
        \draw (8,-1.5) -- (6,1.5);
        \draw (7,1.5) arc (-180:0:0.5);
\end{tikzpicture}
     & $\frac{[4]}{[8]}$ 
     & $\omega^2Z((\rho^{1/2}(T))^2\rho^{1/2}(Q))$
\end{tabular}
\vspace{6pt}
\caption{The terms of $f^{(8)}$ that contribute to $(E,F,G)$. }
\end{table}

We now have the tables (with the $\beta$'s and their coefficients suppressed) for \\
$(T,Q,Q), (T,T,Q),\mbox { and }  (T,Q,T)$:
\begin{table}[ht]
\center
\begin{tabular}{ c | c | c }
$(T,Q,Q)$ & $(T,T,Q)$ & $(T,Q,T)$ \\
\hline
$Z(TQ^2)$     & $Z(T^2Q)$
     & $Z(T^2Q)$ \\
			$\omega^2Z(\rho^{1/2}(T)(\rho^{1/2}(Q))^2)$     
			& $\omega^2Z((\rho^{1/2}(T))^2\rho^{1/2}(Q))$
     & $Z((\rho^{1/2}(T))^2\rho^{1/2}(Q))$ \\
		$\omega^2Z(\rho^{1/2}(T)(\rho^{1/2}(Q))^2)$     
		& $\omega^2Z((\rho^{1/2}(T))^2\rho^{1/2}(Q))$ 
     & $Z((\rho^{1/2}(T))^2\rho^{1/2}(Q))$
\end{tabular}
\vspace{6pt}
\caption{The terms of $f^{(8)}$ that contribute to $(E,F,G)$. }
\end{table}
\newpage
Finally, we have the tables for $(Q,Q,Q),(Q,T,Q)$, and $(Q,Q,T)$:
\begin{table}[ht]
\center
\begin{tabular}{ c | c | c }
$(Q,Q,Q)$ & $(Q,T,Q)$ & $(Q,Q,T)$ \\
\hline
$Z(Q^3)$     & $Z(TQ^2)$
     & $Z(TQ^2)$ \\
			$\omega^2 Z((\rho^{1/2}(Q))^3)$     
			& $\omega^2Z((\rho^{1/2}(Q))^2\rho^{1/2}(T))$
     & $Z((\rho^{1/2}(Q))^2\rho^{1/2}(T))$ \\
		$\omega Z((\rho^{1/2}(Q))^3)$     
		& $\omega Z((\rho^{1/2}(Q))^2\rho^{1/2}(T))$ 
     & $\omega^2Z((\rho^{1/2}(Q))^2\rho^{1/2}(T))$
\end{tabular}
\vspace{6pt}
\caption{The terms of $f^{(8)}$ that contribute to $(E,F,G)$. }
\end{table}

Using these tables, we can easily calculate the inner products.
\end{proof}

\begin{lem} Let $B=(T,T)$.  Then $\ip{B,B} =\frac{9+19\sqrt{21}}{36}$.
\end{lem}
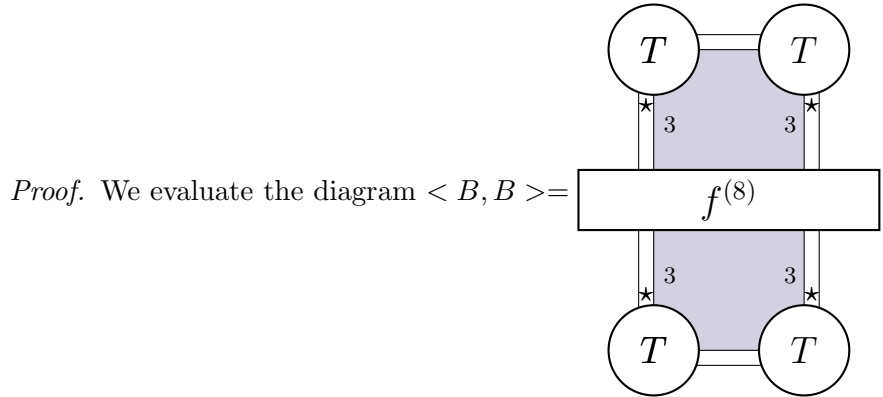
\begin{proof}  We evaluate the diagram 
$<B,B> =
\begin{tikzpicture}[baseline=0,scale=2]
	{%
	\fill[shaded] (-0.5,0) rectangle (0.5,1);
	\draw (-0.5,1) -- (-0.5,0);
	\node[anchor=west] at (-0.5,0.5) {\footnotesize$3$};
	\draw (-0.6,1) -- (-0.6,0);
	\draw (0.5,1) -- (0.5,0);
	\node[anchor=west] at (0.3,0.5) {\footnotesize$3$};
	\draw (0.6,1) -- (0.6,0);
}
 {%
	
	\draw (-0.5,1) -- (0.5, 1);
	\draw (-0.5,1.1) -- (0.5,1.1);
	\foreach \x in {-0.5,0.5} {
		{%
	\filldraw[fill=white,thick] (\x,1) ellipse (3mm and 3mm);
	\node at (-.5,1) {\Large $T$};
	\node at (.5,1) {\Large $T$};
	\path(-.55,1) ++(-90:0.37) node {$\star$};
	\path(0.55,1) ++(-90:0.37) node {$\star$};
}
	}
}       \upsidedown{{%
	\fill[shaded] (-0.5,0) rectangle (0.5,1);
	\draw (-0.5,1) -- (-0.5,0);
	\node[anchor=west] at (-0.5,0.5) {\footnotesize$3$};
	\draw (-0.6,1) -- (-0.6,0);
	\draw (0.5,1) -- (0.5,0);
	\node[anchor=west] at (0.3,0.5) {\footnotesize$3$};
	\draw (0.6,1) -- (0.6,0);
}
 {%
	
	\draw (-0.5,1) -- (0.5, 1);
	\draw (-0.5,1.1) -- (0.5,1.1);
	\foreach \x in {-0.5,0.5} {
		{%
	\filldraw[fill=white,thick] (\x,1) ellipse (3mm and 3mm);
	\node at (-.5,1) {\Large $T$};
	\node at (.5,1) {\Large $T$};
	\path(-.55,1) ++(-90:0.37) node {$\star$};
	\path(0.55,1) ++(-90:0.37) node {$\star$};
}
	}
}    }   
     {%
	\filldraw[fill=white,thick] (-1,-0.2) rectangle (1,0.2);
	\node at (0,0) {\Large$\JW{8}$};
}

\end{tikzpicture}$

Expand $f^{(8)}$ in $\ip{B,B}$ into TL diagrams.  If $\beta$ is a basis element with a cup at position $i\neq4$, then the diagram is 0.  Otherwise, $\beta$ has a cup at position $i=4$.  There are four such diagrams which give non-zero values:
\begin{table}[ht]
\center
\begin{tabular}{ c | c | c }
$\beta$ & $\mathrm{Coeff}_{f^{(8)}}(\beta)$ & value of diagram \\
\hline
\begin{tikzpicture}[scale=.15,baseline=-0.5ex]
    \foreach \x in {1,2,3,4,5,6,7,8} \draw (\x cm , -1.5cm)--(\x cm, 1.5cm);
    
\end{tikzpicture} 
     & $1$
     & $\frac{[4]}{[3]}Z(T^2)$ \\
\begin{tikzpicture}[scale=.15,baseline=-0.5ex]
        \draw (4,1.5) arc (-180:0:0.5);
        \draw (1,-1.5) -- (1,1.5);
        \draw (2,-1.5) -- (2,1.5);
        \draw (3,-1.5) -- (3,1.5);
        \draw (4,-1.5) arc (180:0:0.5);
        \draw (6,-1.5) -- (6,1.5);
        \draw (7,-1.5) -- (7,1.5);
        \draw (8,-1.5) -- (8,1.5);
\end{tikzpicture}
     & $-[4]^2(\frac{1}{[8][7]}+\frac{1}{[7][6]}+\frac{1}{[6][5]}+\frac{1}{[5][4]})$
     & $Z((\rho^{1/2}(T))^4)$ \\
\begin{tikzpicture}[scale=.15,baseline=-0.5ex]
        \draw (1,-1.5) -- (1,1.5);
        \draw (2,-1.5) -- (2,1.5);
        \draw (3,-1.5) arc (180:0:1.5 and 1.1);
        \draw (4,-1.5) arc (180:0:0.5);
        \draw (7,-1.5) -- (7,1.5);
        \draw (8,-1.5) -- (8,1.5);
        \draw (3,1.5) arc (-180:0:1.5 and 1.1);
        \draw (4,1.5) arc (-180:0:0.5);
\end{tikzpicture}
     & $\frac{[4]^2[3]^2}{[8][7][6][5][4][3]}([8][7]+[8][5]+[8][3]+[6][5]+[6][3]+[4][3])$ 
     & $\frac{[4]}{[3]}Z(T^2)$\\
\begin{tikzpicture}[scale=.15,baseline=-0.5ex]
        \draw (1,-1.5) -- (1,1.5);
        \draw (2,-1.5) arc (180:0:2.5 and 1.5);
        \draw (3,-1.5) arc (180:0:1.5 and 1.0);
        \draw (4,-1.5) arc (180:0:0.5 and 0.5);
        \draw (8,-1.5) -- (8,1.5);
        \draw (2,1.5) arc (-180:0:2.5 and 1.5);
        \draw (3,1.5) arc (-180:0:1.5 and 1.0);
        \draw (4,1.5) arc (-180:0:0.5 and 0.5);
\end{tikzpicture}
     & $\frac{-[4][3][2]}{[8][7][6][5]}([8]+[6]+[4]+[2])$ 
     & $\frac{[4]}{[2]}Z(T^2)$  \\
 \begin{tikzpicture}[scale=.15,baseline=-0.5ex]
        \draw (1,-1.5) arc (180:0:3.5 and 1.5);
        \draw (2,-1.5) arc (180:0:2.5 and 1.1);
        \draw (3,-1.5) arc (180:0:1.5 and 0.7);
        \draw (4,-1.5) arc (180:0:0.5 and 0.3);
        \draw (1,1.5) arc (-180:0:3.5 and 1.5);
        \draw (2,1.5) arc (-180:0:2.5 and 1.1);
        \draw (3,1.5) arc (-180:0:1.5 and 0.7);
        \draw (4,1.5) arc (-180:0:0.5 and 0.3);
\end{tikzpicture}
     & $\frac{[4][3][2]}{[8][7][6][5]}$ 
     & $(Z(T^2))^2$            
\end{tabular}
\vspace{6pt}
\caption{The terms of $f^{(8)}$ that contribute to $\ip{B,B}$. }
\end{table}

Adding up the contributions, we get $\ip{B,B}=\frac{9+19\sqrt{21}}{36}$.
\end{proof}

\begin{lem}Here we calculate some further inner products of the form $\ip{(E,F),(G,H)}$, where $E,F,G,H$ $\in \left\{T,Q \right\}$. We list the relevant tables as well.  
\begin{itemize}
\item $\ip{(Q,Q),(Q,Q)}= \frac{1}{252}(399+109\sqrt{21})$
\item $\ip{(T,T),(Q,Q)}=-\frac{5}{252}(21+11\sqrt{21})$
\item $\ip{(T,Q),(T,Q)}=\frac{5}{9}(6+\sqrt{21})$
\item $\ip{(Q,T),(Q,T)}=\frac{5}{9}(6+\sqrt{21})$
\item $\ip{(T,Q),(Q,T)}=\frac{1}{36}(21+11\sqrt{21})-\frac{1}{6}\sqrt{\frac{7}{2}(71+11\sqrt{21}}\cdot i$
\item $\ip{(T,Q),(T,T)}=\sqrt{-\frac{1}{189}+\frac{61}{36\sqrt{21}}}+\sqrt{-\frac{1}{63}+\frac{61}{12\sqrt{21}}}\cdot i$
\end{itemize}
\end{lem}
\newpage
\begin{proof}
To calculate $\ip{(Q,Q),(Q,Q)}$, we need the following table:
\begin{table}[htb]
\center
\begin{tabular}{ c | c | c }
$\beta$ & $\mathrm{Coeff}_{f^{(8)}}(\beta)$ & value of diagram \\
\hline
\begin{tikzpicture}[scale=.15,baseline=-0.5ex]
    \foreach \x in {1,2,3,4,5,6,7,8} \draw (\x cm , -1.5cm)--(\x cm, 1.5cm);
    
\end{tikzpicture} 
     & $1$
     & $\frac{[4]}{[3]}Z(Q^2)$ \\
\begin{tikzpicture}[scale=.15,baseline=-0.5ex]
        \draw (4,1.5) arc (-180:0:0.5);
        \draw (1,-1.5) -- (1,1.5);
        \draw (2,-1.5) -- (2,1.5);
        \draw (3,-1.5) -- (3,1.5);
        \draw (4,-1.5) arc (180:0:0.5);
        \draw (6,-1.5) -- (6,1.5);
        \draw (7,-1.5) -- (7,1.5);
        \draw (8,-1.5) -- (8,1.5);
\end{tikzpicture}
     & $-[4]^2(\frac{1}{[8][7]}+\frac{1}{[7][6]}+\frac{1}{[6][5]}+\frac{1}{[5][4]})$
     & $\omega Z((\rho^{1/2}(Q))^4)$ \\
\begin{tikzpicture}[scale=.15,baseline=-0.5ex]
        \draw (1,-1.5) -- (1,1.5);
        \draw (2,-1.5) -- (2,1.5);
        \draw (3,-1.5) arc (180:0:1.5 and 1.1);
        \draw (4,-1.5) arc (180:0:0.5);
        \draw (7,-1.5) -- (7,1.5);
        \draw (8,-1.5) -- (8,1.5);
        \draw (3,1.5) arc (-180:0:1.5 and 1.1);
        \draw (4,1.5) arc (-180:0:0.5);
\end{tikzpicture}
     & $\frac{[4]^2[3]^2}{[8][7][6][5][4][3]}([8][7]+[8][5]+[8][3]+[6][5]+[6][3]+[4][3])$ 
     & $\frac{[4]}{[3]}Z(Q^2)$\\
\begin{tikzpicture}[scale=.15,baseline=-0.5ex]
        \draw (1,-1.5) -- (1,1.5);
        \draw (2,-1.5) arc (180:0:2.5 and 1.5);
        \draw (3,-1.5) arc (180:0:1.5 and 1.0);
        \draw (4,-1.5) arc (180:0:0.5 and 0.5);
        \draw (8,-1.5) -- (8,1.5);
        \draw (2,1.5) arc (-180:0:2.5 and 1.5);
        \draw (3,1.5) arc (-180:0:1.5 and 1.0);
        \draw (4,1.5) arc (-180:0:0.5 and 0.5);
\end{tikzpicture}
     & $\frac{-[4][3][2]}{[8][7][6][5]}([8]+[6]+[4]+[2])$ 
     & $\frac{[4]}{[2]}Z(Q^2)$  \\
 \begin{tikzpicture}[scale=.15,baseline=-0.5ex]
        \draw (1,-1.5) arc (180:0:3.5 and 1.5);
        \draw (2,-1.5) arc (180:0:2.5 and 1.1);
        \draw (3,-1.5) arc (180:0:1.5 and 0.7);
        \draw (4,-1.5) arc (180:0:0.5 and 0.3);
        \draw (1,1.5) arc (-180:0:3.5 and 1.5);
        \draw (2,1.5) arc (-180:0:2.5 and 1.1);
        \draw (3,1.5) arc (-180:0:1.5 and 0.7);
        \draw (4,1.5) arc (-180:0:0.5 and 0.3);
\end{tikzpicture}
     & $\frac{[4][3][2]}{[8][7][6][5]}$ 
     & $(Z(Q^2))^2$            
\end{tabular}
\vspace{6pt}
\caption{The terms of $f^{(8)}$ that contribute to $\ip{(Q,Q),(Q,Q)}$. }
\end{table}

We have the following table for the inner product $\ip{(T,T),(Q,Q)}$:
\begin{table}[htb]
\center
\begin{tabular}{ c | c | c }
$\beta$ & $\mathrm{Coeff}_{f^{(8)}}(\beta)$ & value of diagram \\
\hline
\begin{tikzpicture}[scale=.15,baseline=-0.5ex]
    \foreach \x in {1,2,3,4,5,6,7,8} \draw (\x cm , -1.5cm)--(\x cm, 1.5cm);
    
\end{tikzpicture} 
     & $1$
     & $0$ \\
\begin{tikzpicture}[scale=.15,baseline=-0.5ex]
        \draw (4,1.5) arc (-180:0:0.5);
        \draw (1,-1.5) -- (1,1.5);
        \draw (2,-1.5) -- (2,1.5);
        \draw (3,-1.5) -- (3,1.5);
        \draw (4,-1.5) arc (180:0:0.5);
        \draw (6,-1.5) -- (6,1.5);
        \draw (7,-1.5) -- (7,1.5);
        \draw (8,-1.5) -- (8,1.5);
\end{tikzpicture}
     & \scriptsize$-[4]^2(\frac{1}{[8][7]}+\frac{1}{[7][6]}+\frac{1}{[6][5]}+\frac{1}{[5][4]})$
     & \scriptsize$\omega^2 Z((\rho^{1/2}(T))^2(\rho^{1/2}(Q))^2)$ \\
\begin{tikzpicture}[scale=.15,baseline=-0.5ex]
        \draw (1,-1.5) -- (1,1.5);
        \draw (2,-1.5) -- (2,1.5);
        \draw (3,-1.5) arc (180:0:1.5 and 1.1);
        \draw (4,-1.5) arc (180:0:0.5);
        \draw (7,-1.5) -- (7,1.5);
        \draw (8,-1.5) -- (8,1.5);
        \draw (3,1.5) arc (-180:0:1.5 and 1.1);
        \draw (4,1.5) arc (-180:0:0.5);
\end{tikzpicture}
     & \scriptsize$\frac{[4]^2[3]^2}{[8][7][6][5][4][3]}([8][7]+[8][5]+[8][3]+[6][5]+[6][3]+[4][3])$ 
     & $\frac{[4]}{[3]}Z(Q^2)$\\
\begin{tikzpicture}[scale=.15,baseline=-0.5ex]
        \draw (1,-1.5) -- (1,1.5);
        \draw (2,-1.5) arc (180:0:2.5 and 1.5);
        \draw (3,-1.5) arc (180:0:1.5 and 1.0);
        \draw (4,-1.5) arc (180:0:0.5 and 0.5);
        \draw (8,-1.5) -- (8,1.5);
        \draw (2,1.5) arc (-180:0:2.5 and 1.5);
        \draw (3,1.5) arc (-180:0:1.5 and 1.0);
        \draw (4,1.5) arc (-180:0:0.5 and 0.5);
\end{tikzpicture}
     & \scriptsize$\frac{-[4][3][2]}{[8][7][6][5]}([8]+[6]+[4]+[2])$ 
     & $\frac{[4]}{[2]}Z(Q^2)$  \\
 \begin{tikzpicture}[scale=.15,baseline=-0.5ex]
        \draw (1,-1.5) arc (180:0:3.5 and 1.5);
        \draw (2,-1.5) arc (180:0:2.5 and 1.1);
        \draw (3,-1.5) arc (180:0:1.5 and 0.7);
        \draw (4,-1.5) arc (180:0:0.5 and 0.3);
        \draw (1,1.5) arc (-180:0:3.5 and 1.5);
        \draw (2,1.5) arc (-180:0:2.5 and 1.1);
        \draw (3,1.5) arc (-180:0:1.5 and 0.7);
        \draw (4,1.5) arc (-180:0:0.5 and 0.3);
\end{tikzpicture}
     & \scriptsize$\frac{[4][3][2]}{[8][7][6][5]}$ 
     & $(Z(T^2))^2$            
\end{tabular}
\vspace{6pt}
\caption{The terms of $f^{(8)}$ that contribute to $\ip{(T,T),(Q,Q)}$. }
\end{table}
\newpage
The remaining tables are for the inner products\\
$\ip{(T,Q),(T,Q)},\ip{(Q,T),(Q,T)},\ip{(T,Q),(Q,T)},\ip{(T,Q),(T,T)}$:
\begin{table}[htb]
\center
\begin{tabular}{ c | c | c }
$\beta$ & $\mathrm{Coeff}_{f^{(8)}}(\beta)$ & value of diagram \\
\hline
\begin{tikzpicture}[scale=.15,baseline=-0.5ex]
    \foreach \x in {1,2,3,4,5,6,7,8} \draw (\x cm , -1.5cm)--(\x cm, 1.5cm);
    
\end{tikzpicture} 
     & $1$
     & $\frac{[4]}{[3]}Z(Q^2)$ \\
\begin{tikzpicture}[scale=.15,baseline=-0.5ex]
        \draw (4,1.5) arc (-180:0:0.5);
        \draw (1,-1.5) -- (1,1.5);
        \draw (2,-1.5) -- (2,1.5);
        \draw (3,-1.5) -- (3,1.5);
        \draw (4,-1.5) arc (180:0:0.5);
        \draw (6,-1.5) -- (6,1.5);
        \draw (7,-1.5) -- (7,1.5);
        \draw (8,-1.5) -- (8,1.5);
\end{tikzpicture}
     & \scriptsize$-[4]^2(\frac{1}{[8][7]}+\frac{1}{[7][6]}+\frac{1}{[6][5]}+\frac{1}{[5][4]})$
     & \scriptsize$\omega^2Z((\rho^{1/2}(T))^2(\rho^{1/2}(Q))^2)$ \\
\begin{tikzpicture}[scale=.15,baseline=-0.5ex]
        \draw (1,-1.5) -- (1,1.5);
        \draw (2,-1.5) -- (2,1.5);
        \draw (3,-1.5) arc (180:0:1.5 and 1.1);
        \draw (4,-1.5) arc (180:0:0.5);
        \draw (7,-1.5) -- (7,1.5);
        \draw (8,-1.5) -- (8,1.5);
        \draw (3,1.5) arc (-180:0:1.5 and 1.1);
        \draw (4,1.5) arc (-180:0:0.5);
\end{tikzpicture}
     & \scriptsize$\frac{[4]^2[3]^2}{[8][7][6][5][4][3]}([8][7]+[8][5]+[8][3]+[6][5]+[6][3]+[4][3])$ 
     & $0$\\
\begin{tikzpicture}[scale=.15,baseline=-0.5ex]
        \draw (1,-1.5) -- (1,1.5);
        \draw (2,-1.5) arc (180:0:2.5 and 1.5);
        \draw (3,-1.5) arc (180:0:1.5 and 1.0);
        \draw (4,-1.5) arc (180:0:0.5 and 0.5);
        \draw (8,-1.5) -- (8,1.5);
        \draw (2,1.5) arc (-180:0:2.5 and 1.5);
        \draw (3,1.5) arc (-180:0:1.5 and 1.0);
        \draw (4,1.5) arc (-180:0:0.5 and 0.5);
\end{tikzpicture}
     & \scriptsize$\frac{-[4][3][2]}{[8][7][6][5]}([8]+[6]+[4]+[2])$ 
     & $0$  \\
 \begin{tikzpicture}[scale=.15,baseline=-0.5ex]
        \draw (1,-1.5) arc (180:0:3.5 and 1.5);
        \draw (2,-1.5) arc (180:0:2.5 and 1.1);
        \draw (3,-1.5) arc (180:0:1.5 and 0.7);
        \draw (4,-1.5) arc (180:0:0.5 and 0.3);
        \draw (1,1.5) arc (-180:0:3.5 and 1.5);
        \draw (2,1.5) arc (-180:0:2.5 and 1.1);
        \draw (3,1.5) arc (-180:0:1.5 and 0.7);
        \draw (4,1.5) arc (-180:0:0.5 and 0.3);
\end{tikzpicture}
     & \scriptsize$\frac{[4][3][2]}{[8][7][6][5]}$ 
     & $0$            
\end{tabular}
\vspace{6pt}
\caption{The terms of $f^{(8)}$ that contribute to $\ip{(T,Q),(T,Q)}$. }
\end{table}

\begin{table}[!htb]
\center
\begin{tabular}{ c | c | c }
$\beta$ & $\mathrm{Coeff}_{f^{(8)}}(\beta)$ & value of diagram \\
\hline
\begin{tikzpicture}[scale=.15,baseline=-0.5ex]
    \foreach \x in {1,2,3,4,5,6,7,8} \draw (\x cm , -1.5cm)--(\x cm, 1.5cm);
    
\end{tikzpicture} 
     & $1$
     & $\frac{[4]}{[3]}Z(T^2)$ \\
\begin{tikzpicture}[scale=.15,baseline=-0.5ex]
        \draw (4,1.5) arc (-180:0:0.5);
        \draw (1,-1.5) -- (1,1.5);
        \draw (2,-1.5) -- (2,1.5);
        \draw (3,-1.5) -- (3,1.5);
        \draw (4,-1.5) arc (180:0:0.5);
        \draw (6,-1.5) -- (6,1.5);
        \draw (7,-1.5) -- (7,1.5);
        \draw (8,-1.5) -- (8,1.5);
\end{tikzpicture}
     & \scriptsize$-[4]^2(\frac{1}{[8][7]}+\frac{1}{[7][6]}+\frac{1}{[6][5]}+\frac{1}{[5][4]})$
     & \scriptsize$\omega^2Z((\rho^{1/2}(T))^2(\rho^{1/2}(Q))^2)$ \\
\begin{tikzpicture}[scale=.15,baseline=-0.5ex]
        \draw (1,-1.5) -- (1,1.5);
        \draw (2,-1.5) -- (2,1.5);
        \draw (3,-1.5) arc (180:0:1.5 and 1.1);
        \draw (4,-1.5) arc (180:0:0.5);
        \draw (7,-1.5) -- (7,1.5);
        \draw (8,-1.5) -- (8,1.5);
        \draw (3,1.5) arc (-180:0:1.5 and 1.1);
        \draw (4,1.5) arc (-180:0:0.5);
\end{tikzpicture}
     & \scriptsize$\frac{[4]^2[3]^2}{[8][7][6][5][4][3]}([8][7]+[8][5]+[8][3]+[6][5]+[6][3]+[4][3])$ 
     & $0$\\
\begin{tikzpicture}[scale=.15,baseline=-0.5ex]
        \draw (1,-1.5) -- (1,1.5);
        \draw (2,-1.5) arc (180:0:2.5 and 1.5);
        \draw (3,-1.5) arc (180:0:1.5 and 1.0);
        \draw (4,-1.5) arc (180:0:0.5 and 0.5);
        \draw (8,-1.5) -- (8,1.5);
        \draw (2,1.5) arc (-180:0:2.5 and 1.5);
        \draw (3,1.5) arc (-180:0:1.5 and 1.0);
        \draw (4,1.5) arc (-180:0:0.5 and 0.5);
\end{tikzpicture}
     & \scriptsize$\frac{-[4][3][2]}{[8][7][6][5]}([8]+[6]+[4]+[2])$ 
     & $0$  \\
 \begin{tikzpicture}[scale=.15,baseline=-0.5ex]
        \draw (1,-1.5) arc (180:0:3.5 and 1.5);
        \draw (2,-1.5) arc (180:0:2.5 and 1.1);
        \draw (3,-1.5) arc (180:0:1.5 and 0.7);
        \draw (4,-1.5) arc (180:0:0.5 and 0.3);
        \draw (1,1.5) arc (-180:0:3.5 and 1.5);
        \draw (2,1.5) arc (-180:0:2.5 and 1.1);
        \draw (3,1.5) arc (-180:0:1.5 and 0.7);
        \draw (4,1.5) arc (-180:0:0.5 and 0.3);
\end{tikzpicture}
     & \scriptsize$\frac{[4][3][2]}{[8][7][6][5]}$ 
     & $0$            
\end{tabular}
\vspace{6pt}
\caption{The terms of $f^{(8)}$ that contribute to $\ip{(Q,T),(Q,T)}$. }
\end{table}

\begin{table}[!htb]
\center
\begin{tabular}{ c | c | c }
$\beta$ & $\mathrm{Coeff}_{f^{(8)}}(\beta)$ & value of diagram \\
\hline
\begin{tikzpicture}[scale=.15,baseline=-0.5ex]
    \foreach \x in {1,2,3,4,5,6,7,8} \draw (\x cm , -1.5cm)--(\x cm, 1.5cm);
    
\end{tikzpicture} 
     & $1$
     & $0$ \\
\begin{tikzpicture}[scale=.15,baseline=-0.5ex]
        \draw (4,1.5) arc (-180:0:0.5);
        \draw (1,-1.5) -- (1,1.5);
        \draw (2,-1.5) -- (2,1.5);
        \draw (3,-1.5) -- (3,1.5);
        \draw (4,-1.5) arc (180:0:0.5);
        \draw (6,-1.5) -- (6,1.5);
        \draw (7,-1.5) -- (7,1.5);
        \draw (8,-1.5) -- (8,1.5);
\end{tikzpicture}
     & \scriptsize$-[4]^2(\frac{1}{[8][7]}+\frac{1}{[7][6]}+\frac{1}{[6][5]}+\frac{1}{[5][4]})$
     & \scriptsize$Z((\rho^{1/2}(T))^2(\rho^{1/2}(Q))^2)$ \\
\begin{tikzpicture}[scale=.15,baseline=-0.5ex]
        \draw (1,-1.5) -- (1,1.5);
        \draw (2,-1.5) -- (2,1.5);
        \draw (3,-1.5) arc (180:0:1.5 and 1.1);
        \draw (4,-1.5) arc (180:0:0.5);
        \draw (7,-1.5) -- (7,1.5);
        \draw (8,-1.5) -- (8,1.5);
        \draw (3,1.5) arc (-180:0:1.5 and 1.1);
        \draw (4,1.5) arc (-180:0:0.5);
\end{tikzpicture}
     & \scriptsize$\frac{[4]^2[3]^2}{[8][7][6][5][4][3]}([8][7]+[8][5]+[8][3]+[6][5]+[6][3]+[4][3])$ 
     & $0$\\
\begin{tikzpicture}[scale=.15,baseline=-0.5ex]
        \draw (1,-1.5) -- (1,1.5);
        \draw (2,-1.5) arc (180:0:2.5 and 1.5);
        \draw (3,-1.5) arc (180:0:1.5 and 1.0);
        \draw (4,-1.5) arc (180:0:0.5 and 0.5);
        \draw (8,-1.5) -- (8,1.5);
        \draw (2,1.5) arc (-180:0:2.5 and 1.5);
        \draw (3,1.5) arc (-180:0:1.5 and 1.0);
        \draw (4,1.5) arc (-180:0:0.5 and 0.5);
\end{tikzpicture}
     & \scriptsize$\frac{-[4][3][2]}{[8][7][6][5]}([8]+[6]+[4]+[2])$ 
     & $0$  \\
 \begin{tikzpicture}[scale=.15,baseline=-0.5ex]
        \draw (1,-1.5) arc (180:0:3.5 and 1.5);
        \draw (2,-1.5) arc (180:0:2.5 and 1.1);
        \draw (3,-1.5) arc (180:0:1.5 and 0.7);
        \draw (4,-1.5) arc (180:0:0.5 and 0.3);
        \draw (1,1.5) arc (-180:0:3.5 and 1.5);
        \draw (2,1.5) arc (-180:0:2.5 and 1.1);
        \draw (3,1.5) arc (-180:0:1.5 and 0.7);
        \draw (4,1.5) arc (-180:0:0.5 and 0.3);
\end{tikzpicture}
     & \scriptsize$\frac{[4][3][2]}{[8][7][6][5]}$ 
     & $0$            
\end{tabular}
\vspace{6pt}
\caption{The terms of $f^{(8)}$ that contribute to $\ip{(T,Q),(Q,T)}$. }
\end{table}
\newpage
\begin{table}[!htb]
\center
\begin{tabular}{ c | c | c }
$\beta$ & $\mathrm{Coeff}_{f^{(8)}}(\beta)$ & value of diagram \\
\hline
\begin{tikzpicture}[scale=.15,baseline=-0.5ex]
    \foreach \x in {1,2,3,4,5,6,7,8} \draw (\x cm , -1.5cm)--(\x cm, 1.5cm);
    
\end{tikzpicture} 
     & $1$
     & $0$ \\
\begin{tikzpicture}[scale=.15,baseline=-0.5ex]
        \draw (4,1.5) arc (-180:0:0.5);
        \draw (1,-1.5) -- (1,1.5);
        \draw (2,-1.5) -- (2,1.5);
        \draw (3,-1.5) -- (3,1.5);
        \draw (4,-1.5) arc (180:0:0.5);
        \draw (6,-1.5) -- (6,1.5);
        \draw (7,-1.5) -- (7,1.5);
        \draw (8,-1.5) -- (8,1.5);
\end{tikzpicture}
     & \scriptsize$-[4]^2(\frac{1}{[8][7]}+\frac{1}{[7][6]}+\frac{1}{[6][5]}+\frac{1}{[5][4]})$
     & \scriptsize$\omega^2\big( \frac{Z(T^2Q)}{[4]}Z(T^3)+\frac{Z(TQ^2)}{[4]}Z(T^2Q)\big)$ \\
\begin{tikzpicture}[scale=.15,baseline=-0.5ex]
        \draw (1,-1.5) -- (1,1.5);
        \draw (2,-1.5) -- (2,1.5);
        \draw (3,-1.5) arc (180:0:1.5 and 1.1);
        \draw (4,-1.5) arc (180:0:0.5);
        \draw (7,-1.5) -- (7,1.5);
        \draw (8,-1.5) -- (8,1.5);
        \draw (3,1.5) arc (-180:0:1.5 and 1.1);
        \draw (4,1.5) arc (-180:0:0.5);
\end{tikzpicture}
     & \scriptsize$\frac{[4]^2[3]^2}{[8][7][6][5][4][3]}([8][7]+[8][5]+[8][3]+[6][5]+[6][3]+[4][3])$ 
     & $0$\\
\begin{tikzpicture}[scale=.15,baseline=-0.5ex]
        \draw (1,-1.5) -- (1,1.5);
        \draw (2,-1.5) arc (180:0:2.5 and 1.5);
        \draw (3,-1.5) arc (180:0:1.5 and 1.0);
        \draw (4,-1.5) arc (180:0:0.5 and 0.5);
        \draw (8,-1.5) -- (8,1.5);
        \draw (2,1.5) arc (-180:0:2.5 and 1.5);
        \draw (3,1.5) arc (-180:0:1.5 and 1.0);
        \draw (4,1.5) arc (-180:0:0.5 and 0.5);
\end{tikzpicture}
     & \scriptsize$\frac{-[4][3][2]}{[8][7][6][5]}([8]+[6]+[4]+[2])$ 
     & $0$  \\
 \begin{tikzpicture}[scale=.15,baseline=-0.5ex]
        \draw (1,-1.5) arc (180:0:3.5 and 1.5);
        \draw (2,-1.5) arc (180:0:2.5 and 1.1);
        \draw (3,-1.5) arc (180:0:1.5 and 0.7);
        \draw (4,-1.5) arc (180:0:0.5 and 0.3);
        \draw (1,1.5) arc (-180:0:3.5 and 1.5);
        \draw (2,1.5) arc (-180:0:2.5 and 1.1);
        \draw (3,1.5) arc (-180:0:1.5 and 0.7);
        \draw (4,1.5) arc (-180:0:0.5 and 0.3);
\end{tikzpicture}
     & \scriptsize$\frac{[4][3][2]}{[8][7][6][5]}$ 
     & $0$            
\end{tabular}
\vspace{6pt}
\caption{The terms of $f^{(8)}$ that contribute to $\ip{(T,Q),(T,T)}$. }
\end{table}
\end{proof}

We are now ready to prove the shaded 4-box relations.
Since there is one new element $M'$ at level 4, we expect to see that $(T,Q)$ can be written as a linear combination of elements in $ATL(T)$, $ATL(Q)$, and the new element $M'$.  More precisely, we expect to see $(T,Q)$ as a linear combination of $u_0$, $v_0$, and $M'$.
Let $$M'=(T,Q)-\proj{u_0}{(T,Q)}-\proj{v_0}{(T,Q)}$$
We want to show that $(Q,T)\in span\left\{u_0,v_0,M'\right\}$.  Similarly, for $(T,T)$ and $(Q,Q)$.  We use Bessel's inequality to prove these relations.

\begin{theorem}We have the following 4-box relations in the precision sufficient for our needs.

\begin{itemize}
\item $(R4(T,Q))_0$ \hspace{2cm}$(T,Q) \in  span\left\{u_0,v_0,M'\right\}$
\item $(R4(Q,T))_0$ \hspace{2cm}$(Q,T) \in  span\left\{u_0,v_0,M'\right\}$
\item $(R4(T,T))_0$ \hspace{2cm}$(T,T) \in  span\left\{u_0,v_0,M'\right\}$
\item $(R4(Q,Q))_0$ \hspace{2cm}$(Q,Q) \in  span\left\{u_0,v_0,M'\right\}$
\end{itemize}
In terms of diagrams, this means:
Each of 
$$
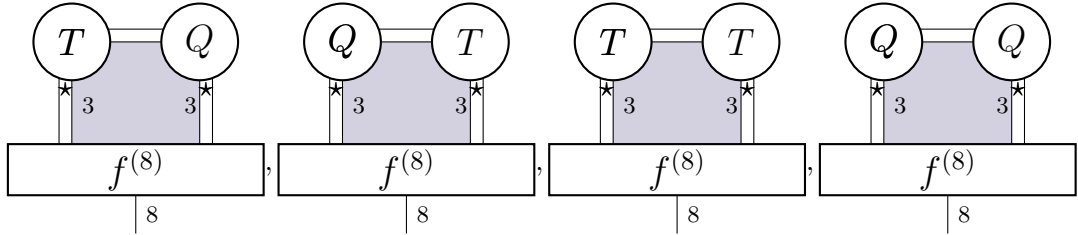
\begin{tikzpicture}[baseline=0,scale=1.7]
	{%
	\fill[shaded] (-0.5,0) rectangle (0.5,1);
	\draw (-0.5,1) -- (-0.5,0);
	\node[anchor=west] at (-0.5,0.5) {\footnotesize$3$};
	\draw (-0.6,1) -- (-0.6,0);
	\draw (0.5,1) -- (0.5,0);
	\node[anchor=west] at (0.3,0.5) {\footnotesize$3$};
	\draw (0.6,1) -- (0.6,0);
}
 {%
	
	\draw (-0.5,1) -- (0.5, 1);
	\draw (-0.5,1.1) -- (0.5,1.1);
	\foreach \x in {-0.5,0.5} {
		{%
	\filldraw[fill=white,thick] (\x,1) ellipse (3mm and 3mm);
	\node at (-.5,1) {\Large $T$};
	\node at (.5,1) {\Large $Q$};
	\path(-.55,1) ++(-90:0.37) node {$\star$};
	\path(0.55,1) ++(-90:0.37) node {$\star$};
}
	}
}
        \draw (0,0)--(0,-0.5);
        \node[anchor=west] at (0,-0.35) {\footnotesize$8$};
	{%
	\filldraw[fill=white,thick] (-1,-0.2) rectangle (1,0.2);
	\node at (0,0) {\Large$\JW{8}$};
}
\end{tikzpicture}, \begin{tikzpicture}[baseline=0,scale=1.7]
	{%
	\fill[shaded] (-0.5,0) rectangle (0.5,1);
	\draw (-0.5,1) -- (-0.5,0);
	\node[anchor=west] at (-0.5,0.5) {\footnotesize$3$};
	\draw (-0.6,1) -- (-0.6,0);
	\draw (0.5,1) -- (0.5,0);
	\node[anchor=west] at (0.3,0.5) {\footnotesize$3$};
	\draw (0.6,1) -- (0.6,0);
}
 {%
	
	\draw (-0.5,1) -- (0.5, 1);
	\draw (-0.5,1.1) -- (0.5,1.1);
	\foreach \x in {-0.5,0.5} {
		{%
	\filldraw[fill=white,thick] (\x,1) ellipse (3mm and 3mm);
	\node at (-.5,1) {\Large $Q$};
	\node at (.5,1) {\Large $T$};
	\path(-.55,1) ++(-90:0.37) node {$\star$};
	\path(0.55,1) ++(-90:0.37) node {$\star$};
}
	}
}
        \draw (0,0)--(0,-0.5);
        \node[anchor=west] at (0,-0.35) {\footnotesize$8$};
	{%
	\filldraw[fill=white,thick] (-1,-0.2) rectangle (1,0.2);
	\node at (0,0) {\Large$\JW{8}$};
}
\end{tikzpicture}, \begin{tikzpicture}[baseline=0,scale=1.7]
	{%
	\fill[shaded] (-0.5,0) rectangle (0.5,1);
	\draw (-0.5,1) -- (-0.5,0);
	\node[anchor=west] at (-0.5,0.5) {\footnotesize$3$};
	\draw (-0.6,1) -- (-0.6,0);
	\draw (0.5,1) -- (0.5,0);
	\node[anchor=west] at (0.3,0.5) {\footnotesize$3$};
	\draw (0.6,1) -- (0.6,0);
}
 {%
	
	\draw (-0.5,1) -- (0.5, 1);
	\draw (-0.5,1.1) -- (0.5,1.1);
	\foreach \x in {-0.5,0.5} {
		{%
	\filldraw[fill=white,thick] (\x,1) ellipse (3mm and 3mm);
	\node at (-.5,1) {\Large $T$};
	\node at (.5,1) {\Large $T$};
	\path(-.55,1) ++(-90:0.37) node {$\star$};
	\path(0.55,1) ++(-90:0.37) node {$\star$};
}
	}
}
        \draw (0,0)--(0,-0.5);
        \node[anchor=west] at (0,-0.35) {\footnotesize$8$};
	{%
	\filldraw[fill=white,thick] (-1,-0.2) rectangle (1,0.2);
	\node at (0,0) {\Large$\JW{8}$};
}
\end{tikzpicture}, \begin{tikzpicture}[baseline=0,scale=1.7]
	{%
	\fill[shaded] (-0.5,0) rectangle (0.5,1);
	\draw (-0.5,1) -- (-0.5,0);
	\node[anchor=west] at (-0.5,0.5) {\footnotesize$3$};
	\draw (-0.6,1) -- (-0.6,0);
	\draw (0.5,1) -- (0.5,0);
	\node[anchor=west] at (0.3,0.5) {\footnotesize$3$};
	\draw (0.6,1) -- (0.6,0);
}
 {%
	
	\draw (-0.5,1) -- (0.5, 1);
	\draw (-0.5,1.1) -- (0.5,1.1);
	\foreach \x in {-0.5,0.5} {
		{%
	\filldraw[fill=white,thick] (\x,1) ellipse (3mm and 3mm);
	\node at (-.5,1) {\Large $Q$};
	\node at (.5,1) {\Large $Q$};
	\path(-.55,1) ++(-90:0.37) node {$\star$};
	\path(0.55,1) ++(-90:0.37) node {$\star$};
}
	}
}
        \draw (0,0)--(0,-0.5);
        \node[anchor=west] at (0,-0.35) {\footnotesize$8$};
	{%
	\filldraw[fill=white,thick] (-1,-0.2) rectangle (1,0.2);
	\node at (0,0) {\Large$\JW{8}$};
}
\end{tikzpicture}$$ lies in the span of 

$$\Big( 
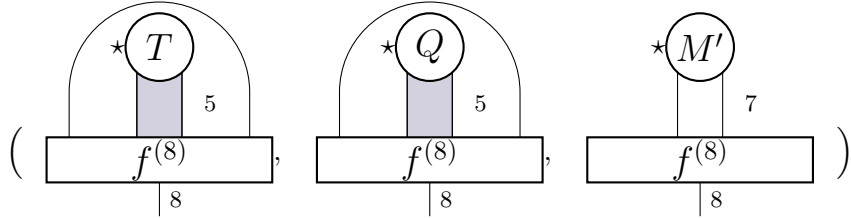
\hspace{3mm}\begin{tikzpicture}[baseline=0,scale=1.5]
	{\draw[shaded] (0.2,0) -- (0.2,1) -- (-0.2,1) -- (-0.2,0);
	\draw (-0.8,0) -- (-0.8,0.6) arc (180:0:0.8) -- (0.8,0);
	\draw(-0.2,0) -- (-0.2,1);
	\draw (0.2,0) -- (0.2,1);
	\node at (0.45,0.5) {\footnotesize$5$};
	{%
	\filldraw[fill=white,thick] (0,1) ellipse (3mm and 3mm);
	\node at (0,1) {\Large $T$};
	\path(0,1) ++(-180:0.37) node {$\star$};
}
}
;
        \draw (0,0)--(0,-0.5);
        \node[anchor=west] at (0,-0.35) {\footnotesize$8$};
	{%
	\filldraw[fill=white,thick] (-1,-0.2) rectangle (1,0.2);
	\node at (0,0) {\Large$\JW{8}$};
}\end{tikzpicture},
\quad \begin{tikzpicture}[baseline=0,scale=1.5]
	{\draw[shaded] (0.2,0) -- (0.2,1) -- (-0.2,1) -- (-0.2,0);
	\draw (-0.8,0) -- (-0.8,0.6) arc (180:0:0.8) -- (0.8,0);
	\draw(-0.2,0) -- (-0.2,1);
	\draw (0.2,0) -- (0.2,1);
	\node at (0.45,0.5) {\footnotesize$5$};
	{%
	\filldraw[fill=white,thick] (0,1) ellipse (3mm and 3mm);
	\node at (0,1) {\Large $Q$};
	\path(0,1) ++(-180:0.37) node {$\star$};
}
}
;
        \draw (0,0)--(0,-0.5);
        \node[anchor=west] at (0,-0.35) {\footnotesize$8$};
	{%
	\filldraw[fill=white,thick] (-1,-0.2) rectangle (1,0.2);
	\node at (0,0) {\Large$\JW{8}$};
}\end{tikzpicture} ,\quad \begin{tikzpicture}[baseline=0,scale=1.5]
	{

	\draw (-0.2,0) -- (-0.2,1);
	\draw (0.2,0) -- (0.2,1);
	\node at (0.45,0.5) {\footnotesize$7$};
	{%
	\filldraw[fill=white,thick] (0,1) ellipse (3mm and 3mm);
	\node at (0,1) {\Large $M'$};
	\path(0,1) ++(-180:0.37) node {$\star$};
}
}
;
        \draw (0,0)--(0,-0.5);
        \node[anchor=west] at (0,-0.35) {\footnotesize$8$};
	{%
	\filldraw[fill=white,thick] (-1,-0.2) rectangle (1,0.2);
	\node at (0,0) {\Large$\JW{8}$};
}\end{tikzpicture} \hspace{3mm}\Big)$$
\end{theorem}

\begin{proof}[Proof of $(R4(T,Q))_0$] This is true by definition of $M'$.
\end{proof}
We use Lemmas 5.2.11, 5.2.13, 5.2.14, 5.2.15, and the reverse shading results of Lemma 5.2.2 for the values of the relevant inner products to prove the remaining relations.
\begin{proof}[Proof of $(R4(Q,T))_0$] Let $w =(T,Q)$ and $w_0=(Q,T)$.  We need to show that $$\ip{w_0,w_0} -\frac{\norm{\ip{w_0,u_0}}^2}{\ip{u_0,u_0}}-\frac{\norm{\ip{w_0,v_0}}^2}{\ip{v_0,v_0}}=\frac{\norm{\ip{w_0,w}-\frac{\ip{w_0,u_0}\overline{\ip{w,u_0}}}{\ip{u_0,u_0}}-\frac{\ip{w_0,v_0}\overline{\ip{w,v_0}}}{\ip{v_0,v_0}}}^2}{\Norm{M'}^2}$$
On the left hand side, we get $\frac{9+19\sqrt{21}}{45}$.  On the right hand side, we get $\frac{\frac{2}{675}(1277+57\sqrt{21}}{\Norm{M'}^2}$.
Since $\Norm{M'}^2=\frac{9+19\sqrt{21}}{45}$, we have equality.  Thus, $(Q,T) \in  span\left\{u_0,v_0,M'\right\}$
\end{proof}
\begin{proof}[Proof of $(R4(T,T))_0$]
We need to show that $\ip{(T,T),(T,T)}-\frac{\norm{\ip{(T,T),u_0}}^2}{\ip{u_0,u_0}}-\frac{\norm{\ip{(T,T),v_0}}^2}{\ip{v_0,v_0}}$
=$$\frac{\norm{\ip{(T,T),(T,Q)}-\frac{\ip{(T,T),u_0}\overline{\ip{(T,Q),u_0}}}{\ip{u_0,u_0}}-\frac{\ip{(T,T),v_0}\overline{\ip{(T,Q),v_0}}}{\ip{v_0,v_0}}}^2}{\Norm{M'}^2}$$

On the left hand side, we get $\frac{363-67\sqrt{21}}{180}$.  On the right hand side, we get $\frac{\frac{-3911+1049\sqrt{21}}{1350}}{\Norm{M'}^2}$.
Since $\Norm{M'}^2=\frac{9+19\sqrt{21}}{45}$, we have equality.  Thus, $(T,T) \in  span\left\{u_0,v_0,M'\right\}$
\end{proof}
\begin{proof}[Proof of $(R4(Q,Q))_0$] 
We need to show that
$ \ip{(Q,Q),(Q,Q)}-\frac{\norm{\ip{(Q,Q),u_0}}^2}{\ip{u_0,u_0}}-\frac{\norm{\ip{(Q,Q),v_0}}^2}{\ip{v_0,v_0}}$= $$\frac{\norm{\ip{(Q,Q),(T,Q)}-\frac{\ip{(Q,Q),u_0}\overline{\ip{(T,Q),u_0}}}{\ip{u_0,u_0}}-\frac{\ip{(Q,Q),v_0}\overline{\ip{(T,Q),v_0}}}{\ip{v_0,v_0}}}^2}{\Norm{M'}^2}$$

On the left hand side, we get $\frac{87+17\sqrt{21}}{180}$.  On the right hand side, we get $\frac{\frac{1261+301\sqrt{21}}{1350}}{\Norm{M'}^2}$.
Since $\Norm{M'}^2=\frac{9+19\sqrt{21}}{45}$, we have equality.  Thus, $(Q,Q) \in  span\left\{u_0,v_0,M'\right\}$
\end{proof}

\chapter{One-strand Braiding Substitutes}
Ultimately, we want to show the subfactor property of our planar algebra and to do this we will use the Jellyfish algorithm, which is introduced in \cite {ExtH}.  In order to do the Jellyfish algorithm, the authors of that paper use one-strand and two-strand braiding substitutes which allow you to pluck out a generator from underneath a single strand and two strands, respectively.  We can find one-strand braiding substitutes using our 4-box relations.  However, since there are four new elements $N_1,N_2,N_3,N_4$ at level $5$, this makes it very difficult to find two-strand braiding substitutes as the number of unknown variables in our technique for finding substitutes is too large.  Fortunately, we do not need the two-strand braiding substitute to do the Jellyfish algorithm, since a one-strand braiding substitute with the shading switched suffices.  In this chapter, we look for unshaded and shaded one-strand braiding substitutes.

To prove the one-strand braiding substitutes, we use Bessel's inequality.

\section{Unshaded One-strand Braiding substitutes}
Recall from Theorem 5.2.9 the unshaded 4-box relations:
$$\begin{tikzpicture}[baseline=0,scale=1.5]
	{%
	\draw (-0.5,1) -- (-0.5,0);
	\node[anchor=west] at (-0.5,0.5) {\footnotesize$4$};
	\draw (0.5,1) -- (0.5,0);
	\node[anchor=west] at (0.5,0.5) {\footnotesize$4$};
}
 {%
	\node[anchor=south] at (0,1) {\footnotesize$2$};
	\draw (-0.5,1) -- (0.5, 1);
	\foreach \x in {-0.5,0.5} {
		{%
	\filldraw[fill=white,thick] (\x,1) ellipse (3mm and 3mm);
	\node at (-.5,1) {\Large $T$};
	\node at (.5,1) {\Large $Q$};
	\path(\x,1) ++(90:0.37) node {$\star$};
}
	}
}
        \draw (0,0)--(0,-0.5);
        \node[anchor=west] at (0,-0.35) {\footnotesize$8$};
	{%
	\filldraw[fill=white,thick] (-1,-0.2) rectangle (1,0.2);
	\node at (0,0) {\Large$\JW{8}$};
}
\end{tikzpicture} = \alpha \cdot \begin{tikzpicture}[baseline=0,scale=1.5]
	{
	\fill[shaded] (-0.8,0) -- (-0.8,0.6) arc (180:0:0.8) -- (0.8,0) -- (0.2,0) -- (0.2,1) -- (-0.2,1) -- (-0.2,0);
	\draw (-0.8,0) -- (-0.8,0.6) arc (180:0:0.8) -- (0.8,0);
	\draw (-0.2,0) -- (-0.2,1);
	\draw (0.2,0) -- (0.2,1);
	\node at (0.45,0.5) {\footnotesize$5$};
	{%
	\filldraw[fill=white,thick] (0,1) ellipse (3mm and 3mm);
	\node at (0,1) {\Large $T$};
	\path(0,1) ++(-90:0.37) node {$\star$};
}
}
;
        \draw (0,0)--(0,-0.5);
        \node[anchor=west] at (0,-0.35) {\footnotesize$8$};
	{%
	\filldraw[fill=white,thick] (-1,-0.2) rectangle (1,0.2);
	\node at (0,0) {\Large$\JW{8}$};
}\end{tikzpicture} +\beta \cdot \begin{tikzpicture}[baseline=0,scale=1.5]
	{
	\fill[shaded] (-0.8,0) -- (-0.8,0.6) arc (180:0:0.8) -- (0.8,0) -- (0.2,0) -- (0.2,1) -- (-0.2,1) -- (-0.2,0);
	\draw (-0.8,0) -- (-0.8,0.6) arc (180:0:0.8) -- (0.8,0);
	\draw (-0.2,0) -- (-0.2,1);
	\draw (0.2,0) -- (0.2,1);
	\node at (0.45,0.5) {\footnotesize$5$};
	{%
	\filldraw[fill=white,thick] (0,1) ellipse (3mm and 3mm);
	\node at (0,1) {\Large $Q$};
	\path(0,1) ++(-90:0.37) node {$\star$};
}
}
;
        \draw (0,0)--(0,-0.5);
        \node[anchor=west] at (0,-0.35) {\footnotesize$8$};
	{%
	\filldraw[fill=white,thick] (-1,-0.2) rectangle (1,0.2);
	\node at (0,0) {\Large$\JW{8}$};
}\end{tikzpicture}+ \begin{tikzpicture}[baseline=0,scale=1.5]
	{\draw[shaded] (0.2,0) -- (0.2,1) -- (-0.2,1) -- (-0.2,0);
	
	\draw(-0.2,0) -- (-0.2,1);
	\draw (0.2,0) -- (0.2,1);
	\node at (0.45,0.5) {\footnotesize$7$};
	{%
	\filldraw[fill=white,thick] (0,1) ellipse (3mm and 3mm);
	\node at (0,1) {\Large $M$};
	\path(0,1) ++(-180:0.37) node {$\star$};
}
}
;
        \draw (0,0)--(0,-0.5);
        \node[anchor=west] at (0,-0.35) {\footnotesize$8$};
	{%
	\filldraw[fill=white,thick] (-1,-0.2) rectangle (1,0.2);
	\node at (0,0) {\Large$\JW{8}$};
}\end{tikzpicture} $$, 
and 
$$\begin{tikzpicture}[baseline=0,scale=1.5]
	{%
	\draw (-0.5,1) -- (-0.5,0);
	\node[anchor=west] at (-0.5,0.5) {\footnotesize$4$};
	\draw (0.5,1) -- (0.5,0);
	\node[anchor=west] at (0.5,0.5) {\footnotesize$4$};
}
 {%
	\node[anchor=south] at (0,1) {\footnotesize$2$};
	\draw (-0.5,1) -- (0.5, 1);
	\foreach \x in {-0.5,0.5} {
		{%
	\filldraw[fill=white,thick] (\x,1) ellipse (3mm and 3mm);
	\node at (-.5,1) {\Large $Q$};
	\node at (.5,1) {\Large $T$};
	\path(\x,1) ++(90:0.37) node {$\star$};
}
	}
}
        \draw (0,0)--(0,-0.5);
        \node[anchor=west] at (0,-0.35) {\footnotesize$8$};
	{%
	\filldraw[fill=white,thick] (-1,-0.2) rectangle (1,0.2);
	\node at (0,0) {\Large$\JW{8}$};
}
\end{tikzpicture} = \alpha_0 \cdot \begin{tikzpicture}[baseline=0,scale=1.5]
	{
	\fill[shaded] (-0.8,0) -- (-0.8,0.6) arc (180:0:0.8) -- (0.8,0) -- (0.2,0) -- (0.2,1) -- (-0.2,1) -- (-0.2,0);
	\draw (-0.8,0) -- (-0.8,0.6) arc (180:0:0.8) -- (0.8,0);
	\draw (-0.2,0) -- (-0.2,1);
	\draw (0.2,0) -- (0.2,1);
	\node at (0.45,0.5) {\footnotesize$5$};
	{%
	\filldraw[fill=white,thick] (0,1) ellipse (3mm and 3mm);
	\node at (0,1) {\Large $T$};
	\path(0,1) ++(-90:0.37) node {$\star$};
}
}
;
        \draw (0,0)--(0,-0.5);
        \node[anchor=west] at (0,-0.35) {\footnotesize$8$};
	{%
	\filldraw[fill=white,thick] (-1,-0.2) rectangle (1,0.2);
	\node at (0,0) {\Large$\JW{8}$};
}\end{tikzpicture} +\beta_0 \cdot \begin{tikzpicture}[baseline=0,scale=1.5]
	{
	\fill[shaded] (-0.8,0) -- (-0.8,0.6) arc (180:0:0.8) -- (0.8,0) -- (0.2,0) -- (0.2,1) -- (-0.2,1) -- (-0.2,0);
	\draw (-0.8,0) -- (-0.8,0.6) arc (180:0:0.8) -- (0.8,0);
	\draw (-0.2,0) -- (-0.2,1);
	\draw (0.2,0) -- (0.2,1);
	\node at (0.45,0.5) {\footnotesize$5$};
	{%
	\filldraw[fill=white,thick] (0,1) ellipse (3mm and 3mm);
	\node at (0,1) {\Large $Q$};
	\path(0,1) ++(-90:0.37) node {$\star$};
}
}
;
        \draw (0,0)--(0,-0.5);
        \node[anchor=west] at (0,-0.35) {\footnotesize$8$};
	{%
	\filldraw[fill=white,thick] (-1,-0.2) rectangle (1,0.2);
	\node at (0,0) {\Large$\JW{8}$};
}\end{tikzpicture}+ \gamma_0 \cdot \begin{tikzpicture}[baseline=0,scale=1.5]
	{\draw[shaded] (0.2,0) -- (0.2,1) -- (-0.2,1) -- (-0.2,0);
	
	\draw(-0.2,0) -- (-0.2,1);
	\draw (0.2,0) -- (0.2,1);
	\node at (0.45,0.5) {\footnotesize$7$};
	{%
	\filldraw[fill=white,thick] (0,1) ellipse (3mm and 3mm);
	\node at (0,1) {\Large $M$};
	\path(0,1) ++(-180:0.37) node {$\star$};
}
}
;
        \draw (0,0)--(0,-0.5);
        \node[anchor=west] at (0,-0.35) {\footnotesize$8$};
	{%
	\filldraw[fill=white,thick] (-1,-0.2) rectangle (1,0.2);
	\node at (0,0) {\Large$\JW{8}$};
}\end{tikzpicture} $$
for some $\alpha,\beta,\alpha_0,\beta_0,\gamma_0 \in \C$.
\begin{remark}Notice that the way we defined $M$ was relative to $(T,Q)$, and so $M$ is not really the underlying self-adjoint generator at level 4.  Let $M_0$ be the underlying self-adjoint generator at level 4.  Then, $M$ is a scalar multiple of $M_0$, and we modify the above two relations by replacing $M$ by a scalar multiple of $M_0$.
\end{remark}
\newpage
Thus, we have the following modified 4-box relations:
$$\begin{tikzpicture}[baseline=0,scale=1.5]
	{%
	\draw (-0.5,1) -- (-0.5,0);
	\node[anchor=west] at (-0.5,0.5) {\footnotesize$4$};
	\draw (0.5,1) -- (0.5,0);
	\node[anchor=west] at (0.5,0.5) {\footnotesize$4$};
}
 {%
	\node[anchor=south] at (0,1) {\footnotesize$2$};
	\draw (-0.5,1) -- (0.5, 1);
	\foreach \x in {-0.5,0.5} {
		{%
	\filldraw[fill=white,thick] (\x,1) ellipse (3mm and 3mm);
	\node at (-.5,1) {\Large $T$};
	\node at (.5,1) {\Large $Q$};
	\path(\x,1) ++(90:0.37) node {$\star$};
}
	}
}
        \draw (0,0)--(0,-0.5);
        \node[anchor=west] at (0,-0.35) {\footnotesize$8$};
	{%
	\filldraw[fill=white,thick] (-1,-0.2) rectangle (1,0.2);
	\node at (0,0) {\Large$\JW{8}$};
}
\end{tikzpicture} = \alpha \cdot \begin{tikzpicture}[baseline=0,scale=1.5]
	{
	\fill[shaded] (-0.8,0) -- (-0.8,0.6) arc (180:0:0.8) -- (0.8,0) -- (0.2,0) -- (0.2,1) -- (-0.2,1) -- (-0.2,0);
	\draw (-0.8,0) -- (-0.8,0.6) arc (180:0:0.8) -- (0.8,0);
	\draw (-0.2,0) -- (-0.2,1);
	\draw (0.2,0) -- (0.2,1);
	\node at (0.45,0.5) {\footnotesize$5$};
	{%
	\filldraw[fill=white,thick] (0,1) ellipse (3mm and 3mm);
	\node at (0,1) {\Large $T$};
	\path(0,1) ++(-90:0.37) node {$\star$};
}
}
;
        \draw (0,0)--(0,-0.5);
        \node[anchor=west] at (0,-0.35) {\footnotesize$8$};
	{%
	\filldraw[fill=white,thick] (-1,-0.2) rectangle (1,0.2);
	\node at (0,0) {\Large$\JW{8}$};
}\end{tikzpicture} +\beta \cdot \begin{tikzpicture}[baseline=0,scale=1.5]
	{
	\fill[shaded] (-0.8,0) -- (-0.8,0.6) arc (180:0:0.8) -- (0.8,0) -- (0.2,0) -- (0.2,1) -- (-0.2,1) -- (-0.2,0);
	\draw (-0.8,0) -- (-0.8,0.6) arc (180:0:0.8) -- (0.8,0);
	\draw (-0.2,0) -- (-0.2,1);
	\draw (0.2,0) -- (0.2,1);
	\node at (0.45,0.5) {\footnotesize$5$};
	{%
	\filldraw[fill=white,thick] (0,1) ellipse (3mm and 3mm);
	\node at (0,1) {\Large $Q$};
	\path(0,1) ++(-90:0.37) node {$\star$};
}
}
;
        \draw (0,0)--(0,-0.5);
        \node[anchor=west] at (0,-0.35) {\footnotesize$8$};
	{%
	\filldraw[fill=white,thick] (-1,-0.2) rectangle (1,0.2);
	\node at (0,0) {\Large$\JW{8}$};
}\end{tikzpicture}+ \gamma \cdot \begin{tikzpicture}[baseline=0,scale=1.5]
	{\draw[shaded] (0.2,0) -- (0.2,1) -- (-0.2,1) -- (-0.2,0);
	
	\draw(-0.2,0) -- (-0.2,1);
	\draw (0.2,0) -- (0.2,1);
	\node at (0.45,0.5) {\footnotesize$7$};
	{%
	\filldraw[fill=white,thick] (0,1) ellipse (3mm and 3mm);
	\node at (0,1) {\Large $M_0$};
	\path(0,1) ++(-180:0.37) node {$\star$};
}
}
;
        \draw (0,0)--(0,-0.5);
        \node[anchor=west] at (0,-0.35) {\footnotesize$8$};
	{%
	\filldraw[fill=white,thick] (-1,-0.2) rectangle (1,0.2);
	\node at (0,0) {\Large$\JW{8}$};
}\end{tikzpicture} $$, 
and 
$$\begin{tikzpicture}[baseline=0,scale=1.5]
	{%
	\draw (-0.5,1) -- (-0.5,0);
	\node[anchor=west] at (-0.5,0.5) {\footnotesize$4$};
	\draw (0.5,1) -- (0.5,0);
	\node[anchor=west] at (0.5,0.5) {\footnotesize$4$};
}
 {%
	\node[anchor=south] at (0,1) {\footnotesize$2$};
	\draw (-0.5,1) -- (0.5, 1);
	\foreach \x in {-0.5,0.5} {
		{%
	\filldraw[fill=white,thick] (\x,1) ellipse (3mm and 3mm);
	\node at (-.5,1) {\Large $Q$};
	\node at (.5,1) {\Large $T$};
	\path(\x,1) ++(90:0.37) node {$\star$};
}
	}
}
        \draw (0,0)--(0,-0.5);
        \node[anchor=west] at (0,-0.35) {\footnotesize$8$};
	{%
	\filldraw[fill=white,thick] (-1,-0.2) rectangle (1,0.2);
	\node at (0,0) {\Large$\JW{8}$};
}
\end{tikzpicture} = \alpha_0 \cdot \begin{tikzpicture}[baseline=0,scale=1.5]
	{
	\fill[shaded] (-0.8,0) -- (-0.8,0.6) arc (180:0:0.8) -- (0.8,0) -- (0.2,0) -- (0.2,1) -- (-0.2,1) -- (-0.2,0);
	\draw (-0.8,0) -- (-0.8,0.6) arc (180:0:0.8) -- (0.8,0);
	\draw (-0.2,0) -- (-0.2,1);
	\draw (0.2,0) -- (0.2,1);
	\node at (0.45,0.5) {\footnotesize$5$};
	{%
	\filldraw[fill=white,thick] (0,1) ellipse (3mm and 3mm);
	\node at (0,1) {\Large $T$};
	\path(0,1) ++(-90:0.37) node {$\star$};
}
}
;
        \draw (0,0)--(0,-0.5);
        \node[anchor=west] at (0,-0.35) {\footnotesize$8$};
	{%
	\filldraw[fill=white,thick] (-1,-0.2) rectangle (1,0.2);
	\node at (0,0) {\Large$\JW{8}$};
}\end{tikzpicture} +\beta_0 \cdot \begin{tikzpicture}[baseline=0,scale=1.5]
	{
	\fill[shaded] (-0.8,0) -- (-0.8,0.6) arc (180:0:0.8) -- (0.8,0) -- (0.2,0) -- (0.2,1) -- (-0.2,1) -- (-0.2,0);
	\draw (-0.8,0) -- (-0.8,0.6) arc (180:0:0.8) -- (0.8,0);
	\draw (-0.2,0) -- (-0.2,1);
	\draw (0.2,0) -- (0.2,1);
	\node at (0.45,0.5) {\footnotesize$5$};
	{%
	\filldraw[fill=white,thick] (0,1) ellipse (3mm and 3mm);
	\node at (0,1) {\Large $Q$};
	\path(0,1) ++(-90:0.37) node {$\star$};
}
}
;
        \draw (0,0)--(0,-0.5);
        \node[anchor=west] at (0,-0.35) {\footnotesize$8$};
	{%
	\filldraw[fill=white,thick] (-1,-0.2) rectangle (1,0.2);
	\node at (0,0) {\Large$\JW{8}$};
}\end{tikzpicture}+ \gamma_0 \cdot \begin{tikzpicture}[baseline=0,scale=1.5]
	{\draw[shaded] (0.2,0) -- (0.2,1) -- (-0.2,1) -- (-0.2,0);
	
	\draw(-0.2,0) -- (-0.2,1);
	\draw (0.2,0) -- (0.2,1);
	\node at (0.45,0.5) {\footnotesize$7$};
	{%
	\filldraw[fill=white,thick] (0,1) ellipse (3mm and 3mm);
	\node at (0,1) {\Large $M_0$};
	\path(0,1) ++(-180:0.37) node {$\star$};
}
}
;
        \draw (0,0)--(0,-0.5);
        \node[anchor=west] at (0,-0.35) {\footnotesize$8$};
	{%
	\filldraw[fill=white,thick] (-1,-0.2) rectangle (1,0.2);
	\node at (0,0) {\Large$\JW{8}$};
}\end{tikzpicture} $$
for some $\alpha,\beta,\gamma,\alpha_0,\beta_0,\gamma_0 \in \C$.
Let $w=(T,Q)$ and $w_0=(Q,T)$. Let $u=(T)$ and $v=(Q)$. We eliminate the underlying generator $M_0$ by multiplying the first equation by $\gamma_0$ and subtracting the second equation multiplied by $\gamma$.  This will give us that $\gamma_0 w -\gamma w_0$ is a linear combination of $u$ and $v$.  Our goal is to find the coefficients $\gamma$ and $\gamma_0$ which makes this relation hold:$$\gamma_0 w -\gamma w_0 \in span(u,v)$$ 
\begin{lem}[Eqn.1]The following equation holds:
$$\overline{\eta}\cdot \begin{tikzpicture}[baseline=0,scale=1.5]
	{%
	\draw (-0.5,1) -- (-0.5,0);
	\node[anchor=west] at (-0.5,0.5) {\footnotesize$4$};
	\draw (0.5,1) -- (0.5,0);
	\node[anchor=west] at (0.5,0.5) {\footnotesize$4$};
}
 {%
	\node[anchor=south] at (0,1) {\footnotesize$2$};
	\draw (-0.5,1) -- (0.5, 1);
	\foreach \x in {-0.5,0.5} {
		{%
	\filldraw[fill=white,thick] (\x,1) ellipse (3mm and 3mm);
	\node at (-.5,1) {\Large $T$};
	\node at (.5,1) {\Large $Q$};
	\path(\x,1) ++(90:0.37) node {$\star$};
}
	}
}
        \draw (0,0)--(0,-0.5);
        \node[anchor=west] at (0,-0.35) {\footnotesize$8$};
	{%
	\filldraw[fill=white,thick] (-1,-0.2) rectangle (1,0.2);
	\node at (0,0) {\Large$\JW{8}$};
}
\end{tikzpicture}-\eta \cdot \begin{tikzpicture}[baseline=0,scale=1.5]
	{%
	\draw (-0.5,1) -- (-0.5,0);
	\node[anchor=west] at (-0.5,0.5) {\footnotesize$4$};
	\draw (0.5,1) -- (0.5,0);
	\node[anchor=west] at (0.5,0.5) {\footnotesize$4$};
}
 {%
	\node[anchor=south] at (0,1) {\footnotesize$2$};
	\draw (-0.5,1) -- (0.5, 1);
	\foreach \x in {-0.5,0.5} {
		{%
	\filldraw[fill=white,thick] (\x,1) ellipse (3mm and 3mm);
	\node at (-.5,1) {\Large $Q$};
	\node at (.5,1) {\Large $T$};
	\path(\x,1) ++(90:0.37) node {$\star$};
}
	}
}
        \draw (0,0)--(0,-0.5);
        \node[anchor=west] at (0,-0.35) {\footnotesize$8$};
	{%
	\filldraw[fill=white,thick] (-1,-0.2) rectangle (1,0.2);
	\node at (0,0) {\Large$\JW{8}$};
}
\end{tikzpicture}=a_0\cdot \begin{tikzpicture}[baseline=0,scale=1.5]
	{
	\fill[shaded] (-0.8,0) -- (-0.8,0.6) arc (180:0:0.8) -- (0.8,0) -- (0.2,0) -- (0.2,1) -- (-0.2,1) -- (-0.2,0);
	\draw (-0.8,0) -- (-0.8,0.6) arc (180:0:0.8) -- (0.8,0);
	\draw (-0.2,0) -- (-0.2,1);
	\draw (0.2,0) -- (0.2,1);
	\node at (0.45,0.5) {\footnotesize$5$};
	{%
	\filldraw[fill=white,thick] (0,1) ellipse (3mm and 3mm);
	\node at (0,1) {\Large $T$};
	\path(0,1) ++(-90:0.37) node {$\star$};
}
}
;
        \draw (0,0)--(0,-0.5);
        \node[anchor=west] at (0,-0.35) {\footnotesize$8$};
	{%
	\filldraw[fill=white,thick] (-1,-0.2) rectangle (1,0.2);
	\node at (0,0) {\Large$\JW{8}$};
}\end{tikzpicture}+b_0 \cdot \begin{tikzpicture}[baseline=0,scale=1.5]
	{
	\fill[shaded] (-0.8,0) -- (-0.8,0.6) arc (180:0:0.8) -- (0.8,0) -- (0.2,0) -- (0.2,1) -- (-0.2,1) -- (-0.2,0);
	\draw (-0.8,0) -- (-0.8,0.6) arc (180:0:0.8) -- (0.8,0);
	\draw (-0.2,0) -- (-0.2,1);
	\draw (0.2,0) -- (0.2,1);
	\node at (0.45,0.5) {\footnotesize$5$};
	{%
	\filldraw[fill=white,thick] (0,1) ellipse (3mm and 3mm);
	\node at (0,1) {\Large $Q$};
	\path(0,1) ++(-90:0.37) node {$\star$};
}
}
;
        \draw (0,0)--(0,-0.5);
        \node[anchor=west] at (0,-0.35) {\footnotesize$8$};
	{%
	\filldraw[fill=white,thick] (-1,-0.2) rectangle (1,0.2);
	\node at (0,0) {\Large$\JW{8}$};
}\end{tikzpicture}$$,where $\eta=1+\sqrt{3}i$, $a_0=\frac{\overline{\eta}\ip{(T,Q),u}-\eta\ip{(Q,T),u}}{\ip{u,u}}$, and $b_0=\frac{\overline{\eta}\ip{(T,Q),v}-\eta\ip{(Q,T),v}}{\ip{v,v}}$.  In fact, $a_0=i\sqrt{\frac{2}{3}(1+\sqrt{21})}$ and $b_0=\sqrt{\frac{1}{2}(17-3\sqrt{21})}+i\Big(\frac{-9+\sqrt{21}}{6}\Big)$.
\end{lem}
\begin{proof}
Let $W=\gamma_0 w -\gamma w_0 $. We can assume that $\overline{\gamma}=\gamma_0$, since $w^*=w_0$.  To show that $W$ lies in the span of $u$ and $v$, we use Bessel's inequality.
We look for $\gamma$ and $\gamma_0$ for which the following holds:
$$\ip{W,W}=\frac{\norm{\ip{W,u}}^2}{\ip{u,u}}+\frac{\norm{\ip{W,v}}^2}{\ip{v,v}}.$$
$$LHS = \norm{\gamma_0}^2\ip{w,w}-2Re(\gamma_0\overline{\gamma}\ip{w,w_0})+\norm{\gamma}^2\ip{w_0,w_0}$$
and
\begin{align*}
\noindent RHS &=   \frac{\norm{\gamma_0}^2\norm{\ip{w,u}}^2}{\ip{u,u}}-\frac{2Re(\gamma_0\overline{\gamma}\ip{w,u}\overline{\ip{w_0,u}})}{\ip{u,u}}+\frac{\norm{\gamma}^2\norm{\ip{w_0,u}}^2}{\ip{u,u}}\\
& +\frac{\norm{\gamma_0}^2\norm{\ip{w,v}}^2}{\ip{v,v}}-\frac{2Re(\gamma_0\overline{\gamma}\ip{w,v}\overline{\ip{w_0,v}})}{\ip{v,v}}+\frac{\norm{\gamma}^2\norm{\ip{w_0,v}}^2}{\ip{v,v}}
\end{align*}
From Lemmas 5.2.2,5.2.6,and 5.2.8, we can easily calculate that $$\ip{w_0,w_0}-\Big(\frac{\norm{\ip{w_0,u}}^2}{\ip{u,u}}+\frac{\norm{\ip{w_0,v}}^2}{\ip{v,v}}\Big)=\frac{9+19\sqrt{21}}{45}$$
and $$\ip{w,w}-\Big(\frac{\norm{\ip{w,u}}^2}{\ip{u,u}}+\frac{\norm{\ip{w,v}}^2}{\ip{v,v}}\Big)=\frac{9+19\sqrt{21}}{45}$$
So $LHS=RHS$ iff \\
$\norm{\gamma_0}^2\cdot \frac{9+19\sqrt{21}}{45}+\norm{\gamma}^2\cdot \frac{9+19\sqrt{21}}{45}-2Re(\gamma_0\overline{\gamma}\ip{w,w_0})
+\frac{2Re(\gamma_0\overline{\gamma}\ip{w,u}\overline{\ip{w_0,u}})}{\ip{u,u}}+\frac{2Re(\gamma_0\overline{\gamma}\ip{w,v}\overline{\ip{w_0,v}})}{\ip{v,v}}= 0\\
\Longleftrightarrow\\
\norm{\gamma_0}^2\cdot \frac{9+19\sqrt{21}}{45}+\norm{\gamma}^2\cdot \frac{9+19\sqrt{21}}{45}-2Re(\overline{\gamma_0}\gamma\ip{w,w_0})
+\frac{2Re(\overline{\gamma_0}\gamma\ip{u,w}\overline{\ip{u,w_0}})}{\ip{u,u}}+\frac{2Re(\overline{\gamma_0}\gamma\ip{v,w}\overline{\ip{v,w_0}})}{\ip{v,v}}= 0\\
\Longleftrightarrow\\
\norm{\gamma_0}^2\cdot \frac{9+19\sqrt{21}}{45}+\norm{\gamma}^2\cdot \frac{9+19\sqrt{21}}{45}+2Re(\overline{\gamma_0}\gamma(c+id))=0,
\mbox{where $c=\frac{9+19\sqrt{21}}{90}$ and $d=\frac{9\sqrt{3}+57\sqrt{7}}{90}$}\\
\Longleftrightarrow\\
2\norm{\gamma_0}^2\cdot\frac{9+19\sqrt{21}}{45}+2Re((\overline{\gamma_0})^2(c+id))=0\\
\Longleftrightarrow\\
\norm{\gamma_0}^2\cdot\frac{9+19\sqrt{21}}{45}+Re((Re(\gamma_0)-iIm(\gamma_0))^2(c+id))=0\\
\Longleftrightarrow\\
Re(\gamma_0)^2+Im(\gamma_0)^2+(Re(\gamma_0)^2-Im(\gamma_0)^2)\frac{1}{2}+2Re(\gamma_0)Im(\gamma_0)\frac{\sqrt{3}}{2}=0\\
\Longleftrightarrow\\
3Re(\gamma_0)^2+Im(\gamma_0)^2+2\sqrt{3}Re(\gamma_0)Im(\gamma_0)=0\\
\Longleftrightarrow\\
(\sqrt{3}Re(\gamma_0)+Im(\gamma_0))^2=0\\
\Longleftrightarrow\\
Im(\gamma_0)=-\sqrt{3}Re(\gamma_0)\\
\Longleftrightarrow\\
\gamma_0=Re(\gamma_0)(1-\sqrt{3}i).$\\
So $LHS=RHS$ if $\overline{\gamma}=\gamma_0$ and $\gamma_0=Re(\gamma_0)(1-\sqrt{3}i)=Re(\gamma_0)\overline{\eta}$, where $\eta = 1+\sqrt{3}i$.\\
To find $\gamma_0$, we impose $$\ip{w_0,w_0}=\frac{\norm{\ip{w_0,u}}^2}{\ip{u,u}}+\frac{\norm{\ip{w_0,v}}^2}{\ip{v,v}}+\norm{\gamma_0}^2\Norm{M_0}^2, \mbox{ where $\Norm{M_0}^2=[5]$}$$
This gives us that $Re(\gamma_0)=\pm\sqrt{-\frac{59}{120}+\frac{43}{120}\sqrt{\frac{7}{3}}}$.\\
If we pick a value for $Re(\gamma_0)$, then the equation we get involving $w$ and $w_0$ has a factor of $Re(\gamma_0)$ on both sides and hence, we can cancel out $Re(\gamma_0)$ to get $$\overline{\eta}(T,Q)-\eta(Q,T)=a_0u+b_0v$$,where $a_0=\frac{\overline{\eta}\ip{(T,Q),u}-\eta\ip{(Q,T),u}}{\ip{u,u}}$ and $b_0=\frac{\overline{\eta}\ip{(T,Q),v}-\eta\ip{(Q,T),v}}{\ip{v,v}}$.
Hence, we have the linear equation (Eqn.1):
$$\overline{\eta}\cdot \begin{tikzpicture}[baseline=0,scale=1.5]
	{%
	\draw (-0.5,1) -- (-0.5,0);
	\node[anchor=west] at (-0.5,0.5) {\footnotesize$4$};
	\draw (0.5,1) -- (0.5,0);
	\node[anchor=west] at (0.5,0.5) {\footnotesize$4$};
}
 {%
	\node[anchor=south] at (0,1) {\footnotesize$2$};
	\draw (-0.5,1) -- (0.5, 1);
	\foreach \x in {-0.5,0.5} {
		{%
	\filldraw[fill=white,thick] (\x,1) ellipse (3mm and 3mm);
	\node at (-.5,1) {\Large $T$};
	\node at (.5,1) {\Large $Q$};
	\path(\x,1) ++(90:0.37) node {$\star$};
}
	}
}
        \draw (0,0)--(0,-0.5);
        \node[anchor=west] at (0,-0.35) {\footnotesize$8$};
	{%
	\filldraw[fill=white,thick] (-1,-0.2) rectangle (1,0.2);
	\node at (0,0) {\Large$\JW{8}$};
}
\end{tikzpicture}-\eta \cdot \begin{tikzpicture}[baseline=0,scale=1.5]
	{%
	\draw (-0.5,1) -- (-0.5,0);
	\node[anchor=west] at (-0.5,0.5) {\footnotesize$4$};
	\draw (0.5,1) -- (0.5,0);
	\node[anchor=west] at (0.5,0.5) {\footnotesize$4$};
}
 {%
	\node[anchor=south] at (0,1) {\footnotesize$2$};
	\draw (-0.5,1) -- (0.5, 1);
	\foreach \x in {-0.5,0.5} {
		{%
	\filldraw[fill=white,thick] (\x,1) ellipse (3mm and 3mm);
	\node at (-.5,1) {\Large $Q$};
	\node at (.5,1) {\Large $T$};
	\path(\x,1) ++(90:0.37) node {$\star$};
}
	}
}
        \draw (0,0)--(0,-0.5);
        \node[anchor=west] at (0,-0.35) {\footnotesize$8$};
	{%
	\filldraw[fill=white,thick] (-1,-0.2) rectangle (1,0.2);
	\node at (0,0) {\Large$\JW{8}$};
}
\end{tikzpicture}=a_0\cdot \begin{tikzpicture}[baseline=0,scale=1.5]
	{
	\fill[shaded] (-0.8,0) -- (-0.8,0.6) arc (180:0:0.8) -- (0.8,0) -- (0.2,0) -- (0.2,1) -- (-0.2,1) -- (-0.2,0);
	\draw (-0.8,0) -- (-0.8,0.6) arc (180:0:0.8) -- (0.8,0);
	\draw (-0.2,0) -- (-0.2,1);
	\draw (0.2,0) -- (0.2,1);
	\node at (0.45,0.5) {\footnotesize$5$};
	{%
	\filldraw[fill=white,thick] (0,1) ellipse (3mm and 3mm);
	\node at (0,1) {\Large $T$};
	\path(0,1) ++(-90:0.37) node {$\star$};
}
}
;
        \draw (0,0)--(0,-0.5);
        \node[anchor=west] at (0,-0.35) {\footnotesize$8$};
	{%
	\filldraw[fill=white,thick] (-1,-0.2) rectangle (1,0.2);
	\node at (0,0) {\Large$\JW{8}$};
}\end{tikzpicture}+b_0 \cdot \begin{tikzpicture}[baseline=0,scale=1.5]
	{
	\fill[shaded] (-0.8,0) -- (-0.8,0.6) arc (180:0:0.8) -- (0.8,0) -- (0.2,0) -- (0.2,1) -- (-0.2,1) -- (-0.2,0);
	\draw (-0.8,0) -- (-0.8,0.6) arc (180:0:0.8) -- (0.8,0);
	\draw (-0.2,0) -- (-0.2,1);
	\draw (0.2,0) -- (0.2,1);
	\node at (0.45,0.5) {\footnotesize$5$};
	{%
	\filldraw[fill=white,thick] (0,1) ellipse (3mm and 3mm);
	\node at (0,1) {\Large $Q$};
	\path(0,1) ++(-90:0.37) node {$\star$};
}
}
;
        \draw (0,0)--(0,-0.5);
        \node[anchor=west] at (0,-0.35) {\footnotesize$8$};
	{%
	\filldraw[fill=white,thick] (-1,-0.2) rectangle (1,0.2);
	\node at (0,0) {\Large$\JW{8}$};
}\end{tikzpicture}$$
\end{proof}

We want to relate $(T,T)$ and $(Q,Q)$ in a similar fashion as we did for $(T,Q)$ and $(Q,T)$.  
Recall from Theorem 5.2.9 the unshaded 4-box relations:
$$\begin{tikzpicture}[baseline=0,scale=1.5]
	{%
	\draw (-0.5,1) -- (-0.5,0);
	\node[anchor=west] at (-0.5,0.5) {\footnotesize$4$};
	\draw (0.5,1) -- (0.5,0);
	\node[anchor=west] at (0.5,0.5) {\footnotesize$4$};
}
 {%
	\node[anchor=south] at (0,1) {\footnotesize$2$};
	\draw (-0.5,1) -- (0.5, 1);
	\foreach \x in {-0.5,0.5} {
		{%
	\filldraw[fill=white,thick] (\x,1) ellipse (3mm and 3mm);
	\node at (-.5,1) {\Large $T$};
	\node at (.5,1) {\Large $T$};
	\path(\x,1) ++(90:0.37) node {$\star$};
}
	}
}
        \draw (0,0)--(0,-0.5);
        \node[anchor=west] at (0,-0.35) {\footnotesize$8$};
	{%
	\filldraw[fill=white,thick] (-1,-0.2) rectangle (1,0.2);
	\node at (0,0) {\Large$\JW{8}$};
}
\end{tikzpicture} = \alpha \cdot \begin{tikzpicture}[baseline=0,scale=1.5]
	{
	\fill[shaded] (-0.8,0) -- (-0.8,0.6) arc (180:0:0.8) -- (0.8,0) -- (0.2,0) -- (0.2,1) -- (-0.2,1) -- (-0.2,0);
	\draw (-0.8,0) -- (-0.8,0.6) arc (180:0:0.8) -- (0.8,0);
	\draw (-0.2,0) -- (-0.2,1);
	\draw (0.2,0) -- (0.2,1);
	\node at (0.45,0.5) {\footnotesize$5$};
	{%
	\filldraw[fill=white,thick] (0,1) ellipse (3mm and 3mm);
	\node at (0,1) {\Large $T$};
	\path(0,1) ++(-90:0.37) node {$\star$};
}
}
;
        \draw (0,0)--(0,-0.5);
        \node[anchor=west] at (0,-0.35) {\footnotesize$8$};
	{%
	\filldraw[fill=white,thick] (-1,-0.2) rectangle (1,0.2);
	\node at (0,0) {\Large$\JW{8}$};
}\end{tikzpicture} +\beta \cdot \begin{tikzpicture}[baseline=0,scale=1.5]
	{
	\fill[shaded] (-0.8,0) -- (-0.8,0.6) arc (180:0:0.8) -- (0.8,0) -- (0.2,0) -- (0.2,1) -- (-0.2,1) -- (-0.2,0);
	\draw (-0.8,0) -- (-0.8,0.6) arc (180:0:0.8) -- (0.8,0);
	\draw (-0.2,0) -- (-0.2,1);
	\draw (0.2,0) -- (0.2,1);
	\node at (0.45,0.5) {\footnotesize$5$};
	{%
	\filldraw[fill=white,thick] (0,1) ellipse (3mm and 3mm);
	\node at (0,1) {\Large $Q$};
	\path(0,1) ++(-90:0.37) node {$\star$};
}
}
;
        \draw (0,0)--(0,-0.5);
        \node[anchor=west] at (0,-0.35) {\footnotesize$8$};
	{%
	\filldraw[fill=white,thick] (-1,-0.2) rectangle (1,0.2);
	\node at (0,0) {\Large$\JW{8}$};
}\end{tikzpicture}+ h\cdot\begin{tikzpicture}[baseline=0,scale=1.5]
	{\draw[shaded] (0.2,0) -- (0.2,1) -- (-0.2,1) -- (-0.2,0);
	
	\draw(-0.2,0) -- (-0.2,1);
	\draw (0.2,0) -- (0.2,1);
	\node at (0.45,0.5) {\footnotesize$7$};
	{%
	\filldraw[fill=white,thick] (0,1) ellipse (3mm and 3mm);
	\node at (0,1) {\Large $M_0$};
	\path(0,1) ++(-180:0.37) node {$\star$};
}
}
;
        \draw (0,0)--(0,-0.5);
        \node[anchor=west] at (0,-0.35) {\footnotesize$8$};
	{%
	\filldraw[fill=white,thick] (-1,-0.2) rectangle (1,0.2);
	\node at (0,0) {\Large$\JW{8}$};
}\end{tikzpicture} $$, 
and 
$$\begin{tikzpicture}[baseline=0,scale=1.5]
	{%
	\draw (-0.5,1) -- (-0.5,0);
	\node[anchor=west] at (-0.5,0.5) {\footnotesize$4$};
	\draw (0.5,1) -- (0.5,0);
	\node[anchor=west] at (0.5,0.5) {\footnotesize$4$};
}
 {%
	\node[anchor=south] at (0,1) {\footnotesize$2$};
	\draw (-0.5,1) -- (0.5, 1);
	\foreach \x in {-0.5,0.5} {
		{%
	\filldraw[fill=white,thick] (\x,1) ellipse (3mm and 3mm);
	\node at (-.5,1) {\Large $Q$};
	\node at (.5,1) {\Large $Q$};
	\path(\x,1) ++(90:0.37) node {$\star$};
}
	}
}
        \draw (0,0)--(0,-0.5);
        \node[anchor=west] at (0,-0.35) {\footnotesize$8$};
	{%
	\filldraw[fill=white,thick] (-1,-0.2) rectangle (1,0.2);
	\node at (0,0) {\Large$\JW{8}$};
}
\end{tikzpicture} = \alpha_0 \cdot \begin{tikzpicture}[baseline=0,scale=1.5]
	{
	\fill[shaded] (-0.8,0) -- (-0.8,0.6) arc (180:0:0.8) -- (0.8,0) -- (0.2,0) -- (0.2,1) -- (-0.2,1) -- (-0.2,0);
	\draw (-0.8,0) -- (-0.8,0.6) arc (180:0:0.8) -- (0.8,0);
	\draw (-0.2,0) -- (-0.2,1);
	\draw (0.2,0) -- (0.2,1);
	\node at (0.45,0.5) {\footnotesize$5$};
	{%
	\filldraw[fill=white,thick] (0,1) ellipse (3mm and 3mm);
	\node at (0,1) {\Large $T$};
	\path(0,1) ++(-90:0.37) node {$\star$};
}
}
;
        \draw (0,0)--(0,-0.5);
        \node[anchor=west] at (0,-0.35) {\footnotesize$8$};
	{%
	\filldraw[fill=white,thick] (-1,-0.2) rectangle (1,0.2);
	\node at (0,0) {\Large$\JW{8}$};
}\end{tikzpicture} +\beta_0 \cdot \begin{tikzpicture}[baseline=0,scale=1.5]
	{
	\fill[shaded] (-0.8,0) -- (-0.8,0.6) arc (180:0:0.8) -- (0.8,0) -- (0.2,0) -- (0.2,1) -- (-0.2,1) -- (-0.2,0);
	\draw (-0.8,0) -- (-0.8,0.6) arc (180:0:0.8) -- (0.8,0);
	\draw (-0.2,0) -- (-0.2,1);
	\draw (0.2,0) -- (0.2,1);
	\node at (0.45,0.5) {\footnotesize$5$};
	{%
	\filldraw[fill=white,thick] (0,1) ellipse (3mm and 3mm);
	\node at (0,1) {\Large $Q$};
	\path(0,1) ++(-90:0.37) node {$\star$};
}
}
;
        \draw (0,0)--(0,-0.5);
        \node[anchor=west] at (0,-0.35) {\footnotesize$8$};
	{%
	\filldraw[fill=white,thick] (-1,-0.2) rectangle (1,0.2);
	\node at (0,0) {\Large$\JW{8}$};
}\end{tikzpicture}+ h_0 \cdot \begin{tikzpicture}[baseline=0,scale=1.5]
	{\draw[shaded] (0.2,0) -- (0.2,1) -- (-0.2,1) -- (-0.2,0);
	
	\draw(-0.2,0) -- (-0.2,1);
	\draw (0.2,0) -- (0.2,1);
	\node at (0.45,0.5) {\footnotesize$7$};
	{%
	\filldraw[fill=white,thick] (0,1) ellipse (3mm and 3mm);
	\node at (0,1) {\Large $M_0$};
	\path(0,1) ++(-180:0.37) node {$\star$};
}
}
;
        \draw (0,0)--(0,-0.5);
        \node[anchor=west] at (0,-0.35) {\footnotesize$8$};
	{%
	\filldraw[fill=white,thick] (-1,-0.2) rectangle (1,0.2);
	\node at (0,0) {\Large$\JW{8}$};
}\end{tikzpicture} $$
for some $\alpha,\beta,h,\alpha_0,\beta_0, h_0 \in \C$.
\begin{remark}Again, we replaced $M$ in the original 4-box relation with a scalar multiple of the underlying generator $M_0$, which is self-adjoint.
\end{remark}
\begin{lem}[Eqn.2]The following equation holds:
$$\tau\cdot \begin{tikzpicture}[baseline=0,scale=1.5]
	{%
	\draw (-0.5,1) -- (-0.5,0);
	\node[anchor=west] at (-0.5,0.5) {\footnotesize$4$};
	\draw (0.5,1) -- (0.5,0);
	\node[anchor=west] at (0.5,0.5) {\footnotesize$4$};
}
 {%
	\node[anchor=south] at (0,1) {\footnotesize$2$};
	\draw (-0.5,1) -- (0.5, 1);
	\foreach \x in {-0.5,0.5} {
		{%
	\filldraw[fill=white,thick] (\x,1) ellipse (3mm and 3mm);
	\node at (-.5,1) {\Large $T$};
	\node at (.5,1) {\Large $T$};
	\path(\x,1) ++(90:0.37) node {$\star$};
}
	}
}
        \draw (0,0)--(0,-0.5);
        \node[anchor=west] at (0,-0.35) {\footnotesize$8$};
	{%
	\filldraw[fill=white,thick] (-1,-0.2) rectangle (1,0.2);
	\node at (0,0) {\Large$\JW{8}$};
}
\end{tikzpicture}- \begin{tikzpicture}[baseline=0,scale=1.5]
	{%
	\draw (-0.5,1) -- (-0.5,0);
	\node[anchor=west] at (-0.5,0.5) {\footnotesize$4$};
	\draw (0.5,1) -- (0.5,0);
	\node[anchor=west] at (0.5,0.5) {\footnotesize$4$};
}
 {%
	\node[anchor=south] at (0,1) {\footnotesize$2$};
	\draw (-0.5,1) -- (0.5, 1);
	\foreach \x in {-0.5,0.5} {
		{%
	\filldraw[fill=white,thick] (\x,1) ellipse (3mm and 3mm);
	\node at (-.5,1) {\Large $Q$};
	\node at (.5,1) {\Large $Q$};
	\path(\x,1) ++(90:0.37) node {$\star$};
}
	}
}
        \draw (0,0)--(0,-0.5);
        \node[anchor=west] at (0,-0.35) {\footnotesize$8$};
	{%
	\filldraw[fill=white,thick] (-1,-0.2) rectangle (1,0.2);
	\node at (0,0) {\Large$\JW{8}$};
}
\end{tikzpicture}=a'\cdot \begin{tikzpicture}[baseline=0,scale=1.5]
	{
	\fill[shaded] (-0.8,0) -- (-0.8,0.6) arc (180:0:0.8) -- (0.8,0) -- (0.2,0) -- (0.2,1) -- (-0.2,1) -- (-0.2,0);
	\draw (-0.8,0) -- (-0.8,0.6) arc (180:0:0.8) -- (0.8,0);
	\draw (-0.2,0) -- (-0.2,1);
	\draw (0.2,0) -- (0.2,1);
	\node at (0.45,0.5) {\footnotesize$5$};
	{%
	\filldraw[fill=white,thick] (0,1) ellipse (3mm and 3mm);
	\node at (0,1) {\Large $T$};
	\path(0,1) ++(-90:0.37) node {$\star$};
}
}
;
        \draw (0,0)--(0,-0.5);
        \node[anchor=west] at (0,-0.35) {\footnotesize$8$};
	{%
	\filldraw[fill=white,thick] (-1,-0.2) rectangle (1,0.2);
	\node at (0,0) {\Large$\JW{8}$};
}\end{tikzpicture}+b' \cdot \begin{tikzpicture}[baseline=0,scale=1.5]
	{
	\fill[shaded] (-0.8,0) -- (-0.8,0.6) arc (180:0:0.8) -- (0.8,0) -- (0.2,0) -- (0.2,1) -- (-0.2,1) -- (-0.2,0);
	\draw (-0.8,0) -- (-0.8,0.6) arc (180:0:0.8) -- (0.8,0);
	\draw (-0.2,0) -- (-0.2,1);
	\draw (0.2,0) -- (0.2,1);
	\node at (0.45,0.5) {\footnotesize$5$};
	{%
	\filldraw[fill=white,thick] (0,1) ellipse (3mm and 3mm);
	\node at (0,1) {\Large $Q$};
	\path(0,1) ++(-90:0.37) node {$\star$};
}
}
;
        \draw (0,0)--(0,-0.5);
        \node[anchor=west] at (0,-0.35) {\footnotesize$8$};
	{%
	\filldraw[fill=white,thick] (-1,-0.2) rectangle (1,0.2);
	\node at (0,0) {\Large$\JW{8}$};
}\end{tikzpicture}$$,where $\tau=\frac{4+\sqrt{21}}{5}$, $a'=\frac{\tau\ip{(T,T),u}-\ip{(Q,Q),u}}{\ip{u,u}}$, and $b'=\frac{\tau\ip{(T,T),v}-\ip{(Q,Q),v}}{\ip{v,v}}$.  In fact, $a'=-\sqrt{\frac{5}{18}+\frac{\sqrt{\frac{7}{3}}}{6}}$ and $b'=-\frac{1}{6}\sqrt{8+3\sqrt{21}}+i\Big(-\frac{1}{2}\sqrt{\frac{8}{3}+\sqrt{21}}\Big)$.
\end{lem}
\begin{proof}Let $w=(T,T)$ and $w_0=(Q,Q)$.  Let $W=h_0w-hw_0$.  We can assume $\overline{h}=h$ and $\overline{h_0}=h_0$, since $w^*=w$, $w_0^*=w_0$.  To show that $W$ lies in the span of $u$ and $v$, we use Bessel's inequality.
We look for $h$ and $h_0$ for which the following holds:
$$\ip{W,W}=\frac{\norm{\ip{W,u}}^2}{\ip{u,u}}+\frac{\norm{\ip{W,v}}^2}{\ip{v,v}}.$$
This holds iff\\
$ \norm{h_0}^2\Big(\ip{w,w}-\Big(\frac{\norm{\ip{w,u}}^2}{\ip{u,u}}+\frac{\norm{\ip{w,v}}^2}{\ip{v,v}}\Big)\Big)+\norm{h}^2\Big(\ip{w_0,w_0}-\Big(\frac{\norm{\ip{w_0,u}}^2}{\ip{u,u}}+\frac{\norm{\ip{w_0,v}}^2}{\ip{v,v}}\Big)\Big)\\
-2Re(h_0\overline{h}\ip{w,w_0})+\frac{2Re(h_0\overline{h}\ip{w,u}\overline{\ip{w_0,u}})}{\ip{u,u}}+\frac{2Re(h_0\overline{h}\ip{w,v}\overline{\ip{w_0,v}})}{\ip{v,v}}=0$

We can calculate that $$\Big(\ip{w,w}-\Big(\frac{\norm{\ip{w,u}}^2}{\ip{u,u}}+\frac{\norm{\ip{w,v}}^2}{\ip{v,v}}\Big)\Big)=\frac{1}{180}(363-67\sqrt{21})$$ and $$\Big(\ip{w_0,w_0}-\Big(\frac{\norm{\ip{w_0,u}}^2}{\ip{u,u}}+\frac{\norm{\ip{w_0,v}}^2}{\ip{v,v}}\Big)\Big)=\frac{1}{180}(87+17\sqrt{21})$$

Thus our equation holds iff\\
$\norm{h_0}^2\Big(\frac{1}{180}(363-67\sqrt{21})\Big)+\norm{h}^2\Big(\frac{1}{180}(87+17\sqrt{21})\Big)-2Re(h_0h\ip{w,w_0})\\
+\frac{2Re(h_0h\overline{\ip{u,w}}\ip{u,w_0}}{\ip{u,u}}+\frac{2Re(h_0h\overline{\ip{v,w}}\ip{v,w_0})}{\ip{v,v}}=0\\
\Longleftrightarrow
h_0^2\Big(\frac{1}{180}(363-67\sqrt{21})\Big)+h^2\Big(\frac{1}{180}(87+17\sqrt{21})\Big)+2Re(h_0h\Big(\frac{-9-19\sqrt{21}}{180}\Big))=0\\
\Longleftrightarrow
h_0=\frac{4+\sqrt{21}}{5}h$\\
To find $h_0$, we impose $$\ip{w_0,w_0}=\frac{\norm{\ip{w_0,u}}^2}{\ip{u,u}}+\frac{\norm{\ip{w_0,v}}^2}{\ip{v,v}}+\norm{h_0}^2\Norm{M_0}^2, \mbox{ where $\Norm{M_0}^2=[5]$}$$\\
This gives us that $h_0=\pm\frac{1}{6}\sqrt{\frac{1}{10}(39-\sqrt{21})}$.  It follows that $h=\pm\frac{\sqrt{\frac{5}{2}(39-\sqrt{21})}}{6(4+\sqrt{21})}$.\\
If we pick a value for $h_0$ then the equation we get involving $w$ and $w_0$ has a common factor of $h$ on both sides. Cancel out $h$ to get $$\tau(T,T)-(Q,Q)=a'u+b'v,$$ where $\tau=\frac{4+\sqrt{21}}{5}$, $a'=\frac{\tau\ip{(T,T),u}-\ip{(Q,Q),u}}{\ip{u,u}}$, and $b'=\frac{\tau\ip{(T,T),v}-\ip{(Q,Q),v}}{\ip{v,v}}$.  
\end{proof}
Now, we easily deduce the one-strand braiding substitute.
\begin{prop}[Unshaded One-strand Braiding Substitutes] $$u,v \in span \{(T,T),(Q,Q),(T,Q),(Q,T)\}$$
\end{prop}
\begin{proof}Consider Eqn.1 and Eqn.2 from Lemmas 6.1.2 and 6.1.4.  This is a system of two equations and two unknowns where $u$ and $v$ are treated as the unknowns.  We can solve this system for $u$ and $v$, since the determinant of this system is $$a_0b'-a'b_0=\frac{1}{3}(6+\sqrt{21}-i\sqrt{19+4\sqrt{21}})\neq 0.$$
 We get expressions for $u$ and $v$ as linear combinations of the diagrams $(T,T),(Q,Q),(T,Q)$, and $(Q,T)$.  These expressions will be our unshaded one-strand braiding substitutes.  
\end{proof}
\section{Shaded One-strand Braiding substitutes}
Recall from Theorem 5.2.16 the shaded 4-box relations:
$$\begin{tikzpicture}[baseline=0,scale=1.5]
	{%
	\fill[shaded] (-0.5,0) rectangle (0.5,1);
	\draw (-0.5,1) -- (-0.5,0);
	\node[anchor=west] at (-0.5,0.5) {\footnotesize$3$};
	\draw (-0.6,1) -- (-0.6,0);
	\draw (0.5,1) -- (0.5,0);
	\node[anchor=west] at (0.3,0.5) {\footnotesize$3$};
	\draw (0.6,1) -- (0.6,0);
}
 {%
	
	\draw (-0.5,1) -- (0.5, 1);
	\draw (-0.5,1.1) -- (0.5,1.1);
	\foreach \x in {-0.5,0.5} {
		{%
	\filldraw[fill=white,thick] (\x,1) ellipse (3mm and 3mm);
	\node at (-.5,1) {\Large $T$};
	\node at (.5,1) {\Large $Q$};
	\path(-.55,1) ++(-90:0.37) node {$\star$};
	\path(0.55,1) ++(-90:0.37) node {$\star$};
}
	}
}
        \draw (0,0)--(0,-0.5);
        \node[anchor=west] at (0,-0.35) {\footnotesize$8$};
	{%
	\filldraw[fill=white,thick] (-1,-0.2) rectangle (1,0.2);
	\node at (0,0) {\Large$\JW{8}$};
}
\end{tikzpicture} = \alpha \cdot \begin{tikzpicture}[baseline=0,scale=1.5]
	{\draw[shaded] (0.2,0) -- (0.2,1) -- (-0.2,1) -- (-0.2,0);
	\draw (-0.8,0) -- (-0.8,0.6) arc (180:0:0.8) -- (0.8,0);
	\draw(-0.2,0) -- (-0.2,1);
	\draw (0.2,0) -- (0.2,1);
	\node at (0.45,0.5) {\footnotesize$5$};
	{%
	\filldraw[fill=white,thick] (0,1) ellipse (3mm and 3mm);
	\node at (0,1) {\Large $T$};
	\path(0,1) ++(-180:0.37) node {$\star$};
}
}
;
        \draw (0,0)--(0,-0.5);
        \node[anchor=west] at (0,-0.35) {\footnotesize$8$};
	{%
	\filldraw[fill=white,thick] (-1,-0.2) rectangle (1,0.2);
	\node at (0,0) {\Large$\JW{8}$};
}\end{tikzpicture} +\beta \cdot \begin{tikzpicture}[baseline=0,scale=1.5]
	{\draw[shaded] (0.2,0) -- (0.2,1) -- (-0.2,1) -- (-0.2,0);
	\draw (-0.8,0) -- (-0.8,0.6) arc (180:0:0.8) -- (0.8,0);
	\draw(-0.2,0) -- (-0.2,1);
	\draw (0.2,0) -- (0.2,1);
	\node at (0.45,0.5) {\footnotesize$5$};
	{%
	\filldraw[fill=white,thick] (0,1) ellipse (3mm and 3mm);
	\node at (0,1) {\Large $Q$};
	\path(0,1) ++(-180:0.37) node {$\star$};
}
}
;
        \draw (0,0)--(0,-0.5);
        \node[anchor=west] at (0,-0.35) {\footnotesize$8$};
	{%
	\filldraw[fill=white,thick] (-1,-0.2) rectangle (1,0.2);
	\node at (0,0) {\Large$\JW{8}$};
}\end{tikzpicture}+ \gamma \begin{tikzpicture}[baseline=0,scale=1.5]
	{

	\draw (-0.2,0) -- (-0.2,1);
	\draw (0.2,0) -- (0.2,1);
	\node at (0.45,0.5) {\footnotesize$7$};
	{%
	\filldraw[fill=white,thick] (0,1) ellipse (3mm and 3mm);
	\node at (0,1) {\Large $M'_0$};
	\path(0,1) ++(-180:0.37) node {$\star$};
}
}
;
        \draw (0,0)--(0,-0.5);
        \node[anchor=west] at (0,-0.35) {\footnotesize$8$};
	{%
	\filldraw[fill=white,thick] (-1,-0.2) rectangle (1,0.2);
	\node at (0,0) {\Large$\JW{8}$};
}\end{tikzpicture} $$, 
and 
$$\begin{tikzpicture}[baseline=0,scale=1.5]
	{%
	\fill[shaded] (-0.5,0) rectangle (0.5,1);
	\draw (-0.5,1) -- (-0.5,0);
	\node[anchor=west] at (-0.5,0.5) {\footnotesize$3$};
	\draw (-0.6,1) -- (-0.6,0);
	\draw (0.5,1) -- (0.5,0);
	\node[anchor=west] at (0.3,0.5) {\footnotesize$3$};
	\draw (0.6,1) -- (0.6,0);
}
 {%
	
	\draw (-0.5,1) -- (0.5, 1);
	\draw (-0.5,1.1) -- (0.5,1.1);
	\foreach \x in {-0.5,0.5} {
		{%
	\filldraw[fill=white,thick] (\x,1) ellipse (3mm and 3mm);
	\node at (-.5,1) {\Large $Q$};
	\node at (.5,1) {\Large $T$};
	\path(-.55,1) ++(-90:0.37) node {$\star$};
	\path(0.55,1) ++(-90:0.37) node {$\star$};
}
	}
}
        \draw (0,0)--(0,-0.5);
        \node[anchor=west] at (0,-0.35) {\footnotesize$8$};
	{%
	\filldraw[fill=white,thick] (-1,-0.2) rectangle (1,0.2);
	\node at (0,0) {\Large$\JW{8}$};
}
\end{tikzpicture} = \alpha_0 \cdot \begin{tikzpicture}[baseline=0,scale=1.5]
	{\draw[shaded] (0.2,0) -- (0.2,1) -- (-0.2,1) -- (-0.2,0);
	\draw (-0.8,0) -- (-0.8,0.6) arc (180:0:0.8) -- (0.8,0);
	\draw(-0.2,0) -- (-0.2,1);
	\draw (0.2,0) -- (0.2,1);
	\node at (0.45,0.5) {\footnotesize$5$};
	{%
	\filldraw[fill=white,thick] (0,1) ellipse (3mm and 3mm);
	\node at (0,1) {\Large $T$};
	\path(0,1) ++(-180:0.37) node {$\star$};
}
}
;
        \draw (0,0)--(0,-0.5);
        \node[anchor=west] at (0,-0.35) {\footnotesize$8$};
	{%
	\filldraw[fill=white,thick] (-1,-0.2) rectangle (1,0.2);
	\node at (0,0) {\Large$\JW{8}$};
}\end{tikzpicture} +\beta_0 \cdot \begin{tikzpicture}[baseline=0,scale=1.5]
	{\draw[shaded] (0.2,0) -- (0.2,1) -- (-0.2,1) -- (-0.2,0);
	\draw (-0.8,0) -- (-0.8,0.6) arc (180:0:0.8) -- (0.8,0);
	\draw(-0.2,0) -- (-0.2,1);
	\draw (0.2,0) -- (0.2,1);
	\node at (0.45,0.5) {\footnotesize$5$};
	{%
	\filldraw[fill=white,thick] (0,1) ellipse (3mm and 3mm);
	\node at (0,1) {\Large $Q$};
	\path(0,1) ++(-180:0.37) node {$\star$};
}
}
;
        \draw (0,0)--(0,-0.5);
        \node[anchor=west] at (0,-0.35) {\footnotesize$8$};
	{%
	\filldraw[fill=white,thick] (-1,-0.2) rectangle (1,0.2);
	\node at (0,0) {\Large$\JW{8}$};
}\end{tikzpicture}+ \gamma_0 \begin{tikzpicture}[baseline=0,scale=1.5]
	{

	\draw (-0.2,0) -- (-0.2,1);
	\draw (0.2,0) -- (0.2,1);
	\node at (0.45,0.5) {\footnotesize$7$};
	{%
	\filldraw[fill=white,thick] (0,1) ellipse (3mm and 3mm);
	\node at (0,1) {\Large $M'_0$};
	\path(0,1) ++(-180:0.37) node {$\star$};
}
}
;
        \draw (0,0)--(0,-0.5);
        \node[anchor=west] at (0,-0.35) {\footnotesize$8$};
	{%
	\filldraw[fill=white,thick] (-1,-0.2) rectangle (1,0.2);
	\node at (0,0) {\Large$\JW{8}$};
}\end{tikzpicture} $$
for some $\alpha,\beta,\gamma,\alpha_0,\beta_0,\gamma_0 \in \C$.  Notice we replaced $M'$ by the underlying self-adjoint generator $M'_0$.\\
\indent Let $w=(T,Q)$ and $w_0=(Q,T)$.  Let $W=\gamma_0w-\gamma w_0$.
\begin{lem}[Eqn.A]The following equation holds:
$$\eta \cdot \begin{tikzpicture}[baseline=0,scale=1.5]
	{%
	\fill[shaded] (-0.5,0) rectangle (0.5,1);
	\draw (-0.5,1) -- (-0.5,0);
	\node[anchor=west] at (-0.5,0.5) {\footnotesize$3$};
	\draw (-0.6,1) -- (-0.6,0);
	\draw (0.5,1) -- (0.5,0);
	\node[anchor=west] at (0.3,0.5) {\footnotesize$3$};
	\draw (0.6,1) -- (0.6,0);
}
 {%
	
	\draw (-0.5,1) -- (0.5, 1);
	\draw (-0.5,1.1) -- (0.5,1.1);
	\foreach \x in {-0.5,0.5} {
		{%
	\filldraw[fill=white,thick] (\x,1) ellipse (3mm and 3mm);
	\node at (-.5,1) {\Large $T$};
	\node at (.5,1) {\Large $Q$};
	\path(-.55,1) ++(-90:0.37) node {$\star$};
	\path(0.55,1) ++(-90:0.37) node {$\star$};
}
	}
}
        \draw (0,0)--(0,-0.5);
        \node[anchor=west] at (0,-0.35) {\footnotesize$8$};
	{%
	\filldraw[fill=white,thick] (-1,-0.2) rectangle (1,0.2);
	\node at (0,0) {\Large$\JW{8}$};
}
\end{tikzpicture}-\overline{\eta} \cdot \begin{tikzpicture}[baseline=0,scale=1.5]
	{%
	\fill[shaded] (-0.5,0) rectangle (0.5,1);
	\draw (-0.5,1) -- (-0.5,0);
	\node[anchor=west] at (-0.5,0.5) {\footnotesize$3$};
	\draw (-0.6,1) -- (-0.6,0);
	\draw (0.5,1) -- (0.5,0);
	\node[anchor=west] at (0.3,0.5) {\footnotesize$3$};
	\draw (0.6,1) -- (0.6,0);
}
 {%
	
	\draw (-0.5,1) -- (0.5, 1);
	\draw (-0.5,1.1) -- (0.5,1.1);
	\foreach \x in {-0.5,0.5} {
		{%
	\filldraw[fill=white,thick] (\x,1) ellipse (3mm and 3mm);
	\node at (-.5,1) {\Large $Q$};
	\node at (.5,1) {\Large $T$};
	\path(-.55,1) ++(-90:0.37) node {$\star$};
	\path(0.55,1) ++(-90:0.37) node {$\star$};
}
	}
}
        \draw (0,0)--(0,-0.5);
        \node[anchor=west] at (0,-0.35) {\footnotesize$8$};
	{%
	\filldraw[fill=white,thick] (-1,-0.2) rectangle (1,0.2);
	\node at (0,0) {\Large$\JW{8}$};
}
\end{tikzpicture}=a_0\cdot \begin{tikzpicture}[baseline=0,scale=1.5]
	{\draw[shaded] (0.2,0) -- (0.2,1) -- (-0.2,1) -- (-0.2,0);
	\draw (-0.8,0) -- (-0.8,0.6) arc (180:0:0.8) -- (0.8,0);
	\draw(-0.2,0) -- (-0.2,1);
	\draw (0.2,0) -- (0.2,1);
	\node at (0.45,0.5) {\footnotesize$5$};
	{%
	\filldraw[fill=white,thick] (0,1) ellipse (3mm and 3mm);
	\node at (0,1) {\Large $T$};
	\path(0,1) ++(-180:0.37) node {$\star$};
}
}
;
        \draw (0,0)--(0,-0.5);
        \node[anchor=west] at (0,-0.35) {\footnotesize$8$};
	{%
	\filldraw[fill=white,thick] (-1,-0.2) rectangle (1,0.2);
	\node at (0,0) {\Large$\JW{8}$};
}\end{tikzpicture}+b_0 \cdot \begin{tikzpicture}[baseline=0,scale=1.5]
	{\draw[shaded] (0.2,0) -- (0.2,1) -- (-0.2,1) -- (-0.2,0);
	\draw (-0.8,0) -- (-0.8,0.6) arc (180:0:0.8) -- (0.8,0);
	\draw(-0.2,0) -- (-0.2,1);
	\draw (0.2,0) -- (0.2,1);
	\node at (0.45,0.5) {\footnotesize$5$};
	{%
	\filldraw[fill=white,thick] (0,1) ellipse (3mm and 3mm);
	\node at (0,1) {\Large $Q$};
	\path(0,1) ++(-180:0.37) node {$\star$};
}
}
;
        \draw (0,0)--(0,-0.5);
        \node[anchor=west] at (0,-0.35) {\footnotesize$8$};
	{%
	\filldraw[fill=white,thick] (-1,-0.2) rectangle (1,0.2);
	\node at (0,0) {\Large$\JW{8}$};
}\end{tikzpicture}$$,where $\eta=1+\sqrt{3}i$, $a_0=\frac{\eta \ip{(T,Q),u_0}-\overline{\eta}\ip{(Q,T),u_0}}{\ip{u_0,u_0}}$, and $b_0=\frac{\eta \ip{(T,Q),v_0}-\overline{\eta}\ip{(Q,T),v_0}}{\ip{v_0,v_0}}$.  In fact, $a_0=i\sqrt{\frac{2}{3}(1+\sqrt{21})}$ and $b_0=i\sqrt{\frac{2}{3}(17-3\sqrt{21})}$.
\end{lem}
\begin{proof}We can assume $\overline{\gamma}=\gamma_0$.  To show that $W$ lies in the span of $u_0$ and $v_0$, we use Bessel's inequality.
We look for $\gamma$ and $\gamma_0$ for which the following holds:
$$\ip{W,W}=\frac{\norm{\ip{W,u_0}}^2}{\ip{u_0,u_0}}+\frac{\norm{\ip{W,v_0}}^2}{\ip{v_0,v_0}}.$$

We can calculate that $$\Big(\ip{w,w}-\Big(\frac{\norm{\ip{w,u_0}}^2}{\ip{u_0,u_0}}+\frac{\norm{\ip{w,v_0}}^2}{\ip{v_0,v_0}}\Big)\Big)=\frac{9+19\sqrt{21}}{45}$$ and $$\Big(\ip{w_0,w_0}-\Big(\frac{\norm{\ip{w_0,u_0}}^2}{\ip{u_0,u_0}}+\frac{\norm{\ip{w_0,v_0}}^2}{\ip{v_0,v_0}}\Big)\Big)=\frac{9+19\sqrt{21}}{45}$$

Thus our equation holds iff \\
$\norm{\gamma_0}^2\cdot \frac{9+19\sqrt{21}}{45}+\norm{\gamma}^2\cdot \frac{9+19\sqrt{21}}{45}-2Re(\gamma_0\overline{\gamma}\ip{w,w_0})
+\frac{2Re(\gamma_0\overline{\gamma}\ip{w,u_0}\overline{\ip{w_0,u_0}})}{\ip{u_0,u_0}}+\frac{2Re(\gamma_0\overline{\gamma}\ip{w,v_0}\overline{\ip{w_0,v_0}})}{\ip{v_0,v_0}}= 0\\
\Longleftrightarrow\\
\norm{\gamma_0}^2\cdot \frac{9+19\sqrt{21}}{45}+\norm{\gamma}^2\cdot \frac{9+19\sqrt{21}}{45}-2Re(\overline{\gamma_0}\gamma\ip{w_0,w})
+\frac{2Re(\overline{\gamma_0}\gamma\ip{u_0,w}\overline{\ip{u_0,w_0}})}{\ip{u_0,u_0}}+\frac{2Re(\overline{\gamma_0}\gamma\ip{v_0,w}\overline{\ip{v_0,w_0}})}{\ip{v_0,v_0}}= 0\\
\Longleftrightarrow\\
\norm{\gamma_0}^2\cdot \frac{9+19\sqrt{21}}{45}+\norm{\gamma}^2\cdot \frac{9+19\sqrt{21}}{45}+2Re(\overline{\gamma_0}\gamma(c-id))=0,
\mbox{where $c=\frac{9+19\sqrt{21}}{90}$ and $d=\frac{9\sqrt{3}+57\sqrt{7}}{90}$}\\
\Longleftrightarrow\\
2\norm{\gamma_0}^2\cdot\frac{9+19\sqrt{21}}{45}+2Re((\overline{\gamma_0})^2(c-id))=0\\
\Longleftrightarrow\\
\norm{\gamma_0}^2\cdot\frac{9+19\sqrt{21}}{45}+Re((Re(\gamma_0)-iIm(\gamma_0))^2(c-id))=0\\
\Longleftrightarrow\\
Re(\gamma_0)^2+Im(\gamma_0)^2+(Re(\gamma_0)^2-Im(\gamma_0)^2)\frac{1}{2}-2Re(\gamma_0)Im(\gamma_0)\frac{\sqrt{3}}{2}=0\\
\Longleftrightarrow\\
3Re(\gamma_0)^2+Im(\gamma_0)^2-2\sqrt{3}Re(\gamma_0)Im(\gamma_0)=0\\
\Longleftrightarrow\\
(\sqrt{3}Re(\gamma_0)-Im(\gamma_0))^2=0\\
\Longleftrightarrow\\
Im(\gamma_0)=\sqrt{3}Re(\gamma_0)\\
\Longleftrightarrow\\
\gamma_0=Re(\gamma_0)(1+\sqrt{3}i).$\\
So we have equality if $\overline{\gamma}=\gamma_0$ and $\gamma_0=Re(\gamma_0)(1+\sqrt{3}i)=Re(\gamma_0)\eta$, where $\eta = 1+\sqrt{3}i$.\\
To find $\gamma_0$, we impose $$\ip{w_0,w_0}=\frac{\norm{\ip{w_0,u_0}}^2}{\ip{u_0,u_0}}+\frac{\norm{\ip{w_0,v_0}}^2}{\ip{v_0,v_0}}+\norm{\gamma_0}^2\Norm{M'_0}^2, \mbox{ where $\Norm{M'_0}^2=[5]$}$$
This gives us that $Re(\gamma_0)=\pm\sqrt{-\frac{59}{120}+\frac{43}{120}\sqrt{\frac{7}{3}}}$.\\
If we pick a value for $Re(\gamma_0)$, then the equation we get involving $w$ and $w_0$ has a factor of $Re(\gamma_0)$ on both sides and hence, we can cancel out $Re(\gamma_0)$ to get $$\eta(T,Q)-\overline{\eta}(Q,T)=a_0u_0+b_0v_0$$,where $a_0=\frac{\eta\ip{(T,Q),u_0}-\overline{\eta}\ip{(Q,T),u_0}}{\ip{u_0,u_0}}$ and $b_0=\frac{\eta\ip{(T,Q),v_0}-\overline{\eta}\ip{(Q,T),v_0}}{\ip{v_0,v_0}}$.  Hence, we have the desired equation.
\end{proof}
Recall from Lemma 5.2.16 the shaded 4-box relations:
$$
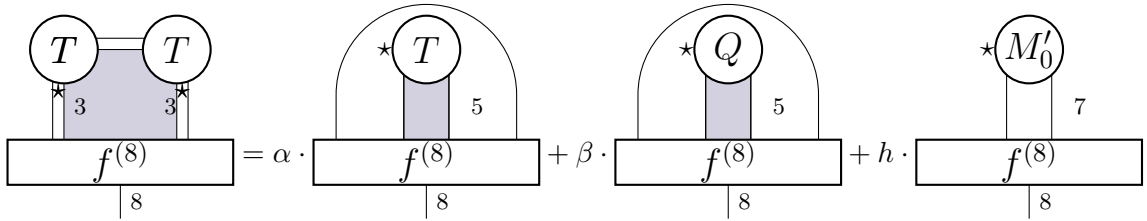
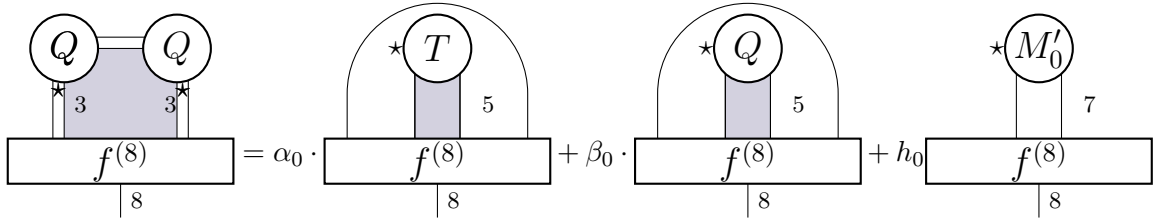
\begin{tikzpicture}[baseline=0,scale=1.5]
	{%
	\fill[shaded] (-0.5,0) rectangle (0.5,1);
	\draw (-0.5,1) -- (-0.5,0);
	\node[anchor=west] at (-0.5,0.5) {\footnotesize$3$};
	\draw (-0.6,1) -- (-0.6,0);
	\draw (0.5,1) -- (0.5,0);
	\node[anchor=west] at (0.3,0.5) {\footnotesize$3$};
	\draw (0.6,1) -- (0.6,0);
}
 {%
	
	\draw (-0.5,1) -- (0.5, 1);
	\draw (-0.5,1.1) -- (0.5,1.1);
	\foreach \x in {-0.5,0.5} {
		{%
	\filldraw[fill=white,thick] (\x,1) ellipse (3mm and 3mm);
	\node at (-.5,1) {\Large $T$};
	\node at (.5,1) {\Large $T$};
	\path(-.55,1) ++(-90:0.37) node {$\star$};
	\path(0.55,1) ++(-90:0.37) node {$\star$};
}
	}
}
        \draw (0,0)--(0,-0.5);
        \node[anchor=west] at (0,-0.35) {\footnotesize$8$};
	{%
	\filldraw[fill=white,thick] (-1,-0.2) rectangle (1,0.2);
	\node at (0,0) {\Large$\JW{8}$};
}
\end{tikzpicture} = \alpha \cdot \begin{tikzpicture}[baseline=0,scale=1.5]
	{\draw[shaded] (0.2,0) -- (0.2,1) -- (-0.2,1) -- (-0.2,0);
	\draw (-0.8,0) -- (-0.8,0.6) arc (180:0:0.8) -- (0.8,0);
	\draw(-0.2,0) -- (-0.2,1);
	\draw (0.2,0) -- (0.2,1);
	\node at (0.45,0.5) {\footnotesize$5$};
	{%
	\filldraw[fill=white,thick] (0,1) ellipse (3mm and 3mm);
	\node at (0,1) {\Large $T$};
	\path(0,1) ++(-180:0.37) node {$\star$};
}
}
;
        \draw (0,0)--(0,-0.5);
        \node[anchor=west] at (0,-0.35) {\footnotesize$8$};
	{%
	\filldraw[fill=white,thick] (-1,-0.2) rectangle (1,0.2);
	\node at (0,0) {\Large$\JW{8}$};
}\end{tikzpicture} +\beta \cdot \begin{tikzpicture}[baseline=0,scale=1.5]
	{\draw[shaded] (0.2,0) -- (0.2,1) -- (-0.2,1) -- (-0.2,0);
	\draw (-0.8,0) -- (-0.8,0.6) arc (180:0:0.8) -- (0.8,0);
	\draw(-0.2,0) -- (-0.2,1);
	\draw (0.2,0) -- (0.2,1);
	\node at (0.45,0.5) {\footnotesize$5$};
	{%
	\filldraw[fill=white,thick] (0,1) ellipse (3mm and 3mm);
	\node at (0,1) {\Large $Q$};
	\path(0,1) ++(-180:0.37) node {$\star$};
}
}
;
        \draw (0,0)--(0,-0.5);
        \node[anchor=west] at (0,-0.35) {\footnotesize$8$};
	{%
	\filldraw[fill=white,thick] (-1,-0.2) rectangle (1,0.2);
	\node at (0,0) {\Large$\JW{8}$};
}\end{tikzpicture}+ h \cdot \begin{tikzpicture}[baseline=0,scale=1.5]
	{

	\draw (-0.2,0) -- (-0.2,1);
	\draw (0.2,0) -- (0.2,1);
	\node at (0.45,0.5) {\footnotesize$7$};
	{%
	\filldraw[fill=white,thick] (0,1) ellipse (3mm and 3mm);
	\node at (0,1) {\Large $M'_0$};
	\path(0,1) ++(-180:0.37) node {$\star$};
}
}
;
        \draw (0,0)--(0,-0.5);
        \node[anchor=west] at (0,-0.35) {\footnotesize$8$};
	{%
	\filldraw[fill=white,thick] (-1,-0.2) rectangle (1,0.2);
	\node at (0,0) {\Large$\JW{8}$};
}\end{tikzpicture} $$, 
and 
$$\begin{tikzpicture}[baseline=0,scale=1.5]
	{%
	\fill[shaded] (-0.5,0) rectangle (0.5,1);
	\draw (-0.5,1) -- (-0.5,0);
	\node[anchor=west] at (-0.5,0.5) {\footnotesize$3$};
	\draw (-0.6,1) -- (-0.6,0);
	\draw (0.5,1) -- (0.5,0);
	\node[anchor=west] at (0.3,0.5) {\footnotesize$3$};
	\draw (0.6,1) -- (0.6,0);
}
 {%
	
	\draw (-0.5,1) -- (0.5, 1);
	\draw (-0.5,1.1) -- (0.5,1.1);
	\foreach \x in {-0.5,0.5} {
		{%
	\filldraw[fill=white,thick] (\x,1) ellipse (3mm and 3mm);
	\node at (-.5,1) {\Large $Q$};
	\node at (.5,1) {\Large $Q$};
	\path(-.55,1) ++(-90:0.37) node {$\star$};
	\path(0.55,1) ++(-90:0.37) node {$\star$};
}
	}
}
        \draw (0,0)--(0,-0.5);
        \node[anchor=west] at (0,-0.35) {\footnotesize$8$};
	{%
	\filldraw[fill=white,thick] (-1,-0.2) rectangle (1,0.2);
	\node at (0,0) {\Large$\JW{8}$};
}
\end{tikzpicture} = \alpha_0 \cdot \begin{tikzpicture}[baseline=0,scale=1.5]
	{\draw[shaded] (0.2,0) -- (0.2,1) -- (-0.2,1) -- (-0.2,0);
	\draw (-0.8,0) -- (-0.8,0.6) arc (180:0:0.8) -- (0.8,0);
	\draw(-0.2,0) -- (-0.2,1);
	\draw (0.2,0) -- (0.2,1);
	\node at (0.45,0.5) {\footnotesize$5$};
	{%
	\filldraw[fill=white,thick] (0,1) ellipse (3mm and 3mm);
	\node at (0,1) {\Large $T$};
	\path(0,1) ++(-180:0.37) node {$\star$};
}
}
;
        \draw (0,0)--(0,-0.5);
        \node[anchor=west] at (0,-0.35) {\footnotesize$8$};
	{%
	\filldraw[fill=white,thick] (-1,-0.2) rectangle (1,0.2);
	\node at (0,0) {\Large$\JW{8}$};
}\end{tikzpicture} +\beta_0 \cdot \begin{tikzpicture}[baseline=0,scale=1.5]
	{\draw[shaded] (0.2,0) -- (0.2,1) -- (-0.2,1) -- (-0.2,0);
	\draw (-0.8,0) -- (-0.8,0.6) arc (180:0:0.8) -- (0.8,0);
	\draw(-0.2,0) -- (-0.2,1);
	\draw (0.2,0) -- (0.2,1);
	\node at (0.45,0.5) {\footnotesize$5$};
	{%
	\filldraw[fill=white,thick] (0,1) ellipse (3mm and 3mm);
	\node at (0,1) {\Large $Q$};
	\path(0,1) ++(-180:0.37) node {$\star$};
}
}
;
        \draw (0,0)--(0,-0.5);
        \node[anchor=west] at (0,-0.35) {\footnotesize$8$};
	{%
	\filldraw[fill=white,thick] (-1,-0.2) rectangle (1,0.2);
	\node at (0,0) {\Large$\JW{8}$};
}\end{tikzpicture}+ h_0 \begin{tikzpicture}[baseline=0,scale=1.5]
	{

	\draw (-0.2,0) -- (-0.2,1);
	\draw (0.2,0) -- (0.2,1);
	\node at (0.45,0.5) {\footnotesize$7$};
	{%
	\filldraw[fill=white,thick] (0,1) ellipse (3mm and 3mm);
	\node at (0,1) {\Large $M'_0$};
	\path(0,1) ++(-180:0.37) node {$\star$};
}
}
;
        \draw (0,0)--(0,-0.5);
        \node[anchor=west] at (0,-0.35) {\footnotesize$8$};
	{%
	\filldraw[fill=white,thick] (-1,-0.2) rectangle (1,0.2);
	\node at (0,0) {\Large$\JW{8}$};
}\end{tikzpicture} $$
for some $\alpha,\beta,h,\alpha_0,\beta_0,h_0 \in \C$.
\begin{lem}[Eqn. B] The following equation holds:
$$\tau\cdot \begin{tikzpicture}[baseline=0,scale=1.5]
	{%
	\fill[shaded] (-0.5,0) rectangle (0.5,1);
	\draw (-0.5,1) -- (-0.5,0);
	\node[anchor=west] at (-0.5,0.5) {\footnotesize$3$};
	\draw (-0.6,1) -- (-0.6,0);
	\draw (0.5,1) -- (0.5,0);
	\node[anchor=west] at (0.3,0.5) {\footnotesize$3$};
	\draw (0.6,1) -- (0.6,0);
}
 {%
	
	\draw (-0.5,1) -- (0.5, 1);
	\draw (-0.5,1.1) -- (0.5,1.1);
	\foreach \x in {-0.5,0.5} {
		{%
	\filldraw[fill=white,thick] (\x,1) ellipse (3mm and 3mm);
	\node at (-.5,1) {\Large $T$};
	\node at (.5,1) {\Large $T$};
	\path(-.55,1) ++(-90:0.37) node {$\star$};
	\path(0.55,1) ++(-90:0.37) node {$\star$};
}
	}
}
        \draw (0,0)--(0,-0.5);
        \node[anchor=west] at (0,-0.35) {\footnotesize$8$};
	{%
	\filldraw[fill=white,thick] (-1,-0.2) rectangle (1,0.2);
	\node at (0,0) {\Large$\JW{8}$};
}
\end{tikzpicture}-  \begin{tikzpicture}[baseline=0,scale=1.5]
	{%
	\fill[shaded] (-0.5,0) rectangle (0.5,1);
	\draw (-0.5,1) -- (-0.5,0);
	\node[anchor=west] at (-0.5,0.5) {\footnotesize$3$};
	\draw (-0.6,1) -- (-0.6,0);
	\draw (0.5,1) -- (0.5,0);
	\node[anchor=west] at (0.3,0.5) {\footnotesize$3$};
	\draw (0.6,1) -- (0.6,0);
}
 {%
	
	\draw (-0.5,1) -- (0.5, 1);
	\draw (-0.5,1.1) -- (0.5,1.1);
	\foreach \x in {-0.5,0.5} {
		{%
	\filldraw[fill=white,thick] (\x,1) ellipse (3mm and 3mm);
	\node at (-.5,1) {\Large $Q$};
	\node at (.5,1) {\Large $Q$};
	\path(-.55,1) ++(-90:0.37) node {$\star$};
	\path(0.55,1) ++(-90:0.37) node {$\star$};
}
	}
}
        \draw (0,0)--(0,-0.5);
        \node[anchor=west] at (0,-0.35) {\footnotesize$8$};
	{%
	\filldraw[fill=white,thick] (-1,-0.2) rectangle (1,0.2);
	\node at (0,0) {\Large$\JW{8}$};
}
\end{tikzpicture}=a_0'\cdot \begin{tikzpicture}[baseline=0,scale=1.5]
	{\draw[shaded] (0.2,0) -- (0.2,1) -- (-0.2,1) -- (-0.2,0);
	\draw (-0.8,0) -- (-0.8,0.6) arc (180:0:0.8) -- (0.8,0);
	\draw(-0.2,0) -- (-0.2,1);
	\draw (0.2,0) -- (0.2,1);
	\node at (0.45,0.5) {\footnotesize$5$};
	{%
	\filldraw[fill=white,thick] (0,1) ellipse (3mm and 3mm);
	\node at (0,1) {\Large $T$};
	\path(0,1) ++(-180:0.37) node {$\star$};
}
}
;
        \draw (0,0)--(0,-0.5);
        \node[anchor=west] at (0,-0.35) {\footnotesize$8$};
	{%
	\filldraw[fill=white,thick] (-1,-0.2) rectangle (1,0.2);
	\node at (0,0) {\Large$\JW{8}$};
}\end{tikzpicture}+b_0' \cdot \begin{tikzpicture}[baseline=0,scale=1.5]
	{\draw[shaded] (0.2,0) -- (0.2,1) -- (-0.2,1) -- (-0.2,0);
	\draw (-0.8,0) -- (-0.8,0.6) arc (180:0:0.8) -- (0.8,0);
	\draw(-0.2,0) -- (-0.2,1);
	\draw (0.2,0) -- (0.2,1);
	\node at (0.45,0.5) {\footnotesize$5$};
	{%
	\filldraw[fill=white,thick] (0,1) ellipse (3mm and 3mm);
	\node at (0,1) {\Large $Q$};
	\path(0,1) ++(-180:0.37) node {$\star$};
}
}
;
        \draw (0,0)--(0,-0.5);
        \node[anchor=west] at (0,-0.35) {\footnotesize$8$};
	{%
	\filldraw[fill=white,thick] (-1,-0.2) rectangle (1,0.2);
	\node at (0,0) {\Large$\JW{8}$};
}\end{tikzpicture}$$,where $\tau=\frac{4+\sqrt{21}}{5}$, $a_0'=\frac{\tau \ip{(T,T),u_0}-\ip{(Q,Q),u_0}}{\ip{u_0,u_0}}$, and $b_0'=\frac{\tau \ip{(T,T),v_0}-\ip{(Q,Q),v_0}}{\ip{v_0,v_0}}$.  In fact, $a_0'=-\sqrt{\frac{5}{18}+\frac{1}{6}\sqrt{\frac{7}{3}}}$ and $b_0'=\sqrt{\frac{8}{9}+\sqrt{\frac{7}{3}}}$.
\end{lem}
\begin{proof}Let $w=(T,T)$ and $w_0=(Q,Q)$.  Let $W=h_0w-hw_0$.  We can assume $\overline{h}=h$ and $\overline{h_0}=h_0$, since $w^*=w$, $w_0^*=w_0$.  To show that $W$ lies in the span of $u_0$ and $v_0$, we use Bessel's inequality.
We look for $h$ and $h_0$ for which the following holds:
$$\ip{W,W}=\frac{\norm{\ip{W,u_0}}^2}{\ip{u_0,u_0}}+\frac{\norm{\ip{W,v_0}}^2}{\ip{v_0,v_0}}.$$
This holds iff\\
$ \norm{h_0}^2\Big(\ip{w,w}-\Big(\frac{\norm{\ip{w,u_0}}^2}{\ip{u_0,u_0}}+\frac{\norm{\ip{w,v_0}}^2}{\ip{v_0,v_0}}\Big)\Big)+\norm{h}^2\Big(\ip{w_0,w_0}-\Big(\frac{\norm{\ip{w_0,u_0}}^2}{\ip{u_0,u_0}}+\frac{\norm{\ip{w_0,v_0}}^2}{\ip{v_0,v_0}}\Big)\Big)\\
-2Re(h_0\overline{h}\ip{w,w_0})+\frac{2Re(h_0\overline{h}\ip{w,u_0}\overline{\ip{w_0,u_0}})}{\ip{u_0,u_0}}+\frac{2Re(h_0\overline{h}\ip{w,v_0}\overline{\ip{w_0,v_0}})}{\ip{v_0,v_0}}=0$

We can calculate that $$\Big(\ip{w,w}-\Big(\frac{\norm{\ip{w,u_0}}^2}{\ip{u_0,u_0}}+\frac{\norm{\ip{w,v_0}}^2}{\ip{v_0,v_0}}\Big)\Big)=\frac{1}{180}(363-67\sqrt{21})$$ and $$\Big(\ip{w_0,w_0}-\Big(\frac{\norm{\ip{w_0,u_0}}^2}{\ip{u_0,u_0}}+\frac{\norm{\ip{w_0,v_0}}^2}{\ip{v_0,v_0}}\Big)\Big)=\frac{1}{180}(87+17\sqrt{21})$$

Thus our equation holds iff\\
$\norm{h_0}^2\Big(\frac{1}{180}(363-67\sqrt{21})\Big)+\norm{h}^2\Big(\frac{1}{180}(87+17\sqrt{21})\Big)-2Re(h_0h\ip{w,w_0})\\
+\frac{2Re(h_0h\overline{\ip{u_0,w}}\ip{u_0,w_0}}{\ip{u_0,u_0}}+\frac{2Re(h_0h\overline{\ip{v_0,w}}\ip{v_0,w_0})}{\ip{v_0,v_0}}=0\\
\Longleftrightarrow
h_0^2\Big(\frac{1}{180}(363-67\sqrt{21})\Big)+h^2\Big(\frac{1}{180}(87+17\sqrt{21})\Big)+2Re(h_0h\Big(\frac{-9-19\sqrt{21}}{180}\Big))=0\\
\Longleftrightarrow
h_0=\frac{4+\sqrt{21}}{5}h$\\
To find $h_0$, we impose $$\ip{w_0,w_0}=\frac{\norm{\ip{w_0,u_0}}^2}{\ip{u_0,u_0}}+\frac{\norm{\ip{w_0,v_0}}^2}{\ip{v_0,v_0}}+\norm{h_0}^2\Norm{M'_0}^2, \mbox{ where $\Norm{M'_0}^2=[5]$}$$\\
This gives us that $h_0=\pm\frac{1}{6}\sqrt{\frac{1}{10}(39-\sqrt{21})}$.  It follows that $h=\pm\frac{\sqrt{\frac{5}{2}(39-\sqrt{21})}}{6(4+\sqrt{21})}$.\\
If we pick a value for $h_0$ then the equation we get involving $w$ and $w_0$ has a common factor of $h$ on both sides. Cancel out $h$ to get $$\tau(T,T)-(Q,Q)=a_0'u+b_0'v,$$ where $\tau=\frac{4+\sqrt{21}}{5}$, $a_0'=\frac{\tau\ip{(T,T),u_0}-\ip{(Q,Q),u_0}}{\ip{u_0,u_0}}$, and $b_0'=\frac{\tau\ip{(T,T),v_0}-\ip{(Q,Q),v_0}}{\ip{v_0,v_0}}$. 
\end{proof}
We can now easily deduce the shaded one-strand braiding substitutes.
\begin{prop}[Shaded One-strand Braiding Substitutes] $$u_0,v_0 \in span \{(T,T),(Q,Q),(T,Q),(Q,T)\}$$
\end{prop}
\begin{proof}Consider Eqn.A and Eqn.B from Lemmas 6.2.1 and 6.2.2.  This is a system of two equations and two unknowns where $u_0$ and $v_0$ are treated as the unknowns.  We can solve this system for $u_0$ and $v_0$, since the determinant of this system is $$a_0b_0'-a_0'b_0=\frac{2}{3}i(2\sqrt{3}+\sqrt{7})\neq 0.$$  We get expressions for $u_0$ and $v_0$ as linear combinations of the diagrams $$(T,T),(Q,Q),(T,Q),(Q,T).$$  These expressions will be our shaded one-strand braiding substitutes.  
\end{proof}
\chapter{Jellyfish Algorithm}
In this chapter, we prove that our planar algebra is subfactor.  To do this, we use the Jellyfish algorithm. Many of the proofs in this chapter are slight modifications of the proofs found in \cite{ExtH}.  We adopt the following definitions of arc, distance, and geodesic from \cite{ExtH}.
\begin{Defn}[\cite{ExtH}]  Suppose $D$ is a diagram in the planar algebra $\mathcal{P}$.  Let $S_0$ be a fixed copy of a generator inside $D$.  Suppose $\gamma$ is an embedded arc in $D$ from a point on the boundary of $S_0$ to a point on the top edge of $D$.  Suppose $\gamma$ is in general position, meaning that it intersects the strands of $D$ transversely, and does not touch any generator except at its initial point on $S_0$.  Let $m$ be the number of points of intersection between $\gamma$ and the strands of $D$.  If $m$ is minimal over all such arcs $\gamma$ then we say $\gamma$ is a {\em geodesic} and $m$ is the {\em distance} from $S_0$ to the top of $D$.
\end{Defn}
\begin{lem}Suppose $X$ is a diagram consisting of one copy of $T$(or $Q$) with all strands pointing down, and $d$ parallel edges forming a ``rainbow'', where $d\geq1$.  Then $X$ is a linear combination of diagrams that contain at most $2$ generators, each having distance less than $d$ from the top of the diagram.  
\end{lem}
\begin{proof}Suppose our diagram consists of one copy of $T$.  The proof for a diagram consisting of one copy of $Q$ is analogous. \\ 
\noindent Case $d=1$. Up to applying the rotation relation $\rho(T)=\omega_{T}(T)$, $X$ is as shown below:
$$
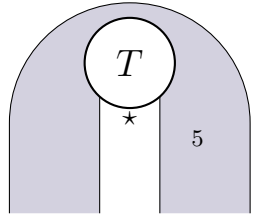
\begin{tikzpicture}[baseline=0,scale=2]
	{
	\fill[shaded] (-0.8,0) -- (-0.8,0.6) arc (180:0:0.8) -- (0.8,0) -- (0.2,0) -- (0.2,1) -- (-0.2,1) -- (-0.2,0);
	\draw (-0.8,0) -- (-0.8,0.6) arc (180:0:0.8) -- (0.8,0);
	\draw (-0.2,0) -- (-0.2,1);
	\draw (0.2,0) -- (0.2,1);
	\node at (0.45,0.5) {\footnotesize$5$};
	{%
	\filldraw[fill=white,thick] (0,1) ellipse (3mm and 3mm);
	\node at (0,1) {\Large $T$};
	\path(0,1) ++(-90:0.37) node {$\star$};
}
}
;
\end{tikzpicture} $$
Recall that we have the following unshaded one-strand braiding substitute:
$$
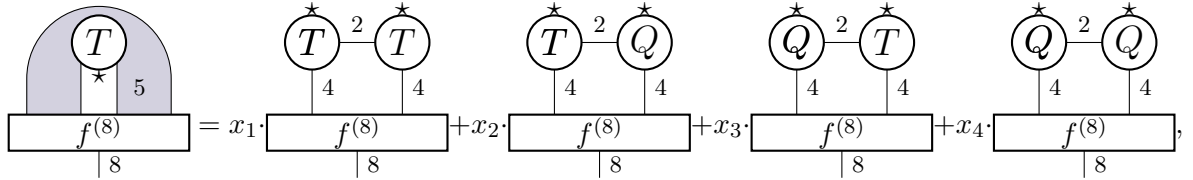
\begin{tikzpicture}[baseline=0,scale=1.2]
	{
	\fill[shaded] (-0.8,0) -- (-0.8,0.6) arc (180:0:0.8) -- (0.8,0) -- (0.2,0) -- (0.2,1) -- (-0.2,1) -- (-0.2,0);
	\draw (-0.8,0) -- (-0.8,0.6) arc (180:0:0.8) -- (0.8,0);
	\draw (-0.2,0) -- (-0.2,1);
	\draw (0.2,0) -- (0.2,1);
	\node at (0.45,0.5) {\footnotesize$5$};
	{%
	\filldraw[fill=white,thick] (0,1) ellipse (3mm and 3mm);
	\node at (0,1) {\Large $T$};
	\path(0,1) ++(-90:0.37) node {$\star$};
}
}
;
        \draw (0,0)--(0,-0.5);
        \node[anchor=west] at (0,-0.35) {\footnotesize$8$};
	{%
	\filldraw[fill=white,thick] (-1,-0.2) rectangle (1,0.2);
	\node at (0,0) {$\JW{8}$};
}\end{tikzpicture}
=x_1 \cdot \begin{tikzpicture}[baseline=0,scale=1.2]
	{%
	\draw (-0.5,1) -- (-0.5,0);
	\node[anchor=west] at (-0.5,0.5) {\footnotesize$4$};
	\draw (0.5,1) -- (0.5,0);
	\node[anchor=west] at (0.5,0.5) {\footnotesize$4$};
}
 {%
	\node[anchor=south] at (0,1) {\footnotesize$2$};
	\draw (-0.5,1) -- (0.5, 1);
	\foreach \x in {-0.5,0.5} {
		{%
	\filldraw[fill=white,thick] (\x,1) ellipse (3mm and 3mm);
	\node at (-.5,1) {\Large $T$};
	\node at (.5,1) {\Large $T$};
	\path(\x,1) ++(90:0.37) node {$\star$};
}
	}
}
        \draw (0,0)--(0,-0.5);
        \node[anchor=west] at (0,-0.35) {\footnotesize$8$};
	{%
	\filldraw[fill=white,thick] (-1,-0.2) rectangle (1,0.2);
	\node at (0,0) {$\JW{8}$};
}
\end{tikzpicture} +x_2 \cdot  \begin{tikzpicture}[baseline=0,scale=1.2]
	{%
	\draw (-0.5,1) -- (-0.5,0);
	\node[anchor=west] at (-0.5,0.5) {\footnotesize$4$};
	\draw (0.5,1) -- (0.5,0);
	\node[anchor=west] at (0.5,0.5) {\footnotesize$4$};
}
 {%
	\node[anchor=south] at (0,1) {\footnotesize$2$};
	\draw (-0.5,1) -- (0.5, 1);
	\foreach \x in {-0.5,0.5} {
		{%
	\filldraw[fill=white,thick] (\x,1) ellipse (3mm and 3mm);
	\node at (-.5,1) {\Large $T$};
	\node at (.5,1) {\Large $Q$};
	\path(\x,1) ++(90:0.37) node {$\star$};
}
	}
}
        \draw (0,0)--(0,-0.5);
        \node[anchor=west] at (0,-0.35) {\footnotesize$8$};
	{%
	\filldraw[fill=white,thick] (-1,-0.2) rectangle (1,0.2);
	\node at (0,0) {$\JW{8}$};
}
\end{tikzpicture} +x_3\cdot  \begin{tikzpicture}[baseline=0,scale=1.2]
	{%
	\draw (-0.5,1) -- (-0.5,0);
	\node[anchor=west] at (-0.5,0.5) {\footnotesize$4$};
	\draw (0.5,1) -- (0.5,0);
	\node[anchor=west] at (0.5,0.5) {\footnotesize$4$};
}
 {%
	\node[anchor=south] at (0,1) {\footnotesize$2$};
	\draw (-0.5,1) -- (0.5, 1);
	\foreach \x in {-0.5,0.5} {
		{%
	\filldraw[fill=white,thick] (\x,1) ellipse (3mm and 3mm);
	\node at (-.5,1) {\Large $Q$};
	\node at (.5,1) {\Large $T$};
	\path(\x,1) ++(90:0.37) node {$\star$};
}
	}
}
        \draw (0,0)--(0,-0.5);
        \node[anchor=west] at (0,-0.35) {\footnotesize$8$};
	{%
	\filldraw[fill=white,thick] (-1,-0.2) rectangle (1,0.2);
	\node at (0,0) {$\JW{8}$};
}
\end{tikzpicture} +x_4\cdot \begin{tikzpicture}[baseline=0,scale=1.2]
	{%
	\draw (-0.5,1) -- (-0.5,0);
	\node[anchor=west] at (-0.5,0.5) {\footnotesize$4$};
	\draw (0.5,1) -- (0.5,0);
	\node[anchor=west] at (0.5,0.5) {\footnotesize$4$};
}
 {%
	\node[anchor=south] at (0,1) {\footnotesize$2$};
	\draw (-0.5,1) -- (0.5, 1);
	\foreach \x in {-0.5,0.5} {
		{%
	\filldraw[fill=white,thick] (\x,1) ellipse (3mm and 3mm);
	\node at (-.5,1) {\Large $Q$};
	\node at (.5,1) {\Large $Q$};
	\path(\x,1) ++(90:0.37) node {$\star$};
}
	}
}
        \draw (0,0)--(0,-0.5);
        \node[anchor=west] at (0,-0.35) {\footnotesize$8$};
	{%
	\filldraw[fill=white,thick] (-1,-0.2) rectangle (1,0.2);
	\node at (0,0) {$\JW{8}$};
}
\end{tikzpicture},$$ for some $x_1,x_2,x_3,x_4\in \C$.
Consider what happens as we expand $f^{(8)}$on the LHS into Temperley-Lieb elements.  The identity element in $f^{(8)}$ gives our diagram $X$.  Any other diagram in the expansion of $\JW{8}$ has a cup.  If the cup lands on $T$, we get 0.  Otherwise, the cup is at the very left or at the very right and hence gives a diagram with the generator $T$ ``exposed''(i.e., a diagram consisting of a single copy of $T$ at distance 0 from the top and a cup on the bottom boundary of the diagram.).  Now, consider what happens as we expand the $\JW{8}$'s on the RHS into TL elements.  Consider just one of the diagrams on the RHS, say (T,T), since they are analogous.  This diagram breaks up into diagrams with two generators, each of which has distance 0 from the top.  Similarly, for the other three diagrams on the RHS.  By solving for $X$ in the resulting equation, we can write $X$ as a linear combination of diagrams that contain one or two generators, each of which has distance 0 from the top.  If the shading is reversed for $X$, then we can apply the shaded one-strand braiding substitute in analogous fashion. Thus, case $d=1$ is done.\\
\indent Now, suppose $d \geq 2$.  Suppose $X$ has unshaded outer boundary.  If $d$ is odd, then $\gamma$ begins in a  shaded region of $X$.  $X$ contains a copy of the diagram with unshaded outer boundary as in the case $d=1$, up to a rotation of the generator.  Apply the result of case $d=1$ to derive $X$ as a linear combination of diagrams that contain one or two generators, each of which has distance $d-1$ from the top.  If $d$ is even, then $\gamma$ begins in an unshaded region of $X$.  $X$ contains a copy of the diagram with shaded outer boundary as in case $d=1$, up to a rotation of the generator.  Apply the result of case $d=1$ to derive $X$ as a linear combination of diagrams that contain one or two generators, each of which has distance $d-1$ from the top.  The case where $X$ has shaded outer boundary is analogous.
\end{proof}
\begin{Defn}[\cite{ExtH}slightly modified]We say a diagram $D$ in $\mathcal{P}$ is in {\em jellyfish form} if all occurrences of each generator lie in a row at the top of $D$, and all strands of $D$ lie entirely below the height of the tops of the copies of each generator.
\end{Defn}
\begin{lem}Every diagram in $\mathcal{P}$ is a linear combination of diagrams in jellyfish form.
\end{lem}
\begin{proof}Let $D$ be a diagram in $\mathcal{P}$ with all endpoints at the bottom edge of $D$. If every copy of each generator is at distance 0 from the top of $D$, then $D$ is already in jellyfish form.  If not, we can use Lemma 7.0.5 to pull each copy of a generator to the top of $D$ as follows.\\
Suppose $S_0$ is a copy of a generator at distance $d\geq 1$ from the top of $D$.  Let $\gamma$ be a geodesic from $S_0$ to the top of $D$.  Let $X$ be a small neighborhood of $S_0\cup \gamma$.  By isotopy, we can assume $X$ is a rectangular neighborhood of $S_0\cup \gamma$ which consists of a single copy of $S_0$ with all strands pointing down and a ``rainbow'' of $d$ strands over it.  By Lemma 7.0.5, $X$ is a linear combination of diagrams that contain at most 2 copies of generators, each having distance less than $d$ from the top of the diagram.  Let $X'$ be one of the terms in this expression for $X$.  Let $D'$ be the result of replacing $X$ by $X'$ in $D$.  Suppose $S_1$ is a copy of the generator in $D'$.  If $S_1$ lies in $X'$, then $S_1$ is at most distance $d-1$.  Now, if $S_1$ does not lie in $X'$, then there is a geodesic in $D$ from $S_1$ to the top of $D$ that does not intersect $\gamma$.  This geodesic intersects strands in the same number of points as before the replacement of $X$ by $X'$.  Thus, the distance from $S_1$ to the top of $D$ does not increase when we perform the replacement of $X$ by $X'$.  Thus far, we know that if we replace $X$ by $X'$ then $S_0$ will be replaced by one or two copies of generators that are closer to the top of $D$; and any other copy of a generator will not be farther from the top of $D$ than before the replacement.  This process will end with a linear combination of diagrams that have all copies of generators at distance 0 from the top.
\end{proof}

\begin{theorem}[Subfactor property]$\mathcal{P}_{0,+}$ is one-dimensional.
\end{theorem}
\begin{proof}Suppose $D$ is a closed diagram with unshaded exterior.  We assume $D$ is in jellyfish form, by Lemma 7.0.7.  Assume that no cups are attached to any copy of a generator and that there are no closed loops.  
\begin{claim}There must be a copy of a generator connected to only its adjacent neighbors.
\end{claim}
\begin{proof}{(of claim)}
Consider all strands that do not connect adjacent vertices and, amongst these, find one that has the smallest number of vertices between its endpoints.  Any vertex between the endpoints of this strand can connect only to its adjacent neighbors.
\end{proof} 
Let $S_0$ be such a copy of a generator.  Then $S_0$ is connected to one of its neighbors, say $S_1$, by at least 3 parallel strands.  Recall that $T^2=\JW{3}+aT+bQ,$ for some $a,b\in\C,$  $Q^2=\JW{3}+cT+dQ,$ for some $c,d\in\C,$  $TQ=a_0T+b_0Q,$ for some $a_0,b_0\in\C,$ and $QT=c_0T+d_0Q,$ for some $c_0,d_0\in\C$.  Thus, we can replace $S_0$ and $S_1$ with the relevant substitution, giving a linear combination of diagrams that are still in jellyfish form but having fewer copies of generators.  By induction on the number of copies of generators, $D$ is a scalar multiple of the empty diagram.
\end{proof}
\begin{theorem}[$\mathcal{P}$ is irreducible]$dim(\mathcal{P}_{1,+})=1$.
\end{theorem}
\begin{proof}Suppose $D$ is a diagram in $\mathcal{P}_{1,+}$.  Then, by Lemma 7.0.7, we can assume $D$ is in jellyfish form.  Assume that no cups are attached to any copy of a generator and that there are no closed diagrams.  Now, there are two endpoints on the bottom of $D$.  We can assume each endpoint connects to a copy of a generator since otherwise $D$ would be in $TL$, and we're done.  \\
\indent If there is a strand that connects two distinct non-adjacent copies of generators, then amongst such strands find one that has the smallest number of vertices between its endpoints.  Any vertex between the endpoints of this strand can connect only to its adjacent neighbors.  Let $S_0$ be such a copy of a generator.  Then $S_0$ is connected to one of its neighbors by at least 3 parallel strands.  (Note that $S_0$ must have two adjacent neighbors.)  Apply the relevant  quadratic relation to get a linear combination of diagrams that have fewer copies of generators.  By induction, each of these is a scalar multiple of the single strand, and hence D is a scalar multiple of the single strand.\\
\indent If there is no strand that connects two distinct non-adjacent copies of generators, then every strand attached to a copy of a generator connects to the bottom of $D$ or connects two adjacent copies of generators.\\
\indent Suppose both endpoints at the bottom of $D$ connect to the same copy $S'$ of a generator.  Since $S'$ has 6 strands and only 2 of them connect to the bottom of $D$, there must be a strand that connects $S'$ to an adjacent neighbor $S''$.  $S''$ is either to the left or to the right of $S'$.  Suppose $S''$ is to the left of $S'$.  Among those copies of generators to the left of $S'$, consider the leftmost copy.  This copy attaches to its adjacent neighbor by 6 strands and hence forms a closed diagram with its neighbor.  This contradicts our assumption that there are no closed diagrams.  Similarly, if $S''$ is to the right of $S'$, we get a contradiction.\\
\indent So each endpoint at the bottom of $D$ connects to a distinct copy of a generator.  Say the two endpoints connect to $S_1$ and $S_2$, where $S_1$ is to the left of $S_2$.  There cannot be a strand from $S_1$ connecting to any copies of generators to the left of $S_1$ since then there would be a leftmost vertex which forms a closed diagram with its adjacent neighbor.  Similarly, there cannot be any strand from $S_2$ conecting to any copies of generators to the right of $S_2$.  Thus, there are five strands connecting $S_1$ to an adjacent neighbor.  (Similarly, for $S_2$.)  Apply a quadratic relation to $S_1$ and its neighbor (or $S_2$ and its neighbor) to get a linear combination of diagrams that have fewer copies of generators. By induction, each of these is a multiple of the single strand, and hence $D$ is as well.
\end{proof}
\begin{theorem}$\mathcal{P}_2=TL_2$.
\end{theorem}
\begin{proof}  Suppose $D$ is a diagram in $\mathcal{P}_{2,+}$. Then, by Lemma 7.0.7, we can assume $D$ is in jellyfish form.  Assume no caps attached to a copy of a generator and no closed diagrams.  There are four endpoints on the bottom of $D$.  \\
\indent Suppose there is a strand that connects two distinct non-adjacent copies of generators.  Then we can find a vertex $S_0$ which is connected solely to its two adjacent neighbors.  $S_0$ is connected to one of its neighbors by at least 3 parallel strands. Apply the relevant quadratic relation to get a linear combination of diagrams that have fewer copies of generators.  By induction, each of these is in $TL_2$, and hence $D$ is in $TL_2$.  \\
\indent Now, suppose there is no strand that connects two distinct non-adjacent copies of generators.  Then every strand attached to a copy of a generator connects to the bottom of $D$ or connects two adjacent copies of generators. \\
\indent If all four endpoints at the bottom of $D$ are connected solely to the bottom of $D$, then $D$ is in $TL_2$.\\
\indent If there is a single cap on the bottom of $D$, then the remaining part of $D$ lies in $\mathcal{P}_1$ and hence is a scalar multiple of a single strand.  It follows that $D$ lies in $TL_2$.  \\
\indent If all endpoints at the bottom of $D$ connect to a copy of a generator, then consider the first endpoint at the bottom of $D$.  Let $S_1$ be the copy of a generator to which it is connected.  There cannot be a strand from $S_1$ conecting to any copies of generators to the left of $S_1$ since then there would be a leftmost vertex which forms a closed diagram with its adjacent neighbor.  If $S_1$ connects to an adjacent neighbor by at least 3 parallel strands, then we can apply a quadratic relation and apply induction to get $D$ in $TL_2$.  If not, then $S_1$ connects to an adjacent neighbor $S_2$ to the right of $S_1$ by 2 strands and has 4 strands connecting to the bottom.  However, there cannot be a strand from $S_1$ connecting to any copies of generators to the right of $S_1$ since there would be a rightmost vertex which forms a closed diagram with its adjacent neighbor.  Contradiction.
\end{proof}

\begin{theorem} $\mathcal{P}_3=TL_3\oplus ATL_3(T)\oplus ATL_3(Q)$.
\end{theorem}
\begin{proof}  Suppose $D$ is a diagram in $\mathcal{P}_{3,+}$.  Then, by Lemma 7.0.7, we can assume $D$ is in jellyfish form.  Assume no caps attached to a copy of a generator and no closed diagrams.  There are 6 endpoints on the bottom of $D$.\\
\indent Suppose there is a strand that connects two distinct non-adjacent copies of generators.  Then we can find a vertex $S_0$ which is connected solely to its two adjacent neighbors.  $S_0$ is connected to one of its neighbors by at least 3 parallel strands.  Apply the relevant quadratic relation to get a linear combination of diagrams that have fewer copies of generators.  By induction, each of these is in  $TL_3\oplus ATL_3(T)\oplus ATL_3(Q)$, and hence $D$ is also.\\
\indent Now, suppose there is no strand that connects two distinct non-adjacent copies of generators.  Then every strand attached to a copy of a generator connects to the bottom of $D$ or connects two adjacent copies of generators.\\
\indent  If all 6 endpoints at the bottom of $D$ are connected solely to the bottom of $D$, then $D$ is in $TL_3$.\\
\indent  If there is a single cap on the bottom of $D$, then the remaining part of $D$ lies in $\mathcal{P}_2=TL_2$ and hence $D$ is in $TL_3$.  \\
\indent If there are two caps on the bottom of $D$, then the remaining part of $D$ lies in $\mathcal{P}_1$ and hence $D$ is in $TL_3$.  \\
\indent  If there are no caps on the bottom of $D$, then each endpoint at the bottom of $D$ connects to a copy of a generator.  Consider the first endpoint at the bottom of $D$.  Let $S_1$ be the copy of a generator to which it is connected.  There cannot be a strand from $S_1$ connecting to any copies of generators to the left of $S_1$ since then there would be a leftmost vertex which forms a closed diagram with its adjacent neighbor.  If $S_1$ connects to an adjacent neighbor on its right by at least 3 parallel strands, then we can apply a quadratic relation and apply induction to get $D$ in $TL_3\oplus ATL_3(T)\oplus ATL_3(Q)$.  If not, then either all strands of $S_1$ connect to the bottom of $D$, in which case $D$ is in $ATL(T)$ or $ATL(Q)$, or $S_1$ connects to an adjacent neighbor $S_2$ to the right of $S_1$ by 1 or 2 strands.\\
\indent  If $S_1$ connects to $S_2$ by 1 strand, then 5 strands connect $S_1$ to the bottom of $D$ and there is only one remaining endpoint on the bottom of $D$.  Suppose this endpoint connects to $S_3$.  Then there are no strands from $S_3$ connecting to any copies of generators to the right of $S_3$.  $S_3$ must have an adjacent neighbor on its left to which it is connected by 5 strands.  Apply a quadratic relation and induction.\\
\indent  If $S_1$ conects to $S_2$ by 2 strands, then 4 strands connect $S_1$ to the bottom of $D$ and there are 2 remaining endpoints on the bottom of $D$.  Consider the very last endpoint and suppose it connects to $S_3$.  There are no strands from $S_3$ connecting to any copies of generators to the right of $S_3$.  $S_3$ must have an adjacent neighbor on its left to which it is connected by at least 4 strands.  Apply a quadratic relation and induction.
\end{proof}
\chapter{The planar algebra $\mathcal{PA}(T,Q)$ is ``2221''}
In this chapter, we show that our planar algebra $\mathcal{PA}(T,Q)$ has the graph ``2221'' as its principal graph.  We consider the tensor category of projections of a planar algebra, as laid out in \cite{D2n}.
\begin{Defn}[\cite{D2n}]
Given a planar algebra $P$ we construct a tensor category $\cC_P$ as follows.
\begin{itemize}
\item An object of $\cC_P$ is a projection in one of the $2n$-box algebras $P_{2n}$.
\item Given two projections $\pi_1 \in P_{2n}$ and $\pi_2 \in P_{2m}$ we define $\Hom {\pi_1,\pi_2}$ to be the space $\pi_2 P_{n \rightarrow m} \pi_1$  ($P_{n \rightarrow m}$ is a convenient way of denoting $P_{(n+m)}$, drawn with $n$ strands going down and $m$ going up.)
\item The tensor product $\pi_1 \otimes \pi_2$ is the disjoint union of the two projections in $P_{2n+2m}$.
\item The trivial object $\id$ is the empty picture (which is a projection in $P_0$).
\item The dual $\overline{\pi}$ of a projection $\pi$ is given by rotating it 180 degrees.
\end{itemize}
\end{Defn}

\begin{remark}[\cite{D2n}]
This category comes with a special object $X \in P_2$ which is the single strand.  Note that $X = \overline{X}$. 
We would like to be able to take direct sums in this category.  If $\pi_1$ and $\pi_2$ are orthogonal projections in the same box space $P_n$ (i.e. if $\pi_1 \pi_2 = 0 = \pi_2 \pi_1$), then their direct sum is just $\pi_1 + \pi_2$.  However, if they are not orthogonal the situation is a bit more difficult.  One solution to this problem is to replace the projections with isomorphic projections which are orthogonal.  However, this construction only makes sense on equivalence classes, so we use another construction.
\end{remark}
\begin{Defn}[\cite{D2n}]
Given a category $\cC$ define its matrix category $\Mat{\cC}$ as follows.
\begin{enumerate}
\item The objects of $\Mat{\cC}$ are formal direct sums of objects of $\cC$
\item A morphism of $\Mat{\cC}$ from $A_1 \oplus \ldots \oplus A_n \rightarrow B_1 \oplus \ldots \oplus B_m$ is an $m$-by-$n$ matrix whose $(i,j)$th entry is in $Hom_{\cC}\left(A_j,B_i\right)$.
\end{enumerate}
\end{Defn}
\begin{remark}[\cite{D2n}]
If $\cC$ is a tensor category then $\Mat{\cC}$ has an obvious tensor
product (on objects, formally distribute, and on morphisms, use the
usual tensor product of matrices and the tensor product for $\cC$ on
matrix entries).  If $\cC$ is spherical then so is $\Mat{\cC}$ (where the
dual on objects is just the dual of each summand and on morphisms the dual transposes
the matrix and dualizes each matrix entry).
\end{remark}
\begin{lem}[\cite{D2n}]
If $\pi_1$ and $\pi_2$ are orthogonal projections in $P_{2n}$, then $\pi_1 \oplus \pi_2 \cong \pi_1 + \pi_2$ in $\Mat{\cC_P}$.
\end{lem}
\begin{proof}
Define $f: \pi_1 \oplus \pi_2 \rightarrow \pi_1 + \pi_2$ and $g: \pi_1 + \pi_2 \rightarrow \pi_1 \oplus \pi_2$ by $f = \begin{pmatrix} \pi_1 & \pi_2 \end{pmatrix}$ and $g=\begin{pmatrix} \pi_1 \\ \pi_2 \end{pmatrix}$.  Then $f g = \id_{\pi_1 + \pi_2}$ and $g f= \id_{\pi_1 \oplus \pi_2}$.
\end{proof}

\begin{Defn}[\cite{D2n}]
An object $\pi \in \Mat{\cC_P}$ is called minimal if $\Hom{\pi,\pi}$ is $1$-dimensional.
\end{Defn}
\begin{Defn}[\cite{D2n}]
A planar algebra is called semisimple if every projection is a direct sum of minimal projections and for any pair of non-isomorphic minimal projections $\pi_1$ and $\pi_2$, we have that $\Hom {\pi_1, \pi_2} = 0$.
\end{Defn}

\begin{Defn}[\cite{D2n}]
The principal (multi-)graph of a semisimple planar algebra has as vertices the isomorphism classes of minimal projections, and there are $$\dim \Hom {\pi_1 \otimes X,\pi_2} (= \dim \Hom {\pi_1,\pi_2 \otimes X})$$ edges between the vertices $\pi_1 \in P_n$ and $\pi_2 \in P_m$.
\end{Defn}
Recall the principal graph ``2221'':
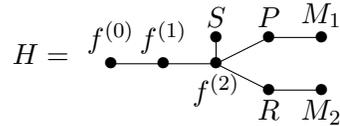
$$H=\begin{tikzpicture}[baseline,scale=.7]
\node at (0,0) {$\bullet$};
\node at (0,0) [above] {$\JW{0}$};
\node at (1,0) {$\bullet$};
\node at (1,0) [above] {$\JW{1}$};
\node at (2,0) {$\bullet$};
\node at (2,0) [below] {$\JW{2}$};
\node at (2,.5) {$\bullet$};
\node at (2,.5) [above] {$S$};
\node at (3,.5) {$\bullet$};
\node at (3,.5) [above] {$P$};
\draw (0,0)--(2,0);
\draw (2,0)--(2,.5);
\node at (4,.5) {$\bullet$};
\node at (4,.5) [above] {$M_1$};
\node at (3,-.5) {$\bullet$};
\node at (3,-.5) [below] {$R$};
\node at (4,-.5) {$\bullet$};
\node at (4,-.5) [below] {$M_2$};
\draw (2,0)--(3,.5)--(4,.5);
\draw (2,0)--(3,-.5)--(4,-.5);
\end{tikzpicture}$$ 

We look for the projections $S$, $P$, and $R$ as linear combinations of $\JW{3}$, $T$, and $Q$.  If we suppose that our projection $P$ is equal to $\alpha \JW{3}+\beta T + \gamma Q$ and impose the condition that $P$ is a projection, then we get the following three equations:
\begin{itemize}
\item $\alpha^2+\beta^2+\gamma^2=\alpha$
\item $2\alpha\beta+\frac{\beta^2Z(T^3)}{[4]}+\frac{\gamma^2Z(TQ^2)}{[4]}+2\beta\gamma\frac{Z(T^2Q)}{[4]}=\beta$
\item $2\alpha\gamma+\frac{\beta^2Z(QT^2)}{[4]}+\frac{\gamma^2Z(Q^3)}{[4]}+2\beta\gamma\frac{Z(Q^2T)}{[4]}=\gamma$
\end{itemize}
Solving this system of equations, we get 6 possible values for $\alpha, \beta, \gamma$ (see below).  Because of the symmetry between the last four values, we predict the first two possible values to be for $S$.\\ 
For $S$, we have the following possible values for $\alpha, \beta, \gamma$ :
\begin{itemize}
\item $\alpha=\frac{6-\sqrt{21}}{5}, \beta= \frac{-1+\sqrt{21}}{10},\gamma= \sqrt{\frac{2}{13+3\sqrt{21}}}$
\item $\alpha=\frac{-1+\sqrt{21}}{5}, \beta= \frac{1-\sqrt{21}}{10}, \gamma= -\sqrt{\frac{2}{13+3\sqrt{21}}}$
\end{itemize}
For $P$, we get the following possible values for $\alpha, \beta, \gamma$ :
\begin{itemize}
\item $\alpha=\frac{-1+\sqrt{21}}{10}, \beta= -\frac{1}{60}(-3+3\sqrt{21}+5\sqrt{6(-3+\sqrt{21})}),\gamma=\frac{2}{\sqrt{27+7\sqrt{21}+\sqrt{6(141+31\sqrt{21})}}}$
\item $\alpha=\frac{11-\sqrt{21}}{10}, \beta= \frac{1}{60}(-3+3\sqrt{21}+5\sqrt{6(-3+\sqrt{21})}),\gamma=-\frac{2}{\sqrt{27+7\sqrt{21}+\sqrt{6(141+31\sqrt{21})}}}$
\end{itemize}
For $R$, we get the following possible values for $\alpha, \beta, \gamma$ :
\begin{itemize}
\item $\alpha=\frac{11-\sqrt{21}}{10}, \beta=\frac{1}{60}(-3+3\sqrt{21}-5\sqrt{6(-3+\sqrt{21})}),\gamma=\frac{2}{\sqrt{27+7\sqrt{21}-\sqrt{6(141+31\sqrt{21})}}}$
\item $\alpha=\frac{-1+\sqrt{21}}{10}, \beta=-\frac{1}{60}(-3+3\sqrt{21}-5\sqrt{6(-3+\sqrt{21})}),\gamma=-\frac{2}{\sqrt{27+7\sqrt{21}-\sqrt{6(141+31\sqrt{21})}}}$
\end{itemize}
Let $$S=\alpha'\JW{3}+\beta'T+\gamma'Q,$$ where $\alpha'=\frac{6-\sqrt{21}}{5}, \beta'=\frac{-1+\sqrt{21}}{10}, \gamma'=\sqrt{\frac{2}{13+3\sqrt{21}}}$.\\
Let $$P = \alpha\JW{3}+\beta T +\gamma Q,$$ where $\alpha=\frac{-1+\sqrt{21}}{10}, 
\beta=-\frac{1}{60}(-3+3\sqrt{21}+5\sqrt{6(-3+\sqrt{21})}),\gamma=\frac{2}{\sqrt{27+7\sqrt{21}+\sqrt{6(141+31\sqrt{21})}}}$.\\
Let $$R = \alpha_0\JW{3}+\beta_0 T +\gamma_0 Q,$$ where $\alpha_0=\frac{-1+\sqrt{21}}{10}, 
\beta_0=-\frac{1}{60}(-3+3\sqrt{21}-5\sqrt{6(-3+\sqrt{21})}),\gamma_0=-\frac{2}{\sqrt{27+7\sqrt{21}-\sqrt{6(141+31\sqrt{21})}}}.$
Let $$M_1=P\otimes id_1 - \frac{1}{\alpha\frac{[4]}{[3]}} \cdot \begin{tikzpicture}[baseline]
   \clip (-.6,-1.3) rectangle (.8,1.3);
   \node at (0,.7) (top) [rectangle,draw] {$ \hspace{.3 cm} P \hspace{.3cm} $};
   \node at (0,-.7) (bottom) [rectangle,draw]{$ \hspace{.3cm} P
\hspace{.3cm} $};

   \draw (top.-60)--(bottom.60);
   \draw (top.-145)--(bottom.145);
   \draw (top.-35) arc (-180:0:.2cm) -- ++(90:1cm);
   \draw (bottom.35) arc (180:0:.2cm) -- ++(-90:1cm);

   \draw (top.35) -- ++(90:1cm);
   \draw (top.60) -- ++(90:1cm);
   \draw (top.145) -- ++(90:1cm);

   \draw (bottom.-35) -- ++(-90:1cm);
   \draw (bottom.-60) -- ++(-90:1cm);
   \draw (bottom.-145) -- ++(-90:1cm);
\end{tikzpicture}$$

\noindent Let $$M_2=R\otimes id_1 - \frac{1}{\alpha_0\frac{[4]}{[3]}} \cdot \begin{tikzpicture}[baseline]
   \clip (-.6,-1.3) rectangle (.8,1.3);
   \node at (0,.7) (top) [rectangle,draw] {$ \hspace{.3 cm} R \hspace{.3cm} $};
   \node at (0,-.7) (bottom) [rectangle,draw]{$ \hspace{.3cm} R
\hspace{.3cm} $};

   \draw (top.-60)--(bottom.60);
   \draw (top.-145)--(bottom.145);
   \draw (top.-35) arc (-180:0:.2cm) -- ++(90:1cm);
   \draw (bottom.35) arc (180:0:.2cm) -- ++(-90:1cm);

   \draw (top.35) -- ++(90:1cm);
   \draw (top.60) -- ++(90:1cm);
   \draw (top.145) -- ++(90:1cm);

   \draw (bottom.-35) -- ++(-90:1cm);
   \draw (bottom.-60) -- ++(-90:1cm);
   \draw (bottom.-145) -- ++(-90:1cm);
\end{tikzpicture}$$
The main theorem of this chapter is the following.
\begin{theorem}The planar algebra ``2221'' is semi-simple, with minimal projections $\JW{0},\JW{1},\JW{2},S,P,R,M_1,M_2$, where $S,P,R,M_1,M_2$ are as defined above.  The principal graph is the graph 
$$H=
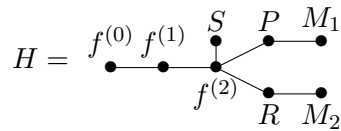
\begin{tikzpicture}[baseline,scale=.7]
\node at (0,0) {$\bullet$};
\node at (0,0) [above] {$\JW{0}$};
\node at (1,0) {$\bullet$};
\node at (1,0) [above] {$\JW{1}$};
\node at (2,0) {$\bullet$};
\node at (2,0) [below] {$\JW{2}$};
\node at (2,.5) {$\bullet$};
\node at (2,.5) [above] {$S$};
\node at (3,.5) {$\bullet$};
\node at (3,.5) [above] {$P$};
\draw (0,0)--(2,0);
\draw (2,0)--(2,.5);
\node at (4,.5) {$\bullet$};
\node at (4,.5) [above] {$M_1$};
\node at (3,-.5) {$\bullet$};
\node at (3,-.5) [below] {$R$};
\node at (4,-.5) {$\bullet$};
\node at (4,-.5) [below] {$M_2$};
\draw (2,0)--(3,.5)--(4,.5);
\draw (2,0)--(3,-.5)--(4,-.5);
\end{tikzpicture}$$ \\
Furthermore, the dual principal graph is also $H$, with the appropriate minimal projections.
\end{theorem}
\begin{proof}The proof depends on lemmas that are proven in the remainder of this chapter.  $S,P$, and $R$ are projections by our choice of coefficients.  Also, it is not difficult to check that $M_1$ and $M_2$ are projections.  Let $\mathcal{M}=\{\JW{0},\JW{1},\JW{2},S,P,R,M_1,M_2\}$.  Lemmas 8.0.22, 8.0.24, and 8.0.26 imply that every projection in $\mathcal{M}$ is minimal.  Lemma 8.0.27 implies that $\Hom{A,B}=0$ for any distinct $A,B\in \mathcal{M}$.  Lemmas 8.0.28, 8.0.29, 8.0.30, 8.0.31, and 8.0.32 imply that for each $Y\in\mathcal{M}$, the projection $Y\otimes\JW{1}$ is isomorphic to a direct sum of projections in $\mathcal{M}$.  Hence, $id_n$ for any $n$, is isomorphic to a direct sum of elements in $\mathcal{M}$.  If $p\in\mathcal{P}_n$ is a projection, then $p+(id_n-p)=id_n$ and $p(id_n-p)=0$.  Hence, 
\begin{align*}
&p+(id_n-p)\cong p\oplus (id_n-p)\mbox{\indent  by Lemma 8.0.17}\\
&\Longrightarrow id_n\cong p\oplus (id_n-p)
\end{align*}
But 
\begin{align*}
&id_n=\left(\JW{1}\right)^{\otimes n}\\
&\Longrightarrow \left(\JW{1}\right)^{\otimes n}\cong p\oplus (id_n-p)\\
&\Longrightarrow p\oplus(id_n-p)\cong \mbox{ a direct sum of elements in } \mathcal{M}
\end{align*}
That is, every projection is a direct summand of $id_n$ for some $n$, and hence is a direct summand of a direct sum of elements in $\mathcal{M}$.  In particular, every minimal projection is a direct summand of a direct sum of elements in $\mathcal{M}$ and hence is isomorphic to an element in $\mathcal{M}$.  That is, all the minimal projections are in $\mathcal{M}$, up to isomorphism.  The fact that the planar algebra ``2221'' has principal graph $H$ follows from Lemmas 8.0.28, 8.0.29, 8.0.30, 8.0.31, and 8.0.32.\\
\indent To show that the dual principal graph is $H$, we make an appropriate choice of two generators $T_0 , Q_0\in \mathcal{P}_{3,-}$ and minimal projections $S_0,P_0,R_0, M_{1_0}, M_{2_0}$.\\
\indent Let $T_0=\rho^{1/2}(T)$ and $Q_0=\omega\cdot\rho^{1/2}(Q)$.  Then $T_0,Q_0$ are self-adjoint, uncappable, rotational eigenvectors with eigenvalues $1$ and $\omega$, respectively.  All the trace values of Lemma 5.0.14 hold for $T_0,Q_0$.  It follows that the 3-box relations of Theorem 5.0.16 hold for $T_0,Q_0$ (ignoring the twisted versions).  Lemmas 5.2.2,5.2.4,5.2.6,5.2.7,5.2.8 all hold for $T_0,Q_0$.  Hence, the 4-box relations hold for $T_0,Q_0$.  It follows that the one-strand braiding substitutes hold as well.  The jellyfish algorithm applies.  That is, Lemmas 7.0.5 and 7.0.7 hold for $T_0,Q_0$.  Consequently, Theorem 7.0.8 holds for $\mathcal{P}_{0,-}$ , Theorem 7.0.10 holds for $\mathcal{P}_{1,-}$ , Theorem 7.0.11 holds for $\mathcal{P}_{2,-}$ , and Theorem 7.0.12 holds for $\mathcal{P}_{3,-}$ with $T,Q$ replaced by $T_0,Q_0$.\\
\indent Let $S_0,P_0,R_0,M_{1_0},M_{2_0}$ be $S,P,R,M_1,M_2$ ,respectively, with $T,Q$ replaced by $T_0,Q_0$.  Then $\JW{0},\JW{1},\JW{2},S_0,P_0,R_0,M_{1_0},M_{2_0}$ are minimal projections and the dual principal graph is ``2221''.  The proof is analogous to the proof for the principal graph-- just replace every $T,Q$ with $T_0,Q_0$ and replace $S,P,R,M_1,M_2$ with $S_0,P_0,R_0,M_{1_0},M_{2_0}$.
\end{proof}
\begin{lem}The Jones-Wenzl idempotents $\JW{k}$, for $k=0,1,2$ are minimal.
\end{lem}
\begin{proof}$\JW{0}$ is minimal because $\Hom{\JW{0},\JW{0}}$ is one-dimensional.  That is, every closed diagram $x\in \mathcal{PA}(T,Q)_0$ is a multiple of the empty diagram.\\
\indent $\Hom{\JW{i},\JW{i}}$ consists of all diagrams of the form 

$$
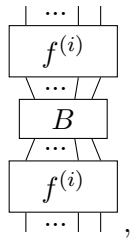
\begin{tikzpicture}
   \clip (-1,-1.5) rectangle (.8,1.5);
   \node at (0,.9) (top) [rectangle,draw] { $\hspace{.3 cm} \JW{i} \hspace{.3cm} $};
   \node at (0,0) (mid) [rectangle,draw] {$\hspace{.3 cm} B \hspace{.3cm} $};
   \node at (0,-.9) (bottom) [rectangle,draw]{$ \hspace{.3cm} \JW{i}
\hspace{.3cm} $};
	 \draw (top.-60)--(mid.60);
   \draw (top.-145)--(mid.145);
   \draw (top.-35)--(mid.35);
   \node at (-.1,.4) {...};
   \node at (-.1,-.4) {...};
   \draw (mid.-60)--(bottom.60);
   \draw (mid.-145)--(bottom.145);
   \draw (mid.-35)--(bottom.35);
   \draw (top.35) -- ++(90:1cm);
   \draw (top.60) -- ++(90:1cm);
   \node at (-.1,1.4) {...};
   \draw (top.145) -- ++(90:1cm);

   \draw (bottom.-35) -- ++(-90:1cm);
   \draw (bottom.-60) -- ++(-90:1cm);
   \node at (-.1,-1.4) {...};
   \draw (bottom.-145) -- ++(-90:1cm);
\end{tikzpicture},$$
for some diagram $B$ in $\mathcal{P}(T,Q)_i$.  We will see what happens as we plug in various diagrams for $B$.
\begin{claim}Any diagram of the above form which is non-zero is equal to a multiple of the diagram obtained by inserting the identity in for $B$.
\end{claim}
\begin{proof}$\mathcal{P}_1=TL_1$ and $\mathcal{P}_2=TL_2$ by Lemmas 7.0.10 and 7.0.11.  If we plug in a TL diagram, then the diagram must not have any caps/cups because the $\JW{i}$'s will give 0.  Hence, any TL diagram we plug in must be a multiple of the identity.
\end{proof}
Hence, each Hom space $\Hom{\JW{k},\JW{k}}$ is one-dimensional, for $k=0,1,2$.
\end{proof}
\begin{lem}The projections $S,P,R$ are minimal.
\end{lem}
\begin{proof}Consider $P$ first.  Consider $\Hom{P,P}$, which consists of all diagrams of the form 
$$
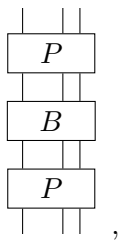
\begin{tikzpicture}
   \clip (-1,-1.5) rectangle (.8,1.5);
   \node at (0,.9) (top) [rectangle,draw] { $\hspace{.3 cm} P \hspace{.3cm} $};
   \node at (0,0) (mid) [rectangle,draw] {$\hspace{.3 cm} B \hspace{.3cm} $};
   \node at (0,-.9) (bottom) [rectangle,draw]{$ \hspace{.3cm} P
\hspace{.3cm} $};
	 \draw (top.-60)--(mid.60);
   \draw (top.-145)--(mid.145);
   \draw (top.-35)--(mid.35);
   \draw (mid.-60)--(bottom.60);
   \draw (mid.-145)--(bottom.145);
   \draw (mid.-35)--(bottom.35);
   \draw (top.35) -- ++(90:1cm);
   \draw (top.60) -- ++(90:1cm);
   \draw (top.145) -- ++(90:1cm);

   \draw (bottom.-35) -- ++(-90:1cm);
   \draw (bottom.-60) -- ++(-90:1cm);
   \draw (bottom.-145) -- ++(-90:1cm);
\end{tikzpicture},$$
for some diagram $B$ in $\mathcal{P}(T,Q)_3$.
\begin{claim}Any such diagram (which is non-zero) is equal to a multiple of the diagram with the identity inserted in for $B$.  
\end{claim}
\begin{proof}First, consider inserting a Temperley-Lieb diagram.  The TL diagram must not have any caps/cups on it because $P$ with a cap/cup would give 0.  Hence, the TL diagram must be a multiple of the identity.  Now, consider any diagram with exactly one $T$ or $Q$ (since any diagram in $\mathcal{P}_3$ is in $TL_3\oplus ATL_3(T)\oplus ATL_3(Q)$).  Suppose we had a diagram with exactly one $T$.  Then we have (up to rotation)  
$$
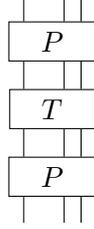
\begin{tikzpicture}
   \clip (-1,-1.5) rectangle (.8,1.5);
   \node at (0,.9) (top) [rectangle,draw] { $\hspace{.3 cm} P \hspace{.3cm} $};
   \node at (0,0) (mid) [rectangle,draw] {$\hspace{.3 cm} T \hspace{.3cm} $};
   \node at (0,-.9) (bottom) [rectangle,draw]{$ \hspace{.3cm} P
\hspace{.3cm} $};
	 \draw (top.-60)--(mid.60);
   \draw (top.-145)--(mid.145);
   \draw (top.-35)--(mid.35);
   \draw (mid.-60)--(bottom.60);
   \draw (mid.-145)--(bottom.145);
   \draw (mid.-35)--(bottom.35);
   \draw (top.35) -- ++(90:1cm);
   \draw (top.60) -- ++(90:1cm);
   \draw (top.145) -- ++(90:1cm);

   \draw (bottom.-35) -- ++(-90:1cm);
   \draw (bottom.-60) -- ++(-90:1cm);
   \draw (bottom.-145) -- ++(-90:1cm);
\end{tikzpicture}$$
But $PT=\frac{\beta}{\alpha}P$, where $P=\alpha \JW{3}+\beta T+\gamma Q$.  To see this, we have that $$PT=\Big(\alpha +\beta\frac{Z(T^3)}{[4]}+\gamma\frac{Z(T^2Q)}{[4]}\Big)T+\Big(\beta\frac{Z(QT^2)}{[4]}+\gamma\frac{Z(Q^2T)}{[4]}\Big)Q+\beta\JW{3}$$ and $$\beta=\frac{\alpha}{\beta}\Big( \alpha+\beta\frac{Z(T^3)}{[4]}+\gamma\frac{Z(T^2Q)}{[4]}\Big),$$ $$\gamma=\frac{\alpha}{\beta}\Big(\beta\frac{Z(QT^2)}{[4]}+\gamma\frac{Z(Q^2T)}{[4]}\Big).$$

So we get $PTP=\frac{\beta}{\alpha}P P= \frac{\beta}{\alpha}P$, which is a multiple of the diagram with identity inserted. \\
Similarly, we have that $PQ=\frac{\gamma}{\alpha}P$. To see this we check that $$\beta=\frac{\alpha}{\gamma}\Big(\beta\frac{Z(T^2Q)}{[4]}+\gamma\frac{Z(TQ^2)}{[4]}\Big),$$
and $$\gamma=\frac{\alpha}{\gamma}\Big(\alpha+\beta\frac{Z(TQ^2)}{[4]}+\gamma\frac{Z(Q^3)}{[4]}\Big).$$ So we get $PQP=\frac{\gamma}{\alpha}PP=\frac{\gamma}{\alpha}P$, which is a multiple of the diagram with identity inserted.
\end{proof}
Similarly, for $S$ and $R$, we have that $$ST=\frac{\beta'}{\alpha'}S, \hspace{2mm}SQ=\frac{\gamma'}{\alpha'}S,$$
 and $$RT=\frac{\beta_0}{\alpha_0}R, \hspace{2mm}RQ=\frac{\gamma_0}{\alpha_0}R.$$  The proofs for $S$ and $R$ are analogous.
\end{proof}
\begin{lem}$M_1$ and $M_2$ are minimal.
\end{lem}
\begin{proof}Consider $\Hom{M_1,M_1}$.  Let $D$ be a morphism from $M_1$ to $M_1$.  $D$ consists of a diagram $D_0$ in $\Hom{id_8,\id}$ standing on $M_1\otimes\overline{M_1}$.  Assume $D_0$ is in jellyfish form.  Assume no cups/caps attached to a copy of a generator and no closed diagrams.  If $D_0$ is a TL element, view it as an element in $TL_4$.  Any non-identity element caps off $M_1$ and hence gives 0.  The identity element gives $M_1^2=M_1$.\\
\indent Suppose there is a strand that connects two distinct non-adjacent copies of generators.  Then we can find a vertex $S_0$ which is connected solely to its two adjacent neighbors.  $S_0$ is connected to one of its neighbors by at least 3 strands.  Apply the relevant quadratic relation to get a linear combination of diagrams that have fewer copies of generators.  By induction, each of these is a multiple of $M_1$ and hence so is $D$.\\
\indent Now, suppose there is no strand that connects two distinct non-adjacent copies of generators.  Then every strand attached to a copy of a generator connects to the bottom of $D_0$ or connects two adjacent copies of generators.\\
\indent If all 8 endpoints at the bottom of $D_0$ are connected solely to the bottom of $D_0$, then $D_0$ is in TL and we're done by our analysis of the case where $D_0$ lies in TL.  \\
\indent If there are exactly 3 caps on the bottom of $D_0$, the only way that this can happen so that $M_1$ doesn't immediately get capped off is if the caps occur at positions (2,7),(3,6), and (4,5).  In this case, the remaining part of $D_0$ lies in $\mathcal{P}_1$ and hence $D_0$ is in TL and we're done.\\
\indent If there are exactly two caps on the bottom of $D_0$, the only way that this can happen without immediately capping off $M_1$ is if the caps occur at positions (3,6) and (4,5).  In this case, the remaining part of $D_0$ lies in $\mathcal{P}_2=TL_2$.  Hence, $D_0$ lies in TL and we're done.\\
\indent If there is exactly one cap on the bottom of $D_0$, the only way this can happen without immediately capping off $M_1$ is if the cap occurs at position (4,5).  In this case, the remaining part $R_0$ of $D_0$ lies in $\mathcal{P}_3=TL_3\oplus ATL_3(T)\oplus ATL_3(Q)$.  The component of $R_0$ in $TL_3$ gives $D_0$ as an element in TL and we're done.  The component of $R_0$ in $ATL_3(T)$ gives $D$ itself as a multiple of $M_1$.  To see this, expand $M_1$ and $\overline{M_1}$ and use the fact that $TP=\frac{\beta}{\alpha}P$ and the fact that $P$ is a projection to get $\frac{\beta}{\alpha}M_1$.  Similarly, for the component of $R_0$ in $ATL_3(Q)$, use the fact that $QP=\frac{\gamma}{\alpha}P$ to get $D$ as a multiple of $\frac{\gamma}{\alpha}M_1$.\\
\indent If there are no caps at the bottom of $D_0$, then each endpoint at the bottom of $D_0$ connects to a copy of a generator.  Consider the first endpoint on the bottom of $D_0$.  Let $S_1$ be the copy of a generator to which it is connected.  There cannot be a strand from $S_1$ connecting to any copies of generators to the left of $S_1$ since then there would be a leftmost vertex which forms a closed diagram with its adjacent neighbor.  If $S_1$ connects to an adjacent neighbor on its right by at least 3 parallel strands, then we can apply a quadratic relation and induction to get $D$ as a multiple of $M_1$.  If not, then either all strands of $S_1$ connect to the bottom of $D_0$ or $S_1$ connects to an adjacent neighbor $S_2$ to the right of $S_1$ by 1 or 2 strands.  \\
\indent If all strands of $S_1$ connect to the bottom of $D_0$, then there are two remaining endpoints at the bottom of $D_0$.  Consider the very last endpoint on the bottom of $D_0$, and suppose it connects to $S_3$.  There are no strands from $S_3$ connecting to any copies of generators to the right of $S_3$.  $S_3$ must have an adjacent neighbor on its left to which it is connected by at least 4 strands.  Apply a quadratic relation and induction to get $D$ as a multiple of $M_1$.  \\
\indent If $S_1$ connects to an adjacent neighbor $S_2$ to the right of $S_1$ by 1 strand, then $S_1$ connects to the bottom of $D_0$ by 5 strands.  There are 3 remaining endpoints on the bottom of $D_0$.  Consider the very last endpoint on the bottom of $D_0$, and suppose it connects to $S_3$.  $S_3$ must have an adjacent neighbor on its left to which it is connected by at least 3 parallel strands.  Apply a quadratic relation and induction to get $D$ as a multiple of $M_1$.  \\
\indent If $S_1$ connects to an adjacent neighbor $S_2$ to the right of $S_1$ by 2 strands, then $S_1$ connects to the bottom of $D_0$ by 4 strands.  Consider the very last endpoint on the bottom of $D_0$ and suppose it connects to $S_3$.  $S_3$ connects to the bottom of $D_0$ by either, 1,2,3, or 4 strands.  If $S_3$ connects to the bottom by 1,2, or 3 strands, then $S_3$ must have an adjacent neighbor on its left to which it is connected by at least 3 parallel strands.  Apply a quadratic relation and induction to get $D$ as a multiple of $M_1$.  If $S_3$ connects to the bottom of $D_0$ by 4 strands, then $S_3$ has an adjacent neighbor $S_4$ on its left to which it is connected by 2 strands.  $S_4$ is either equal to $S_1$ or is connected to an adjacent neighbor by 4 strands.  If the latter, we're done by applying a quadratic relation and induction.  If $S_4$ is equal to $S_1$, then $D$ evaluates to 0.  To see this, expand out $M_1$ and use the fact that $TP=\frac{\beta}{\alpha}P$ (or $QP=\frac{\gamma}{\alpha}P$) and the partial trace relation.\\
\indent The proof for $M_2$ is analogous.  Use the fact that $TR=\frac{\beta_0}{\alpha_0}R$ (or $QR=\frac{\gamma_0}{\alpha_0}R$).
\end{proof}
\begin{lem}If $A$ and $B$ are two distinct projections from the set  $$\{\JW{0},\JW{1},\JW{2},S,P,R,M_1,M_2\},$$  then $\Hom{A,B}=0$.
\end{lem}
\begin{proof}Suppose $A$ and $B$ are distinct Jones-Wenzl projections.  Then $\Hom{\JW{0},\JW{1}}=0$ and $\Hom{\JW{1},\JW{2}}=0$ because any morphism in $\Hom{\JW{0},\JW{1}}$ can be seen as lying in $\Hom{id_1,\id}$, which is 0 by convention and any morphism in $\Hom{\JW{1},\JW{2}}=0$ can be seen as lying in $\Hom{id_3,\id}$, which is 0 by convention.  In fact, $\Hom{id_n,\id}$ is 0 for any n odd.  It remains to show that $\Hom{\JW{0},\JW{2}}$ is 0.  Let $D$ be a morphism from $\JW{0}$ to $\JW{2}$.  Then $D$ can be seen as a diagram $D_0\in \Hom{id_2,\id}$ standing on top of $\JW{2}$.  $D_0$ is in $\mathcal{P}_1$ and hence $D_0$ is a multiple of the single strand.  This puts a cap on $\JW{2}$ and so we get 0.\\
\indent Suppose $A$ is a Jones-Wenzl projection, while $B$ is $S, P,\mbox{ or } R$.   If $A=\JW{0}$, then any morphism $D$ from $A$ to $B$ lies in $\Hom{id_3,\id}$, and hence is 0.  If $A=\JW{2}$, then any morphism $D$ from $A$ to $B$ lies in $\Hom{id_5,\id}$, and hence is 0.  If $A=\JW{1}$ and $B\in \{S,P,R\}$, then let $D$ be a morphism from $A$ to $B$.  $D$ consists of $D_0\in\Hom{id_4,\id}$ standing on top of $B\otimes id_1$. Assume $D_0$ is in jellyfish form. $D_0$ lies in $\mathcal{P}_2=TL_2$, and hence $D_0$ puts a cap on $B$.  So we get 0.\\
\indent Suppose $A=\JW{0}$ and $B\in\{M_1,M_2\}$.  Assume $D$ is a morphism from $A$ to $B$.  $D$ consists of a diagram $D_0\in\Hom{id_4,\id}$ standing on $B$.  Assume $D_0$ is in jellyfish form.  $D_0$ lies in $\mathcal{P}_2=TL_2$, and hence $D_0$ puts a cap on $P$ in $M_1$ (or a cap on $R$ in $M_2$).  So we get 0.  \\
\indent Suppose $A=\JW{1}$ and $B\in\{M_1,M_2\}$.  Then any morphism $D$ from $A$ to $B$ lies in $\Hom{id_5,\id}$ and hence is 0.\\
\indent Suppose $A=\JW{2}$ and $B\in\Hom{M_1,M_2}$.  Say $B=M_1$.  Let $D$ be a morphism from $A$ to $B$.  Then $D$ consists of a diagram $D_0\in\Hom{id_6,\id}$ standing on top of $B\otimes\JW{2}$.  Assume $D_0$ is in jellyfish form.  $D_0$ lies in $\mathcal{P}_3=TL_3\oplus ATL_3(T)\oplus ATL_3(Q)$.  Any TL element caps off $M_1$ and hence gives 0.  Any $ATL_3(T)$ element gives 0, by a simple calculation using the fact that $TP=\frac{\beta}{\alpha}P$ and the partial trace relation.  Similarly, any $ATL_3(Q)$ element gives 0, using the fact that $QP=\frac{\gamma}{\alpha}P$ and the partial trace relation.  The case $B=M_2$ is similar; use the fact that $TR=\frac{\beta_0}{\alpha_0}R$ and $QR=\frac{\gamma_0}{\alpha_0}R$.  Thus, $D$ is 0.\\
\indent Suppose $A\in\{S,P,R\}$ and $B\in\{S,P,R\}$ and $A\neq B$.  Let $D$ be a morphism from $A$ to $B$.  Then $D$ consists of $A$, $B$, and a diagram $D_0\in\mathcal{P}_3$ between them.  Since $D_0\in\mathcal{P}_3=TL_3\oplus ATL_3(T)\oplus ATL_3(Q)$, consider what happens to any diagram in $TL_3$, $ATL_3(T)$, and $ATL_3(Q)$.  Any $TL_3$ diagram either caps off $A$ and $B$ or is the identity, in which case we get $AB=0$ ( since $AB=0$ for any $A,B\in\{S,P,R\}$ with $A\neq B$).  Any diagram in $ATL_3(T)$ gives a multiple of $ATB$, which is a multiple of $AB=0$ since $PT=\frac{\beta}{\alpha}P$ (and similarly for $S,R$).  Any diagram in $ATL_3(Q)$ gives a multiple of $AQB$, which is a multiple of $AB=0$ since $PQ=\frac{\gamma}{\alpha}P$ (and similarly for $S,R$).\\
\indent Suppose $A\in\{S,P,R\}$ and $B\in\{M_1,M_2\}$.  Any morphism $D$ from $A$ to $B$ lies in $\Hom{id_7,\id}$ and hence is 0.\\
\indent Suppose $A=M_1$ and $B=M_2$.  Consider a morphism $D$ from $A$ to $B$.  $D$ consists of a diagram $D_0$ in $\Hom{id_8,\id}$ standing on $M_1\otimes\overline{M_2}$.  Assume $D_0$ is in jellyfish form.  Assume no caps/cups are attached to any copy of a generator and no closed diagrams.  \\
\indent If $D_0$ lies in TL, view it as an element in $TL_4$.  Any non-identity element caps off $M_1$ and hence gives 0.  The identity element gives $M_2M_1=0$.  \\
\indent Supose there is a strand that connects two distinct non-adjacent copies of generators.  Then we can find a vertex $S_0$ which is connected solely to its two adjacent neighbors.  $S_0$ is connected to one of its neighbors by at least 3 parallel strands.  Apply the relevant quadratic relation to get a linear combination of diagrams that have fewer copies of generators.  By induction, each of these is 0.  Hence, so is $D$.\\
\indent Now, suppose there is no strand that connects two distinct non-adjacent vertices.  Then every strand attached to a copy of a generator connects to the bottom of $D_0$ or connects two adjacent copies of generators. \\
\indent If all 8 endpoints at the bottom of $D_0$ are connected solely to the bottom of $D_0$, then $D_0$ is in TL, and we're done by our analysis of the case where $D_0$ is in TL.\\
\indent If there are exactly 3 caps on the bottom of $D_0$, the only way that this can happen so that $M_1$ and $\overline{M_2}$ don't immediately get capped off is if the caps occur at positions (2,7),(3,6),(4,5).  In this case, the remaining part of $D_0$ lies in $\mathcal{P}_1$ and hence $D_0$ is in TL, and we're done.\\
\indent If there are exactly 2 caps on the bottom of $D_0$, the only way that this can happen without immediately capping off $M_1$ or $\overline{M_2}$ is if the caps occur at positions (3,6) and (4,5).  In this case, the remaining part of $D_0$ lies in $\mathcal{P}_2=TL_2$.  Hence, $D_0$ lies in TL and we're done.\\
\indent If there is exactly one cap, the only way this can happen without immediately capping off $M_1$ or $\overline{M_2}$ is if the cap occurs at position (4,5).  In this case, the remaining part $R_0$ of $D_0$ lies in $\mathcal{P}_3=TL_3\oplus ATL_3(T)\oplus ATL_3(Q)$.  The component of $R_0$ in $TL_3$ gives $D_0$ as an element in TL and we're done.  The component of $R_0$ in $ATL_3(T)$ gives $D$ itself as 0.  To see this, expand $M_1$ and $\overline{M_2}$ and use the fact that $TP=\frac{\beta}{\alpha}P$ and the fact that $RP=0$.  Similarly, for the component of $R_0$ in $ATL_3(Q)$, use the fact that $QP=\frac{\gamma}{\alpha}P$ to get $D$ as 0.\\
\indent If there are no caps at the bottom of $D_0$, then each endpoint at the bottom of $D_0$ connects to a copy of a generator.  Consider the first endpoint at the bottom of $D_0$.  Let $S_1$ be the copy of a generator to which it is connected.  There cannot be a strand from $S_1$ connecting to any copies of generators to the left of $S_1$ since then there would be a leftmost vertex which forms a closed diagram with its adjacent neighbor.  If $S_1$ connects to an adjacent neighbor on its right by at least 3 parallel strands, then we can apply a quadratic relation and induction to get $D$ as 0.  If not, then either all strands of $S_1$ connect to the bottom of $D_0$ or $S_1$ connects to an adjacent neighbor $S_2$ to the right of $S_1$ by 1 or 2 strands.  \\
\indent If all strands of $S_1$ connect to the bottom of $D_0$, then there are two remaining endpoints at the bottom of $D_0$.  Consider the very last endpoint on the bottom of $D_0$, and suppose it connects to $S_3$.  There are no strands from $S_3$ connecting to any copies of generators to the right of $S_3$.  $S_3$ must have an adjacent neighbor on its left to which it is connected by at least 4 strands.  Apply a quadratic relation and induction to get $D$ as 0. \\
\indent If $S_1$ connects to an adjacent neighbor $S_2$ to the right of $S_1$ by 1 strand, then $S_1$ connects to the bottom of $D_0$ by 5 strands.  There are 3 remaining endpoints at the bottom of $D_0$.  Consider the very last endpoint on the bottom of $D_0$, and suppose it connects to $S_3$.  $S_3$ must have an adjacent neighbor on its left to which it is connected by at least 3 parallel strands.  Apply a quadratic relation and induction to get $D$ as 0.\\
\indent If $S_1$ connects to an adjacent neighbor $S_2$ to the right of $S_1$ by 2 strands, then $S_1$ connects to the bottom of $D_0$ by 4 strands.  Consider the very last endpoint on the bottom of $D_0$, and suppose it connects to $S_3$.  $S_3$ connects to the bottom of $D_0$ by either 1,2,3 or 4 strands.  If $S_3$ connects to the bottom by 1,2, or 3 strands, then $S_3$ must have an adjacent neighbor on its left to which it is connected by at least 3 parallel strands.  Apply a quadratic relation and induction to get $D$ as 0.  If $S_3$ connects to the bottom of $D_0$ by 4 strands, then $S_3$ has an adjacent neighbor $S_4$ on its left to which it is connected by 2 strands.  $S_4$ is either equal to $S_1$ or is connected to an adjacent neighbor by 4 strands.  If the latter, we're done by applying a quadratic relation and induction.  If $S_4$ is equal to $S_1$, then $D$ evaluates to 0.  To see this, expand out $M_1$ and use the fact that $TP=\frac{\beta}{\alpha}P$ ( or $QP=\frac{\gamma}{\alpha}P$) and the partial trace relation.
\end{proof}
We have the following well-known fact, which we state without proof.  The proof can be found in \cite{D2n}.
\begin{lem}$\JW{k}\otimes \JW{1}$ is isomorphic to $\JW{k-1}\oplus\JW{k+1}$ for $k=1$.
\end{lem}
\noindent Of course, $\JW{0}\otimes id_1\cong \JW{1}$.
\begin{lem}$\JW{2}\otimes id_1\cong\JW{1}\oplus S\oplus P\oplus R$.
\end{lem}
\begin{proof} Let $f:\JW{2}\otimes id_1\rightarrow \JW{1}\oplus S\oplus P\oplus R$ be the morphism given by 
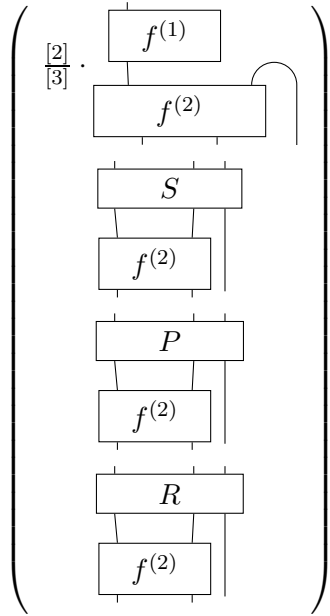
\[\left(\begin{array}{c}
\vspace{2mm}
\frac{[2]}{[3]}\cdot \begin{tikzpicture}[baseline]
    \node at (0,.5) (top) [rectangle,draw] {$ \hspace{.3cm}\JW{1} \hspace{.3cm}$};
    \node at (.2,-.5) (mid) [rectangle, draw] {$ \hspace{.7cm} \JW{2} \hspace{.7cm} $};
    \draw (top.-145)--(mid.153);
    \draw (mid.20) arc (180:0:.3cm) -- ++(-90:.8cm);
    \draw (mid.-35) -- ++(-90:1mm);
    \draw (mid.-145) -- ++(-90:1mm);
    \draw (top.145) -- ++(90:1mm);
\end{tikzpicture}\\
\vspace{2mm}
\begin{tikzpicture}[baseline]
    \node at (.2,.5) (mid) [rectangle, draw] {$ \hspace{.7cm} S \hspace{.7cm} $};
    \node at (0,-.5) (bottom) [rectangle,draw]{$ \hspace{.3cm}\JW{2} \hspace{.3cm}$};
    \draw (mid.-40)--(bottom.35);
    \draw (mid.40) -- ++(90:1mm);
    \draw (mid.160) -- ++(90:1mm);
    \draw (mid.20) -- ++(90:1mm);
    \draw (mid.-20) -- ++(-90:1.1cm);
    \draw (mid.-160)--(bottom.145);
    \draw (bottom.-35) -- ++(-90:1mm);
    \draw (bottom.-145) -- ++(-90:1mm);
\end{tikzpicture}\\
\vspace{2mm}
\begin{tikzpicture}[baseline]
    \node at (.2,.5) (mid) [rectangle, draw] {$ \hspace{.7cm} P \hspace{.7cm} $};
    \node at (0,-.5) (bottom) [rectangle,draw]{$ \hspace{.3cm} \JW{2} \hspace{.3cm} $};
    \draw (mid.-40)--(bottom.35);
    \draw (mid.40) -- ++(90:1mm);
    \draw (mid.160) -- ++(90:1mm);
    \draw (mid.20) -- ++(90:1mm);
    \draw (mid.-20) -- ++(-90:1.1cm);
    \draw (mid.-160)--(bottom.145);
    \draw (bottom.-35) -- ++(-90:1mm);
    \draw (bottom.-145) -- ++(-90:1mm);
\end{tikzpicture}\\
\vspace{2mm}
\begin{tikzpicture}[baseline]
    \node at (.2,.5) (mid) [rectangle, draw] {$ \hspace{.7cm} R \hspace{.7cm} $};
    \node at (0,-.5) (bottom) [rectangle,draw]{$ \hspace{.3cm} \JW{2} \hspace{.3cm}$};
    \draw (mid.-40)--(bottom.35);
    \draw (mid.40) -- ++(90:1mm);
    \draw (mid.160) -- ++(90:1mm);
    \draw (mid.20) -- ++(90:1mm);
    \draw (mid.-20) -- ++(-90:1.1cm);
    \draw (mid.-160)--(bottom.145);
    \draw (bottom.-35) -- ++(-90:1mm);
    \draw (bottom.-145) -- ++(-90:1mm);
\end{tikzpicture}\end{array}\right)\]
Let $g:\JW{1}\oplus S\oplus P\oplus R\rightarrow \JW{2}\otimes id_1$ be given by
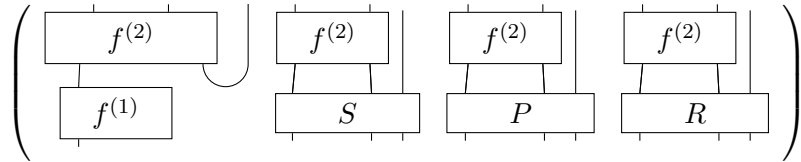
\[\left(\begin{array}{cccc}
\vspace{2mm}
\begin{tikzpicture}[baseline]
    \node at (.2,.5) (mid) [rectangle, draw] {$ \hspace{.7cm} \JW{2} \hspace{.7cm} $};
    \node at (0,-.5) (bottom) [rectangle,draw]{$ \hspace{.3cm}\JW{1} \hspace{.3cm}$};
    \draw (mid.-20) arc (-180:0:.3cm) -- ++(90:.8cm);
    \draw (mid.-153)--(bottom.145);
    \draw (mid.35) -- ++(90:1mm);
    \draw (mid.145) -- ++(90:1mm);
    \draw (bottom.-145) -- ++(-90:1mm);
\end{tikzpicture} & 
\begin{tikzpicture}[baseline]
    \node at (0,.5) (top) [rectangle,draw] {$ \hspace{.3cm}\JW{2}\hspace{.3cm} $};
    \node at (.2,-.5) (mid) [rectangle, draw] {$ \hspace{.7cm} S \hspace{.7cm} $};
    \draw (top.-35)--(mid.40);
    \draw (top.-145)--(mid.160);
    \draw (mid.-40) -- ++(-90:1mm);
    \draw (mid.-160) -- ++(-90:1mm);
    \draw (mid.-20) -- ++(-90:1mm);
    \draw (top.-35)--(mid.40);
     \draw (top.-145)--(mid.160);
      \draw (top.35) -- ++(90:1mm);
    \draw (top.145) -- ++(90:1mm);
    \draw (mid.20) -- ++(90:1.1cm);
\end{tikzpicture} & 
\begin{tikzpicture}[baseline]
    \node at (0,.5) (top) [rectangle,draw] {$ \hspace{.3cm}\JW{2}\hspace{.3cm} $};
    \node at (.2,-.5) (mid) [rectangle, draw] {$ \hspace{.7cm} P \hspace{.7cm} $};
    \draw (top.-35)--(mid.40);
    \draw (top.-145)--(mid.160);
    \draw (mid.-40) -- ++(-90:1mm);
    \draw (mid.-160) -- ++(-90:1mm);
    \draw (mid.-20) -- ++(-90:1mm);
    \draw (top.-35)--(mid.40);
     \draw (top.-145)--(mid.160);
      \draw (top.35) -- ++(90:1mm);
    \draw (top.145) -- ++(90:1mm);
    \draw (mid.20) -- ++(90:1.1cm);
\end{tikzpicture} &
\begin{tikzpicture}[baseline]
    \node at (0,.5) (top) [rectangle,draw] {$ \hspace{.3cm}\JW{2}\hspace{.3cm} $};
    \node at (.2,-.5) (mid) [rectangle, draw] {$ \hspace{.7cm} R \hspace{.7cm} $};
    \draw (top.-35)--(mid.40);
    \draw (top.-145)--(mid.160);
    \draw (mid.-40) -- ++(-90:1mm);
    \draw (mid.-160) -- ++(-90:1mm);
    \draw (mid.-20) -- ++(-90:1mm);
    \draw (top.-35)--(mid.40);
     \draw (top.-145)--(mid.160);
      \draw (top.35) -- ++(90:1mm);
    \draw (top.145) -- ++(90:1mm);
    \draw (mid.20) -- ++(90:1.1cm);
\end{tikzpicture}\end{array}\right)\]
Then $gf=\JW{2}\otimes id_1=id_{\JW{2}\otimes id_1}$, by Wenzl's relation.  $fg=id_{\JW{1}\oplus S\oplus P\oplus R}$  follows from the fact that $SP=PS=SR=RS=PR=RP=0$.  Hence, $f$ and $g$ are inverses of each other.
\end{proof}
\begin{lem}$S\otimes id_1\cong\JW{2}$.
\end{lem}
\begin{proof}Let $f:S\otimes id_1\rightarrow\JW{2}$ be given by

$$\frac{1}{\alpha'\frac{[4]}{[3]}}\cdot 
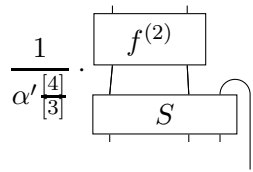
\begin{tikzpicture}[baseline]
    \node at (0,.5) (top) [rectangle,draw] {$ \hspace{.3cm}\JW{2}\hspace{.3cm} $};
    \node at (.2,-.5) (mid) [rectangle, draw] {$ \hspace{.7cm} S \hspace{.7cm} $};
    \draw (top.-35)--(mid.40);
    \draw (top.-145)--(mid.160);
    \draw (mid.-40) -- ++(-90:1mm);
    \draw (mid.-160) -- ++(-90:1mm);
    \draw (mid.-20) -- ++(-90:1mm);
    \draw (top.-35)--(mid.40);
     \draw (top.-145)--(mid.160);
      \draw (top.35) -- ++(90:1mm);
    \draw (top.145) -- ++(90:1mm);
    \draw (mid.20) arc (180:0:.2cm) -- ++(-90:1cm);

\end{tikzpicture}$$
Let $g:\JW{2}\rightarrow S\otimes id_1$ be given by
$$\begin{tikzpicture}[baseline]
    \node at (.2,.5) (mid) [rectangle, draw] {$ \hspace{.7cm} S \hspace{.7cm} $};
    \node at (0,-.5) (bottom) [rectangle,draw]{$ \hspace{.3cm}\JW{2} \hspace{.3cm}$};
    \draw (mid.-40)--(bottom.35);
    \draw (mid.40) -- ++(90:1mm);
    \draw (mid.160) -- ++(90:1mm);
    \draw (mid.20) -- ++(90:1mm);
    \draw (mid.-20) arc (-180:0:.2cm) -- ++(90:1cm);
    \draw (mid.-160)--(bottom.145);
    \draw (bottom.-35) -- ++(-90:1mm);
    \draw (bottom.-145) -- ++(-90:1mm);
\end{tikzpicture}$$
Then $fg=\JW{2}$ by the fact that $S$ is a projection and the partial trace relation.  For $gf$, we have $$gf=\frac{1}{\alpha'\frac{[4]}{[3]}}\cdot \begin{tikzpicture}[baseline]
   \clip (-.6,-1.3) rectangle (.8,1.3);
   \node at (0,.7) (top) [rectangle,draw] {$ \hspace{.3 cm} S \hspace{.3cm} $};
   \node at (0,-.7) (bottom) [rectangle,draw]{$ \hspace{.3cm} S
\hspace{.3cm} $};

   \draw (top.-60)--(bottom.60);
   \draw (top.-145)--(bottom.145);
   \draw (top.-35) arc (-180:0:.2cm) -- ++(90:1cm);
   \draw (bottom.35) arc (180:0:.2cm) -- ++(-90:1cm);

   \draw (top.35) -- ++(90:1cm);
   \draw (top.60) -- ++(90:1cm);
   \draw (top.145) -- ++(90:1cm);

   \draw (bottom.-35) -- ++(-90:1cm);
   \draw (bottom.-60) -- ++(-90:1cm);
   \draw (bottom.-145) -- ++(-90:1cm);
\end{tikzpicture}$$
However, $\begin{tikzpicture}[baseline]
   \clip (-.6,-1.3) rectangle (.8,1.3);
   \node at (0,.7) (top) [rectangle,draw] {$ \hspace{.3 cm} S \hspace{.3cm} $};
   \node at (0,-.7) (bottom) [rectangle,draw]{$ \hspace{.3cm} S
\hspace{.3cm} $};

   \draw (top.-60)--(bottom.60);
   \draw (top.-145)--(bottom.145);
   \draw (top.-35) arc (-180:0:.2cm) -- ++(90:1cm);
   \draw (bottom.35) arc (180:0:.2cm) -- ++(-90:1cm);

   \draw (top.35) -- ++(90:1cm);
   \draw (top.60) -- ++(90:1cm);
   \draw (top.145) -- ++(90:1cm);

   \draw (bottom.-35) -- ++(-90:1cm);
   \draw (bottom.-60) -- ++(-90:1cm);
   \draw (bottom.-145) -- ++(-90:1cm);
\end{tikzpicture}$ is a scalar multiple of $S\otimes id_1$.  To see this, we let $A=\begin{tikzpicture}[baseline]
   \clip (-.6,-1.3) rectangle (.8,1.3);
   \node at (0,.7) (top) [rectangle,draw] {$ \hspace{.3 cm} S \hspace{.3cm} $};
   \node at (0,-.7) (bottom) [rectangle,draw]{$ \hspace{.3cm} S
\hspace{.3cm} $};

   \draw (top.-60)--(bottom.60);
   \draw (top.-145)--(bottom.145);
   \draw (top.-35) arc (-180:0:.2cm) -- ++(90:1cm);
   \draw (bottom.35) arc (180:0:.2cm) -- ++(-90:1cm);

   \draw (top.35) -- ++(90:1cm);
   \draw (top.60) -- ++(90:1cm);
   \draw (top.145) -- ++(90:1cm);

   \draw (bottom.-35) -- ++(-90:1cm);
   \draw (bottom.-60) -- ++(-90:1cm);
   \draw (bottom.-145) -- ++(-90:1cm);
\end{tikzpicture}$ and $B=S\otimes id_1$, and check to see that $\ip{A,A}\ip{B,B}=\ip{A,B}\overline{\ip{A,B}}$.  If so, we have that $A=\frac{\ip{A,B}}{\ip{B,B}}B$.  We calculate that $\ip{A,A}=\alpha'\frac{[4]}{[3]}Z(S), \ip{B,B}=\delta Z(S)$, and $\ip{A,B}=Z(S)$.  Since $\alpha'\frac{[4]}{[3]}\delta=1$, the equation $\ip{A,A}\ip{B,B}=\ip{A,B}\overline{\ip{A,B}}$ holds and $A=\frac{1}{\delta}B$.  Hence, $$gf=\delta \cdot \begin{tikzpicture}[baseline]
   \clip (-.6,-1.3) rectangle (.8,1.3);
   \node at (0,.7) (top) [rectangle,draw] {$ \hspace{.3 cm} S \hspace{.3cm} $};
   \node at (0,-.7) (bottom) [rectangle,draw]{$ \hspace{.3cm} S
\hspace{.3cm} $};

   \draw (top.-60)--(bottom.60);
   \draw (top.-145)--(bottom.145);
   \draw (top.-35) arc (-180:0:.2cm) -- ++(90:1cm);
   \draw (bottom.35) arc (180:0:.2cm) -- ++(-90:1cm);

   \draw (top.35) -- ++(90:1cm);
   \draw (top.60) -- ++(90:1cm);
   \draw (top.145) -- ++(90:1cm);

   \draw (bottom.-35) -- ++(-90:1cm);
   \draw (bottom.-60) -- ++(-90:1cm);
   \draw (bottom.-145) -- ++(-90:1cm);
\end{tikzpicture}=S\otimes id_1$$
It follows that $g$ and $f$ are inverses of each other.  
\end{proof}
\begin{lem}$P\otimes id_1 \cong\JW{2}\oplus M_1$ and $R\otimes id_1 \cong\JW{2}\oplus M_2$.
\end{lem}
\begin{proof}  First, we establish that $$\frac{1}{\alpha\frac{[4]}{[3]}}\cdot \begin{tikzpicture}[baseline]
   \clip (-.6,-1.3) rectangle (.8,1.3);
   \node at (0,.7) (top) [rectangle,draw] {$ \hspace{.3 cm} P \hspace{.3cm} $};
   \node at (0,-.7) (bottom) [rectangle,draw]{$ \hspace{.3cm} P
\hspace{.3cm} $};

   \draw (top.-60)--(bottom.60);
   \draw (top.-145)--(bottom.145);
   \draw (top.-35) arc (-180:0:.2cm) -- ++(90:1cm);
   \draw (bottom.35) arc (180:0:.2cm) -- ++(-90:1cm);

   \draw (top.35) -- ++(90:1cm);
   \draw (top.60) -- ++(90:1cm);
   \draw (top.145) -- ++(90:1cm);

   \draw (bottom.-35) -- ++(-90:1cm);
   \draw (bottom.-60) -- ++(-90:1cm);
   \draw (bottom.-145) -- ++(-90:1cm);
\end{tikzpicture}\cong \JW{2}$$
Let $f:\frac{1}{\alpha\frac{[4]}{[3]}}\cdot \begin{tikzpicture}[baseline]
   \clip (-.6,-1.3) rectangle (.8,1.3);
   \node at (0,.7) (top) [rectangle,draw] {$ \hspace{.3 cm} P \hspace{.3cm} $};
   \node at (0,-.7) (bottom) [rectangle,draw]{$ \hspace{.3cm} P
\hspace{.3cm} $};

   \draw (top.-60)--(bottom.60);
   \draw (top.-145)--(bottom.145);
   \draw (top.-35) arc (-180:0:.2cm) -- ++(90:1cm);
   \draw (bottom.35) arc (180:0:.2cm) -- ++(-90:1cm);

   \draw (top.35) -- ++(90:1cm);
   \draw (top.60) -- ++(90:1cm);
   \draw (top.145) -- ++(90:1cm);

   \draw (bottom.-35) -- ++(-90:1cm);
   \draw (bottom.-60) -- ++(-90:1cm);
   \draw (bottom.-145) -- ++(-90:1cm);
\end{tikzpicture}\longrightarrow \JW{2}$ be given by 
$$
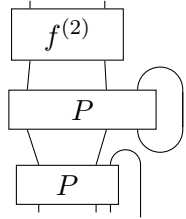
\begin{tikzpicture}[baseline]
    \node at (0,1) (top) [rectangle,draw] {$ \hspace{.3cm}\JW{2} \hspace{.3cm}$};
    \node at (.2,0) (mid) [rectangle, draw] {$ \hspace{.7cm} P \hspace{.7cm} $};
    \node at (0,-1) (bottom) [rectangle,draw]{$ \hspace{.4cm}P \hspace{.4cm}$};
    \draw (top.-35)--(mid.40);
    \draw (top.-145)--(mid.160);
    \draw (mid.-40)--(bottom.35);
    \draw (mid.20) arc (180:0:.3cm) -- ++(-90:.5cm);
    \draw (mid.-20) arc (-180:0:.3cm) -- ++(90:.5cm);
    \draw (bottom.25) arc (180:0:.2cm) --++(-90:.7cm);
    \draw (mid.-160)--(bottom.145);
      \draw (top.35) -- ++(90:.1cm);
   \draw (top.145) -- ++(90:.1cm);
    \draw (bottom.-35) -- ++(-90:.1cm);
   \draw (bottom.-25) -- ++(-90:.1cm);
   \draw (bottom.-145) -- ++(-90:.1cm);
\end{tikzpicture}$$
and let $g:\JW{2}\longrightarrow \frac{1}{\alpha\frac{[4]}{[3]}}\cdot \begin{tikzpicture}[baseline]
   \clip (-.6,-1.3) rectangle (.8,1.3);
   \node at (0,.7) (top) [rectangle,draw] {$ \hspace{.3 cm} P \hspace{.3cm} $};
   \node at (0,-.7) (bottom) [rectangle,draw]{$ \hspace{.3cm} P
\hspace{.3cm} $};

   \draw (top.-60)--(bottom.60);
   \draw (top.-145)--(bottom.145);
   \draw (top.-35) arc (-180:0:.2cm) -- ++(90:1cm);
   \draw (bottom.35) arc (180:0:.2cm) -- ++(-90:1cm);

   \draw (top.35) -- ++(90:1cm);
   \draw (top.60) -- ++(90:1cm);
   \draw (top.145) -- ++(90:1cm);

   \draw (bottom.-35) -- ++(-90:1cm);
   \draw (bottom.-60) -- ++(-90:1cm);
   \draw (bottom.-145) -- ++(-90:1cm);
\end{tikzpicture}$ be given by 
$$\frac{1}{\Big(\alpha\frac{[4]}{[3]}\Big)^2}\cdot
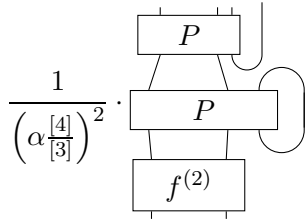
\begin{tikzpicture}[baseline]
    \node at (0,1) (top) [rectangle,draw] {$ \hspace{.4cm}P \hspace{.4cm}$};
    \node at (.2,0) (mid) [rectangle, draw] {$ \hspace{.7cm} P \hspace{.7cm} $};
    \node at (0,-1) (bottom) [rectangle,draw]{$ \hspace{.3cm}\JW{2} \hspace{.3cm}$};
    \draw (top.-35)--(mid.40);
    \draw (top.-145)--(mid.160);
    \draw (mid.-40)--(bottom.35);
    \draw (mid.20) arc (180:0:.3cm) -- ++(-90:.5cm);
    \draw (mid.-20) arc (-180:0:.3cm) -- ++(90:.5cm);
    \draw (top.-25) arc (-180:0:.2cm) --++(90:.7cm);
    \draw (mid.-160)--(bottom.145);
      \draw (top.35) -- ++(90:.1cm);
   \draw (top.145) -- ++(90:.1cm);
    \draw (bottom.-35) -- ++(-90:.1cm);
   \draw (top.25) -- ++(90:.1cm);
   \draw (bottom.-145) -- ++(-90:.1cm);
\end{tikzpicture}$$
We can check that $fg=id_{\JW{2}}$ and $gf=id_a$, where $a=\frac{1}{\alpha\frac{[4]}{[3]}}\cdot \begin{tikzpicture}[baseline]
   \clip (-.6,-1.3) rectangle (.8,1.3);
   \node at (0,.7) (top) [rectangle,draw] {$ \hspace{.3 cm} P \hspace{.3cm} $};
   \node at (0,-.7) (bottom) [rectangle,draw]{$ \hspace{.3cm} P
\hspace{.3cm} $};

   \draw (top.-60)--(bottom.60);
   \draw (top.-145)--(bottom.145);
   \draw (top.-35) arc (-180:0:.2cm) -- ++(90:1cm);
   \draw (bottom.35) arc (180:0:.2cm) -- ++(-90:1cm);

   \draw (top.35) -- ++(90:1cm);
   \draw (top.60) -- ++(90:1cm);
   \draw (top.145) -- ++(90:1cm);

   \draw (bottom.-35) -- ++(-90:1cm);
   \draw (bottom.-60) -- ++(-90:1cm);
   \draw (bottom.-145) -- ++(-90:1cm);
\end{tikzpicture}$.  The claim that $fg=id_{\JW{2}}$ follows from the fact that $P$ is a projection, along with the partial trace relation used three times.  The claim that $gf=id_a$ follows from the fact that $\JW{2}$ is a projection, along with the partial trace relation used twice.  We have thus established that $a\cong \JW{2}$.  We now prove the lemma.\\
Recall that $$M_1=P\otimes id_1 - \frac{1}{\alpha\frac{[4]}{[3]}} \cdot \begin{tikzpicture}[baseline]
   \clip (-.6,-1.3) rectangle (.8,1.3);
   \node at (0,.7) (top) [rectangle,draw] {$ \hspace{.3 cm} P \hspace{.3cm} $};
   \node at (0,-.7) (bottom) [rectangle,draw]{$ \hspace{.3cm} P
\hspace{.3cm} $};

   \draw (top.-60)--(bottom.60);
   \draw (top.-145)--(bottom.145);
   \draw (top.-35) arc (-180:0:.2cm) -- ++(90:1cm);
   \draw (bottom.35) arc (180:0:.2cm) -- ++(-90:1cm);

   \draw (top.35) -- ++(90:1cm);
   \draw (top.60) -- ++(90:1cm);
   \draw (top.145) -- ++(90:1cm);

   \draw (bottom.-35) -- ++(-90:1cm);
   \draw (bottom.-60) -- ++(-90:1cm);
   \draw (bottom.-145) -- ++(-90:1cm);
\end{tikzpicture}$$
So 
\begin{align*}
P\otimes id_1 &=M_1+a\\
&\cong M_1\oplus a \mbox{\indent  because $M_1$ and $a$ are orthogonal and Lemma 8.0.17}\\
&\cong M_1\oplus \JW{2} \mbox{\indent  because $a\cong \JW{2}$}\\
\end{align*}
Thus, $P\otimes id_1 \cong M_1\oplus \JW{2}\cong \JW{2}\oplus M_1$.  The proof of the claim that $R\otimes id_1\cong\JW{2}\oplus M_2$ is analogous to the above proof.  We use the fact that 
$$\frac{1}{\alpha_0\frac{[4]}{[3]}}\cdot \begin{tikzpicture}[baseline]
   \clip (-.6,-1.3) rectangle (.8,1.3);
   \node at (0,.7) (top) [rectangle,draw] {$ \hspace{.3 cm} R \hspace{.3cm} $};
   \node at (0,-.7) (bottom) [rectangle,draw]{$ \hspace{.3cm} R
\hspace{.3cm} $};

   \draw (top.-60)--(bottom.60);
   \draw (top.-145)--(bottom.145);
   \draw (top.-35) arc (-180:0:.2cm) -- ++(90:1cm);
   \draw (bottom.35) arc (180:0:.2cm) -- ++(-90:1cm);

   \draw (top.35) -- ++(90:1cm);
   \draw (top.60) -- ++(90:1cm);
   \draw (top.145) -- ++(90:1cm);

   \draw (bottom.-35) -- ++(-90:1cm);
   \draw (bottom.-60) -- ++(-90:1cm);
   \draw (bottom.-145) -- ++(-90:1cm);
\end{tikzpicture}\cong \JW{2}.$$
\end{proof}
\begin{lem}$M_1\otimes id_1\cong P$ and $M_2\otimes id_1\cong R$.
\end{lem}
\begin{proof}Let $f:M_1\otimes id_1 \rightarrow P$ be given by
$$\delta \cdot \begin{tikzpicture}[baseline]
    \node at (0,.5) (top) [rectangle,draw] {$ \hspace{.5cm}P\hspace{.5cm} $};
    \node at (.2,-.5) (mid) [rectangle, draw] {$ \hspace{.7cm} M_1 \hspace{.7cm} $};
    \draw (top.-35)--(mid.40);
    \draw (top.-145)--(mid.160);
    \draw (top.-100)--(mid.135);
    \draw (mid.-40) -- ++(-90:1mm);
    \draw (mid.-160) -- ++(-90:1mm);
    \draw (mid.-20) -- ++(-90:1mm);
    \draw (mid.-135) -- ++(-90:1mm);
    \draw (top.-35)--(mid.40);
     \draw (top.-145)--(mid.160);
      \draw (top.35) -- ++(90:1mm);
    \draw (top.145) -- ++(90:1mm);
    \draw (top.100)-- ++(90:1mm);
    \draw (mid.20) arc (180:0:.2cm) -- ++(-90:1cm);

\end{tikzpicture}$$
and let $g:P\rightarrow M_1\otimes id_1$ be given by
$$\begin{tikzpicture}[baseline]
    \node at (.2,.5) (mid) [rectangle, draw] {$ \hspace{.7cm} M_1 \hspace{.7cm} $};
    \node at (0,-.5) (bottom) [rectangle,draw]{$ \hspace{.5cm}P \hspace{.5cm}$};
    \draw (mid.-40)--(bottom.35);
     \draw (mid.-135)--(bottom.100);
    \draw (mid.40) -- ++(90:1mm);
    \draw (mid.160) -- ++(90:1mm);
    \draw (mid.20) -- ++(90:1mm);
    \draw (mid.135) -- ++(90:1mm);
    \draw (mid.-20) arc (-180:0:.2cm) -- ++(90:1cm);
    \draw (mid.-160)--(bottom.145);
    \draw (bottom.-35) -- ++(-90:1mm);
    \draw (bottom.-145) -- ++(-90:1mm);
    \draw (bottom.-100)-- ++(-90:1mm);
\end{tikzpicture}$$
We have that $fg=P$.  To see this, use the fact that $M_1$ is a projection, then expand $M_1$ while taking the partial trace of $M_1$.  Use the fact that $\delta\Big(\delta - \frac{1}{\alpha\frac{[4]}{[3]}}\Big)=1$.
Now, $$gf= \delta \cdot \begin{tikzpicture}[baseline]
   \clip (-.8,-1.3) rectangle (1,1.3);
   \node at (0,.7) (top) [rectangle,draw] {$ \hspace{.3 cm} M_1 \hspace{.3cm} $};
   \node at (0,-.7) (bottom) [rectangle,draw]{$ \hspace{.3cm} M_1
\hspace{.3cm} $};

   \draw (top.-60)--(bottom.60);
   \draw (top.-145)--(bottom.145);
   \draw (top.-100)--(bottom.100);
   \draw (top.-35) arc (-180:0:.2cm) -- ++(90:1cm);
   \draw (bottom.35) arc (180:0:.2cm) -- ++(-90:1cm);

   \draw (top.35) -- ++(90:1cm);
   \draw (top.60) -- ++(90:1cm);
   \draw (top.145) -- ++(90:1cm);
   \draw (top.100)-- ++(90:1cm);
   \draw (bottom.-100)-- ++(-90:1cm);

   \draw (bottom.-35) -- ++(-90:1cm);
   \draw (bottom.-60) -- ++(-90:1cm);
   \draw (bottom.-145) -- ++(-90:1cm);
\end{tikzpicture}=\delta \cdot \frac{1}{\delta}M_1\otimes id_1=M_1\otimes id_1,$$
if $\begin{tikzpicture}[baseline]
   \clip (-.8,-1.3) rectangle (1,1.3);
   \node at (0,.7) (top) [rectangle,draw] {$ \hspace{.3 cm} M_1 \hspace{.3cm} $};
   \node at (0,-.7) (bottom) [rectangle,draw]{$ \hspace{.3cm} M_1
\hspace{.3cm} $};

   \draw (top.-60)--(bottom.60);
   \draw (top.-145)--(bottom.145);
   \draw (top.-100)--(bottom.100);
   \draw (top.-35) arc (-180:0:.2cm) -- ++(90:1cm);
   \draw (bottom.35) arc (180:0:.2cm) -- ++(-90:1cm);

   \draw (top.35) -- ++(90:1cm);
   \draw (top.60) -- ++(90:1cm);
   \draw (top.145) -- ++(90:1cm);
   \draw (top.100)-- ++(90:1cm);
   \draw (bottom.-100)-- ++(-90:1cm);

   \draw (bottom.-35) -- ++(-90:1cm);
   \draw (bottom.-60) -- ++(-90:1cm);
   \draw (bottom.-145) -- ++(-90:1cm);
\end{tikzpicture}=\frac{1}{\delta}\cdot M_1\otimes id_1$.\\

\indent It suffices to show $$\begin{tikzpicture}[baseline]
   \clip (-.8,-1.3) rectangle (1,1.3);
   \node at (0,.7) (top) [rectangle,draw] {$ \hspace{.3 cm} M_1 \hspace{.3cm} $};
   \node at (0,-.7) (bottom) [rectangle,draw]{$ \hspace{.3cm} M_1
\hspace{.3cm} $};

   \draw (top.-60)--(bottom.60);
   \draw (top.-145)--(bottom.145);
   \draw (top.-100)--(bottom.100);
   \draw (top.-35) arc (-180:0:.2cm) -- ++(90:1cm);
   \draw (bottom.35) arc (180:0:.2cm) -- ++(-90:1cm);

   \draw (top.35) -- ++(90:1cm);
   \draw (top.60) -- ++(90:1cm);
   \draw (top.145) -- ++(90:1cm);
   \draw (top.100)-- ++(90:1cm);
   \draw (bottom.-100)-- ++(-90:1cm);

   \draw (bottom.-35) -- ++(-90:1cm);
   \draw (bottom.-60) -- ++(-90:1cm);
   \draw (bottom.-145) -- ++(-90:1cm);
\end{tikzpicture}=\frac{1}{\delta}\cdot M_1\otimes id_1$$
Let $A=\begin{tikzpicture}[baseline]
   \clip (-.8,-1.3) rectangle (1,1.3);
   \node at (0,.7) (top) [rectangle,draw] {$ \hspace{.3 cm} M_1 \hspace{.3cm} $};
   \node at (0,-.7) (bottom) [rectangle,draw]{$ \hspace{.3cm} M_1
\hspace{.3cm} $};

   \draw (top.-60)--(bottom.60);
   \draw (top.-145)--(bottom.145);
   \draw (top.-100)--(bottom.100);
   \draw (top.-35) arc (-180:0:.2cm) -- ++(90:1cm);
   \draw (bottom.35) arc (180:0:.2cm) -- ++(-90:1cm);

   \draw (top.35) -- ++(90:1cm);
   \draw (top.60) -- ++(90:1cm);
   \draw (top.145) -- ++(90:1cm);
   \draw (top.100)-- ++(90:1cm);
   \draw (bottom.-100)-- ++(-90:1cm);

   \draw (bottom.-35) -- ++(-90:1cm);
   \draw (bottom.-60) -- ++(-90:1cm);
   \draw (bottom.-145) -- ++(-90:1cm);
\end{tikzpicture}$ and $B=M_1\otimes id_1$.  We need to check that $\ip{A,A}\ip{B,B}=\ip{A,B}\overline{\ip{A,B}}$.  \\

If so, $A=\frac{\ip{A,B}}{\ip{B,B}}B$.  
We calculate
\begin{align*}
\ip{A,A}&=\left(\delta-\frac{1}{\alpha\frac{[4]}{[3]}}\right)Z(M_1)\\
&=\left(\delta-\frac{1}{\alpha\frac{[4]}{[3]}}\right)^2Z(P)\\
&=\left(\frac{1}{\delta}\right)^2Z(P)
\end{align*}
We calculate 
\begin{align*}
\ip{B,B}&=\delta Z(M_1)\\
&= \delta \frac{1}{\delta}Z(P)\\
&= Z(P)
\end{align*}
We calculate
$\ip{A,B}= Z(M_1)=\frac{1}{\delta}Z(P)$.  On the one hand, we have
$\ip{A,A}\ip{B,B}= \frac{\left(Z(P)\right)^2}{\delta^2}$.  On the other hand, $\ip{A,B}\overline{\ip{A,B}}=\frac{\left(Z(P)\right)^2}{\delta^2}$.  So equality holds, and $A=\frac{1}{\delta}B$.
The proof that $M_2\otimes id_1\cong R$ is analogous.
\end{proof}
\chapter{Uniqueness}
In this chapter, we explore the other possible pair of generators which generates a ``2221'' subfactor planar algebra.  All of the equations used to find the generator in $X^{3,\omega^2}$ are the same as those used to find $S$ and $Q$, except with every $\eta$ and $\omega$ conjugated.  Let $U$ be the generator coming from $X^{3,\omega^2}$ which corresponds to $S$, and let $V$ be the generator corresponding to $Q$. Note that $V$ is just $Q$, with $\eta$ and $\omega$ conjugated.  Hence, the matrices used to calculate traces for $V$ are the transposes of the matrices used to calculate traces for $Q$.  That is, if $L_i$ $(i=1,\ldots, 8)$ are the matrices in the decomposition of the matrix representation for $V$, then $L_i=N_i^T$ $(i=1,\ldots, 8)$, where $N_i$ are the matrices in the decomposition of the matrix representation for $Q$.  (Similarly, for $U$.)  Denote $V$ by $\overline{Q}$. \\
\indent Recall that when we identified the generator $T$ in Theorem 4.0.12, we had the freedom to choose the imaginary part of $s$ to be positive or negative.  Let $\overline{T}$ denote the generator with $Im(s)$ negative.  Then $\overline{T}$ is just $T$ with its coefficients conjugated.  Hence, the matrices in the decomposition of its matrix representation are the transposes of those for $T$.  The trace values for $\overline{T}$ and $\overline{Q}$ (along with the the trace values of their mixtures) are the same as those for $T$ and $Q$.  \\
\indent For the half-rotations of $\overline{T}$ and $\overline{Q}$, note that the matrices for $\rho^{1/2}(\overline{T})$ are the transposes of those for $\rho^{1/2}(T)$.  The matrices for $\rho^{1/2}(\overline{Q})$ are the conjugates of the matrices for $\rho^{1/2}(Q)$.  Hence, the trace values for the pair ($\rho^{1/2}(\overline{T})$ , $\rho^{1/2}(\overline{Q})$) are conjugates of the trace values for the pair ($\rho^{1/2}(T)$ , $\rho^{1/2}(Q)$).\\
\indent  The 3-box relations 1,2,5,6 of Theorem 5.0.16 hold for $\overline{T}$ , $\overline{Q}$ since the trace values are exactly the same as those for $T$ , $Q$.  The twisted 3-box relations 3,4,7,8 of Theorem 5.0.16 also hold for $\overline{T}$ , $\overline{Q}$ (with $\omega$ replaced by $\omega^2$ everywhere) since the trace values for half-rotations of $\overline{T}$ , $\overline{Q}$ are the conjugates of those for half-rotations of $T$ , $Q$.\\
\indent Lemma 5.2.2 holds for $\overline{T}$ , $\overline{Q}$.  Lemmas 5.2.4 and 5.2.6 hold ( with conjugation everywhere) for $\overline{T}$ , $\overline{Q}$.  Lemmas 5.2.7 and 5.2.8 hold for $\overline{T}$ , $\overline{Q}$.  Since all the inner product values for $\overline{T}$ , $\overline{Q}$ are the conjugates of those for $T$ , $Q$, the unshaded 4-box relations hold for $\overline{T}$ , $\overline{Q}$ ( where $u$, $v$, and $M$ are defined accordingly).  Similarly, the shaded 4-box relations hold for $\overline{T}$ , $\overline{Q}$.\\
\indent The one-strand braiding substitutes also hold for $\overline{T}$ , $\overline{Q}$.  Eqn.1 of Lemma 6.1.2 holds where $\eta, a_0, b_0$ are conjugated.  Eqn. 2 of Lemma 6.1.4 holds where $\tau, a', b'$ are conjugated.  We have the unshaded one-strand braiding substitute for $\overline{T}$ , $\overline{Q}$ because the determinant of the 2 by 2 system of equations for $\overline{T}$ , $\overline{Q}$ is just the conjugate of the determinant of the system in Proposition 6.1.5, and hence non-zero.  Similarly, we have the shaded one-strand braiding substitute.  \\
\indent All the theorems of Chapter 7 hold for $\overline{T}$ , $\overline{Q}$ and we can show that $\mathcal{P}(\overline{T},\overline{Q})$ is a subfactor planar algebra.  Since the trace values for $\overline{T}, \overline{Q}$ are the same as for $T, Q$ , all the calculations of Chapter 8 remain unchanged and $\mathcal{P}(\overline{T},\overline{Q})$ has principal graph ``2221''.  It also has dual principal graph ``2221''; the proof is analogous to the proof for the dual principal graph of $\mathcal{P}(T,Q)$, except we define $Q_0=\omega^2\rho^{1/2}(\overline{Q})$. Thus, there are actually at least two subfactor planar algebras having principal graphs ``2221'', namely $\mathcal{P}(T,Q)$ and $\mathcal{P}(\overline{T},\overline{Q})$.\\
\indent Recall from Theorem 4.0.12 that there is another possible generator $P$ of rotational eigenvalue 1.  It happens to be the case that the pair $(P,-S)$ has exactly the same trace values as those for the pair $(T,Q)$.  That is, Lemma 5.0.14 holds with $T$ replaced by $P$ and $Q$ replaced by $-S$.  Hence, $\mathcal{P}(P,-S)$ is a subfactor planar algebra having principal graphs ``2221''.  Analogous to the pair $(\overline{T},\overline{Q})$, we have that $\mathcal{P}(\overline{P},\overline{-S})$ is also a subfactor planar algebra having principal graphs ``2221''.  Thus, $\mathcal{P}(P,-S)$ is isomorphic to $\mathcal{P}(T,Q)$, and $\mathcal{P}(\overline{P},\overline{-S})$ is isomorphic to $\mathcal{P}(\overline{T},\overline{Q})$.

Now, we want to prove that the pairs $(1,\omega)$ and $(1,\omega^2)$ are the only two possible pairs for the rotational eigenvalues $\omega_{T}, \omega_{Q}$.  That is, we want to prove uniqueness up to conjugation of ``2221''.  In order to show this, we will rule out each pair $(\omega_T,\omega_Q)$ except for $(1,\omega)$ and $(1,\omega^2)$.  We will first prove that $\omega_T$ and $\omega_Q$ cannot be equal.  Then, it suffices to rule out the pair $(\omega, \omega^2)$; this is done by looking at the two unique(up to conjugation) elements $Q$ and $S$ in the graph planar algebra, as given in Theorem 4.0.13, and checking to see that neither of them can pair up with their conjugates $V=\overline{Q}$ and $U=\overline{S}$.  That is, we rule out the pairs of generators $(Q,\overline{Q}),(Q,\overline{S}),(S,\overline{Q})$, and $(S,\overline{S})$.\\
\indent The following claim proves that the rotational eigenvalue of the new element $M$ at level $4$ is a $2$-th root of unity.
\begin{claim}$\omega_M^2=1$
\end{claim}
\begin{proof} Let $K=\begin{tikzpicture}[baseline=0,scale=1.5]
       {%
	\draw (-1,1) -- (0,0);
	\draw (1,1) -- (0,0);
	\node[anchor=west] at (-0.5,0.6) {\footnotesize$4$};
	\node[anchor=west] at (0.3,0.6) {\footnotesize$4$};
}
 {%
	\node[anchor=south] at (0,1) {\footnotesize$4$};
	\draw (-1,1) -- (1, 1);
	\foreach \x in {-1,1} {
		{%
	\filldraw[fill=white,thick] (\x,1) ellipse (3mm and 3mm);
	\node at (\x,1) {\Large $M$};
	\path(\x,1) ++(90:0.37) node {$\star$};
}
	}
}
       \upsidedown{{
	{%
}
}}
      {%
	\filldraw[fill=white,thick] (0,0) ellipse (3mm and 3mm);
	\node at (0,0) {\Large $M$};
	\path(0,0) ++(-90:0.37) node {$\star$};
}

\end{tikzpicture}$.  \\
Then $K=\omega_M^6\cdot \begin{tikzpicture}[baseline=0,scale=1.5]
       {%
	\draw (-1,1) -- (0,0);
	\draw (1,1) -- (0,0);
	\node[anchor=west] at (-0.5,0.6) {\footnotesize$4$};
	\node[anchor=west] at (0.3,0.6) {\footnotesize$4$};
}
 {%
	\node[anchor=south] at (0,1) {\footnotesize$4$};
	\draw (-1,1) -- (1, 1);
	\foreach \x in {-1,1} {
		{%
	\filldraw[fill=white,thick] (\x,1) ellipse (3mm and 3mm);
	\node at (\x,1) {\Large $M$};
	\path(-1,1) ++(-30:0.37) node {$\star$};
	\path(1,1) ++(210:0.37) node {$\star$};
}
	}
}
       \upsidedown{{
	{%
}
}}
      {%
	\filldraw[fill=white,thick] (0,0) ellipse (3mm and 3mm);
	\node at (0,0) {\Large $M$};
	\path(0,0) ++(90:0.37) node {$\star$};
}

\end{tikzpicture}$\\
We can turn the above diagram inside out to get just $K$ by using sphericality and isotopy.  So the above expression is equal to 
$\omega_M^6\cdot \begin{tikzpicture}[baseline=0,scale=1.5]
       {%
	\draw (-1,1) -- (0,0);
	\draw (1,1) -- (0,0);
	\node[anchor=west] at (-0.5,0.6) {\footnotesize$4$};
	\node[anchor=west] at (0.3,0.6) {\footnotesize$4$};
}
 {%
	\node[anchor=south] at (0,1) {\footnotesize$4$};
	\draw (-1,1) -- (1, 1);
	\foreach \x in {-1,1} {
		{%
	\filldraw[fill=white,thick] (\x,1) ellipse (3mm and 3mm);
	\node at (\x,1) {\Large $M$};
	\path(\x,1) ++(90:0.37) node {$\star$};
}
	}
}
       \upsidedown{{
	{%
}
}}
      {%
	\filldraw[fill=white,thick] (0,0) ellipse (3mm and 3mm);
	\node at (0,0) {\Large $M$};
	\path(0,0) ++(-90:0.37) node {$\star$};
}

\end{tikzpicture}$\\
That is, $K=\omega_M^6\cdot K$.  Since $K \neq 0$, $\omega_M^6=1$.  Furthermore, since $\omega_M$ is a $4$-th root of unity, we have $\omega_M^2=1$. 
\end{proof}
The next claim is proven by Jones in Theorem 5.2.3 of \cite{VJ4} in a general setting.  I will prove it for our specific situation using the same idea.
\begin{claim}Suppose $T'$ and $Q'$ are the two generators of a ``2221'' subfactor planar algebra.  Then $\omega_{T'}\neq \omega_{Q'}$
\end{claim}
\begin{proof}
Let $T'\circ Q'$ be defined as
$\begin{tikzpicture}[baseline]
   \clip (-.8,-1.3) rectangle (.8,1.3);
   \node at (0,.7) (top) [rectangle,draw] {$ \hspace{.3 cm} T' \hspace{.3cm} $};
   \node at (0,-.7) (bottom) [rectangle,draw]{$ \hspace{.3cm} Q'
\hspace{.3cm} $};

   \draw (top.-60)--(bottom.60);
   \draw (top.-120)--(bottom.120);
   \draw (top.120) -- ++(90:1cm);

   \draw (top.35) -- ++(90:1cm);
   \draw (top.60) -- ++(90:1cm);
   \draw (top.145) -- ++(90:1cm);

   \draw (bottom.-35) -- ++(-90:1cm);
   \draw (bottom.-60) -- ++(-90:1cm);
   \draw (bottom.-120) -- ++(-90:1cm);
   \draw (bottom.-145) -- ++(-90:1cm);
\end{tikzpicture}$.  \\
Then $T'\circ Q'\in span\{id_1\otimes \rho^{1/2}(T'),T'\otimes id_1, id_1\otimes \rho^{1/2}(Q'), Q'\otimes id_1, M\}$ and $$Coeff_{T'\circ Q'}(id_1\otimes\rho^{1/2}(T'))=\omega_{T'}^2\omega_{Q'}^2\cdot Coeff_{T'\circ Q'}(T'\otimes id_1),$$
$$Coeff_{T'\circ Q'}(id_1\otimes\rho^{1/2}(Q'))=\omega_{T'}\cdot Coeff_{T'\circ Q'}(Q'\otimes id_1)$$\\
On one hand, $\rho^2(T'\circ Q')=Q'\circ T'$ if $\omega_{T'}=\omega_{Q'}$.  One can see this by double rotating the diagram $T'\circ Q'$.\\
On the other hand, $\rho^2(T'\circ Q') = T'\circ Q'$ if $\omega_{T'}=\omega_{Q'}$.  To see this, start by writing out $T'\circ Q'$ as a linear combination of ATL elements and $M$.  Then take the double rotation of both sides.  Use the fact that $\omega_{T'}=\omega_{Q'}$ is a $3$-th root of unity and the fact, by Claim 9.0.33, that $\omega_M^2=1$.\\
Hence, $T'\circ Q' = Q'\circ T'$ if $\omega_{T'}=\omega_{Q'}$.  On one hand, $T'\circ Q'$ and $Q'\circ T'$ are orthogonal because $T'Q'=Q'T'\in span\{T',Q'\}$.  It follows that $\ip{T'\circ Q',T'\circ Q'}=0$.  On the other hand, $\ip{T'\circ Q', T'\circ Q'}=\frac{[4]^2}{[3]}\neq 0$.  Contradiction.  Thus, $\omega_{T'}\neq \omega_{Q'}$.
\end{proof}
Now, we are ready to prove uniqueness up to conjugation.
\begin{theorem}The ``2221'' subfactor planar algebra is unique up to conjugation.
\end{theorem}
\begin{proof}According to Claim 9.0.34, the rotational eigenvalues $\omega_{T'}$ and $\omega_{Q'}$ of the two generators $T',Q'$ of our planar algebra cannot be equal.  Hence, the only possible pairs $(\omega_{T'},\omega_{Q'})$ are $(1,\omega),(1,\omega^2),$ and $(\omega,\omega^2)$.  It suffices to rule out $(\omega,\omega^2)$.  By the embedding theorem for finite depth subfactor planar algebras, which states that any finite depth subfactor planar algebra $\mathcal{P}$ with principal graph $\Gamma$ can be injectively mapped into the graph planar algebra of $\Gamma$, it suffices to rule out $(\omega,\omega^2)$ in the graph planar algebra of ``2221''.(see \cite{VJ6} for more details on the embedding theorem.)  According to Theorem 4.0.13, the only possible generators at level 3 of a ``2221'' graph planar algebra that have rotational eigenvalue $\omega$ are (multiples of) $Q$ and $S$.  Their conjugates $V=\overline{Q}$ and $U=\overline{S}$ (or multiples thereof) are the only possible generators at level 3 that have rotational eigenvalue $\omega^2$, as explained in the remark immediately following Theorem 4.0.13.  Suppose our pair of generators has the pair of rotational eigenvalues ($\omega$, $\omega^2$).  Then our pair must be (up to pairs of multiples) among the pairs $(Q,V),(Q,U),(S,V),(S,U)$.  Suppose $Q$ and $V$ are our potential generators.  Then we expect various quadratic relations to hold.  In particular, we expect $QV=\frac{Z(Q^2V)}{[4]}Q+\frac{Z(QV^2)}{[4]}V$.  In order for this to hold, we must have $Z(Q^2V^2)=\frac{\norm{Z(Q^2V)}^2}{[4]}+\frac{\norm{Z(QV^2)}^2}{[4]}$.  However, $Z(Q^2V^2)=\frac{91\sqrt{3}+43\sqrt{7}}{42}$ while $Z(Q^2V)=0$ and $Z(QV^2)=0$.  Thus, $(Q,V)$ cannot be a pair of generators.  Similarly,  $Z(Q^2U^2)=\frac{91\sqrt{3}+43\sqrt{7}}{42}$ while $Z(Q^2U)=0$ and $Z(QU^2)=0$.  So $(Q,U)$ cannot be a pair of generators.  Analogously, for the pairs $(S,V)$ and $(S,U)$, we find that $Z(S^2U^2)=Z(S^2V^2)=\frac{91\sqrt{3}+43\sqrt{7}}{42}$ while $Z(S^2U)=Z(SU^2)=Z(S^2V)=Z(SV^2)=0$.  Hence, neither of $(S,V),(S,U)$ can be pairs of generators.  Pairs of the form $(\alpha X, \beta Y$) (where $X\in\{Q,S\},Y\in\{V,U\}$ and $\alpha,\beta\in\C$ are nonzero) also do not satisfy the requisite equation since $Z((\alpha X)^2(\beta Y)^2)=\alpha^2\beta^2Z(X^2Y^2)$ is nonzero while $Z((\alpha X)^2\beta Y)=\alpha^2\beta Z(X^2Y)=0$ and $Z(\alpha X (\beta Y)^2)=\alpha\beta^2 Z(XY^2)=0$.  Since none of the possible pairs satisfies the quadratic relation, we cannot have $(\omega,\omega^2)$ as a possible pair of rotational eigenvalues for our generators.
\end{proof}
\begin{cor}The ``2221'' subfactor is unique up to conjugation.  
\end{cor}
\begin{proof}According to a theorem by Popa(\cite{Pop4}), if a subfactor satisfies strong amenability, then its standard invariant is a complete invariant. Any subfactor of finite depth is strongly amenable, and hence its standard invariant is a complete invariant.  The planar algebra of a subfactor is an equivalent formulation of the standard invariant of a subfactor.  Hence, uniqueness of the planar algebra having principal graph ``2221'' implies uniqueness of the ``2221'' subfactor. It follows from Theorem 9.0.35 that the ``2221'' subfactor is unique up to conjugation.
\end{proof}
\chapter{Conclusion}
In conclusion, we have proven the main theorem of the paper, Theorem 4.0.7, which gives a presentation of the ``2221'' planar algebra consisting of two generators and quadratic relations.  To do this, we have used the jellyfish algorithm to prove the subfactor property as well as several facts at higher-level n-box spaces.  We have also identified the minimal projections in the tensor category associated with the planar algebra and have shown that the planar algebra we have constructed has ``2221'' as its principal graph.  It follows as a corollary that there exists a subfactor with principal graph ``2221'', since the existence of a planar algebra with principal graph ``2221'' implies the existence of a subfactor with principal graph ``2221''.  Lastly, we have shown uniqueness up to conjugation of ``2221''. 
\nocite{*}
\bibliographystyle{aip}
\bibliography{uctest}


\appendix
\chapter{Calculating Traces}

$PABG_{k,\pm}$ is a finite-dimensional algebra with a positive definite inner product and is therefore isomorphic, by the Artin-Wedderburn theorem, to a direct sum of matrix algebras.  The minimal central idempotents are projections onto the space of loops starting at $v$ and having midpoint $w$.  The matrix corresponding to loops based at $v$ with midpoint $w$ has rows labelled by length-k paths from $v$ to $w$, and columns labelled by the same paths traversed backwards.
$$PABG_{k,\pm} \simeq \bigoplus_{( v, w) \in U_\pm \times U_\pm} M_{i_{ v, w} }(\mathbb{C})$$

  Given $m \in PABG_{k,+}$, let $\tilde{m}=\oplus \tilde{m}_{ v, w }$ be its image in $\bigoplus M_{i_{ v, w} }(\mathbb{C})$.  Then 
$$Z(m)=\sum_{( v, w) } \lambda(v) \lambda(w)Tr( \tilde{m}_{v, w }),$$ where $Tr$ is the usual matrix trace.  For this reason, working in $\bigoplus M_{i_{ v, w} }(\mathbb{C})$ is sometimes computationally convenient.

\section{Calculating Traces for $T$ and $Q$}

For the purposes of constructing the ``2221'' subfactor, we work in $PABG(H)_{3,+}$. 
We order the set of pairs of vertices lexicographically, with the ``alphabet'' ordered 
$b_0 < b_1 <  b_2 < d$ and $z_0 < z_1 < z_2 < c$.
The rows of $M_{i_{z_0,v}}(\mathbb{C})$ are labelled

\[\left(
\begin{array}{c}
z_0 b_0 c b_0  \\
z_0 b_0 z_0 b_0  \\
\end{array} \right),
\left(
\begin{array}{c}
z_0 b_0 c b_1  \\
\end{array} \right),
\left(
\begin{array}{c}
z_0 b_0 c  b_2 \\
\end{array} \right),
\left(
\begin{array}{c}
z_0 b_0 c  d \\
\end{array} \right).\]
Similarly for $z_1$ and $z_2$ ; the labels of the rows of $M_{i_{z_1,v}}(\mathbb{C})$ are as above, but with the permutation $(012)$ applied to the indices of $z$; the labels of the rows of $M_{i_{z_2,v}}(\mathbb{C})$ are also as above, but with the permutation $(021)$ applied to the indices.

The rows of $M_{i_{c,v}}(\mathbb{C})$ are labelled
\[\left(
\begin{array}{c}
c b_0 c b_0 \\
c b_1 c b_0 \\
c b_2 c b_0 \\
c d c b_0 \\
c b_0 z_0 b_0\\
\end{array} \right),
\left(
\begin{array}{c}
c b_0 c b_1 \\
c b_1 c b_1 \\
c b_2 c b_1 \\
c d c b_1 \\
c b_1 z_1 b_1\\
\end{array} \right),
\left(
\begin{array}{c}
c b_0 c b_2 \\
c b_1 c b_2 \\
c b_2 c b_2 \\
c d c b_2 \\
c b_2 z_2 b_2\\
\end{array} \right),
\left(
\begin{array}{c}
c b_0 c d \\
c b_1 c d \\
c b_2 c d\\
c d c d \\
\end{array} \right).\]

Let $\theta=\sqrt{\frac{\lambda(c)}{\lambda(z)}}=\sqrt{\frac{3+\sqrt{21}}{2}}$.  Then $\frac{1}{R}=\sqrt{\frac{\lambda(b)}{\lambda(d)}}=\frac{\theta}{\sqrt{3}}$.  We want to find the matrix representation for T.
We have the following matrices for $M_{i_{z_0,v}}(\mathbb{C})$:

\[\left(
\begin{array}{cc}
 0 & 0  \\
 0 & 0  \\
\end{array}
\right),
\left(
\begin{array}{c}
-\theta^2t_0 \\
\end{array}
\right),
\left(
\begin{array}{c}
 -\theta^2t_0' \\
\end{array}
\right),
\left(
\begin{array}{c}
 \frac{\theta^2}{R^2}(t_0+t_0')\\
\end{array}
\right).\]

We have the following matrices for $M_{i_{z_1,v}}(\mathbb{C})$:

\[\left(
\begin{array}{cc}
 0 & 0  \\
 0 & 0  \\
\end{array}
\right),
\left(
\begin{array}{c}
-\theta^2t_1 \\
\end{array}
\right),
\left(
\begin{array}{c}
 -\theta^2(t_0'+t_2-t_1) \\
\end{array}
\right),
\left(
\begin{array}{c}
 \frac{\theta^2}{R^2}(t_2+t_0')\\
\end{array}
\right).\]

We have the following matrices for $M_{i_{z_2,v}}(\mathbb{C})$:

\[\left(
\begin{array}{cc}
 0 & 0  \\
 0 & 0  \\
\end{array}
\right),
\left(
\begin{array}{c}
-\theta^2t_2 \\
\end{array}
\right),
\left(
\begin{array}{c}
 -\theta^2(t_0'+t_0-t_1) \\
\end{array}
\right),
\left(
\begin{array}{c}
 \frac{\theta^2}{R^2}(t_0-t_1+t_2+t_0')\\
\end{array}
\right).\]
It is easy to check that \[Coeff_{tr(T)}(z_0)=Coeff_{tr(T)}(z_1)= Coeff_{tr(T)}(z_2)=0.\]
Consider $Coeff_{tr(T^n)}(z_0)$, for $n\geq 1.$  We calculate \[Coeff_{tr(T^n)}(z_0)=\Big((-\theta^2t_0)^n+(-\theta^2t_0')^n\Big)\frac{\theta^2}{\sqrt{3}}+\Big(\frac{\theta^2}{R^2}(t_0+t_0')\Big)^n\sqrt{3}.\]
Similarly, \[Coeff_{tr(T^n)}(z_1)=\Big((-\theta^2t_1)^n+(-\theta^2(t_0'+t_2-t_1))^n\Big)\frac{\theta^2}{\sqrt{3}}+\Big(\frac{\theta^2}{R^2}(t_2+t_0')\Big)^n\sqrt{3},\] 
and \[Coeff_{tr(T^n)}(z_2)=\Big((-\theta^2t_2)^n+(-\theta^2(t_0'+t_0-t_1))^n\Big)\frac{\theta^2}{\sqrt{3}}+\Big(\frac{\theta^2}{R^2}(t_0-t_1+t_2+t_0')\Big)^n\sqrt{3}.\]  We want $Coeff_{tr(T^n)}(z_0)=Coeff_{tr(T^n)}(z_1)=Coeff_{tr(T^n)}(z_2)$ to hold for all $n \geq 1$.  Equivalently, we want $Coeff_{tr(T^n)}(z_0)=Coeff_{tr(T^n)}(z_1)$, $Coeff_{tr(T^n)}(z_0)=Coeff_{tr(T^n)}(z_2)$, and $Coeff_{tr(T^n)}(z_1)=Coeff_{tr(T^n)}(z_2)$ to hold for all $n \geq 1$.  Denote each equation by $E(z_0,z_1)$, $E(z_0,z_2)$, and $E(z_1,z_2)$, respectively.  By inspection, it seems intuitively obvious that $t_0,t_1,t_2$ must all be equal.  It is our goal to show that this must be the case.  First, notice that each equation $E(z_i,z_j)$ (where $(i,j)\in \{(0,1),(0,2),(1,2)\}$) is of the form $(A^n+B^n)\frac{\theta^2}{\sqrt{3}}+(-\frac{\theta^2}{3}(A+B))^n\sqrt{3}=(C^n+D^n)\frac{\theta^2}{\sqrt{3}}+(-\frac{\theta^2}{3}(C+D))^n\sqrt{3}$ for some $A,B,C,D\in\R$.  For instace, for $E(z_0,z_1)$, $A=-\theta^2 t_0, B=-\theta^2 t_0', C=-\theta^2 t_1, D= -\theta^2(t_0'+t_2-t_1)$.  We will show that the only possibilities, for each $E(z_i,z_j)$, are that either ($A=D$ and $B=C$) or ($A=C$ or $B=D$).
\begin{claim}If $$ (A^n+B^n)\frac{\theta^2}{\sqrt{3}}+(-\frac{\theta^2}{3}(A+B))^n\sqrt{3}=(C^n+D^n)\frac{\theta^2}{\sqrt{3}}+(-\frac{\theta^2}{3}(C+D))^n\sqrt{3} \mbox{\hspace{3mm} (Eqn.*)  }$$ holds for all $n\geq 1$, where $A,B,C,D\in\R$, then either ($A=D$ and $B=C$) or ($A=C$ or $B=D$).
\end{claim}
\begin{proof}Suppose $A\neq C$ and $B\neq D$.  We will show that $A=D$ and $B=C$. Since $A\neq C$, either $A > C$ or $A < C$.  Suppose $A>C$.  Since $B\neq D$, either $B>D$ or $B<D$.\\
Case a:$A>C$ and $B>D$.\\
\indent	Case a(i):$ A=D$\\
	\indent \hspace{3mm}	Case a(i)I:$B=C$\\
	\indent	\hspace{3mm}  Case a(i)II: $B<C$\\
	\indent \hspace{3mm}  Case a(i)III:$B>C$\\
\indent	Case a(ii): $A>D$\\
	\indent \hspace{3mm}  Case a(ii)I: $B=C$\\
	\indent	\hspace{3mm}  Case a(ii)II:$ B<C$\\
	\indent \hspace{3mm}  Case a(ii)III: $B>C$\\
\indent	Case a(iii):$A<D$\\
	\indent \hspace{3mm}	Case a(iii)I: $B=C$\\
	\indent \hspace{3mm}	Case a(iii)II: $B<C$\\
	\indent \hspace{3mm}	Case a(iii)III:$ B>C$\\
We use mathematica to perform the necessary calculations.  Each of the cases in Case a, except for Case a(ii)II and Case a(iii)III, are ruled out immediately by applying Eqn.* simultaneously for $n=2,3$.  We will deal with the exceptional cases last.\\
Case b:$A>C$ and $B<D$.\\
\indent	Case b(i):$ A=D$\\
	\indent \hspace{3mm}	Case b(i)I:$B=C$\\
	\indent	\hspace{3mm}  Case b(i)II: $B<C$\\
	\indent \hspace{3mm}  Case b(i)III:$B>C$\\
\indent	Case b(ii): $A>D$\\
	\indent \hspace{3mm}  Case b(ii)I: $B=C$\\
	\indent	\hspace{3mm}  Case b(ii)II:$ B<C$\\
	\indent \hspace{3mm}  Case b(ii)III: $B>C$\\
\indent	Case b(iii):$A<D$\\
	\indent \hspace{3mm}	Case b(iii)I: $B=C$\\
	\indent \hspace{3mm}	Case b(iii)II: $B<C$\\
	\indent \hspace{3mm}	Case b(iii)III:$ B>C$\\
Each of the cases in Case b, except for Case b(i)I, Case b(ii)III, and Case b(iii)II, are ruled out immediately by applying Eqn.* simultaneously for $n=2,3$.  Since Case b(i)I is what we wanted, the only exceptional cases under Case b are Case b(ii)III and Case b(iii)II.\\
Now, suppose $A<C$.  We have the following cases.\\
Case c:$A<C$ and $B>D$.\\
\indent	Case c(i):$ A=D$\\
	\indent \hspace{3mm}	Case c(i)I:$B=C$\\
	\indent	\hspace{3mm}  Case c(i)II: $B<C$\\
	\indent \hspace{3mm}  Case c(i)III:$B>C$\\
\indent	Case c(ii): $A>D$\\
	\indent \hspace{3mm}  Case c(ii)I: $B=C$\\
	\indent	\hspace{3mm}  Case c(ii)II:$ B<C$\\
	\indent \hspace{3mm}  Case c(ii)III: $B>C$\\
\indent	Case c(iii):$A<D$\\
	\indent \hspace{3mm}	Case c(iii)I: $B=C$\\
	\indent \hspace{3mm}	Case c(iii)II: $B<C$\\
	\indent \hspace{3mm}	Case c(iii)III:$ B>C$\\
Each of the cases in Case c, except for Case c(i)I, Case c(ii)III, and Case c(iii)II, are ruled out immediately by applying Eqn.* simultaneously for $n=2,3$.  Since Case c(i)I is what we wanted, the only exceptional cases under Case c are Case c(ii)III and Case c(iii)II.\\
Case d:$A<C$ and $B<D$.\\
\indent	Case d(i):$ A=D$\\
	\indent \hspace{3mm}	Case d(i)I:$B=C$\\
	\indent	\hspace{3mm}  Case d(i)II: $B<C$\\
	\indent \hspace{3mm}  Case d(i)III:$B>C$\\
\indent	Case d(ii): $A>D$\\
	\indent \hspace{3mm}  Case d(ii)I: $B=C$\\
	\indent	\hspace{3mm}  Case d(ii)II:$ B<C$\\
	\indent \hspace{3mm}  Case d(ii)III: $B>C$\\
\indent	Case d(iii):$A<D$\\
	\indent \hspace{3mm}	Case d(iii)I: $B=C$\\
	\indent \hspace{3mm}	Case d(iii)II: $B<C$\\
	\indent \hspace{3mm}	Case d(iii)III:$ B>C$\\
Each of the cases in Case d, except for Case d(ii)II and Case d(iii)III, are ruled out immediately by applying Eqn.* simultaneously for $n=2,3$. \\
\indent We now want to rule out the exceptional cases.  For Case a(ii)II and Case a(iii)III, applying Eqn.* simultaneously for $n=2,3,4$ gives us that $B=-C$ and $A=-D$.  Furthermore, $B$ is a multiple of $A$.  Applying Eqn.* for $n=5$ implies that $A=0$. It follows that $A,B,C,D$ are all $0$, which is a contradiction to our assumption that $A\neq C$ and $B\neq D$.  For Cases b(ii)III and b(iii)II, applying Eqn.* for $n=2,3,4$ gives us that $C=-A$ and $B=-D$.  Furthermore, $B$ is a multiple of $A$.  Applying Eqn.* for $n=5$ implies that $A=0$, from which it follows that $A,B,C,D$ are all $0$.  This contradicts our assumption that $A\neq C$ and $B\neq D$.  For Cases c(ii)III and c(iii)II, applying Eqn* for $n=2,3,4$ implies that $B=-D$ and $C=-A$.  Furthermore, $B$ is a multiple of $A$.  Applying Eqn* for $n=5$ implies $A=0$, which gives a contradiction.  For Cases d(ii)II and d(iii)III, applying Eqn.* for $n=2,3,4$ implies that $B=-C$ and $D=-A$.  Furthermore, $B$ is a multiple of $A$.  Applying Eqn.* for $n=5$ implies $A=0$.  This gives a contradiction.  This exhausts all the possibilities.  Hence, $A=D$ and $B=C$.
\end{proof}
\indent So for each equation $E(z_0,z_1),E(z_0,z_2),E(z_1,z_2)$, we either have ($A=D$ and $B=C$) or ($A=C$ or $B=D$).  We will use this fact to prove that $t_0=t_1=t_2$.  First, we need to be able to calculate the coefficient of $c$ in $tr(T^n)$ so that we can utilize the equation $Coeff_{tr(T^n)}(c)=Coeff_{tr(T^n)}(z_0)$.  In order to do this, we will need the following matrices for $M_{i_{c,v}}(\mathbb{C})$:
\footnotesize
\[\left(
\begin{array}{ccccc}
 0 & t_0 & t_0' & -\frac{(t_0+t_0')}{R} & 0 \\
 t_0 & t_0'+t_2-t_1 & -(2t_0+3t_0'+2t_2-t_1+s) & \frac{(2t_0'+t_2+t_0+s)}{R} & -\theta t_0 \\
 t_0' & -(2t_0+3t_0'+2t_2-t_1+\bar{s}) & t_2 & -\frac{(t_2+t_0'+s)}{R} & -\theta t_0' \\
 -\frac{(t_0+t_0')}{R} & \frac{(2t_0'+t_2+t_0+\bar{s})}{R} & -\frac{(t_2+t_0'+\bar{s})}{R} & 0 & \frac{\theta(t_0+t_0')}{R} \\
 0 & -\theta t_0 & -\theta t_0' & \frac{\theta(t_0+t_0')}{R} & 0 \\
\end{array}
\right),\]

\[\left(
\begin{array}{ccccc}
 t_0 & t_0'+t_2-t_1 & s & -\frac{(t_0-t_1+t_2+t_0'+s)}{R} & -\theta (t_0'+t_2-t_1) \\
 t_0'+t_2-t_1 & 0 & t_1 & -\frac{(t_2+t_0')}{R} & 0 \\
 \bar{s} & t_1 & t_0'+t_0-t_1 & \frac{2(t_2+t_0')-t_1+t_0+s}{R} & -\theta t_1 \\
 -\frac{(t_0-t_1+t_2+t_0'+\bar{s})}{R} & -\frac{(t_2+t_0')}{R} & \frac{2(t_2+t_0')-t_1+t_0+\bar{s}}{R} & 0 & \frac{\theta(t_2+t_0')}{R} \\
 -\theta (t_0'+t_2-t_1) & 0 & -\theta t_1 & \frac{\theta(t_2+t_0')}{R} & 0 \\
\end{array}
\right),\]
\tiny
\[\left(
\begin{array}{ccccc}
 t_0' & -(2t_0+3t_0'+2t_2-t_1+s) & t_2 & \frac{2(t_0+t_0')-t_1+t_2+s}{R} & -\theta t_2 \\
 -(2t_0+3t_0'+2t_2-t_1+\bar{s}) & t_1 & t_0'+t_0-t_1 & -\frac{(t_0+t_0'+s)}{R} & -\theta (t_0'+t_0-t_1) \\
 t_2 & t_0'+t_0-t_1 & 0 & -\frac{(t_0-t_1+t_2+t_0')}{R} & 0 \\
 \frac{2(t_0+t_0')-t_1+t_2+\bar{s}}{R} & -\frac{(t_0+t_0'+\bar{s})}{R} & -\frac{(t_0-t_1+t_2+t_0')}{R} & 0 & \frac{\theta(t_0-t_1+t_2+t_0')}{R} \\
 -\theta t_2 & -\theta (t_0'+t_0-t_1) & 0 & \frac{\theta(t_0-t_1+t_2+t_0')}{R} & 0 \\
\end{array}
\right),\]
\normalsize
\[\left(
\begin{array}{cccc}
 -\frac{(t_0+t_0')}{R^2} & \frac{2t_0'+t_2+t_0+s}{R^2} & -\frac{(t_2+t_0'+s)}{R^2} & 0 \\
 \frac{2t_0'+t_2+t_0+\bar{s}}{R^2} & -\frac{(t_2+t_0')}{R^2} & \frac{2(t_2+t_0')-t_1+t_0+s}{R^2} & 0 \\
 -\frac{(t_2+t_0'+\bar{s})}{R^2} & \frac{2(t_2+t_0')-t_1+t_0+\bar{s}}{R^2} & -\frac{(t_0-t_1+t_2+t_0')}{R^2} & 0 \\
 0 & 0 & 0 & 0 \\
\end{array}
\right).\]

If we call the above four matrices $M_1,M_2,M_3,M_4$, then $$Coeff_{tr(T^n)}(c)=(Tr(M_1^n)+Tr(M_2^n)+Tr(M_3^n))\cdot\frac{\lambda(b)}{\lambda(c)} +Tr(M_4^n)\cdot\frac{\lambda(d)}{\lambda(c)} .$$
That is, $$Coeff_{tr(T^n)}(c)=(Tr(M_1^n)+Tr(M_2^n)+Tr(M_3^n))\cdot\frac{1}{\sqrt{3}} +Tr(M_4^n)\cdot\frac{\sqrt{3}}{\theta^2}. $$

\noindent Now, we want to show that $t_0=t_1=t_2$.  Recall that for each equation $E(z_0,z_1)$, $E(z_0,z_2)$, $E(z_1,z_2)$, we have that either ($A=D$ and $B=C$) or ($A=C$ or $B=D$).\\
Suppose $A=D$ and $B=C$ for each pair $E(z_0,z_1),E(z_0,z_2),E(z_1,z_2)$.  Then $$t_0=t_0'+t_2-t_1 \mbox{ and } t_0'=t_1,$$ and $$t_0=t_0'+t_0-t_1 \mbox{ and } t_0'=t_2,$$ and $$t_1=t_0'+t_0-t_1 \mbox{ and } t_0'+t_2-t_1=t_2.$$  That is, $$t_0=t_0'+t_2-t_1 \mbox{ and } t_0'=t_1,   \hspace{3mm} \mbox{  (Eqn.1) }$$ and $$t_0'=t_1 \mbox{ and } t_0'=t_2, \mbox{ \hspace{3mm} (Eqn.2) }$$
and $$t_0'+t_0=2t_1 \mbox{ and } t_0'=t_1 \hspace{3mm} \mbox{  (Eqn.3) }.$$

$\Longrightarrow$ $t_0'=t_0=t_1=t_2$.
Recall from the remark following Theorem 4.0.10 that $\bar{s}=-(2t_0+3t_0'+2t_2-t_1+s)$.  It follows that $t_0'=-\frac{s+\bar{s}}{6}$.  Setting $Coeff_{tr(T^n)}(c)=Coeff_{tr(T^n)}(z_0)$ for $n=2,3$ simultaneously, we get that $s=0$. Hence, $t_0'=0$ and we have the null solution.\\
Now, suppose only two equations $E(z_i,z_j)$ $((i,j)\in\{(0,1),(0,2),(1,2)\})$ satisfy $A=D$ and $B=C$.  First, suppose $E(z_0,z_1)$ and $E(z_0,z_2)$ satisfy $A=D$ and $B=C$.  Then Eqn.1 and Eqn.2 are satisfied.  It follows that $t_0'=t_0=t_1=t_2$.  As in the previous case, we get the null solution.  Next, suppose $E(z_0,z_1)$ and $E(z_1,z_2)$ satisfy $A=D$ and $B=C$.  Then Eqn.1 and Eqn.3 are satisfied.  It follows that $t_0'=t_0=t_1=t_2$.  Hence, we get the null solution.  Suppose $E(z_0,z_2)$ and $E(z_1,z_2)$ satisfy $A=D$ and $B=C$.  Then Eqn.2 and Eqn.3 are satisifed.  It follows that $t_0'=t_0=t_1=t_2$.  Hence, we get the null solution.\\
Now, suppose only one equation $E(z_i,z_j)$ $((i,j)\in\{(0,1),(0,2),(1,2)\})$ satisfies $A=D$ and $B=C$.  Suppose only $E(z_0,z_1)$ satisfies  $A=D$ and $B=C$.  Then Eqn.1 holds.  Since $A=C$ or $B=D$ for $E(z_0,z_2)$ and $E(z_1,z_2)$, we have that $$(t_0=t_2 \mbox{ or } t_0'=t_0'+t_0-t_1),$$ and $$t_1=t_2 \mbox{ or } t_0'+t_2-t_1=t_0'+t_0-t_1.$$  That is, $$(t_0=t_2 \mbox{ or } t_0=t_1),$$ and $$(t_1=t_2 
\mbox{ or } t_2=t_0).$$
Suppose $t_0=t_2$ and $t_1=t_2$.  Then $t_0'=t_0=t_1=t_2$, and we get the null solution.  Suppose $t_0=t_2$ and $t_2=t_0$.  Then $t_0'=t_1$ and $t_0=t_2$ and $\bar{s}=-(4t_0+2t_0'+s)$
$\Longrightarrow t_0'=-(2t_0+Re(s))$.  Applying $E(z_0,z_2)$ for $n=2,3$ simultaneously,\\
$\Longrightarrow Re(s)=-3t_0$\\
$\Longrightarrow t_0'=t_0$\\
$\Longrightarrow t_0'=t_0=t_1=t_2$\\
$\Longrightarrow$ null solution.\\
Suppose $t_0=t_1$ and $t_1=t_2$.  Then $t_0'=t_0=t_1=t_2.$\\
$\Longrightarrow$ null solution.\\
Suppose $t_0=t_1$ and $t_2=t_0$.\\
$\Longrightarrow t_0'=t_0=t_1=t_2$\\
$\Longrightarrow$ null solution.\\
Suppose only $E(z_0,z_2)$ satisfies $A=D$ and $B=C$.  Then, $t_0'=t_1$ and $t_0'=t_2$.  Furthermore, $$(t_0=t_1 \mbox{ or } t_0'=t_0'+t_2-t_1),$$ and $$(t_1=t_2 \mbox{ or } t_2=t_0).$$  That is, $$t_0=t_1 \mbox{ or } t_2=t_1,$$ and $$t_1=t_2 \mbox{ or } t_2=t_0.$$ \\
If $t_0=t_1$ and $t_1=t_2$, then $t_0'=t_0=t_1=t_2$.\\
$\Longrightarrow$ null solution.\\
If $t_0=t_1$ and $t_2=t_0$, then $t_0'=t_0=t_1=t_2$.\\
$\Longrightarrow$ null solution.\\
If $t_2=t_1$ and $t_1=t_2$, then $t_0'=t_1=t_2$ and $\bar{s}=-(2t_0+4t_1+s)$.\\
$\Longrightarrow$ $t_0=-(2t_1+Re(s))$.\\
Applying $E(z_0,z_1)$ for $n=2,3$ simultaneously\\
$\Longrightarrow Re(s)=-3t_1$\\
$\Longrightarrow t_0=t_1$\\
$\Longrightarrow t_0'=t_0=t_1=t_2$\\
$\Longrightarrow$ null solution.\\
If $t_2=t_1$ and $t_2=t_0$, then $t_0'=t_0=t_1=t_2$.\\
$\Longrightarrow$ null solution.\\
Finally, suppose only $E(z_1,z_2)$ satisfies $A=D$ and $B=C$.  Then $t_0'+t_0=2t_1$  and $t_0'=t_1$, and $$t_0=t_1 \mbox{ or } t_2=t_1,$$ and $$t_0=t_2 \mbox{ or } t_0=t_1.$$\\
If $t_0=t_1$ and $t_0=t_2$, then $t_0'=t_1=t_0=t_2$.\\
$\Longrightarrow$ null solution\\
If $t_0=t_1$ and $t_0=t_1$, then $t_0'=t_1=t_0$ and $\bar{s}=-(4t_0+2t_2+s).$\\
$\Longrightarrow$ $t_2=-(2t_0+Re(s)).$\\
Applying $E(z_0,z_1)$ for $n=2,3$ simultaneously\\
$\Longrightarrow$ $Re(s)=-3t_0$\\
$\Longrightarrow$ $t_2=t_0$\\
$\Longrightarrow$ $t_0'=t_1=t_0=t_2$\\
$\Longrightarrow$ null solution\\
If $t_2=t_1$ and $t_0=t_2$, then $t_0'=t_1=t_2=t_0.$\\
$\Longrightarrow$ null solution.\\
If $t_2=t_1$ and $t_0=t_1$, then $t_0'=t_1=t_0=t_2.$\\
$\Longrightarrow$ null solution.\\
\indent Now suppose ($A=D$ and $B=C$) holds for no $E(z_i,z_j)$ $((i,j)\in\{(0,1),(0,2),(1,2)\})$.  Then, for each $E(z_i,z_j)$, we have ($A=C$ or $B=D$).  That is, $$t_0=t_1 \mbox{ or }t_2=t_1,$$ and $$t_0=t_2 \mbox{ or }t_0=t_1,$$ and $$t_1=t_2 \mbox{ or }t_2=t_0.$$
Any of the eight possibilities gives $t_0=t_1=t_2$, which is what we wanted.\\
Let $t_0=t_1=t_2=t$.  Then our matrices $M_1,M_2,M_3,M_4$ simplify as follows:

\[\left(
\begin{array}{ccccc}
 0 & t & t_0' & -\frac{(t+t_0')}{R} & 0 \\
 t & t_0' & -(3t+3t_0'+s) & \frac{(2t_0'+2t+s)}{R} & -\theta t \\
 t_0' & -(3t+3t_0'+\bar{s}) & t & -\frac{(t+t_0'+s)}{R} & -\theta t_0' \\
 -\frac{(t+t_0')}{R} & \frac{(2t_0'+2t+\bar{s})}{R} & -\frac{(t+t_0'+\bar{s})}{R} & 0 & \frac{\theta(t+t_0')}{R} \\
 0 & -\theta t & -\theta t_0' & \frac{\theta(t+t_0')}{R} & 0 \\
\end{array}
\right),\]

\[\left(
\begin{array}{ccccc}
 t & t_0' & s & -\frac{(t+t_0'+s)}{R} & -\theta t_0' \\
 t_0' & 0 & t & -\frac{(t+t_0')}{R} & 0 \\
 \bar{s} & t & t_0' & \frac{2(t+t_0)+s}{R} & -\theta t \\
 -\frac{(t+t_0'+\bar{s})}{R} & -\frac{(t+t_0')}{R} & \frac{2(t+t_0)+\bar{s}}{R} & 0 & \frac{\theta(t+t_0')}{R} \\
 -\theta t_0' & 0 & -\theta t & \frac{\theta(t+t_0')}{R} & 0 \\
\end{array}
\right),\]

\[\left(
\begin{array}{ccccc}
 t_0' & -(3t+3t_0'+s) & t & \frac{2(t+t_0)+s}{R} & -\theta t \\
 -(3t+3t_0'+\bar{s}) & t & t_0' & -\frac{(t+t_0'+s)}{R} & -\theta t_0' \\
 t & t_0' & 0 & -\frac{(t+t_0')}{R} & 0 \\
 \frac{2(t+t_0)+\bar{s}}{R} & -\frac{(t+t_0'+\bar{s})}{R} & -\frac{(t+t_0')}{R} & 0 & \frac{\theta(t+t_0')}{R} \\
 -\theta t & -\theta t_0' & 0 & \frac{\theta(t+t_0')}{R} & 0 \\
\end{array}
\right),\]

\[\left(
\begin{array}{cccc}
 -\frac{(t+t_0')}{R^2} & \frac{2(t+t_0')+s}{R^2} & -\frac{(t+t_0'+s)}{R^2} & 0 \\
 \frac{2(t+t_0')+\bar{s}}{R^2} & -\frac{(t+t_0')}{R^2} & \frac{2(t+t_0')+s}{R^2} & 0 \\
 -\frac{(t+t_0'+\bar{s})}{R^2} & \frac{2(t+t_0')+\bar{s}}{R^2} & -\frac{(t+t_0')}{R^2} & 0 \\
 0 & 0 & 0 & 0 \\
\end{array}
\right).\]

One can calculate $Tr(M_1^2)= 3(t^2+(t_0')^2)+2\theta^2(t^2+(t_0')^2)+\frac{2}{R^2}(t+t_0')^2(1+\theta^2)+2\vert s \vert^2\\
+\frac{4}{R^2}(s+t+t_0')(\bar{s}+t+t_0')$.  Furthermore, $Tr(M_2^2)=Tr(M_3^2)=Tr(M_1^2)$.  We calculate $Tr(M_4^2)=\frac{3(t+t_0')^2}{R^4}+\frac{6(t+t_0'+s)(t+t_0'+\bar{s})}{R^4}$.  Thus, \[Coeff_{tr(T^2)}(c)=3Tr(M_1^2)\frac{\lambda(b)}{\lambda(c)}+Tr(M_4^2)\frac{\lambda(d)}{\lambda(c)}\]
\[=3Tr(M_1^2)\frac{1}{\sqrt{3}}+Tr(M_4^2)\frac{\sqrt{3}}{\theta^2}\]
\[=\Big((9+6\theta^2)(t^2+(t_0')^2)+\theta^2(3+2\theta^2)(t+t_0')^2+6\vert s \vert^2 +6\theta^2(s+t+t_0')(\bar{s}+t+t_0')\Big)\frac{1}{\sqrt{3}}\] and \[Coeff_{tr(T^2)}(z_0)=(\theta^4 t^2 +\theta^4 (t_0')^2)\frac{\lambda(b)}{\lambda(z)}+\frac{\theta^4}{R^4}(t+t_0')^2\frac{\lambda(d)}{\lambda(z)}\]
\[=(t^2+(t_0')^2)\frac{\theta^6}{\sqrt{3}}+\theta^8(t+t_0')^2\frac{\sqrt{3}}{9}.\]
Setting $Coeff_{tr(T^2)}(z_0)=Coeff_{tr(T^2)}(c)$, we get 
\[(t^2+(t_0')^2)\frac{\theta^6}{\sqrt{3}}+\theta^8(t+t_0')^2\frac{\sqrt{3}}{9}\] \[=\Big((9+6\theta^2)(t^2+(t_0')^2)+\theta^2(3+2\theta^2)(t+t_0')^2+6\vert s \vert^2 +6\theta^2(s+t+t_0')(\bar{s}+t+t_0')\Big)\frac{1}{\sqrt{3}}\]
\[\Longrightarrow (9+3\sqrt{21})(t^2+(t_0')^2)\frac{1}{\sqrt{3}}+(15+3\sqrt{21})(t+t_0')^2\frac{1}{\sqrt{3}}\]
\[=(6\vert s\vert^2 +6\theta^2(s+t+t_0')(\bar{s}+t+t_0'))\frac{1}{\sqrt{3}}=(6\vert s\vert^2 +6\theta^2\vert s \vert^2 -12\theta^2(t+t_0')^2)\frac{1}{\sqrt{3}}.\]
\[\Longrightarrow (9+3\sqrt{21})(t^2+(t_0')^2)\frac{1}{\sqrt{3}} +(33+9\sqrt{21})(t+t_0')^2\frac{1}{\sqrt{3}}\]
\[=(6+6\theta^2)\vert s\vert^2 \frac{1}{\sqrt{3}}\]
\[\Longrightarrow (9+3\sqrt{21})(t^2+(t_0')^2)+(33+9\sqrt{21})(t+t_0')^2=(15+3\sqrt{21})\vert s \vert^2\]
\[\Longrightarrow (3+\sqrt{21})(t^2+(t_0')^2)+(11+3\sqrt{21})(t+t_0')^2=(5+\sqrt{21})\vert s\vert^2\]
\[\Longrightarrow \Big(\frac{-3+\sqrt{21}}{2}\Big)(t^2+(t_0')^2)+(-2+\sqrt{21})(t+t_0')^2=\vert s\vert^2\]
\[\Longrightarrow Im(s)=\pm\sqrt{\frac{-3+\sqrt{21}}{2}(t^2+(t_0')^2)+(-\frac{17}{4}+\sqrt{21})(t+t_0')^2},\] since $Re(s)=-\frac{3}{2}(t+t_0')$.  Hence, $s=Re(s)\pm i\cdot Im(s)$, where $Re(s)$ and $Im(s)$ are given as above.   
By setting $Coeff_{tr(T^3)}(c)=Coeff_{tr(T^3)}(z_0),$
we obtain a quadratic equation in the variables $t$ and $t_0'$:\[-\frac{3}{2}(t+t_0')\big((21\sqrt{2}+13\sqrt{7})t^2+7(17\sqrt{3}+11\sqrt{7})t\cdot t_0'+(21\sqrt{3}+13\sqrt{7})(t_0')^2\big)=0\]
\[\Longrightarrow t_0'=-t \mbox { or } t_0'=\frac{1}{4}\big(-7-\sqrt{21} \pm \sqrt{54+14\sqrt{21}}\big)\cdot t\]

In sum, we have the matrix representation for $T$:
\begin{align*}
\widetilde{T}= &
\left(
\begin{array}{ccccc}
 0 & t & t_0' & -\frac{(t+t_0')}{R} & 0 \\
 t & t_0' & -(3t+3t_0'+s) & \frac{(2t_0'+2t+s)}{R} & -\theta t \\
 t_0' & -(3t+3t_0'+\bar{s}) & t & -\frac{(t+t_0'+s)}{R} & -\theta t_0' \\
 -\frac{(t+t_0')}{R} & \frac{(2t_0'+2t+\bar{s})}{R} & -\frac{(t+t_0'+\bar{s})}{R} & 0 & \frac{\theta(t+t_0')}{R} \\
 0 & -\theta t & -\theta t_0' & \frac{\theta(t+t_0')}{R} & 0 \\
\end{array}
\right)
\\
& \oplus
\left(
\begin{array}{ccccc}
 t & t_0' & s & -\frac{(t+t_0'+s)}{R} & -\theta t_0' \\
 t_0' & 0 & t & -\frac{(t+t_0')}{R} & 0 \\
 \bar{s} & t & t_0' & \frac{2(t+t_0)+s}{R} & -\theta t \\
 -\frac{(t+t_0'+\bar{s})}{R} & -\frac{(t+t_0')}{R} & \frac{2(t+t_0)+\bar{s}}{R} & 0 & \frac{\theta(t+t_0')}{R} \\
 -\theta t_0' & 0 & -\theta t & \frac{\theta(t+t_0')}{R} & 0 \\
\end{array}\right)
\\
& \oplus 
\left(
\begin{array}{ccccc}
 t_0' & -(3t+3t_0'+s) & t & \frac{2(t+t_0)+s}{R} & -\theta t \\
 -(3t+3t_0'+\bar{s}) & t & t_0' & -\frac{(t+t_0'+s)}{R} & -\theta t_0' \\
 t & t_0' & 0 & -\frac{(t+t_0')}{R} & 0 \\
 \frac{2(t+t_0)+\bar{s}}{R} & -\frac{(t+t_0'+\bar{s})}{R} & -\frac{(t+t_0')}{R} & 0 & \frac{\theta(t+t_0')}{R} \\
 -\theta t_0' & -\theta t_0' & 0 & \frac{\theta(t+t_0')}{R} & 0 \\
\end{array}\right) 
\\
& \oplus 
\left(
\begin{array}{cccc}
 -\frac{(t+t_0')}{R^2} & \frac{2(t+t_0')+s}{R^2} & -\frac{(t+t_0'+s)}{R^2} & 0 \\
 \frac{2(t+t_0')+\bar{s}}{R^2} & -\frac{(t+t_0')}{R^2} & \frac{2(t+t_0')+s}{R^2} & 0 \\
 -\frac{(t+t_0'+\bar{s})}{R^2} & \frac{2(t+t_0')+\bar{s}}{R^2} & -\frac{(t+t_0')}{R^2} & 0 \\
 0 & 0 & 0 & 0 \\
\end{array}\right)
\\
& \oplus 
\biggl(
\left(
\begin{array}{cc}
 0 & 0  \\
 0 & 0  \\
\end{array}\right) 
 \oplus 
 \left(
\begin{array}{c}
-\theta^2t \\
\end{array}\right)
\oplus 
\left(
\begin{array}{c}
 -\theta^2t_0' \\
\end{array}\right)
\oplus
\left(
\begin{array}{c}
 \frac{\theta^2}{R^2}(t+t_0')\\
\end{array}
\right) 
\biggr)^{\otimes 3}
\end{align*}
We want to find the matrix representation for $Q$.  
We have the following matrices for $M_{i_{z_0,v}}(\mathbb{C})$:

\[\left(
\begin{array}{cc}
 0 & 0  \\
 0 & 0  \\
\end{array}
\right),
\left(
\begin{array}{c}
2\theta^2t_0 \\
\end{array}
\right),
\left(
\begin{array}{c}
 2\theta^2 t_0' \\
\end{array}
\right),
\left(
\begin{array}{c}
 -\frac{2}{3}\theta^4(t_0+t_0')\\
\end{array}
\right).\]

We have the following matrices for $M_{i_{z_1,v}}(\mathbb{C})$:

\[\left(
\footnotesize\begin{array}{cc}
 0 & 0  \\
 0 & 0  \\
\end{array}
\right),
\left(
\footnotesize\begin{array}{c}
2\theta^2t_1 \\
\end{array}
\right),
\left(
\footnotesize\begin{array}{c}
 2\theta^2(t_0+t_0'-t_2)+\omega(\bar{s}+\omega s)\theta^2 \\
\end{array}
\right),
\left(
\scriptsize\begin{array}{c}
 \frac{\theta^4}{3}(-2(t_1+t_0+t_0'-t_2)-\omega(\bar{s}+\omega s))\\
\end{array}
\right).\]

We have the following matrices for $M_{i_{z_2,v}}(\mathbb{C})$:

\[\left(
\footnotesize\begin{array}{cc}
 0 & 0  \\
 0 & 0  \\
\end{array}
\right),
\left(
\footnotesize\begin{array}{c}
2\theta^2t_2 \\
\end{array}
\right),
\left(
\footnotesize\begin{array}{c}
 2\theta^2(t_1+t_0'-t_2)-\omega(\omega \bar{s}+s)\theta^2) \\
\end{array}
\right),
\left(
\footnotesize\begin{array}{c}
 \frac{\theta^4}{3}(-2(t_1+t_0')+\omega(\omega \bar{s}+s))\\
\end{array}
\right).\]
It is easy to check that \[Coeff_{tr(Q)}(z_0)=Coeff_{tr(Q)}(z_1)= Coeff_{tr(Q)}(z_2)=0.\]

Consider $Coeff_{tr(Q^n)}(z_0)$, for $n\geq 1.$  We calculate \[Coeff_{tr(Q^n)}(z_0)=\Big((2\theta^2t_0)^n+(2\theta^2t_0')^n\Big)\frac{\theta^2}{\sqrt{3}}+\Big(\frac{-2}{3}\theta^4(t_0+t_0')\Big)^n\sqrt{3}.\]
Similarly, \[Coeff_{tr(Q^n)}(z_1)=\Big((2\theta^2t_1)^n+(2\theta^2(t_0+t_0'-t_2)+\omega(\bar{s}+\omega s)\theta^2)^n\Big)\frac{\theta^2}{\sqrt{3}}\]
\[+\Big(\frac{\theta^4}{3}(-2(t_1+t_0+t_0'-t_2)-\omega(\bar{s}+\omega s))\Big)^n\sqrt{3},\] 
and \footnotesize\[Coeff_{tr(Q^n)}(z_2)=\Big((2\theta^2t_2)^n+(2\theta^2(t_1+t_0'-t_2)-\omega(\omega\bar{s}+s)\theta^2)^n\Big)\frac{\theta^2}{\sqrt{3}}+\Big(\frac{\theta^4}{3}(-2(t_1+t_0')+\omega(\omega\bar{s}+s))\Big)^n\sqrt{3}.\]  
\normalsize 
We must have $Coeff_{tr(Q^n)}(z_0)=Coeff_{tr(Q^n)}(z_1)=Coeff_{tr(Q^n)}(z_2)$.  By inspection, it seems intuitively clear that $t_0=t_1=t_2$ and $s=0$.  It is our goal to show that this must be the case.\\
\indent By Claim A.1.1, for each equation $E(z_0,z_1),E(z_0,z_2),E(z_1,z_2)$, we either have ($A=D$ and $B=C$) or ($A=C$ or $B=D$).  We will use this fact to prove that $t_0=t_1=t_2$ and $s=0$.  First, we need to be able to calculate the coefficient of $c$ in $tr(Q^n)$ so that we can utilize the equation $Coeff_{tr(Q^n)}(c)=Coeff_{tr(Q^n)}(z_0)$.  In order to do this, we will need the matrices for $M_{i_{c,v}}(\mathbb{C})$.
Let $\eta = 1+i\sqrt{3}$.
We have the following matrices for $M_{i_{c,v}}(\mathbb{C})$:
\tiny
\[\left(
\begin{array}{ccccc}
 0 & \eta t_0 & \eta t_0' & -\frac{\eta \theta }{\sqrt{3}}(t_0+t_0') & 0 \\
 \bar{\eta}t_0 & -2(t_0+t_0'-t_2)-\omega(\bar{s}+\omega s) & \omega\bar{s} & \frac{2\theta}{\sqrt{3}}(-\omega^2 t_0 +t_0'-t_2)+\frac{\omega^2\theta}{\sqrt{3}}s & -\bar{\eta}\theta t_0 \\
 \bar{\eta}t_0' & \omega^2 s & -2t_2 & \frac{2\theta}{\sqrt{3}}(t_2+\omega t_0')-\frac{\omega^2\theta}{\sqrt{3}}s & -\bar{\eta} \theta t_0' \\
 -\frac{\bar{\eta} \theta }{\sqrt{3}}(t_0+t_0') & \frac{2\theta}{\sqrt{3}}(-\omega t_0 +t_0'-t_2)+\frac{\omega\theta}{\sqrt{3}}\bar{s} & \frac{2\theta}{\sqrt{3}}(t_2+\omega^2 t_0')-\frac{\omega\theta}{\sqrt{3}}\bar{s} & 0 & \frac{\bar{\eta}\theta^2(t_0+t_0')}{\sqrt{3}} \\
 0 & -\eta \theta t_0 & -\eta \theta t_0' & \frac{\eta \theta^2(t_0+t_0')}{\sqrt{3}} & 0 \\
\end{array}
\right),\]
\tiny
\[\left(
\begin{array}{ccccc}
 -2t_0 & \bar{\eta}(t_0+t_0'-t_2)-\omega^2(\bar{s}+\omega s) & 0 & M_{2_{(1,4)}} & M_{2_{(1,5)}} \\
 \eta (t_0+t_0'-t_2)-\omega(s+\omega^2 \bar{s}) & 0 & \eta t_1 & M_{2_{(2,4)}} & 0 \\
 0 & \bar{\eta}t_1 & -2(t_1+t_0'-t_2)+\omega(\omega\bar{s}+s) & M_{2_{(3,4)}} & -\bar{\eta}\theta t_1 \\
 \overline{M_{2_{(1,4)}}} & \overline{M_{2_{(2,4)}}} & \overline{M_{2_{(3,4)}}} & 0 & M_{2_{(4,5)}} \\
 \overline{M_{2_{(1,5)}}} & 0 & -\eta \theta t_1 & \overline{M_{2_{(4,5)}}} & 0 \\
\end{array}
\right),\]
\scriptsize where $M_{2_{(1,4)}}=\frac{2\theta}{\sqrt{3}}(-\omega^2 t_0+\omega t_0'-\omega t_2)+\frac{\omega^2\theta}{\sqrt{3}}\bar{s}$, $M_{2_{(1,5)}}=-\bar{\eta}\theta(t_0+t_0'-t_2)+\omega^2\theta\bar{s}+\theta s$, $M_{2_{(2,4)}}=-\frac{\eta\theta}{\sqrt{3}}(t_1+t_0+t_0'-t_2)+\frac{\theta}{\sqrt{3}}(\bar{s}+\omega s)$, $M_{2_{(3,4)}}=\frac{2\theta}{\sqrt{3}}(t_0'-\omega^2 t_1-t_2)-\frac{\omega\theta}{\sqrt{3}}(s-\bar{s})$, $M_{2_{(4,5)}}=\frac{\bar{\eta} \theta^2 }{\sqrt{3}}(t_1+t_0+t_0'-t_2)-\frac{\theta^2\omega^2}{\sqrt{3}}(\bar{s}+\omega s)$
\scriptsize
\[\left(
\begin{array}{ccccc}
 -2t_0' & \omega^2\bar{s} & \bar{\eta} t_2 & M_{3_{(1,4)}} & -\bar{\eta}\theta t_2 \\
 \omega s & -2t_1 & \bar{\eta}(t_1+t_0'-t_2)+\omega^2(\omega\bar{s}+s) & M_{3_{(2,4)}} & M_{3_{(2,5)}} \\
 \eta t_2 & \eta (t_1+t_0'-t_2)+\omega(\omega^2 s+\bar{s}) & 0 & M_{3_{(3,4)}} & 0 \\
 \overline{M_{3_{(1,4)}}} & \overline{M_{3_{(2,4)}}} & \overline{M_{3_{(3,4)}}} & 0 & M_{3_{(4,5)}}\\
 -\eta \theta t_2 & \overline{M_{3_{(2,5)}}} & 0 & \overline{M_{3_{(4,5)}}} & 0 \\
\end{array}
\right),\]
where
$M_{3_{(1,4)}}=\frac{2\theta}{\sqrt{3}}(t_0'+\omega t_2)-\frac{\omega^2\theta}{\sqrt{3}}\bar{s}$,
$M_{3_{(2,4)}}=\frac{2\theta}{\sqrt{3}}(-\omega^2 t_1+\omega t_0'-\omega t_2)-\frac{\theta}{\sqrt{3}}(\bar{s}-s)$,
$M_{3_{(2,5)}}=-\theta\bar{\eta}(t_1+t_0'-t_2)-\theta\omega^2(\omega\bar{s}+s)$,
$M_{3_{(3,4)}}=-\frac{\eta \theta }{\sqrt{3}}(t_1+t_0')-\frac{\theta}{\sqrt{3}}(\omega\bar{s}+s)$, $M_{3_{(4,5)}}=\frac{\theta^2\bar{\eta}}{\sqrt{3}}(t_1+t_0')+\frac{\omega^2\theta^2}{\sqrt{3}}(\omega\bar{s}+s)$

\tiny
\[\left(
\begin{array}{cccc}
 \frac{2\theta^2}{3}(t_0+t_0') & -\frac{\bar{\eta}\theta^2}{3}(-\omega^2 t_0+t_0'- t_2)+\frac{\theta^2}{3}s & -\frac{\bar{\eta}\theta^2}{3}(t_2+\omega t_0')-\frac{\theta^2}{3}s & 0 \\
 -\frac{\eta \theta^2}{3}(-\omega t_0+t_0'-t_2)+\frac{\theta^2}{3}\bar{s} & \frac{2\theta^2}{3}(t_1+t_0+t_0'-t_2)+\frac{\omega\theta^2}{3}(\bar{s}+\omega s) & -\frac{\bar{\eta}\theta^2}{3}(t_0'-\omega^2t_1-t_2)-\frac{\omega^2\theta^2}{3}(s-\bar{s}) & 0 \\
 -\frac{\eta \theta^2}{3}(t_2+\omega^2 t_0')-\frac{\theta^2}{3}\bar{s} & -\frac{\eta\theta^2}{3}(t_0'-\omega t_1- t_2)-\frac{\omega\theta^2}{3}(\bar{s}-s) &  \frac{2\theta^2}{3}(t_1+t_0')-\frac{\omega\theta^2}{3}(\omega\bar{s}+s) & 0 \\
 0 & 0 & 0 & 0 \\
\end{array}
\right).\]
\normalsize

If we call the above four matrices $N_1,N_2,N_3,N_4$, then $$Coeff_{tr(Q^n)}(c)=(Tr(N_1^n)+Tr(N_2^n)+Tr(N_3^n))\cdot\frac{\lambda(b)}{\lambda(c)} +Tr(N_4^n)\cdot\frac{\lambda(d)}{\lambda(c)} .$$
That is, $$Coeff_{tr(Q^n)}(c)=(Tr(N_1^n)+Tr(N_2^n)+Tr(N_3^n))\cdot\frac{1}{\sqrt{3}} +Tr(N_4^n)\cdot\frac{\sqrt{3}}{\theta^2}. $$

We now prove that $t_0=t_1=t_2$ and $s=0$.
Suppose $A=D$ and $B=C$ for each equation $E(z_0,z_1),E(z_0,z_2),E(z_1,z_2)$.  Then $$0=t_0'-t_2+\frac{\omega(\bar{s}+\omega s)}{2} \mbox{ and } t_0'=t_1$$ and $$t_0=t_1+t_0'-t_2-\frac{\omega(\omega\bar{s}+s)}{2} \mbox{ and } t_0'=t_2$$ and $$0=t_0'-t_2-\frac{\omega(\omega\bar{s}+s)}{2} \mbox{ and } t_0+t_0'-t_2+\frac{\omega(\bar{s}+\omega s)}{2}=t_2$$
So $t_0'=t_1=t_2$ and $\omega(\bar{s}+\omega s)=0$ and $\omega(\omega\bar{s}+s)=0$\\
$\Longrightarrow s=0$\\
So $t_0'=t_0=t_1=t_2$ and $s=0$.  Setting $Coeff_{tr(Q^n)}(c)=Coeff_{tr(Q^n)}(z_0)$ for $n=2$, implies $t_0=0$.\\
$\Longrightarrow$ we have the null solution.\\
Now, suppose only two equations $E(z_i,z_j)$ (where $(i,j)\in\{(0,1),(0,2),(1,2)\}$) satisfy $A=D$ and $B=C$.\\
First, suppose only $E(z_0,z_1)$ and $E(z_0,z_2)$ satisfy $A=D$ and $B=C$.\\
Then $$0=t_0'-t_2+\frac{\omega(\bar{s}+\omega s)}{2} \mbox{ and } t_0'=t_1,$$\\
 and $$t_0=t_1+t_0'-t_2-\frac{\omega(\omega\bar{s}+s)}{2} \mbox{ and } t_0'=t_2.$$\\
$\Longrightarrow t_0'=t_1=t_2$ and $\bar{s}+\omega s=0$ and $t_0=t_1-\frac{\omega(\omega\bar{s}+s)}{2}$\\
$\Longrightarrow t_0=t_1-\frac{\omega(-\omega^2+1)}{2}s$\\
$\Longrightarrow t_1-t_0=\frac{\omega(1-\omega^2)}{2}s$\\
$\Longrightarrow s=\frac{2}{\omega-1}(t_1-t_0)$\\
Setting $Coeff_{tr(Q^n)}(c)=Coeff_{tr(Q^n)}(z_0)$ for $n=2$,\\
$\Longrightarrow t_0=t_1=0$\\
$\Longrightarrow$ null solution.\\
Now, suppose only $E(z_0,z_1)$ and $E(z_1,z_2)$ satisfy $A=D$ and $B=C$.\\
Then $$0=t_0'-t_2+\frac{\omega(\bar{s}+\omega s)}{2} \mbox{ and } t_0'=t_1,$$
and $$0=t_0'-t_2-\frac{\omega(\omega\bar{s}+s)}{2} \mbox{ and } t_0+t_0'-t_2+\frac{\omega(\bar{s}+\omega s)}{2}=t_2$$\\
$\Longrightarrow t_0=t_2$ and $\omega(\bar{s}+\omega s) +\omega(\omega\bar{s}+s)=0$\\
$\Longrightarrow (\omega+\omega^2)(\bar{s}+s)=0$\\
$\Longrightarrow \bar{s}+s=0$\\
$\Longrightarrow \bar{s}=-s$\\
and $t_0+t_0'-t_2+\frac{\omega(-1+\omega)}{2}s=t_2$\\
$\Longrightarrow t_0+t_1-t_2+\frac{\omega(\omega-1)}{2}s=t_2$\\
$\Longrightarrow s=\frac{2}{\omega(\omega-1)}(2t_2-t_0-t_1)=\frac{2}{\omega(\omega-1)}(t_0-t_1)$\\
Setting $Coeff_{tr(Q^n)}(c)=Coeff_{tr(Q^n)}(z_0)$ for $n=2$ implies $t_0=t_1=0$\\
$\Longrightarrow$ null solution.\\
Suppose only $E(z_0,z_2)$ and $E(z_1,z_2)$ satisfy $A=D$ and $B=C$.\\
Then $$t_0=t_1+t_0'-t_2-\frac{\omega(\omega\bar{s}+s)}{2} \mbox{ and } t_0'=t_2$$\\
and $$0=t_0'-t_2-\frac{\omega(\omega\bar{s}+s)}{2} \mbox{ and } t_0+t_0'-t_2+\frac{\omega(\bar{s}+\omega s)}{2}=t_2$$\\
$\Longrightarrow \frac{\omega(\omega\bar{s}+s)}{2}=0$ and $t_0=t_1$\\
It follows that $\bar{s}=-\omega^2 s$.  Hence $t_0+\frac{\omega(\bar{s}+\omega s)}{2}=t_2$\\
$\Longrightarrow t_0+\frac{\omega(-\omega^2+\omega)}{2}s=t_2$\\
$\Longrightarrow t_0+\frac{(-1+\omega^2)}{2}s=t_2$\\
$\Longrightarrow s=\frac{2}{(\omega^2-1)}(t_2-t_0)$\\
Setting $Coeff_{tr(Q^n)}(c)=Coeff_{tr(Q^n)}(z_0)$ for $n=2$ implies $t_0=t_2=0$\\
$\Longrightarrow$ null solution\\
Now suppose only one equation $E(z_i,z_j)$ (where $(i,j)\in\{(0,1),(0,2),(1,2)\}$) satisfies $A=D$ and $B=C$.\\
First, suppose $E(z_0,z_1)$ satisfies $A=D$ and $B=C$. Then $$0=t_0'-t_2+\frac{\omega(\bar{s}+\omega s)}{2} \mbox{ and } t_0'=t_1,$$\\
and since ($A=C$ or $B=D$) for $E(z_0,z_2)$ and $E(z_1,z_2)$, we have $$\mbox{ either } t_0=t_2 \mbox{ or } t_0'=t_1+t_0'-t_2-\frac{\omega(\omega\bar{s}+s)}{2}$$\\
and $$\mbox{ either } t_1=t_2 \mbox{ or } t_0+t_0'-t_2+\frac{\omega(\bar{s}+\omega s)}{2}=t_1+t_0'-t_2-\frac{\omega(\omega\bar{s}+s)}{2}$$\\
That is, $$\mbox{ either } t_0=t_2 \mbox { or } 0=t_1-t_2-\frac{\omega(\omega\bar{s}+s)}{2}$$\\
and $$\mbox{ either } t_1=t_2 \mbox{ or } t_0-t_1-\frac{(\bar{s}+s)}{2}=0$$\\
If $t_0=t_2$ and $t_1=t_2$, then $t_0'=t_1=t_2=t_0$ and $\bar{s}+\omega s=0$.\\
Setting $Coeff_{tr(Q^n)}(z_0)=Coeff_{tr(Q^n)}(z_2)$ for $n=2,3$ simultaneously \\
$\Longrightarrow s=0$\\
$\Longrightarrow$ null solution.\\
If $t_0=t_2 \mbox{ and } t_0-t_1-\frac{(\bar{s}+s)}{2}=0,$ then since $0=t_0'-t_2+\frac{\omega(\bar{s}+\omega s)}{2},$\\
we have 
$0=t_1-t_2+\frac{\omega(\bar{s}+\omega s)}{2}$ and $0=t_2-t_1-\frac{(\bar{s}+s)}{2}$\\
Adding the two together, we get \\
$0=(\omega-1)\bar{s}+(\omega^2-1)s$\\
$\Longrightarrow \bar{s}=\frac{(1-\omega^2)}{\omega-1}s$\\
$\Longrightarrow \bar{s}=-(\omega+1)s$\\
$\Longrightarrow t_0-t_1-\frac{(-(\omega+1)s+s)}{2}=0$\\
$\Longrightarrow t_0-t_1-\frac{(-\omega)}{2}s=0$\\
$\Longrightarrow s=\frac{2}{\omega}(t_1-t_0)$\\
So $t_0'=t_1$, $t_0=t_2$, and $s=\frac{2}{\omega}(t_1-t_0)$.\\
Setting $Coeff_{tr(Q^n)}(z_0)=Coeff_{tr(Q^n)}(z_2)$ for $n=2,3$ simultaneously \\
$\Longrightarrow t_0=t_1$\\
$\Longrightarrow t_0'=t_1=t_0=t_2$ and $s=0$\\
$\Longrightarrow$ null solution.\\
If $0=t_1-t_2-\frac{\omega(\omega\bar{s}+s)}{2}$ and $t_1=t_2$,\\
$\Longrightarrow \omega\bar{s}+s=0$\\
But $0=t_0'-t_2+\frac{\omega(\bar{s}+\omega s)}{2}$\\
$=t_1-t_2+\frac{\omega(\bar{s}+\omega s)}{2}$\\
$=0+\frac{\omega(\bar{s}+\omega s)}{2}$\\
$\Longrightarrow \bar{s}+\omega s=0$\\
$\Longrightarrow \omega\bar{s}+s=0$ and $\omega\bar{s}+\omega^2 s=0$\\
$\Longrightarrow (1-\omega^2)s=0$\\
$\Longrightarrow s=0$\\
So $t_0'=t_1=t_2$ and $s=0$.\\
Setting $Coeff_{tr(Q^n)}(z_0)=Coeff_{tr(Q^n)}(z_2)$ for $n=2,3$ simultaneously \\
$\Longrightarrow t_0=t_1$\\
$\Longrightarrow t_0'=t_1=t_2=t_0$ and $s=0$\\
$\Longrightarrow$ null solution.\\
If $$0=t_1-t_2-\frac{\omega(\omega\bar{s}+s)}{2} \mbox{ and } 0=t_0-t_1-\frac{(\bar{s}+s)}{2},$$
then since $0=t_1-t_2+\frac{\omega(\bar{s}+\omega s)}{2}$ and $0=t_1-t_2-\frac{\omega(\omega\bar{s}+s)}{2}$,\\
subtracting the second equation from the first gives $0=\frac{\omega+\omega^2}{2}(\bar{s}+s)$\\
$\Longrightarrow \bar{s}+s=0$\\
$\Longrightarrow t_0=t_1$\\
Also, $0=t_1-t_2-\frac{\omega(\omega(-s)+s)}{2}$\\
$\Longrightarrow s=\frac{2}{\omega(1-\omega)}(t_1-t_2)$\\
So $t_0'=t_1$ and $t_0=t_1$ and $s=\frac{2}{\omega(1-\omega)}(t_1-t_2)$\\
Setting $Coeff_{tr(Q^n)}(z_0)=Coeff_{tr(Q^n)}(z_2)$ for $n=2,3$ simultaneously \\
$\Longrightarrow t_0=t_2$\\
So $t_0'=t_1=t_0=t_2$ and $s=0$\\
$\Longrightarrow$ null solution.\\
Next, suppose only $E(z_0,z_2)$ satisfies $A=D$ and $B=$C.  Then $t_0=t_1+t_0'-t_2-\frac{\omega(\omega\bar{s}+s)}{2}$ and $t_0'=t_2$, \\
and either $t_0=t_1$ or $t_0'=t_0+t_0'-t_2+\frac{\omega(\bar{s}+\omega s)}{2}$,\\
and either $t_1=t_2$ or $t_0-t_1-\frac{(\bar{s}+s)}{2}=0$.\\
That is, $t_0=t_1$ or $0=t_0-t_2+\frac{\omega(\bar{s}+\omega s)}{2}$,\\
and $t_1=t_2$ or $t_0-t_1-\frac{\bar{s}+s}{2}=0$.\\
If $t_0=t_1$ and $t_1=t_2$, then $t_0'=t_2=t_1=t_0$ and $\omega\bar{s}+s=0$.\\
Setting $Coeff_{tr(Q^n)}(z_0)=Coeff_{tr(Q^n)}(z_1)$ for $n=2,3$ simultaneously \\
$\Longrightarrow s=0$\\
$\Longrightarrow$ null solution.\\
If $t_0=t_1$ and $t_0-t_1-\frac{\bar{s}+s}{2}=0$, then $\bar{s}+s=0$ and $\omega\bar{s}+s=0$\\
$\Longrightarrow s=0$\\
So $t_0=t_1$ and $t_0'=t_2$ and $s=0$.\\
Setting $Coeff_{tr(Q^n)}(z_0)=Coeff_{tr(Q^n)}(z_1)$ for $n=2,3$ simultaneously \\
$\Longrightarrow t_0=t_2$\\
$\Longrightarrow t_0'=t_2=t_0=t_1$ and $s=0$\\
$\Longrightarrow$ null solution.\\
If $0=t_0-t_2+\frac{\omega(\bar{s}+\omega s)}{2}$ and $t_1=t_2$,\\
then $t_0=t_1-\frac{\omega(\omega\bar{s}+s)}{2}$ and $0=t_0-t_2+\frac{\omega(\bar{s}+\omega s)}{2}$\\
$\Longrightarrow 0=t_2-t_0-\frac{\omega(\omega\bar{s}+s)}{2}$ and $0=t_0-t_2+\frac{\omega(\bar{s}+\omega s)}{2}$\\
Adding the above two equations gives $(\omega-\omega^2)\bar{s}+(\omega^2-\omega)s=0$\\
$\Longrightarrow \bar{s}=s$\\
$\Longrightarrow 0=t_0-t_2+\frac{\omega(1+\omega)}{2}s$\\
$\Longrightarrow s=\frac{2}{\omega(1+\omega)}(t_2-t_0)$.\\
Setting $Coeff_{tr(Q^n)}(z_0)=Coeff_{tr(Q^n)}(z_1)$ for $n=2,3$ simultaneously \\
$\Longrightarrow t_0=t_1$\\
$\Longrightarrow t_0'=t_0=t_1=t_2$ and $s=0$\\
$\Longrightarrow$ null solution.\\
If $0=t_0-t_2+\frac{\omega(\bar{s}+\omega s)}{2}$ and $0=t_0-t_1-\frac{\bar{s}+s}{2}$, \\
then $0=t_1-t_0-\frac{\omega(\omega\bar{s}+s)}{2}$ and $0=t_0-t_1-\frac{\bar{s}+s}{2}$\\
$\Longrightarrow 0=-(\omega^2+1)\bar{s}+(-\omega-1)s$\\
$\Longrightarrow 0=-(\omega^2+1)\bar{s}-(\omega+1)s$\\
$\Longrightarrow \bar{s}=-\frac{\omega+1}{\omega^2+1}s$\\
$=-\frac{\omega+1}{-\omega}s$\\
$=-\frac{-\omega^2}{-\omega}s$\\
$=-\omega s$\\
$\Longrightarrow \bar{s}+\omega s=0$\\
$\Longrightarrow t_0=t_2$ and $0=t_0-t_1-\frac{-\omega+1}{2}s$\\
$\Longrightarrow s=\frac{2}{1-\omega}(t_0-t_1)$\\
So $t_0'=t_2=t_0$ and $s=\frac{2}{1-\omega}(t_0-t_1)$\\
Setting $Coeff_{tr(Q^n)}(z_0)=Coeff_{tr(Q^n)}(z_1)$ for $n=2,3$ simultaneously \\
$\Longrightarrow t_0=t_1$\\
$\Longrightarrow t_0'=t_0=t_1=t_2$ and $s=0$\\
$\Longrightarrow$ null solution.\\
Finally, suppose only $E(z_1,z_2)$ satisfies $A=D$ and $B=C$.\\
Then $0=t_0'-t_2-\frac{\omega(\omega\bar{s}+s)}{2}$ and $t_0+t_0'-2t_2+\frac{\omega(\bar{s}+\omega s)}{2}=0$,\\
and either $t_0=t_1$ or $0=t_0-t_2+\frac{\omega(\bar{s}+\omega s)}{2}$\\
and either $t_0=t_2$ or $0=t_1-t_2-\frac{\omega(\omega\bar{s}+s)}{2}$.\\
If $t_0=t_1$ and $t_0=t_2$, then $0=t_0'-t_2-\frac{\omega(\omega\bar{s}+s)}{2}$ and $0=t_0'-t_2+\frac{\omega(\bar{s}+\omega s)}{2}$\\
and subtracting the first equation from the second equation gives $0=(\omega +\omega^2)\bar{s}+(\omega^2+\omega)s$\\
$\Longrightarrow \bar{s}+s=0$\\
$\Longrightarrow 0=t_0'-t_2-\frac{\omega(-\omega+1)}{2}s$\\
$\Longrightarrow s=\frac{2}{\omega(1-\omega)}(t_0'-t_2)$\\
Setting $Coeff_{tr(Q^n)}(z_0)=Coeff_{tr(Q^n)}(z_1)$ for $n=2,3$ simultaneously \\
$\Longrightarrow t_0'=t_0$\\
$\Longrightarrow t_0'=t_0=t_1=t_2$ and $s=0$\\
$\Longrightarrow$ null solution.\\
If $t_0=t_1$ and $0=t_1-t_2-\frac{\omega(\omega\bar{s}+s)}{2}$,\\
then $t_0'=t_1$.\\
Furthermore, $0=t_0'-t_2-\frac{\omega(\omega\bar{s}+s)}{2}$ and $0=t_1+t_0'-2t_2+\frac{\omega(\bar{s}+\omega s)}{2}$\\
Subtracting the first equation from the second gives $0=t_1-t_2+\frac{(\omega+\omega^2)}{2}(\bar{s}+s)$\\
$\Longrightarrow 0=t_1-t_2-\frac{\bar{s}+s}{2}$\\
But $0=t_1-t_2-\frac{\omega(\omega\bar{s}+s)}{2}$.\\
Subtracting the first equation from the second gives $\frac{(-\omega^2+1)\bar{s}+(-\omega+1)s}{2}$\\
$\Longrightarrow 0=(1-\omega^2)\bar{s}+(1-\omega)s$\\
$\Longrightarrow 0=(1+\omega)\bar{s}+s$\\
$\Longrightarrow 0=-\omega^2\bar{s}+s$\\
$\Longrightarrow \bar{s}=\omega s$\\
$\Longrightarrow 0=t_1-t_2-\frac{\omega(\omega^2+1)}{2}s$\\
$\Longrightarrow 0=t_1-t_2-\frac{\omega(-\omega)}{2}s$\\
$\Longrightarrow s=\frac{2}{\omega^2}(t_2-t_1)$\\
So $t_0'=t_1=t_0$ and $s=2\omega(t_2-t_1)$\\
Setting $Coeff_{tr(Q^n)}(z_0)=Coeff_{tr(Q^n)}(z_1)$ for $n=2,3$ simultaneously \\
$\Longrightarrow t_0=t_2$\\
$\Longrightarrow t_0'=t_1=t_0=t_2$ and $s=0$\\
$\Longrightarrow$ null solution.\\
If $t_0-t_2+\frac{\omega(\bar{s}+\omega s)}{2}=0$ and $t_0=t_2$,\\
then $\bar{s}+\omega s=0$\\
$\Longrightarrow 0=t_0+t_0'-2t_2+0$\\
$\Longrightarrow 0=t_0'-t_2$\\
$\Longrightarrow t_0'=t_2$\\
$\Longrightarrow \omega\bar{s}+s=0$\\
But $\bar{s}+\omega s=0$\\
$\Longrightarrow s=0$\\
So $t_0'=t_2=t_0$ and $s=0$\\
Setting $Coeff_{tr(Q^n)}(z_0)=Coeff_{tr(Q^n)}(z_1)$ for $n=2,3$ simultaneously \\
$\Longrightarrow t_1=t_0$\\
$\Longrightarrow t_0'=t_0=t_1=t_2$ and $s=0$\\
$\Longrightarrow$ null solution.\\
If $0=t_0-t_2+\frac{\omega(\bar{s}+\omega s)}{2}$ and $0=t_1-t_2-\frac{\omega(\omega\bar{s}+s)}{2}$:
Recall that $0=t_0'-t_2-\frac{\omega(\omega\bar{s}+s)}{2}$ and $0=t_0+t_0'-2t_2+\frac{\omega(\bar{s}+\omega s)}{2}$.\\
Subtracting the first equation from the second gives $0=t_0-t_2+\frac{\omega^2+\omega}{2}(\bar{s}+s)$\\
$\Longrightarrow$ $0=t_0-t_2-\frac{\bar{s}+s}{2}$\\
But $0=t_0-t_2+\frac{\omega(\bar{s}+\omega s)}{2}$\\
Subtracting the first equation from the second gives $0=\frac{(\omega+1)\bar{s}+(\omega^2+1)s}{2}$\\
$\Longrightarrow 0=(\omega+1)\bar{s}-\omega s$\\
$\Longrightarrow 0=-\omega^2\bar{s}-\omega s$\\
$\Longrightarrow 0=\omega\bar{s}+s$\\
$\Longrightarrow \bar{s}=-\omega^2 s$\\
Plugging this into $0=t_0-t_2+\frac{\omega(\bar{s}+\omega s)}{2}$, we get $0=t_0-t_2+\frac{\omega(-\omega^2+\omega)}{2}s$\\
$\Longrightarrow 0=t_0-t_2+\frac{\omega^2-1}{2}s$\\
$\Longrightarrow s=\frac{2}{\omega^2-1}(t_2-t_0)$\\
Also, using the fact that $\omega\bar{s}+s=0$ and substituting into $0=t_1-t_2-\frac{\omega(\omega\bar{s}+s)}{2}$, we get $0=t_1-t_2+0$\\
$\Longrightarrow t_1=t_2$\\
Similarly, $0=t_0'-t_2-\frac{\omega(\omega\bar{s}+s)}{2}$ becomes $0=t_0'-t_2-0$\\
$\Longrightarrow t_0'=t_2$\\
Setting $Coeff_{tr(Q^n)}(z_0)=Coeff_{tr(Q^n)}(z_1)$ for $n=2,3$ simultaneously \\
$\Longrightarrow t_0=t_2$\\
$\Longrightarrow s=0$\\
$\Longrightarrow t_0=t_2=t_1=t_0'$ and $s=0$\\
$\Longrightarrow$ null solution.\\
Now, suppose no equation $E(z_i,z_j)$ (where $(i,j)\in\{(0,1),(0,2),(1,2)\}$) satisfies $A=D$ and $B=C$.  Then, for each $E(z_i,z_j)$, we have that either $A=C$ or $B=D$.\\
That is, either $t_0=t_1$ or $0=t_0-t_2+\frac{\omega(\bar{s}+\omega s)}{2}$, \\
and either $t_0=t_2$ or $0=t_1-t_2-\frac{\omega(\omega\bar{s}+s)}{2}$,\\
and either $t_1=t_2$ or $0=t_0-t_1-\frac{\bar{s}+s}{2}$.\\
Suppose $t_0=t_1$.  If $t_0=t_2$, then $t_0=t_1=t_2$.\\
Setting $Coeff_{tr(Q^n)}(z_1)=Coeff_{tr(Q^n)}(z_2)$ for $n=2,3$ simultaneously \\
$\Longrightarrow Re(s)=0$.\\
Setting $Coeff_{tr(Q^n)}(z_0)=Coeff_{tr(Q^n)}(z_1)$ for $n=2,3$ simultaneously \\
$\Longrightarrow$ either $s=0$ or ($s=0$ and $t_0=-\frac{1}{2}(-1+\sqrt{21})t_0'$)\\
If the latter case, then setting $Coeff_{tr(Q^n)}(c)=Coeff_{tr(Q^n)}(z_0)$ for $n=2$\\
$\Longrightarrow t_0'=0$\\
$\Longrightarrow$ null solution\\
So $s=0$ and $t_0=t_1=t_2$.\\
$\Longrightarrow$ done.\\
If $0=t_1-t_2-\frac{\omega(\omega\bar{s}+s)}{2}$, then consider the two possibilities for $E(z_1,z_2)$: either $t_1=t_2$ or $0=t_0-t_1-\frac{\bar{s}+s}{2}$.\\
If $t_1=t_2$, then $t_0=t_1=t_2$ and we're done by the previous case.\\
If $0=t_0-t_1-\frac{\bar{s}+s}{2}$, then $Re(s)=0$ because $t_0=t_1$.\\
Furthermore $t_2=t_1-\frac{\omega(\omega\bar{s}+s)}{2}$\\
$=t_0-\frac{\omega(\omega\bar{s}+s)}{2}$.\\
Setting $Coeff_{tr(Q^n)}(z_0)=Coeff_{tr(Q^n)}(z_2)$ for $n=2,3$ simultaneously \\
$\Longrightarrow s=0$ or ($t_0=-\frac{1}{10}(1+\sqrt{21})t_0'$ and $s=0$).\\
If the latter case, setting $Coeff_{tr(Q^n)}(c)=Coeff_{tr(Q^n)}(z_0)$ for $n=2$\\
$\Longrightarrow t_0'=0$\\
$\Longrightarrow$ null solution.\\
So $s=0$ and $t_0=t_1=t_2$.\\
$\Longrightarrow$ done.\\
Now suppose $0=t_0-t_2+\frac{\omega(\bar{s}+\omega s)}{2}$.\\
If $t_0=t_2$, then $\bar{s}+\omega s=0$.\\
We have either $t_1=t_2$ or $0=t_0-t_1-\frac{\bar{s}+s}{2}$.\\
If $t_1=t_2$, then $t_0=t_1=t_2$, and we're done.\\
If $0=t_0-t_1-\frac{\bar{s}+s}{2}$, then $0=t_0-t_1-\frac{(-\omega+1)}{2}s$\\
$\Longrightarrow s=\frac{2}{1-\omega}(t_0-t_1)$\\
So $t_0=t_2$ and $s=\frac{2}{1-\omega}(t_0-t_1)$.\\
Setting $Coeff_{tr(Q^n)}(z_0)=Coeff_{tr(Q^n)}(z_1)$ for $n=2,3$ simultaneously \\
$\Longrightarrow t_0=t_1$\\
$\Longrightarrow t_0=t_1=t_2$ and $s=0$\\
$\Longrightarrow$ done.\\
If $0=t_1-t_2-\frac{\omega(\omega\bar{s}+s)}{2}$, then consider the two possibilities for $E(z_1,z_2)$: Either $t_1=t_2$ or $0=t_0-t_1-\frac{\bar{s}+s}{2}$.\\
If $t_1=t_2$, then $\omega\bar{s}+s=0$\\
$\Longrightarrow \bar{s}=-\omega^2 s$\\
But $0=t_0-t_2+\frac{\omega(\bar{s}+\omega s)}{2}$\\
$\Longrightarrow 0=t_0-t_2+\frac{\omega(-\omega^2+\omega)}{2}s$\\
$\Longrightarrow 0=t_0-t_2+\frac{(\omega^2-1)}{2}s$\\
$\Longrightarrow s=\frac{2}{\omega^2-1}(t_2-t_0)$\\
Setting $Coeff_{tr(Q^n)}(z_0)=Coeff_{tr(Q^n)}(z_1)$ for $n=2,3$ simultaneously \\
$\Longrightarrow t_0=t_1$\\
$\Longrightarrow t_0=t_1=t_2$ and $s=0$\\
$\Longrightarrow$ done.\\
If $0=t_0-t_1-\frac{\bar{s}+s}{2}$, then $t_1=t_0-Re(s)$ and $t_2=t_1-\frac{\omega(\omega\bar{s}+s)}{2}$.\\
Setting $Coeff_{tr(Q^n)}(z_0)=Coeff_{tr(Q^n)}(z_1)$ for $n=2,3$ simultaneously \\
$\Longrightarrow Re(s)=0$\\
$\Longrightarrow t_0=t_1$\\
Setting $Coeff_{tr(Q^n)}(z_0)=Coeff_{tr(Q^n)}(z_2)$ for $n=2,3$ simultaneously \\
$\Longrightarrow s=0$ or ($s=0$ and $t_0=-\frac{1}{10}(1+\sqrt{21})t_0'$)\\
If the latter, then setting $Coeff_{tr(Q^n)}(c)=Coeff_{tr(Q^n)}(z_0)$ for $n=2$ \\
$\Longrightarrow t_0'=0$\\
$\Longrightarrow$ null solution.\\
So $s=0$ and $t_0=t_1=t_2$\\
$\Longrightarrow$ done.\\
This concludes our proof that $t_0=t_1=t_2$ and $s=0$.\\

The matrices $N_1,N_2,N_3,N_4$ simplify as follows:

\[\left(
\begin{array}{ccccc}
 0 & \eta t & \eta t_0' & -\frac{\eta \theta }{\sqrt{3}}(t+t_0') & 0 \\
 \bar{\eta}t & -2t_0' & 0 & \frac{2\theta}{\sqrt{3}}(\omega t +t_0') & -\bar{\eta}\theta t \\
 \bar{\eta}t_0' & 0 & -2t & \frac{2\theta}{\sqrt{3}}(t+\omega t_0') & -\bar{\eta} \theta t_0' \\
 -\frac{\bar{\eta} \theta }{\sqrt{3}}(t+t_0') & \frac{2\theta}{\sqrt{3}}(\bar{\omega} t +t_0') & \frac{2\theta}{\sqrt{3}}(t+\bar{\omega} t_0') & 0 & \frac{\bar{\eta}\theta^2(t+t_0')}{\sqrt{3}} \\
 0 & -\eta \theta t & -\eta \theta t_0' & \frac{\eta \theta^2(t+t_0')}{\sqrt{3}} & 0 \\
\end{array}
\right),\]

\[\left(
\begin{array}{ccccc}
 -2t & \bar{\eta}t_0' & 0 & \frac{2\theta}{\sqrt{3}}(t+\omega t_0') & -\bar{\eta} \theta t_0' \\
 \eta t_0' & 0 & \eta t & -\frac{\eta \theta }{\sqrt{3}}(t+t_0') & 0 \\
 0 & \bar{\eta}t & -2t_0' & \frac{2\theta}{\sqrt{3}}(t_0'+\omega t) & -\bar{\eta}\theta t \\
 \frac{2\theta}{\sqrt{3}}(t+\bar{\omega} t_0') & -\frac{\bar{\eta} \theta }{\sqrt{3}}(t+t_0') & \frac{2\theta}{\sqrt{3}}(t_0'+\bar{\omega} t) & 0 & \frac{\bar{\eta} \theta^2 }{\sqrt{3}}(t+t_0') \\
 -\eta \theta t_0' & 0 & -\eta \theta t & \frac{\eta \theta^2 }{\sqrt{3}}(t+t_0') & 0 \\
\end{array}
\right),\]

\[\left(
\begin{array}{ccccc}
 -2t_0' & 0 & \bar{\eta} t & \frac{2\theta}{\sqrt{3}}(t_0'+\omega t) & -\bar{\eta}\theta t \\
 0 & -2t & \bar{\eta}t_0' & \frac{2\theta}{\sqrt{3}}(\omega t_0'+t) & -\bar{\eta}\theta t_0' \\
 \eta t & \eta t_0' & 0 & -\frac{\eta \theta }{\sqrt{3}}(t+t_0') & 0 \\
 \frac{2\theta}{\sqrt{3}}(t_0'+\bar{\omega} t) & \frac{2\theta}{\sqrt{3}}(\bar{\omega} t_0'+t) & -\frac{\bar{\eta} \theta }{\sqrt{3}}(t+t_0') & 0 & \frac{\bar{\eta}\theta^2(t+t_0')}{\sqrt{3}} \\
 -\eta \theta t & -\eta \theta t_0' & 0 & \frac{\eta \theta^2(t+t_0')}{\sqrt{3}} & 0 \\
\end{array}
\right),\]

\[\left(
\begin{array}{cccc}
 \frac{2\theta^2}{3}(t+t_0') & -\frac{\bar{\eta}\theta^2}{3}(t_0'+\omega t) & -\frac{\bar{\eta}\theta^2}{3}(t+\omega t_0') & 0 \\
 -\frac{\eta \theta^2}{3}(t_0'+\bar{\omega} t) & \frac{2\theta^2}{3}(t+t_0') & -\frac{\bar{\eta}\theta^2}{3}(t_0'+\omega t) & 0 \\
 -\frac{\eta \theta^2}{3}(t+\bar{\omega} t_0') & -\frac{\eta\theta^2}{3}(t_0'+\bar{\omega} t) &  \frac{2\theta^2}{3}(t+t_0') & 0 \\
 0 & 0 & 0 & 0 \\
\end{array}
\right).\]

One can calculate $Tr(N_1^2)= 12t^2+12(t_0')^2+16\theta^2t^2+16\theta^2 (t_0')^2+\frac{8}{3}\theta^4(t+t_0')^2$. One finds that $Tr(N_2^2)=Tr(N_3^2)=Tr(N_1^2)$, and $Tr(N_4^2)=4\theta^4(t^2+(t_0')^2)$.  Hence, $Coeff_{tr(Q^2)}(c)=\big(12t^2+12(t_0')^2+20\theta^2t^2+20\theta^2(t_0')^2+\frac{8}{3}\theta^4(t+t_0')^2\big)\sqrt{3}$ and $Coeff_{tr(Q^2)}(z_0)=\big(\frac{4}{3}\theta^6(t^2+(t_0')^2)+\frac{4}{9}\theta^8(t+t_0')^2\big)\sqrt{3}$.
Setting $Coeff_{tr(Q^2)}(c)=Coeff_{tr(Q^2)}(z_0),$
we get a quadratic equation in the variables $t$ and $t_0'$:
\[t^2+(t_0')^2+\Big(\frac{1+\sqrt{21}}{2}\Big)t\cdot t_0' =0.\]  By the quadratic formula, we get $t_0'=\Big(\frac{1-\theta^2\pm \theta}{2}\Big)t$.

In sum, we have the matrix representation for $Q$:
\begin{align*}
\widetilde{Q}= &
\left(
\begin{array}{ccccc}
 0 & \eta t & \eta t_0' & -\frac{\eta \theta }{\sqrt{3}}(t+t_0') & 0 \\
 \bar{\eta}t & -2t_0' & 0 & \frac{2\theta}{\sqrt{3}}(\omega t +t_0') & -\bar{\eta}\theta t \\
 \bar{\eta}t_0' & 0 & -2t & \frac{2\theta}{\sqrt{3}}(t+\omega t_0') & -\bar{\eta} \theta t_0' \\
 -\frac{\bar{\eta} \theta }{\sqrt{3}}(t+t_0') & \frac{2\theta}{\sqrt{3}}(\bar{\omega} t +t_0') & \frac{2\theta}{\sqrt{3}}(t+\bar{\omega} t_0') & 0 & \frac{\bar{\eta}\theta^2(t+t_0')}{\sqrt{3}} \\
 0 & -\eta \theta t & -\eta \theta t_0' & \frac{\eta \theta^2(t+t_0')}{\sqrt{3}} & 0 \\
\end{array}\right)
\\
& \oplus
\left(
\begin{array}{ccccc}
 -2t & \bar{\eta}t_0' & 0 & \frac{2\theta}{\sqrt{3}}(t+\omega t_0') & -\bar{\eta} \theta t_0' \\
 \eta t_0' & 0 & \eta t & -\frac{\eta \theta }{\sqrt{3}}(t+t_0') & 0 \\
 0 & \bar{\eta}t & -2t_0' & \frac{2\theta}{\sqrt{3}}(t_0'+\omega t) & -\bar{\eta}\theta t \\
 \frac{2\theta}{\sqrt{3}}(t+\bar{\omega} t_0') & -\frac{\bar{\eta} \theta }{\sqrt{3}}(t+t_0') & \frac{2\theta}{\sqrt{3}}(t_0'+\bar{\omega} t) & 0 & \frac{\bar{\eta} \theta^2 }{\sqrt{3}}(t+t_0') \\
 -\eta \theta t_0' & 0 & -\eta \theta t & \frac{\eta \theta^2 }{\sqrt{3}}(t+t_0') & 0 \\
\end{array}\right)
\\
& \oplus 
\left(
\begin{array}{ccccc}
 -2t_0' & 0 & \bar{\eta} t & \frac{2\theta}{\sqrt{3}}(t_0'+\omega t) & -\bar{\eta}\theta t \\
 0 & -2t & \bar{\eta}t_0' & \frac{2\theta}{\sqrt{3}}(\omega t_0'+t) & -\bar{\eta}\theta t_0' \\
 \eta t & \eta t_0' & 0 & -\frac{\eta \theta }{\sqrt{3}}(t+t_0') & 0 \\
 \frac{2\theta}{\sqrt{3}}(t_0'+\bar{\omega} t) & \frac{2\theta}{\sqrt{3}}(\bar{\omega} t_0'+t) & -\frac{\bar{\eta} \theta }{\sqrt{3}}(t+t_0') & 0 & \frac{\bar{\eta}\theta^2(t+t_0')}{\sqrt{3}} \\
 -\eta \theta t & -\eta \theta t_0' & 0 & \frac{\eta \theta^2(t+t_0')}{\sqrt{3}} & 0 \\
\end{array}\right)
\\
& \oplus 
\left(
\begin{array}{cccc}
 \frac{2\theta^2}{3}(t+t_0') & -\frac{\bar{\eta}\theta^2}{3}(t_0'+\omega t) & -\frac{\bar{\eta}\theta^2}{3}(t+\omega t_0') & 0 \\
 -\frac{\eta \theta^2}{3}(t_0'+\bar{\omega} t) & \frac{2\theta^2}{3}(t+t_0') & -\frac{\bar{\eta}\theta^2}{3}(t_0'+\omega t) & 0 \\
 -\frac{\eta \theta^2}{3}(t+\bar{\omega} t_0') & -\frac{\eta\theta^2}{3}(t_0'+\bar{\omega} t) &  \frac{2\theta^2}{3}(t+t_0') & 0 \\
 0 & 0 & 0 & 0 \\
\end{array}\right)
\\
& \oplus 
\biggl(
\left(
\begin{array}{cc}
 0 & 0  \\
 0 & 0  \\
\end{array}\right) 
 \oplus 
 \left(
\begin{array}{c}
2\theta^2t \\
\end{array}\right)
\oplus 
\left(
\begin{array}{c}
 2\theta^2 t_0' \\
\end{array}\right)
\oplus
\left(
\begin{array}{c}
 -\frac{2}{3}\theta^4(t+t_0')\\
\end{array}\right) 
\biggr)^{\otimes 3}
\end{align*}
From the above representations for $T$ and $Q$, we compute the following trace values for $T$:
\begin{itemize}
\item $Z(T)=0$
\item $Z(T^2)=[4]$
\item $Z(T^3)=\frac{-3\sqrt{3}+\sqrt{7}}{6}$
\item $Z(T^4)=\frac{5}{3}\sqrt{\frac{2}{3}(23+5\sqrt{21})},$
\end{itemize}
and the following trace values for $Q$:
\begin{itemize}
\item $Z(Q)=0$
\item $Z(Q^2)=[4]$
\item $Z(Q^3)=-\frac{2}{3}\sqrt{8+3\sqrt{21}}$
\item $Z(Q^4)=\frac{2}{3}\sqrt{\frac{253}{3}+16\sqrt{21}},$
\end{itemize}
and the following trace values for mixtures of $T$ and $Q$:
\begin{itemize}
\item $Z(TQ)=0$
\item $Z(T^2Q)=\sqrt{\frac{151}{18}+\frac{41}{6}\sqrt{\frac{7}{3}}}$
\item $Z(TQ^2)=\frac{1}{3}(2\sqrt{3}+\sqrt{7})$
\item $Z(T^2Q^2)=\frac{7}{3}\sqrt{\frac{1}{6}(11+\sqrt{21})}.$
\end{itemize}

\section{Calculating Traces for $\rho^{\frac{1}{2}}(T)$ and $\rho^{\frac{1}{2}}(Q)$}

Now, we want to calculate the traces for $\rho^{\frac{1}{2}}(T)$ and $\rho^{\frac{1}{2}}(Q)$
We order the set of pairs of vertices lexicographically, with the ``alphabet'' ordered 
$b_0 < b_1 <  b_2 < d$ for the odd vertices $b_0,b_1,b_2,d$.
The rows of $M_{i_{b_0,v}}(\mathbb{C})$ are labelled

\[\left(
\begin{array}{cccc}
b_0 c b_0 c  \\
b_0 c b_1 c  \\
b_0 c b_2 c \\
b_0 c d c \\
b_0 z_0 b_0 c \\
\end{array} \right),
\left(
\begin{array}{cccc}
b_0 c b_0 z_0  \\
b_0 z_0 b_0 z_0  \\
\end{array} \right),
\left(
\begin{array}{cccc}
b_0 c b_1 z_1  \\
\end{array} \right),
\left(
\begin{array}{cccc}
b_0 c b_2 z_2  \\
\end{array} \right). \]

Similarly for $z_1$ and $z_2$ ; the labels of the rows of $M_{i_{b_1,v}}(\mathbb{C})$ are as above, but with the permutation $(012)$ applied to the indices of $z$; the labels of the rows of $M_{i_{b_2,v}}(\mathbb{C})$ are also as above, but with the permutation $(021)$ applied to the indices. 

The rows of $M_{i_{d,v}}(\mathbb{C})$ are labelled

\[\left(
\begin{array}{cccc}
d c d c  \\
d c b_0 c  \\
d c b_1 c \\
d c b_2 c \\
\end{array} \right),
\left(
\begin{array}{cccc}
d c b_0 z_0  \\
\end{array} \right),
\left(
\begin{array}{cccc}
d c b_1 z_1  \\
\end{array} \right),
\left(
\begin{array}{cccc}
d c b_2 z_2  \\
\end{array} \right). \]

To find the entries in the matrices for $\rho^{\frac{1}{2}}(T)$ and $\rho^{\frac{1}{2}}(Q)$, we will take advantage of our knowledge of the entries in the matrices for $T$ and $Q$.  First, we observe that for a length-6 path $p_1p_2p_3p_4p_5p_6$,  \[p_1p_2p_3p_4p_5p_6 = \rho^{\frac{1}{2}}\left (\left(\frac{1}{\sqrt{\frac{\lambda(p_2)}{\lambda(p_1}\frac{\lambda(p_5)}{\lambda(p_4)}}}\right)(p_2p_3p_4p_5p_6p_1)\right)\]
\[\Longrightarrow Coeff_{\rho^{\frac{1}{2}}(T)}(p_1p_2p_3p_4p_5p_6)= \left(\frac{1}{\sqrt{\frac{\lambda(p_1)}{\lambda(p_2}\frac{\lambda(p_4)}{\lambda(p_5)}}}\right)\cdot Coeff_{T}((p_2p_3p_4p_5p_6p_1).\]
Thus, the entry (for $\rho^{\frac{1}{2}}(T)$) with row labelled by $p_1p_2p_3p_4$ and column labelled by $p_4p_5p_6p_1$ is given by $\frac{1}{\sqrt{\frac{\lambda(p_1)}{\lambda(p_2}\frac{\lambda(p_4)}{\lambda(p_5)}}}$ times the entry (for $T$) with row labelled by $p_2p_3p_4p_5$ and column labelled by $p_5p_6p_1p_2$.   
We have the following matrix, which is the same for $M_{i_{b_0,c}}(\mathbb{C})$, $M_{i_{b_1,c}}(\mathbb{C})$, and $M_{i_{b_2,c}}(\mathbb{C})$ :

\[\left(
\begin{array}{ccccc}
 0 & t & t_0' & -\frac{(t+t_0')}{R^2}\cdot \frac{\sqrt{3}}{\theta} & 0 \\
 t & t_0' & s & -\frac{(t+t_0'+s)}{R^2}\cdot \frac{\sqrt{3}}{\theta} & -\theta t \\
 t_0' & \bar{s} & t & \frac{2(t+t_0')+s}{R^2}\cdot \frac{\sqrt{3}}{\theta} & -\theta t_0' \\
 -\frac{(t+t_0')}{R} & \frac{2(t+t_0')+s}{R} & -\frac{(t+t_0'+s)}{R} & 0 & \frac{\theta(t+t_0')}{R} \\
 0 & -\theta^2 t\cdot \frac{1}{\theta} & -\theta^2 t_0'\cdot \frac{1}{\theta} & \frac{\theta^2}{R^2}(t+t_0')\cdot \frac{\sqrt{3}}{\theta^2} & 0 \\
\end{array}
\right).\]
We have the following matrix, which is the same for $M_{i_{b_0,z_0}}(\mathbb{C})$, $M_{i_{b_1,z_1}}(\mathbb{C})$, and $M_{i_{b_2,z_2}}(\mathbb{C})$:
\[\left(
\begin{array}{cc}
	0 & 0 \\
	0 & 0 \\
\end{array}
\right).\]

\noindent We have the following matrix, which is the same for $M_{i_{b_0,z_1}}(\mathbb{C})$, $M_{i_{b_1,z_2}}(\mathbb{C})$, and $M_{i_{b_2,z_0}}(\mathbb{C})$:
\[\left(
\begin{array}{c}
	-t_0'\cdot \theta^2 \\
\end{array}
\right).\]

\noindent We have the following matrix, which is the same for $M_{i_{b_0,z_2}}(\mathbb{C})$, $M_{i_{b_1,z_0}}(\mathbb{C})$, and $M_{i_{b_2,z_1}}(\mathbb{C})$:
\[\left(
\begin{array}{c}
	-t\cdot \theta^2 \\
\end{array}
\right).\]
Now, for paths beginning at $d$, we have the following matrix for $M_{i_{d,c}}(\mathbb{C})$:
\[\left(
\begin{array}{cccc}
	0 & 0 & 0 & 0\\
	0 & -\frac{(t+t_0')}{R}\cdot \frac{1}{R} & -\frac{(t+t_0'+s)}{R}\cdot \frac{1}{R} & \frac{2(t+t_0')+s}{R}\cdot \frac{1}{R} \\
	
	0 & \frac{2(t+t_0')+s}{R}\cdot \frac{1}{R} & -\frac{(t+t_0')}{R}\cdot \frac{1}{R} & -\frac{(t+t_0'+s)}{R}\cdot \frac{1}{R} \\
	0 & -\frac{(t+t_0'+s)}{R}\cdot \frac{1}{R} & \frac{2(t+t_0')+s}{R}\cdot \frac{1}{R} & -\frac{(t+t_0')}{R}\cdot \frac{1}{R}\\
\end{array}
\right).\]
We have the following matrix, which is the same for $M_{i_{d,z_0}}(\mathbb{C})$, $M_{i_{d,z_1}}(\mathbb{C})$, and $M_{i_{d,z_2}}(\mathbb{C})$ :
\[\left(
\begin{array}{c}
	\frac{\theta (t+t_0')}{R}\cdot \frac{\theta}{R} \\
\end{array}
\right).\]

In sum, we have the matrix representation for $\rho^{\frac{1}{2}}(T)$:
\begin{align*}
\widetilde{\rho^{\frac{1}{2}}(T)}= &
\left(
\begin{array}{cccc}
 	0 & 0 & 0 & 0\\
	0 & -\frac{(t+t_0')}{R}\cdot \frac{1}{R} & -\frac{(t+t_0'+s)}{R}\cdot \frac{1}{R} & \frac{2(t+t_0')+s}{R}\cdot \frac{1}{R} \\
	0 & \frac{2(t+t_0')+s}{R}\cdot \frac{1}{R} & -\frac{(t+t_0')}{R}\cdot \frac{1}{R} & -\frac{(t+t_0'+s)}{R}\cdot \frac{1}{R} \\
	0 & -\frac{(t+t_0'+s)}{R}\cdot \frac{1}{R} & \frac{2(t+t_0')+s}{R}\cdot \frac{1}{R} & -\frac{(t+t_0')}{R}\cdot \frac{1}{R}\\
\end{array}\right)
\\
& \oplus
\biggl(
\left(
\begin{array}{c}
	\frac{\theta (t+t_0')}{R}\cdot \frac{\theta}{R} \\
\end{array}\right) 
\biggr)^{\otimes 3}
\\
& \oplus 
\left(
\begin{array}{ccccc}
 0 & t & t_0' & -\frac{(t+t_0')}{R^2}\cdot \frac{\sqrt{3}}{\theta} & 0 \\
 t & t_0' & s & -\frac{(t+t_0'+s)}{R^2}\cdot \frac{\sqrt{3}}{\theta} & -\theta t \\
 t_0' & \bar{s} & t & \frac{2(t+t_0')+s}{R^2}\cdot \frac{\sqrt{3}}{\theta} & -\theta t_0' \\
 -\frac{(t+t_0')}{R} & \frac{2(t+t_0')+s}{R} & -\frac{(t+t_0'+s)}{R} & 0 & \frac{\theta(t+t_0')}{R} \\
 0 & -\theta^2 t\cdot \frac{1}{\theta} & -\theta^2 t_0'\cdot \frac{1}{\theta} & \frac{\theta^2}{R^2}(t+t_0')\cdot \frac{\sqrt{3}}{\theta^2} & 0 \\
\end{array}\right)^{\otimes 3} 
\\
& \oplus 
\biggl( \left(
\begin{array}{cc}
	0 & 0 \\
	0 & 0 \\
\end{array}\right)
\oplus 
\left(
\begin{array}{c}
	-t_0'\cdot \theta^2 \\
\end{array}\right)
\oplus
\left(
\begin{array}{c}
	-t\cdot \theta^2 \\
\end{array}\right) 
\biggr)^{\otimes 3},
\end{align*}
where $$t=\frac{1}{12}\big(-3+\sqrt{21}+\sqrt{-114+26\sqrt{21}}\big),$$ $$t_0'=\frac{1}{4}\big(-7-\sqrt{21}+\sqrt{54+14\sqrt{21}}\big)t,$$ and $$s=\frac{1}{4}\big(3-\sqrt{21}\big)+\frac{1}{2}\sqrt{\frac{-57+13\sqrt{21}}{6}}\cdot i.$$
We find the matrix representation for $\rho^{\frac{1}{2}}(Q)$ in a similar manner.
We have the following matrix, which is the same for $M_{i_{b_0,c}}(\mathbb{C})$, $M_{i_{b_1,c}}(\mathbb{C})$, and $M_{i_{b_2,c}}(\mathbb{C})$:

\[\left(
\begin{array}{ccccc}
 0 & -2t & -2t_0' & \frac{2\theta^2(t+t_0')}{3}\cdot \frac{\sqrt{3}}{\theta} & 0 \\
 \omega^2\eta t & \eta t_0' & 0 & -\frac{\eta \theta^2(t_0'+\omega^2 t)}{3}\cdot \frac{\sqrt{3}}{\theta} & -\omega^2 \eta \theta t \\
 \omega^2 \eta t_0' & 0 & \eta t & -\frac{\eta \theta^2(t+\omega^2t_0')}{3}\cdot \frac{\sqrt{3}}{\theta} & -\omega^2 \eta \theta t_0' \\
 -\frac{\omega^2 \eta \theta(t+t_0')}{\sqrt{3}} & \frac{2\theta(t+\omega^2 t_0')}{\sqrt{3}} & \frac{2\theta (t_0'+\omega^2 t)}{\sqrt{3}} & 0 & \frac{\omega^2 \eta \theta^2(t+t_0')}{\sqrt{3}} \\
 0 & 2\theta^2 t\cdot \frac{1}{\theta} & 2\theta^2 t_0'\cdot \frac{1}{\theta} & -\frac{2 \theta^4(t+t_0')}{3}\cdot \frac{\sqrt{3}}{\theta^2} & 0 \\
\end{array}
\right).\]
\noindent We have the following matrix, which is the same for $M_{i_{b_0,z_0}}(\mathbb{C})$, $M_{i_{b_1,z_1}}(\mathbb{C})$, and $M_{i_{b_2,z_2}}(\mathbb{C})$:
\[\left(
\begin{array}{cc}
	0 & 0 \\
	0 & 0 \\
\end{array}
\right).\]

\noindent We have the following matrix, which is the same for $M_{i_{b_0,z_1}}(\mathbb{C})$, $M_{i_{b_1,z_2}}(\mathbb{C})$, and $M_{i_{b_2,z_0}}(\mathbb{C})$:
\[\left(
\begin{array}{c}
	-\eta t_0'\cdot \theta^2 \\
\end{array}
\right).\]
We have the following matrix, which is the same for $M_{i_{b_0,z_2}}(\mathbb{C})$, $M_{i_{b_1,z_0}}(\mathbb{C})$, and $M_{i_{b_2,z_1}}(\mathbb{C})$:
\[\left(
\begin{array}{c}
	-\eta t\cdot \theta^2 \\
\end{array}
\right).\]

Now, for paths beginning at $d$, we have the following matrix for $M_{i_{d,c}}(\mathbb{C})$:
\[\left(
\begin{array}{cccc}
	0 & 0 & 0 & 0\\
	0 & -\frac{\eta \theta(t+t_0')}{\sqrt{3}}\cdot \frac{1}{R} & \frac{2\theta(t+\omega t_0')}{\sqrt{3}}\cdot \frac{1}{R} & \frac{2\theta(t_0'+\omega t)}{\sqrt{3}}\cdot \frac{1}{R} \\
	
	0 & \frac{2\theta(t_0'+\omega t)}{\sqrt{3}}\cdot \frac{1}{R} & -\frac{\eta \theta(t+t_0')}{\sqrt{3}}\cdot \frac{1}{R} & \frac{2\theta(t+\omega t_0')}{\sqrt{3}}\cdot \frac{1}{R} \\
	0 & \frac{2\theta(t+\omega t_0')}{\sqrt{3}}\cdot \frac{1}{R} & \frac{2\theta(t_0'+\omega t)}{\sqrt{3}}\cdot \frac{1}{R} & -\frac{\eta \theta(t+t_0')}{\sqrt{3}}\cdot \frac{1}{R}\\
\end{array}
\right).\]
We have the following matrix, which is the same for $M_{i_{d,z_0}}(\mathbb{C})$, $M_{i_{d,z_1}}(\mathbb{C})$, and $M_{i_{d,z_2}}(\mathbb{C})$ :
\[\left(
\begin{array}{c}
	\frac{\eta \theta^2 (t+t_0')}{\sqrt{3}}\cdot \frac{\theta}{R} \\
\end{array}
\right).\]

In sum, we have the matrix representation for $\rho^{\frac{1}{2}}(Q)$:
\begin{align*}
\widetilde{\rho^{\frac{1}{2}}(Q)}= &
\left(
\begin{array}{cccc}
	0 & 0 & 0 & 0\\
	0 & -\frac{\eta \theta(t+t_0')}{\sqrt{3}}\frac{1}{R} & \frac{2\theta(t+\omega t_0')}{\sqrt{3}} \frac{1}{R} & \frac{2\theta(t_0'+\omega t)}{\sqrt{3}} \frac{1}{R} \\
	0 & \frac{2\theta(t_0'+\omega t)}{\sqrt{3}} \frac{1}{R} & -\frac{\eta \theta(t+t_0')}{\sqrt{3}}\frac{1}{R} & \frac{2\theta(t+\omega t_0')}{\sqrt{3}} \frac{1}{R} \\
	0 & \frac{2\theta(t+\omega t_0')}{\sqrt{3}} \frac{1}{R} & \frac{2\theta(t_0'+\omega t)}{\sqrt{3}} \frac{1}{R} & -\frac{\eta \theta(t+t_0')}{\sqrt{3}}\frac{1}{R}\\
\end{array}\right)\\
& \oplus
\biggl(
\left(
\begin{array}{c}
	\frac{\eta \theta^2 (t+t_0')}{\sqrt{3}}\cdot \frac{\theta}{R} \\
\end{array}\right) 
\biggr)^{\otimes 3}
\\
& \oplus 
\left(
\begin{array}{ccccc}
 0 & -2t & -2t_0' & \frac{2\theta^2(t+t_0')}{3}\cdot \frac{\sqrt{3}}{\theta} & 0\\
 \omega^2\eta t & \eta t_0' & 0 & -\frac{\eta \theta^2(t_0'+\omega^2 t)}{3}\cdot \frac{\sqrt{3}}{\theta} & -\omega^2 \eta \theta t\\
 \omega^2 \eta t_0' & 0 & \eta t & -\frac{\eta \theta^2(t+\omega^2t_0')}{3}\cdot \frac{\sqrt{3}}{\theta} & -\omega^2 \eta \theta t_0'\\
 -\frac{\omega^2 \eta \theta(t+t_0')}{\sqrt{3}} & \frac{2\theta(t+\omega^2 t_0')}{\sqrt{3}} & \frac{2\theta (t_0'+\omega^2 t)}{\sqrt{3}} & 0 & \frac{\omega^2 \eta \theta^2(t+t_0')}{\sqrt{3}}\\
 0 & 2\theta^2 t\cdot \frac{1}{\theta} & 2\theta^2 t_0'\cdot \frac{1}{\theta} & -\frac{2 \theta^4(t+t_0')}{3}\cdot \frac{\sqrt{3}}{\theta^2} & 0\\
\end{array}\right)^{\otimes 3}\\
& \oplus 
\biggl(
 \left(
\begin{array}{cc}
	0 & 0 \\
	0 & 0 \\
\end{array}\right)
\oplus 
\left(
\begin{array}{c}
	-\eta t_0'\cdot \theta^2 \\
\end{array}\right)
\oplus
\left(
\begin{array}{c}
	-\eta t\cdot \theta^2 \\
\end{array}\right) 
\biggr)^{\otimes 3},
\end{align*}
where $$t=\sqrt{\frac{[4]}{1+\Psi^2}\frac{5}{(246+54\sqrt{21})\sqrt{3}}},$$
 $$\Psi=\frac{1-\theta^2-\theta}{2},$$ and $$t_0'=\Psi t.$$

From the above representations for $\rho ^{\frac{1}{2}}(T)$ and $\rho^{\frac{1}{2}}(Q)$, we compute the following trace values for  $\rho ^{\frac{1}{2}}(T)$:
\begin{itemize}
\item $Z(\rho ^{\frac{1}{2}}(T))=0$
\item $Z((\rho ^{\frac{1}{2}}(T))^2)=[4]$
\item $Z((\rho ^{\frac{1}{2}}(T))^3)=\frac{1}{6}(-3\sqrt{3}+\sqrt{7})=Z(T^3)$
\item $Z((\rho ^{\frac{1}{2}}(T))^4)=\frac{5}{3}\sqrt{\frac{2}{3}(23+5\sqrt{21})}=Z(T^4),$
\end{itemize}
and the following trace values for $\rho ^{\frac{1}{2}}(Q)$:
\begin{itemize}
\item $Z(\rho ^{\frac{1}{2}}(Q))=0$
\item $Z((\rho ^{\frac{1}{2}}(Q))^2)=\omega [4]$
\item $Z((\rho ^{\frac{1}{2}}(Q))^3)=-\frac{2}{3}\sqrt{8+3\sqrt{21}}=Z(Q^3)$
\item $Z((\rho ^{\frac{1}{2}}(Q))^4)=-\frac{i(-i+\sqrt{3})(8+3\sqrt{21})}{3\sqrt{3}}=\omega^2 Z(Q^4),$
\end{itemize}
and the following trace values for mixtures of $\rho ^{\frac{1}{2}}(T)$ and $\rho^{\frac{1}{2}}(Q)$:
\begin{itemize}
\item $Z(\rho ^{\frac{1}{2}}(T)\rho^{\frac{1}{2}}(Q))=0$
\item $Z((\rho ^{\frac{1}{2}}(T)(\rho^{\frac{1}{2}}(Q))^2)=\frac{i(i+\sqrt{3})(697+152\sqrt{21})^{\frac{1}{4}}}{6}=\omega Z(TQ^2)$
\item $Z((\rho ^{\frac{1}{2}}(Q)(\rho^{\frac{1}{2}}(T))^2)=-\frac{1}{12}\sqrt{302+82\sqrt{21}}-\frac{1}{12}\sqrt{906+246\sqrt{21}}\cdot i= \omega^2 Z(QT^2)$
\item $Z((\rho ^{\frac{1}{2}}(T))^2(\rho^{\frac{1}{2}}(Q))^2)=\frac{7i(i+\sqrt{3})(1+\sqrt{21})}{12\sqrt{3}}=\omega Z(T^2Q^2).$
\end{itemize}


\end{document}

%% file: TikzStyles.tex
\tikzstyle{shaded}=[fill=red!10!blue!20!gray!30!white]
\tikzstyle{shaded line}=[double=red!10!blue!20!gray!30!white, double distance=1.5mm, draw=black]
\tikzstyle{unshaded}=[fill=white]
\tikzstyle{unshaded line}=[double=white, double distance=1.5mm, draw=black]
\tikzstyle{Tbox}=[circle, draw, thick, fill=white, opaque,]
\tikzstyle{empty box}=[circle, draw, thick, fill=white, opaque, inner sep=2mm]
\tikzstyle{background rectangle}= [fill=red!10!blue!20!gray!40!white,rounded corners=2mm] 
\tikzstyle{on}=[very thick, red!50!blue!50!black]
\tikzstyle{off}=[gray]

\tikzstyle{traces}=[scale=.2, inner sep=1mm]
\tikzstyle{quadratic}=[scale=.4, inner sep=1mm, baseline]
\tikzstyle{annular}=[scale=.7, inner sep=1mm, baseline]
\tikzstyle{make triple edge size}= [scale=.4, inner sep=1mm,baseline] 
\tikzstyle{icosahedron network}=[scale=.3, inner sep=1mm, baseline]
\tikzstyle{ATLsix}=[scale=.25, baseline]
\tikzstyle{TL12}=[scale=.15,baseline]
\tikzstyle{PAdefn}=[scale=.7,baseline]
\tikzstyle{TLEG}=[scale=.5,baseline]